\documentclass[]{preprint}
\usepackage[hidelinks,draft=false]{hyperref}
\usepackage{times}
\usepackage{shuffle}
\usepackage{booktabs}
\usepackage{microtype}
\usepackage[final]{graphicx}
\usepackage{amsmath}
\makeatletter
\let\over\@@over
\let\atop\@@atop
\makeatother
\usepackage{amssymb}
\usepackage{mathrsfs}
\usepackage{mhequ}
\usepackage{mhenvs}
\usepackage{mhfig}
\usepackage{mhsymb}

\newtheorem{assumption}[lemma]{Assumption}
\newtheorem{example}[lemma]{Example}
\newtheorem{metatheorem}[lemma]{Metatheorem}

\definecolor{darkred}{rgb}{0.9,0.1,0.1}

\newop{diag}

\newcommand\minus{%
  \setbox0=\hbox{-}%
  \vcenter{%
    \hrule width\wd0 height \the\fontdimen8\textfont3%
  }%
}

\def\PR{\boldsymbol{R}}
\def\M{\mathfrak{M}}
\def\RR{\mathfrak{R}}
\def\Ren{\mathscr{R}}
\def\DD{\mathscr{D}}
\def\D{\mathscr{D}}
\def\F{\mathscr{F}}
\def\XX{{\boldsymbol X}}

\def\E{\mathbf{E}}
\def\P{\mathbf{P}}
\def\/{\,\rule[-0.25em]{2pt}{1em}\,}
\def\bstar{\mathbin{\bar \star}}

\def\reg{\zeta}
\def\m{\mathfrak{m}}
\def\n#1{\lfloor \hspace{-0.29em} \rceil #1 \lfloor \hspace{-0.29em} \rceil}

\def\LL{\mathfrak{L}}
\def\Aut{\mop{Aut}}

\def\FF{\mathfrak{F}}
\def\GG{\mathscr{G}}
\def\TT{\mathscr{T}}

\def\T{\mathbf{T}}
\def\s{\mathfrak{s}}
\def\lspan{\mathop{\mathrm{span}}}
\def\Alg{\mathbin{\mathrm{Alg}}}

\def\Lip{\mathcal{C}}
\def\MM{\mathscr{M}}
\def\PP{\mathscr{P}}
\def\SS{\mathscr{S}}
\def\${|\!|\!|}

\def\g{\mathfrak{g}}
\def\m{\mathfrak{m}}
\def\DeltaM{\Delta^{\!M}}
\def\Deltap{\Delta^{\!+}}
\def\hDeltaM{\hat \Delta^{\!M}}
\def\K{\mathfrak{K}}
\def\ls{\lesssim}
\def\PPi{\boldsymbol{\Pi}}

\def\Dom{D}

\def\MHFigArg#1#2{\def\arg{#1}\mhpastefig{#2}}

\begin{document}

\title{A theory of regularity structures}
\author{M.~Hairer}
\institute{Mathematics Department, University of Warwick
 \\ \email{M.Hairer@Warwick.ac.uk}}
\titleindent=0.65cm

\maketitle
\thispagestyle{empty}

\begin{abstract}
We introduce a new notion of ``regularity structure'' that provides an 
algebraic framework allowing to describe functions and / or distributions
via a kind of ``jet'' or local Taylor expansion around each point. The main novel idea is to replace the 
classical polynomial model which is suitable for describing smooth functions
by arbitrary models that are purpose-built for the problem at hand. 
In particular, this allows to describe the local behaviour not only of functions
but also of large classes of distributions. 

We then build a calculus allowing to perform the various operations
(multiplication, composition with smooth functions, integration against
singular kernels) necessary to formulate fixed point equations for 
a very large class of semilinear PDEs driven by some very singular (typically random)
input. This allows, for the first time, 
to give a mathematically rigorous meaning to many interesting stochastic PDEs
arising in physics.
The theory comes with convergence results that allow to interpret the
solutions obtained in this way as limits of classical solutions to 
regularised problems, possibly modified by the addition of diverging counterterms.
These counterterms arise naturally through the action of a ``renormalisation group''
which is defined canonically in terms of the regularity structure associated to
the given class of PDEs. 

Our theory also allows to easily recover many existing results on
singular stochastic PDEs (KPZ equation, stochastic quantisation equations, Burgers-type 
equations) and to understand them as particular instances of a 
unified framework. One surprising insight is that in all of these instances local solutions
are actually ``smooth'' in the sense that they can be approximated locally to arbitrarily high
degree as linear combinations of a fixed family of random functions / distributions that play
the role of ``polynomials'' in the theory.

As an example of a novel application, we solve the long-standing problem
of building a natural Markov process that is symmetric with respect to the
(finite volume) measure describing the $\Phi^4_3$ Euclidean quantum field theory.
It is natural to conjecture that the Markov process built in this way describes 
the Glauber dynamic of $3$-dimensional ferromagnets near their critical temperature.
\end{abstract}

\keywords {\small Stochastic PDEs, Renormalisation, Wick products, quantum field theory}

\noindent{\small \textit{MSC class:} 60H15, 81S20, 82C28}

\setcounter{tocdepth}{2}
\tableofcontents

\section{Introduction}

The purpose of this article is to develop a general theory allowing to formulate, solve
and analyse solutions to semilinear stochastic partial differential equations of the type
\begin{equ}[e:generalProblem]
\CL u = F(u,\xi)\;,
\end{equ}
where $\CL$ is a (typically parabolic but possibly elliptic) differential operator, $\xi$ is a (typically very irregular) random input,
and $F$ is some nonlinearity. The nonlinearity $F$ does not necessarily need to be local,
and it is also allowed to depend on some partial derivatives of $u$, as long as these
are of strictly lower order than $\CL$. One example of random input that is of particular interest 
in many situations arising from the large-scale behaviour of some physical microscopic model
is that of white noise (either space-time or just in space), but let us stress immediately that Gaussianity is \textit{not}
essential to the theory, although it simplifies certain arguments. Furthermore, we will assume that $F$ 
depends on $\xi$ in an affine way, although this could in principle be relaxed to some
polynomial dependencies. 

Our main assumption will be that the equation described by \eref{e:generalProblem}
is \textit{locally subcritical} (see Assumption~\ref{ass:powercount} below). 
Roughly speaking, this means that if one rescales \eref{e:generalProblem}
in a way that keeps both $\CL u$ and $\xi$ invariant then, at small scales, all nonlinear terms
formally disappear. 
A ``na\"\i ve'' approach to such a problem is to consider a sequence of regularised problems
given by
\begin{equ}[e:mainProb]
\CL u_\eps = F(u_\eps,\xi_\eps)\;,
\end{equ}
where $\xi_\eps$ is some smoothened version of $\xi$ (obtained for example by convolution
with a smooth mollifier), and to show that $u_\eps$ converges to some limit $u$ which is independent of
the choice of mollifier.

This approach does in general fail, even under the assumption of local subcriticality.
Indeed, consider the KPZ equation on the line \cite{KPZOrig}, which is the stochastic PDE formally given by
\begin{equ}[e:basicKPZ]
\d_t h = \d_x^2 h + (\d_x h)^2 + \xi\;,
\end{equ}
where $\xi$ denotes space-time white noise. This is indeed of the form \eref{e:generalProblem}
with $\CL = \d_t - \d_x^2$ and $F(h,\xi) = (\d_x h)^2 + \xi$ and it is precisely this kind of problem that we have in mind. Furthermore, 
if we zoom into the small scales by writing $\tilde h(x,t) = \delta^{-1/2} h(\delta x, \delta^2 t)$ and $\tilde \xi(x,t) = \delta^{3/2}\xi(\delta x, \delta^2 t)$
for some small parameter $\delta$, then we have that on the one hand $\tilde \xi$ equals
$\xi$ in distribution, and on the other hand $\tilde h$ solves
\begin{equ}
\d_t \tilde h = \d_x^2 \tilde h +  \delta^{1/2} (\d_x \tilde h)^2 + \tilde \xi\;.
\end{equ}
As $\delta \to 0$ (which corresponds to probing solutions at very small scales), we see
that, at least at a formal level, the nonlinearity vanishes and we simply recover the stochastic 
heat equation. This shows that the KPZ equation is indeed locally subcritical in dimension $1$.
On the other hand, if we simply replace $\xi$ by $\xi_\eps$ in \eref{e:basicKPZ} and try to take
the limit $\eps \to 0$, solutions diverge due to the ill-posedness of the term $(\d_x h)^2$.

However, in this case, it is possible to devise a suitable renormalisation procedure 
\cite{MR1462228,KPZ},
which essentially amounts to subtracting a very large constant to the right hand side of 
a regularised version of \eref{e:basicKPZ}.
This then ensures that the corresponding sequence of solutions converges to a finite limit.
The purpose of this article is to build a general framework that goes far beyond the example of 
the KPZ equation and allows to provide a robust notion of solution to a very large
class of locally subcritical stochastic PDEs that are classically ill-posed.

\begin{remark}
In the language of quantum field theory (QFT), equations that are subcritical in the way just described
give rise to ``superrenormalisable'' theories. 
One major difference between the results presented in this article
and most of the literature on quantum field theory is that the approach explored here
is truly non-perturbative and therefore allows one to deal also with some non-polynomial equations
like \eref{e:PAMGen} or \eref{e:KPZ} below.
We furthermore consider parabolic problems, where we need to deal with the problem
of initial conditions and local (rather than global) solutions.
Nevertheless, the mathematical analysis of QFT was one of the main 
inspirations in the development of the techniques
and notations presented in Sections~\ref{sec:SPDE} and \ref{sec:Gaussian}.
\end{remark}

Conceptually, the approach developed in this article for formulating and solving problems of the
type \eref{e:generalProblem} consists of three steps.

\begin{list}{\labelitemi}{\leftmargin=\parindent}
\item[1.] In an \textit{algebraic} step, one first builds a ``regularity structure'', which is 
sufficiently rich to be able to describe the fixed point problem associated to \eref{e:generalProblem}.
Essentially, a regularity structure is a vector space that allows to describe the coefficients
in a kind of ``Taylor expansion'' of the solution around any point in space-time.
The twist is that the ``model'' for the Taylor expansion does not only consist of polynomials,
but can in general contain other functions and / or distributions built from multilinear expressions
involving $\xi$. 

\item[2.] In an \textit{analytical} step, one solves the fixed point problem formulated in the algebraic step. 
This allows to build 
an ``abstract'' solution map to \eref{e:generalProblem}. In a way, this is a closure procedure: the abstract solution
map essentially describes all ``reasonable'' limits that can be obtained when solving \eref{e:generalProblem}
for sequences of regular driving noises that converge to something very rough.

\item[3.] In a final \textit{probabilistic} step, one builds a ``model'' corresponding to the Gaussian 
process $\xi$ we are really interested in. In this step, one typically has to choose a renormalisation
procedure allowing to make sense of finitely many products of distributions that have no
classical meaning. Although there is some freedom involved, there usually is a canonical model,
which is ``almost unique'' in the sense that it is naturally parametrized by elements in some
finite-dimensional Lie group, which has an interpretation as a ``renormalisation group''
for \eref{e:generalProblem}.
\end{list}

We will see that there is a very general theory that allows to build a ``black box'', which performs
the first two steps for a very large class of stochastic PDEs. For the last step, we do not have
a completely general theory at the moment, but we have a general methodology, as well as a general
toolbox, which seem to be very useful in practice.

\subsection{Some examples of interesting stochastic PDEs}
\label{sec:examples}

Some examples of physically relevant equations that in principle fall into the category of problems 
amenable to analysis via the techniques developed in this article include:
\begin{list}{\labelitemi}{\leftmargin=\parindent}
\item The stochastic quantisation of $\Phi^4$ quantum field theory in dimension $3$. This
formally corresponds to the equation
\begin{equ}[e:Phi4]
\d_t \Phi = \Delta \Phi - \Phi^3 + \xi\;,\tag{$\Phi^4$}
\end{equ}
where $\xi$ denotes space-time white noise and the spatial variable takes values in the $3$-dimensional
torus, see \cite{ParisiWu}. Formally, the invariant measure of \eref{e:Phi4}
(or rather a suitably renormalised version of it) is the measure on Schwartz distributions
associated to Bosonic Euclidean quantum field theory in $3$ space-time dimensions.
The construction of this measure was one of the major achievements of the programme
of constructive quantum field theory, see the articles \cite{MR0231601,MR0359612,MR0408581,MR0416337,MR0384003}, as well as
the monograph \cite{MR887102} and the references therein.

In two spatial dimensions, this problem was previously treated in \cite{AlbRock91,MR2016604}.
It has also been argued more recently in \cite{MR2214506} that even though it is formally symmetric,
the $3$-dimensional version of this model is not amenable to 
analysis via Dirichlet forms. 
In dimension $4$, the model \eref{e:Phi4} becomes critical and one
does not expect to be able to give it any non-trivial (i.e. non-Gaussian in this case)
meaning as a random field for $d \ge 4$, see for example \cite{MR643591,MR678000,MR701675}.

Another reason why \eref{e:Phi4} is a very interesting equation to consider is that 
it is related to the behaviour of the $3D$ Ising model under Glauber dynamic near its
critical temperature. For example, it was shown in \cite{MR1317994} that the one-dimensional version
of this equation describes the Glauber dynamic of an Ising chain with a Kac-type interaction
at criticality. In \cite{MR1661764}, it is argued that the same should hold true in higher
dimensions and an argument is given that relates the renormalisation procedure required
to make sense of \eref{e:Phi4} to the precise choice of length scale as
a function of the distance from criticality.

\item The continuous parabolic Anderson model
\begin{equ}[e:PAM]
\d_t u = \Delta u + \xi u\;,\tag{PAM}
\end{equ}
where $\xi$ denotes spatial white noise that is constant in time. For smooth noise, this problem
has been treated extensively in \cite{MR1185878}. While the problem with $\xi$ given by spatial white noise
is well-posed in dimension $1$ (and a good approximation theory exists, see \cite{MR2451056}), 
it becomes ill-posed already in dimension $2$.
One does however expect this problem
to be renormalisable with the help of the techniques presented here in spatial dimensions
$2$ and $3$. Again, dimension $4$ is critical and one does not expect any continuous version
of the model for $d \ge 4$.

\item KPZ-type equations of the form
\begin{equ}[e:KPZ]
\d_t h = \d_x^2 h + g_1(h) \bigl(\d_x h\bigr)^2 + g_2(h) \d_x h + g_3(h) + g_4(h)\xi\;, \tag{KPZ}
\end{equ}
where $\xi$ denotes space-time white noise and the $g_i$ are smooth functions. 
While the classical KPZ equation can be made sense of via the Cole-Hopf transform
\cite{MR0042889,MR0047234,MR1462228}, this trick fails in the more general situation 
given above or in the case of a system of coupled KPZ equations, which arises naturally in
the study of chains of nonlinearly interacting oscillators \cite{Milton}.

A more robust concept of solution for the KPZ equation where $g_4 = g_1 = 1$ and $g_2 = g_3 = 0$,
as well as for a number of other equations belonging to the class \eref{e:KPZ} was given recently in the series
of articles \cite{MR2875759,BurgersRough,Hendrik,KPZ}, using ideas from the theory of rough paths that eventually lead
to the development of the theory presented here. 
The  more general class
of equations \eref{e:KPZ} is of particular interest since it is formally invariant under 
changes of coordinates
and would therefore be a good candidate for describing a natural ``free evolution'' for loops
on a manifold, which generalises the stochastic heat equation. See \cite{Funaki} for a previous
attempt in this direction and \cite{Zdzislaw} for some closely related work. 
\item The Navier-Stokes equations with very singular forcing
\begin{equ}[e:NS]\tag{SNS}
\d_t v = \Delta v - P(v\cdot \nabla)v + \xi\;,
\end{equ}
where $P$ is Leray's projection onto the space of divergence-free vector fields.
If we take $\xi$ to have the regularity of space-time white noise, \eref{e:NS} is already 
classically ill-posed in dimension $2$, although one can circumvent this problem, see \cite{MR1051499,MR1941997,MR2060312}.
However, it turns out that the actual critical dimension is $4$ again, so that we can hope to make sense
of \eref{e:NS} in a suitably renormalised sense in dimension $3$ and construct local solutions there.
\end{list}

One common feature of all of these problems is that they involve products between terms
that are too irregular for such a product to make sense as a continuous bilinear form
defined on some suitable function space. Indeed, denoting by $\CC^\alpha$ for $\alpha < 0$ the
Besov space $B^\alpha_{\infty,\infty}$, it is well-known that, for non-integer values of $\alpha$ 
and $\beta$, the map $(u,v) \mapsto uv$ is
well defined from $\CC^\alpha \times \CC^\beta$ into some space of Schwartz distributions
if and only if $\alpha + \beta > 0$ (see for example \cite{BookChemin}), which is quite easily seen to be
violated in all of these examples.

In the case of second-order parabolic equations, 
it is straightforward to verify (see also Section~\ref{sec:singular} below) that, for fixed time,  
the solutions to the linear equation
\begin{equ}
\d_t X = \Delta X + \xi\;,
\end{equ}
belong to $\CC^\alpha$ for $\alpha < 1-{d\over 2}$ when $\xi$ is space-time white noise
and $\alpha < 2-{d\over 2}$ when $\xi$ is purely spatial white noise.
As a consequence, one expects $\Phi$ to take values in $\CC^\alpha$ with
$\alpha < -1/2$, so that $\Phi^3$ is ill-defined. In the case of \eref{e:PAM}, one expects $u$
to take values in $\CC^\alpha$ with $\alpha < 2-d/2$, so that the product $u \xi$ is well-posed
only for $d < 2$. As in the case of \eref{e:Phi4}, dimension $2$ is ``borderline'' with the appearance
of logarithmic divergencies, while dimension $3$ sees the appearance of algebraic divergencies
and logarithmic subdivergencies.
Note also that, since $\xi$ is white noise \textit{in space}, there is no theory of
stochastic integration available to make sense of the product $u\xi$, unlike in the case when $\xi$ is
space-time white noise. (See however \cite{PAMPreprint} for a very recent 
article solving this particular
problem in dimension $2$.)
Finally, one expects the function $h$ in \eref{e:KPZ} to take values in $\CC^\alpha$
for $\alpha < {1\over 2}$, so that all the terms 
appearing in \eref{e:KPZ} are ill-posed, except for the term involving $g_3$.

Historically, such situations have been dealt with by 
replacing the products in question by their Wick ordering with respect to the Gaussian
structure given by the solution to the linear problem $\CL u = \xi$, see for example
\cite{MR815192,AlbRock91,MR1941997,MR2016604,MR2365646} and references therein. 
In many of the problems mentioned above, such a technique is bound to fail due to the presence of
additional subdivergencies. Furthermore, we would like to be able to consider terms like $g_1(h) (\d_x h)^2$ in 
\eref{e:KPZ} where $g_1$ is an arbitrary smooth function, so that it is not clear at all what a Wick ordering
would mean.
Over the past few years, it has transpired that the theory of controlled rough paths \cite{Lyons,MR2091358,Trees} 
could be used
in certain situations to provide a meaning to the ill-posed nonlinearities  arising in a class
of Burgers-type equations \cite{MR2860933,BurgersRough,Hendrik,JanHendrik}, as 
well as in the KPZ equation \cite{KPZ}.
That theory however is intrinsically a one-dimensional theory, which is why it has so far only been
successfully applied to stochastic evolution equations with one spatial dimension. 

In general, the theory of rough paths and its variants 
do however allow to deal with processes
taking values in an infinite-dimensional space. It has therefore been applied 
successfully to stochastic PDEs
driven by signals that are very rough in time (i.e. rougher than white noise), but at the expense
of requiring additional spatial regularity \cite{MaxSam,Oberhaus,Josef}.

One very recent attempt to use related ideas in higher dimensions was made in \cite{PAMPreprint}
by using a novel theory of ``controlled distributions''. With the help of this theory, which
relies heavily on the use of Bony's paraproduct, the authors can treat for example
\eref{e:PAM} (as well as some nonlinear variant thereof) in dimension $d = 2$.
The present article can be viewed as a far-reaching generalisation of related ideas,
in a way which will become clearer in Section~\ref{sec:regStructure} below.

\subsection{On regularity structures}
\label{sec:regStruct}

The main idea developed in the present work is that of describing the ``regularity'' of a function
or distribution in a way that is adapted to the problem at hand.
Traditionally, the regularity of a function is measured by its proximity to polynomials. Indeed,
we say that a function $u\colon \R^d \to\R$ is of class $\CC^\alpha$ with $\alpha > 0$ if, for every point
$x \in \R^d$, it is possible to find a polynomial $P_x$ such that
\begin{equ}
|f(y) - P_x(y)| \lesssim |x-y|^\alpha\;.
\end{equ}
What is so special about polynomials? For one, they have very nice algebraic properties:
products of polynomials are again polynomials, and so are their translates and derivatives.
Furthermore, a monomial is a homogeneous function: it behaves at the origin in a self-similar way
under rescalings. The latter property however does rely on the choice of a base point:
the polynomial $y \mapsto (y-x)^k$ is homogeneous of degree $k$ when viewed
around $x$, but it is made up from a sum of monomials with different homogeneities
when viewed around the origin.

In all of the examples considered in the previous subsection, solutions are expected to
be extremely irregular (at least in the classical sense!), 
so that polynomials alone are a very poor model for trying to describe them. 
However, because of local subcriticality, one expects the solutions to look
at smallest scales like solutions to the corresponding linear problems, so we are in situations
where it might be possible to make a good ``guess'' for a much more adequate model 
allowing to describe the small-scale structure of solutions.

\begin{remark}
In the particular case of functions of one variable, this point of view has been advocated
by Gubinelli  in \cite{MR2091358,Trees} (and to some extent by Davie in \cite{MR2387018}) 
as a way of interpreting Lyons's theory
of rough paths. (See also \cite{MR2036784,MR2314753,MR2604669} for 
some recent monographs surveying that theory.) That theory does however rely very strongly
on the notion of ``increments'' which is very one-dimensional in nature and forces
one to work with functions, rather than general distributions. In a more subtle
way, it also relies on the fact that one-dimensional integration can be viewed as 
convolution with the Heaviside function, which is locally constant away from $0$,
another typically one-dimensional feature.
\end{remark}

This line of reasoning is the motivation behind the introduction of 
the main novel abstract structure proposed in this work, which is that of a ``regularity structure''.
The precise definition will be given in Definition~\ref{def:regStruct} below, but 
the basic idea is to fix a finite family of functions (or distributions!) that will play the role
of polynomials. Typically, this family contains all polynomials, but it may contain more
than that. A simple way of formalising this is that one fixes some abstract vector space $T$
where each basis vector represents one of these distributions.
A ``Taylor expansion'' (or ``jet'') is then described by an element $a \in T$ which, via some
``model'' $\PPi\colon T \to \CS'(\R^d)$, one can interpret as determining some
distribution $\PPi a \in \CS'(\R^d)$.
In the case of polynomials, $T$ would be the space of abstract polynomials in $d$ commuting
indeterminates and $\PPi$ would be the map that realises such an abstract polynomial
as an actual function on $\R^d$.

As in the case of polynomials, different distributions have different homogeneities
(but these can now be arbitrary real numbers!), so we have a splitting of $T$ into
``homogeneous subspaces'' $T_\alpha$. Again, as in the case of polynomials, 
the 	homogeneity of an element $a$ describes the behaviour of $\PPi a$ around some
base point, say the origin $0$. Since we want to be able to place this base point
at an arbitrary location we also postulate that one has a family of invertible linear maps
$F_x \colon T \to T$ such that if $a \in T_\alpha$, then $\PPi F_x a$ exhibits
behaviour ``of order $\alpha$'' (this will be made precise below in the case of distribution)
near the point $x$.
In this sense, the map $\Pi_x = \PPi \circ F_x$ plays the role of the ``polynomials based at $x$'',
while the map $\Gamma_{xy} = F_x^{-1} \circ F_y$ plays the role of a ``translation operator''
that allows to rewrite a ``jet based at $y$'' into a ``jet based at $x$''.

We will endow the space of all models $(\PPi, F)$ as above with a topology 
that enforces the correct behaviour of $\Pi_x$ near each point $x$, and furthermore enforces
some natural notion of regularity of the map $x \mapsto F_x$.
The important remark is that although this turns the space of models into a complete
metric space, it does \textit{not} turn it into a linear (Banach) space!
It is the intrinsic nonlinearity of this space which allows to encode 
the subtle cancellations that one needs to be able to keep track of in order to treat the
examples mentioned in Section~\ref{sec:examples}.
Note that the algebraic structure arising in the theory of rough paths
(truncated tensor algebra, together with its group-like elements) can be viewed
as one particular example of an abstract regularity structure. The space of rough paths
with prescribed H\"older regularity is then precisely the corresponding space
of models. See Section~\ref{sec:RP} for a more detailed description of
this correspondence.

\subsection{Main results: abstract theory}

Let us now expose some of the main abstract results obtained in this article. 
Unfortunately, since the precise set-up requires a number of rather lengthy definitions, 
we cannot give precise statements here. However, we would like to provide the
reader with a flavour of the theory and refer to the main text for more details.

One of the main novel definitions consists in spaces $\CD^\gamma$ and $\CD_\alpha^\gamma$
(see Definition~\ref{def:Dgamma} and Remark~\ref{rem:notationReg} below)
which are the equivalent in our framework to the usual spaces $\CC^\gamma$.
They are given in terms of a ``local Taylor expansion of order $\gamma$''
at every point, together with suitable regularity assumption. 
Here, the index $\gamma$ measures the order of the expansion, while the index $\alpha$ (if present)
denotes the lowest homogeneity of the different terms appearing in the expansion. 
In the case of regular Taylor expansions, the term with the lowest homogeneity is always the
constant term, so one has $\alpha = 0$. However, since we allow elements of negative homogeneity,
one can have $\alpha \le 0$ in general.
Unlike the case of regular Taylor expansions where the first term always consists
of the value of the function itself, we are here in a situation where, due to the
fact that our ``model'' might contain elements that are distributions, it is not clear
at all whether these ``jets'' can actually be patched together to represent an actual
distribution. The reconstruction theorem, Theorem~\ref{theo:reconstruction} below,
states that this is always the case as soon as $\gamma > 0$. Loosely speaking, it states 
the following, where we again write $\CC^\alpha$ for the Besov space $\CB^\alpha_{\infty,\infty}$.
(Note that with this notation $\CC^0$ really denotes the space $L^\infty$, $\CC^1$ the space
of Lipschitz continuous functions, etc. This is consistent with the usual notation for non-integer
values of $\alpha$.)

\begin{theorem}[Reconstruction]
For every $\gamma > 0$ and $\alpha \le 0$, there exists a unique continuous linear 
map $\CR \colon \CD_\alpha^\gamma \to \CC^\alpha(\R^d)$ with the property that, in a neighbourhood
of size $\eps$ around any $x \in \R^d$, $\CR f$ is approximated by $\Pi_x f(x)$, the jet
described by $f(x)$, up to an error of order $\eps^\gamma$. 
\end{theorem}

The reconstruction theorem shows that elements $f \in \CD^\gamma$ uniquely describe distributions
that are modelled locally on the distributions described by $\Pi_x f(x)$.
We therefore call such an element $f$ a ``modelled distribution''.
At this stage, the theory is purely descriptive: given a model of a 
regularity structure, it allows to describe
a large class of functions and / or distributions that ``locally look like'' linear combinations
of the elements in the model. 
We now argue that it is possible to construct a whole calculus that makes the
theory operational, and in particular sufficiently rich to allow to formulate and solve
large classes of semilinear PDEs.

One of the most important and non-trivial operations required for this is multiplication.
Indeed, one of the much lamented drawbacks of the classical theory
of Schwartz distributions is that there is no canonical way of multiplying them \cite{MR0064324}. 
As a matter of fact,
it is in general not even possible to multiply a distribution with a continuous function,
unless the said function has sufficient regularity.

The way we use here to circumvent this problem is to \textit{postulate} the values of
the products between elements of our model. If the regularity structure is
sufficiently large to also contain all of these products (or at least sufficiently
many of them in a sense to be made precise), then one can simply perform a pointwise
multiplication of the jets of two modelled distributions at each point.
Our main result in this respect is that, under some very natural structural assumptions,
such a product is again a modelled distribution. The following is a loose statement of
this result, the precise formulation of which is given in Theorem~\ref{theo:mult} below.

\begin{theorem}[Multiplication]\label{thm:mult}
Let $\star$ be a suitable product on $T$ and let $f_1 \in \CD_{\alpha_1}^{\gamma_1}$
and $f_2 \in \CD_{\alpha_2}^{\gamma_2}$ with $\gamma_i > 0$. 
Set $\alpha = \alpha_1 + \alpha_2$ and $\gamma = (\gamma_1 + \alpha_2) \wedge (\gamma_2 + \alpha_1)$.
Then, the pointwise product $f_1 \star f_2$ belongs to $\CD^\gamma_\alpha$.
\end{theorem}

In the case of $f \in \CD^\gamma_0$, all terms in the local expansion have positive
homogeneity, so that $\CR f$ is actually a function. It is then of course
possible to compose this function with any smooth function $g$. The non-trivial fact is
that the new function obtained in this way does also have a local ``Taylor expansion''
around every point which is typically of the same order as for the original function $f$.
The reason why this statement is not trivial is that the function $\CR f$ does in general
not possess much ``classical'' regularity, so that $\CR f$ typically does \textit{not} 
belong to $\CC^\gamma$. Our precise result is the content of Theorem~\ref{theo:smooth} 
below, which can be stated loosely as follows.

\begin{theorem}[Smooth functions]\label{thm:smooth}
Let $g \colon \R \to \R$ be a smooth function and consider a regularity structure endowed with a 
product $\star$ satisfying suitable compatibility assumptions. Then, for $\gamma > 0$,
one can build a map $\CG \colon \CD^\gamma_0 \to \CD^\gamma_0$ such that the identity
$\bigl(\CR \CG (f)\bigr)(x) = g\bigl((\CR f)(x)\bigr)$ holds for every $x \in \R^d$.
\end{theorem}

The final ingredient that is required in any general solution theory for semilinear PDEs
consists in some regularity improvement arising from the linear part of the equation.
One of the most powerful class of such statements is given by the Schauder estimates.
In the case of convolution with the Green's function $G$ of the Laplacian, the Schauder 
estimates state that if $f \in \CC^\alpha$, then $G * f \in \CC^{\alpha + 2}$,
unless $\alpha + 2 \in \N$. (In which case some additional logarithms appear in the 
modulus of continuity of $G*f$.)
One of the main reasons why the theory developed in this article is useful is that such
an estimate still holds when $f \in \CD^\alpha$. This is highly non-trivial since
it involves ``guessing'' an expansion for the local behaviour of $G*\CR f$
up to sufficiently high order. Somewhat surprisingly, it turns out that
even though the convolution with $G$ is not a local operator at all, its action on the 
local expansion of a function is local, except for those coefficients that correspond
to the usual polynomials.

One way of stating our result is the following, which will be reformulated more 
precisely in Theorem~\ref{theo:Int} below.

\begin{theorem}[Multi-level Schauder estimate]\label{thm:Schauder}
Let $K \colon \R^d \setminus \{0\} \to \R$ be a smooth kernel with a singularity 
of order $\beta - d$ at the origin for some $\beta > 0$.
Then, under certain natural assumptions on the regularity structure and 
the model realising it, and provided that $\gamma + \beta \not \in \N$, 
one can construct for $\gamma > 0$ a linear operator
$\CK_\gamma \colon \CD^\gamma_{\alpha} \to \CD^{\gamma+\beta}_{(\alpha + \beta) \wedge 0}$
such that the identity
\begin{equ}
\CR \CK_\gamma f = K * \CR f\;,
\end{equ}
holds for every $f \in \CD^\gamma_\alpha$. Here, $*$ denotes the usual convolution between
two functions / distributions.
\end{theorem}

We call this a ``multi-level'' Schauder estimate because it is a statement not just about
$f$ itself but about every ``layer'' appearing in its local expansion.

\begin{remark}
The precise formulation of the multi-level Schauder estimate allows to
specify a non-uniform scaling of $\R^d$. This is very useful for example when
considering the heat kernel which scales differently in space and in time. 
In this case, Theorem~\ref{thm:Schauder} still holds, but all regularity
statements have to be interpreted in a suitable sense. 
See Sections~\ref{sec:realisation} and \ref{sec:integral} below for more details.
\end{remark}

At this stage, we appear to possibly rely very strongly on the various still unspecified 
structural assumptions that are required
of the regularity structure and of the model realising it.
The reason why, at least to some extent, this can be ``brushed under the rug'' 
without misleading the reader
is the following result, which is a synthesis of Proposition~\ref{prop:extendMult} and 
Theorem~\ref{theo:extension} below.

\begin{theorem}[Extension theorem]\label{thm:extend}
It is always possible to extend a given regularity structure in such a way that 
the assumptions implicit in the statements of Theorems~\ref{thm:mult}--\ref{thm:Schauder} 
do hold.
\end{theorem}

Loosely speaking, the idea is then to start with the ``canonical'' regularity structure
corresponding to classical Taylor expansions and to enlarge it by successively applying 
the extension theorem, until it is large enough to allow a closed formulation of the
problem one wishes to study as a fixed point map.

\subsection{On renormalisation procedures}

The main problem with the strategy outlined above 
is that while the extension of an abstract regularity structure
given by Theorem~\ref{thm:extend} is actually very explicit and rather canonical, the corresponding
extension of the model $(\Pi,F)$ is unique (and continuous) only in the case of the multi-level
Schauder theorem and the composition by smooth functions, but \textit{not} 
in the case of multiplication when some of the homogeneities are strictly negative.
This is a reflection of the fact that multiplication between distributions and
functions that are too rough simply cannot be defined in any canonical way \cite{MR0064324}.
Different non-canonical choices of product then yield truly different solutions, 
so one might think that the theory is useless at selecting one ``natural'' solution
process.

If the driving noise $\xi$ in any of the equations from Section~\ref{sec:examples} 
is replaced by a smooth approximation $\xi^{(\eps)}$, 
then the associated model for the corresponding regularity structure
also consists of smooth functions. In this case, there is of course no problem in multiplying 
these functions, and one obtains a \textit{canonical} sequence of models 
$(\Pi^{(\eps)}, F^{(\eps)})$ realising our regularity structure. (See Section~\ref{sec:realAlg} for
details of this construction.)
At fixed $\eps$, our theory then simply yields some very local description 
of the corresponding classical solutions. 
In some special cases, the sequence $(\Pi^{(\eps)}, F^{(\eps)})$
converges to a limit that is independent of the regularisation procedure 
for a relatively large class of such regularisations. In particular, due to the 
symmetry of finite-dimensional control systems under time reversal, this is often
the case in the classical theory of rough paths, see \cite{Lyons,MR1883719,MR2667703}.

One important feature of the regularity structures arising naturally in the context of
solving semilinear PDEs is that they come with a natural \textit{finite-dimensional}
group $\RR$ of transformations that act on the space of models. 
In some examples (we will treat the case of \eref{e:Phi4} with $d=3$ in Section~\ref{sec:Phi4} and a 
generalisation of \eref{e:PAM} with $d=2$ in Section~\ref{sec:PAMGenRen}), one can explicitly exhibit a
subgroup $\RR_0$ of $\RR$ and a sequence of elements $M_\eps \in \RR_0$
such that the ``renormalised'' sequence $M_\eps(\Pi^{(\eps)}, F^{(\eps)})$
converges to a finite limiting model $(\hat \Pi, \hat F)$.
In such a case, the set of possible limits is parametrised 
by elements of $\RR_0$, which in our setting is always 
just a finite-dimensional nilpotent Lie group. 
In the two cases mentioned above, one can furthermore reinterpret solutions corresponding
to the ``renormalised''  model $M_\eps(\Pi^{(\eps)}, F^{(\eps)})$ as solutions
corresponding to the ``bare'' model $(\Pi^{(\eps)}, F^{(\eps)})$, but for a modified equation.

In this sense, $\RR$ (or a subgroup thereof) has an interpretation as a \textit{renormalisation group}
acting on some space of formal equations, which is a very common viewpoint in the
physics literature. (See for example \cite{Renorm} for a short introduction.)
This thus allows to usually reinterpret the objects constructed by our theory
as limits of solutions to equations that are modified by the addition of finitely many
diverging counterterms.
In the case of \eref{e:PAM} with $d=2$, the corresponding renormalisation
procedure is essentially a type of Wick ordering and therefore yields the appearance of
counterterms that are very similar in nature to those arising in the It\^o-Stratonovich
conversion formula for regular SDEs. (But with the crucial difference that they 
diverge logarithmically instead of being constant!) In the case of \eref{e:Phi4} with
$d=3$, the situation is much more delicate because of the appearance of a logarithmic subdivergence
``below'' the leading order divergence that cannot be dealt with by a Wick-type renormalisation.
For the invariant (Gibbs) measure corresponding to \eref{e:Phi4}, this fact is well-known and
had previously been observed in the context of constructive Euclidean QFT in 
\cite{MR0231601,MR0384003,MR0416337}.

\begin{remark}\label{rem:symmetry}
Symmetries typically play an important role in the analysis of the 
renormalisation group $\RR$. Indeed, if the equation under consideration exhibits
some symmetry, at least at a formal level, then it is natural to approximate 
it by regularised versions with the same symmetry. This then often places some
natural restrictions on $\RR_0 \subset \RR$, ensuring that the renormalised
version of the equation is still symmetric. For example, in the case of the
KPZ equation, it was already remarked in \cite{KPZ} that regularisation via
a non-symmetric mollifier can cause the appearance in the limiting solution of
an additional transport term, thus breaking the invariance under left / right 
reflection. In Section~\ref{sec:PAMGen} below, we will consider a class of equations
which, via the chain rule, is formally invariant under composition by 
diffeomorphisms. This ``symmetry''
again imposes a restriction on $\RR_0$ ensuring that the renormalised equations
again satisfy the chain rule.
\end{remark}

\begin{remark}
If an equation needs to be renormalised in order to have a finite limit, it typically
yields a whole family of limits parametrised by $\RR$ (or rather $\RR_0$ in the presence of
symmetries). Indeed, if $M_\eps(\Pi^{(\eps)}, F^{(\eps)})$ converges to a finite limit and
$M$ is any fixed element of $\RR_0$, then $M M_\eps(\Pi^{(\eps)}, F^{(\eps)})$ obviously
also converges to a finite limit. At first sight, this might look like a serious shortcoming
of the theory: our equations still aren't well-posed after all! It turns out that this state
of affairs is actually very natural. Even the very well-understood situation of
one-dimensional SDEs of the type
\begin{equ}[e:SDE]
dx = f(x)\,dt + \sigma(x)\,dW(t)\;,
\end{equ}
exhibits this phenomena: solutions are different whether we interpret the stochastic integral
as an It\^o integral, a Stratonovich integral, etc. In this particular case, one would
have $\RR \approx \R$ endowed with addition as its group structure and the action of $\RR$ onto
the space of equations is given by $M_c (f,\sigma) = (f, \sigma + c \sigma \sigma')$, where $M_c \in \RR$
is the group element corresponding to the real constant $c$. 
Switching between the It\^o and Stratonovich formulations is indeed a transformation of this type
with $c \in \{\pm {1\over 2}\}$.

If the equation is driven by more than one
Brownian motion, our renormalisation group increases in size: one now has a choice of stochastic integral
for each of the integrals appearing in the equation. On symmetry grounds however, we would of course
work with the subgroup $\RR_0 \subset \RR$ which corresponds to the same choice for each.
If we additionally exploit the fact that the class of equations \eref{e:SDE} is formally invariant
under the action of the group of diffeomorphisms of $\R$ (via the chain rule), then we could
reduce $\RR_0$ further by postulating that the renormalised solutions should also transform
under the classical chain rule. This would then reduce $\RR_0$ to the trivial group, thus leading
to a ``canonical'' choice (the Stratonovich integral). In this particular case, we could of course
also have imposed instead that the integral $\int W\,dW$ has no component in the 
$0$th Wiener chaos, thus leading 
to Wick renormalisation with the It\^o integral as a second ``canonical'' choice. 
\end{remark}

\subsection{Main results: applications}

We now show what kind of convergence results can be obtained by
concretely applying the theory developed in this article to two examples of stochastic PDEs
that cannot be interpreted by any classical means.
The precise type of convergence will be detailed in the main body of the article,
but it is essentially a convergence in probability on spaces of continuous
trajectories with values in $\CC^\alpha$ for a suitable (possibly negative)
value of $\alpha$. A slight technical difficulty arises due to the fact that the limit
processes do not necessarily have global solutions, but could exhibit
blow-ups in finite time. In such a case, we know that the blow-up time is almost surely
strictly positive and we have convergence ``up to the blow-up time''.

\subsubsection{Generalisation of the parabolic Anderson model}
\label{sec:PAMGen}

First, we consider the following generalisation of \eref{e:PAM}:
\begin{equ}[e:PAMGen]
\d_t u = \Delta u + f_{ij}(u) \,\d_i u\, \d_j u  + g(u)\xi\;,\qquad u(0) = u_0\;,\tag{PAMg}
\end{equ}
where $f$ and $g$ are smooth function and summation of the indices $i$ and $j$
is implicit.
Here, $\xi$ denotes spatial white noise. This notation is of course only formal since
neither the product $g(u)\xi$, nor the product $\d_i u\, \d_j u$ make any sense
classically. Here, we view $u$ as a function of time $t \ge 0$ and of
$x \in \T^2$, the two-dimensional torus.

It is then natural to replace $\xi$ by a smooth approximation $\xi_\eps$ which is
given by the convolution of $\xi$ with a rescaled mollifier $\rho$. Denote
by $u_\eps$ the solution to the equation
\begin{equ}[e:PAMRenorm]
\d_t u_\eps = \Delta u_\eps + f_{ij}(u_\eps) \,\bigl(\d_i u_\eps\, \d_j u_\eps - \delta_{ij} C_\eps g^2(u_\eps)\bigr) 
+ g(u_\eps)\bigl(\xi_\eps - 2 C_\eps g'(u_\eps)\bigr)\;,
\end{equ}
again with initial condition $u_0$. Then, we have the following result:

\begin{theorem}\label{theo:mainConvPAM}
Let $\alpha \in ({1\over 2}, 1)$.
There exists a choice of constants $C_\eps$ such that, for every initial 
condition $u_0 \in \CC^\alpha(\T^2)$, the sequence 
of solutions $u_\eps$ to \eref{e:PAMRenorm} converges to a limit $u$. 
Furthermore, there is an explicit constant $K_\rho$ depending on $\rho$ such
that if one sets $C_\eps = -{1\over \pi} \log \eps + K_\rho$, then the limit obtained in this way is 
independent of the choice of mollifier $\rho$.
\end{theorem}

\begin{proof}
This is a combination of Corollary~\ref{cor:locSolPAMGen} (well-posedness of the abstract formulation
of the equation), Theorem~\ref{theo:convModelPAM} 
(convergence of the renormalised models to a limiting model) and
Proposition~\ref{prop:identifyRenormPAM} (identification of the renormalised solutions with \eref{e:PAMRenorm}). The explicit value of the constant $C_\eps$ is given in \eref{e:CEpsPAM}.
\end{proof}

\begin{remark}
In the case $f = 0$, this result has recently been obtained by different (though
related in spirit) techniques in \cite{PAMPreprint}.
\end{remark}

\begin{remark}\label{rem:convergenceNotion}
Since solutions might blow up in finite time, the notion of convergence considered
here is to fix some large cut-off $L >0$ and terminal time $T$
and to stop the solutions $u_\eps$ as soon
as $\|u_\eps(t)\|_\alpha \ge L$, and similarly for the limiting process $u$.
The convergence is then convergence in probability in $\CC^\alpha_\s([0,T]\times \T^2)$ 
for the stopped process. Here elements in $\CC^\alpha_\s$ are $\alpha$-H\"older continuous in space
and ${\alpha \over 2}$-H\"older continuous in time, see Definition~\ref{def:Calphas} below.
\end{remark}

%
\begin{remark}
It is lengthy but straightforward to verify that the additional 
diverging terms in the renormalised equation
\eref{e:PAMRenorm} are precisely such that if $\psi \colon \R \to \R$ is a smooth diffeomorphism,
then $v_\eps \eqdef \psi(u_\eps)$ solves again an equation of the type
\eref{e:PAMRenorm}. Furthermore, this equation is precisely the renormalised version of the equation
that one obtains by just formally applying the chain rule to \eref{e:PAMGen}!
This gives a rigorous justification of the chain rule for \eref{e:PAMGen}.
In the case \eref{e:KPZ}, one expects a similar phenomenon, which would then
allow to interpret the Cole-Hopf transform rigorously as a particular case of a general
change of variables formula.
\end{remark}

\subsubsection{The dynamical $\Phi^4_3$ model}

A similar convergence result can be obtained for \eref{e:Phi4}. This time, the renormalised
equation takes the form
\begin{equ}[e:Phi4bis]
\d_t u_\eps = \Delta u_\eps + C_\eps u_\eps - u_\eps^3 + \xi_\eps\;,
\end{equ}
where $u_\eps$ is a function of time $t \ge 0$ and space $x \in \T^3$, the three-dimensional
torus.
It turns out that the simplest class of approximating noise is to consider a space-time
mollifier $\rho(x,t)$ and to set $\xi_\eps \eqdef \xi * \rho_\eps$, where $\rho_\eps$
is the rescaled mollifier given by $\rho_\eps(x,t) = \eps^{-5} \rho(x/\eps, t/\eps^2)$.

With this notation, we then have the following convergence result, which is the content of Section~\ref{sec:Phi4} below.

\begin{theorem}\label{thm:Phi4}
Let $\alpha \in (-{2\over 3}, -{1\over 2})$.
There exists a choice of constants $C_\eps$ such that, for every initial 
condition $u_0 \in \CC^\alpha(\T^3)$, the sequence 
of solutions $u_\eps$ converges to a limit $u$. Furthermore, if $C_\eps$
are chosen suitably, then this limit is again independent of the choice of mollifier $\phi$.
\end{theorem}

\begin{proof}
This time, the statement is a consequence of 
Proposition~\ref{prop:locSolPhi4} (well-posedness of the abstract formulation), 
Theorem~\ref{theo:convPhi4} 
(convergence of the renormalised models) and
Proposition~\ref{prop:identifyRenormPhi4} (identification of renormalised 
solutions with \eref{e:Phi4bis}).
\end{proof}

\begin{remark}
It turns out that the limiting solution $u$ is almost surely a continuous function in time
with values in $\CC^\alpha(\T^3)$.
The notion of convergence is then as in Remark~\ref{rem:convergenceNotion}.
Here, we wrote again $\CC^\alpha$ as a shorthand for the Besov space $B^\alpha_{\infty,\infty}$.
\end{remark}

\begin{remark}
As already noted in \cite{MR0384003} 
(but for a slightly different regularisation procedure, which is more natural
 for the static version of the model considered there), 
the correct choice of constants $C_\eps$ is of the form
\begin{equ}
C_\eps = {C_1\over \eps} + C_2 \log \eps + C_3\;,
\end{equ}
where $C_1$ and $C_3$ depend on the choice of $\rho$ in a way that is 
explicitly computable, and the constant $C_2$ is independent of the choice of $\rho$.
It is the presence of this additional logarithmic 
divergence that makes the analysis of \eref{e:Phi4} highly non-trivial. In particular,
it was recently remarked in \cite{MR2214506} that this 
seems to rule out the use of Dirichlet form techniques for interpreting \eref{e:Phi4}.
\end{remark}

\begin{remark}
Again, we do not claim that the solutions constructed here are global. Indeed, the convergence
holds in the space $\CC([0,T], \CC^{\alpha})$, but only up to some possibly 
finite explosion time. It is very likely that one can show that the solutions are global
for almost every choice of initial condition, where ``almost every'' refers to the
measure built in \cite{MR0384003}. This is because that measure is expected to be invariant for the limiting
process constructed in Theorem~\ref{thm:Phi4}.
\end{remark}

\subsubsection{General methodology}

Our methodology for proving the kind of convergence results mentioned above is the following. 
First, given a locally subcritical SPDE
of the type \eref{e:mainProb}, we build a regularity structure $\TT_F$ which takes into account
the structure of the nonlinearity $F$ (as well as the regularity index of the driving noise and
the local scaling properties of the linear operator $\CL$), together with
a class $\MM_F$ of ``admissible models'' on $\TT_F$ which are defined using the abstract properties
of $\TT_F$ and the Green's function of $\CL$. The general construction of such a 
structure is performed in Section~\ref{sec:SPDE}. We then also build a natural
``lift map'' $Z\colon \CC(\R^d) \to \MM_F$ (see Section~\ref{sec:realAlg}), 
where $d$ is the dimension of the underlying 
space-time, as well as an abstract solution map $\CS \colon \CC^\alpha\times \MM_F \to \CD^\gamma$,
with the property that $\CR \CS\bigl(u_0, Z(\xi_\eps)\bigr)$ yields the classical 
(local) solution to \eref{e:mainProb} with initial condition $u_0$ and noise $\xi_\eps$.
Here, $\CR$ is the ``reconstruction operator'' already mentioned earlier.
A general result showing that $\CS$ can be built for ``most'' subcritical semilinear evolution
problems is provided in Section~\ref{sec:solSPDE}. This relies fundamentally on the
multi-level Schauder estimate of Section~\ref{sec:integral}, as well as the results of 
Section~\ref{sec:singular} dealing with singular modelled distributions, which is required
in order to deal with the behaviour near time $0$. 

The main feature of this construction is that both the abstract solution map $\CS$ and the
reconstruction operator $\CR$ are continuous. In most cases of interest they are even locally
Lipschitz continuous in a suitable sense. Note that we made a rather serious abuse of notation here, since
the very definition of the space $\CD^\gamma$ does actually depend on the particular 
model $Z(\xi_\eps)$! This will not bother us unduly since one could very easily remedy this
by having the target space be ``$\MM_F\ltimes \CD^\gamma$'', with the understanding that each
``fiber'' $\CD^\gamma$ is modelled on the corresponding model in $\MM_F$. The map $\CS$ would
then simply act as the identity on $\MM_F$.

Finally, we show that it is possible to find a sequence of elements $M_\eps \in \RR$ such
that the sequence of renormalised models $M_\eps Z(\xi_\eps)$ converge to some limiting
model $\hat Z$ and we identify $\CR \CS\bigl(u_0, M_\eps Z(\xi_\eps)\bigr)$ with the classical
solution to a modified equation. The proof of this fact is the only part of the whole theory which
is not ``automated'', but has to be performed by hand for each class of problems. However,
if two problems give rise to the same structure $\MM_F$ and are based on the same linear operator
$\CL$, then they can be treated with the same procedure, since it is only the details of the
solution map $\CS$ that change from one problem to the other. We treat two classes of problems in
detail in Sections~\ref{sec:renProc} and \ref{sec:Gaussian}. Section~\ref{sec:Gaussian} also
contains a quite general toolbox that is very useful for treating the renormalisation of 
many equations with Gaussian driving noise.

\subsection{Alternative theories}

Before we proceed to the meat of this article, let us give a quick review of some of the main 
existing theories allowing to make sense of products of distributions. For each of these theories,
we will highlight the differences with the theory of regularity structures.

\subsubsection{Bony's paraproduct}

Denoting by $\Delta_j f$ the $j$th Paley-Littlewood block of a distribution $f$, one can define the 
 bilinear operators
\begin{equ}
\pi_<(f,g) = \sum_{i < j-1} \Delta_i f \Delta_j g\;,\quad
\pi_>(f,g) = \pi_<(g,f)\;,\quad
\pi_o(f,g) = \sum_{|i - j| \le 1} \Delta_i f \Delta_j g\;,
\end{equ}
so that, at least formally, one has $fg = \pi_<(f,g) + \pi_>(f,g) + \pi_o(f,g)$. (See \cite{MR631751} for the original article and some applications to the analysis of solutions to fully nonlinear PDEs, 
as well as the monograph
and review article \cite{BookChemin,MR2682821}. The notation of this section is borrowed from 
the recent work \cite{PAMPreprint}.)
It turns out that $\pi_<$ and $\pi_>$ make sense for \textit{any} two distributions
$f$ and $g$. Furthermore, if $f \in \CC^\alpha$ and $g \in \CC^\beta$ with $\alpha + \beta > 0$,
then 
\begin{equ}[e:paraprod]
\pi_<(f,g) \in \CC^\beta\;,\qquad \pi_>(f,g) \in \CC^\alpha\;,\qquad
\pi_o(f,g)\in \CC^{\alpha+\beta}\;,
\end{equ}
so that one has a gain of regularity there, but one does
again encounter a ``barrier'' at $\alpha + \beta = 0$.

The idea exploited in \cite{PAMPreprint} is to consider a ``model distribution'' $\eta$ and 
to consider ``controlled distributions'' of the type
\begin{equ}
f = \pi_{<}(f^\eta, \eta) + f^\sharp\;,
\end{equ}
where both $f^\eta$ and $f^\sharp$ are more regular than $\eta$. The construction is such that,
at small scales, irregularities of $f$ ``look like'' irregularities of $\eta$. The hope is then 
that if $f$ is controlled by $\eta$, $g$ is controlled by $\zeta$, and one knows of a
renormalisation procedure allowing to make
sense of the product $\eta \zeta$ (by using tools from stochastic analysis for example),
then one can also give a consistent meaning to the product $fg$.
This is the philosophy that was implemented in \cite[Theorems~9 and 31]{PAMPreprint}.

This approach is very close to the one taken in the present work, and indeed it is possible
to recover the results of \cite{PAMPreprint} in the context of regularity structures, modulo
slight modifications in the precise rigorous formulation of the convergence results. 
There are also some formal similarities:
compare for example \eref{e:paraprod} with the bounds on each of the three terms appearing
in \eref{e:splitProd}.
The main philosophical difference is that the approach presented
here is very local in nature, as opposed to the more global approach used in Bony's paraproduct.
It is also more general, allowing for an arbitrary number of controls which do themselves 
have small-scale structures that are linked to each other. 
As a consequence, the current work also puts a strong emphasis on the highly non-trivial
algebraic structures underlying our construction. In particular, we allow
for rather sophisticated renormalisation procedures going beyond the usual Wick ordering, which is
something that is required in several of the examples presented above.

\subsubsection{Colombeau's generalised functions}

In the early eighties, Colombeau introduced an algebra $\GG(\R^d)$ of generalised functions
on $\R^d$ (or an open subset thereof) with the property that $\CS'(\R^d) \subset \GG(\R^d)$
where $\CS'$ denotes the usual Schwartz distributions \cite{MR701451,MR738781}. 
Without entering into too much detail, $\GG(\R^d)$
is essentially defined as the set of smooth functions from $\CS(\R^d)$, the set of Schwartz test functions,
into $\R$, quotiented by a certain natural equivalence relation.

Some (but not all) generalised functions have an ``associated distribution''. In other words,
the theory comes with a kind of ``projection operator'' $P \colon \GG(\R^d) \to \CS'(\R^d)$ which is a 
left inverse for the injection $\iota \colon \CS'(\R^d) \hookrightarrow \GG(\R^d)$. However, it is important to note that
the domain of definition of $P$ is \textit{not} all of $\GG(\R^d)$.
Furthermore, the product in $\GG(\R^d)$ behaves as one would expect on the images of objects
that one would classically know how to multiply. For example, if $f$ and $g$ are continuous functions,
then $P((\iota f) (\iota g)) = fg$. The same holds true if $f$ is a  smooth function and $g$ is a distribution.

There are some similarities between the theory of regularity structures and that of Colombeau generalised 
functions. For example, just like elements in $\GG$,
elements in the spaces $\CD^\alpha$ (see Definition~\ref{def:Dgamma} below) contain more information
than what is strictly required in order to reconstruct the corresponding distribution.
The theory of regularity structures involves a reconstruction operator $\CR$,
which plays a very similar role to the operator $P$ from the theory of Colombeau's 
generalised functions by allowing to discard that
additional information. Also, both theories allow to provide a rigorous mathematical interpretation
of some of the calculations performed in the context of quantum field theory.

One major difference between the two theories is that the theory of regularity structures has more flexibility
built in. Indeed, it allows some freedom in the definition of the product between elements of the
``model'' used for performing the local Taylor expansions. This allows to account for the fact that 
taking limits along different smooth approximations might in general yield different answers. 
(A classical example is the fact that $\sin(x/\eps) \to 0$ in any reasonable topology where it does converge,
while $\sin^2(x/\eps) \to 1/2$. More sophisticated effects of this kind can easily be encoded in
a regularity structure, but are invisible to the theory of Colombeau's generalised functions.)
This could be viewed as a disadvantage of the theory of regularity structures:  
it requires substantially more effort on the part of the ``user'' in order to specify the theory
completely in a given example. Also, there isn't just ``one'' regularity structure: the precise
algebraic structure that is suitable for analysing a given problem does depend a lot on the problem
in question. However, we will see in Section~\ref{sec:SPDE} that there is a general procedure
allowing to build a large class of regularity structures arising in the analysis of semilinear SPDEs in
a unified way.

\subsubsection{White noise analysis}

One theory that in principle allows to give some meaning to \eref{e:Phi4}, \eref{e:PAM}, and \eref{e:NS}
(but to the best of the author's knowledge
not to \eref{e:PAMGen} or \eref{e:KPZ} with non-constant coefficients) is the theory of ``white noise analysis'' (WNA),
exposed for example in \cite{MR2571742} (see also \cite{MR0451429,MR1160391} for some of 
the earlier works). For example, the case of the stochastic Navier-Stokes equations has been
considered in \cite{MR2050201}, while the case of a stochastic version of the nonlinear heat equation was
considered in \cite{MR1451997}.
Unfortunately, WNA has a number of severe drawbacks
that are not shared by the theory of regularity structures:
\begin{list}{\labelitemi}{\leftmargin=\parindent}
\item Solutions in the WNA sense typically do not consist of random variables but of 
``Hida distributions''. As a consequence, only some suitable moments are obtained by this theory,
but no actual probability distributions and / or random variables.
\item Solutions in the WNA sense are typically \textit{not} obtained as limits of classical solutions
to some regularised version of the problem. As a consequence, their physical interpretation is
unclear. As a matter of fact, it was shown in \cite{MR1743612} that the WNA solution to the KPZ equation
exhibits a physically incorrect large-time behaviour, while the Cole-Hopf solution
(which can also be obtained via a suitable regularity structure, see \cite{KPZ})
is the physically relevant solution \cite{MR1462228}.
\end{list}
There are exceptions to these two rules (usually when the only ill-posed product is of the 
form $F(u)\cdot \xi$ with $\xi$ some white noise, and the problem is parabolic), 
and in such cases the solutions obtained 
by the theory of regularity structures typically ``contain'' the solutions obtained by WNA.
On the other hand, white noise analysis (or, in general, the Wiener chaos decomposition
of random variables)
is a very useful tool when building explicit models associated to a Gaussian noise. This will
be exploited in Section~\ref{sec:Gaussian} below.

\subsubsection{Rough paths}

The theory of rough paths was originally developed in \cite{Lyons} in order to 
interpret solutions to controlled differential equations of the type
\begin{equ}
dY(t) = F(Y)\,dX(t)\;,
\end{equ}
where $X \colon \R^+ \to \R^m$ is an irregular function and $F\colon \R^d \to \R^{dm}$ is
a sufficiently regular collection of vector fields on $\R^d$.
This can be viewed as an instance of the general problem \eref{e:generalProblem} if we
set $\CL = \d_t$ and $\xi = {dX \over dt}$, which is now a rather irregular distribution.
It turns out that, in the case of H\"older-regular
rough paths, the theory of rough paths can be recast into our framework. It can then
be interpreted as one particular class of regularity structures (one for each pair $(\alpha,m)$,
where $m$ is the dimension of the rough path and $\alpha$ its index of H\"older regularity),
with the corresponding space of rough paths being identified with the associated space of
models. Indeed, the theory of rough paths, and particularly the theory of controlled
rough paths as developed in \cite{MR2091358,Trees}, was one major source of inspiration of the 
present work. See Section~\ref{sec:RP} below for more details on the link between
the two theories.

\subsection{Notations}

Given a distribution $\xi$ and a test function $\phi$, we will use indiscriminately the notations
$\scal{\xi,\phi}$ and $\xi(\phi)$ for the evaluation of $\xi$ against $\phi$. We will also sometimes use
the abuse of notation $\int \phi(x)\,\xi(x)\,dx$ or $\int \phi(x)\,\xi(dx)$.

Throughout this article, we will always work with multiindices on $\R^d$. A multiindex $k$
is given by a vector $(k_1, \ldots,k_d)$ with each $k_i \ge 0$ a positive integer.
For $x \in \R^d$, we then write $x^k$ as a shorthand for $x_1^{k_1}\cdots x_d^{k_d}$.
The same notation will still be used when $X \in T^d$ for some algebra $T$. For a sufficiently
regular function $g \colon \R^d \to \R$, we write $D^k g(x)$ as a shorthand for
$\d_{x_1}^{k_1}\cdots \d_{x_d}^{k_d} g(x)$. We also write $k!$ as a shorthand for $k_1!\cdots k_d!$.

Finally, we will write $a \wedge b$ for the minimum of $a$ and $b$ and $a \vee b$ for the maximum.

\subsection*{Acknowledgements}

{\small
I am very grateful to M.~Gubinelli and to H.~Weber for our numerous discussions on quantum field
theory, renormalisation, rough paths, paraproducts, Hopf algebras, etc. 
These discussions were of enormous help
in clarifying the concepts presented in this article. Many other people provided valuable input that helped
shaping the theory. In particular, I would like to mention A.~Debussche, B.~Driver,
P.~Friz, J.~Jones, D.~Kelly, X.-M.~Li, M.~Lewin, T.~Lyons, J.~Maas, K.~Matetski, J.-C.~Mourrat, N.~Pillai, 
D.~Simon, T.~Souganidis, J.~Unterberger, and L.~Zambotti. Special thanks are due 
to L.~Zambotti
for pointing out a mistake in an earlier version of 
the definition of the class of regularity structures
considered in Section~8.

Financial support was kindly
provided by the Royal Society through a Wolfson Research Merit Award and by the Leverhulme 
Trust through a Philip Leverhulme Prize.
}

\section{Abstract regularity structures}
\label{sec:regStructure}

We start by introducing the abstract notion of a ``regularity structure'', which was already
mentioned in a loose way in the introduction, and which permeates the entirety of this work.

\begin{definition}\label{def:regStruct}
A \textit{regularity structure} $\TT = (A, T, G)$ consists of the following elements:
\begin{mylist}
\item An index set $A \subset \R$ such that $0 \in A$, $A$ is bounded from below, and $A$ is locally finite.
\item A \textit{model space} $T$, which is a graded vector space $T = \bigoplus_{\alpha \in A} T_\alpha$,
with each $T_\alpha$ a Banach space. Furthermore, $T_0 \approx \R$ and its unit vector
is denoted by $\one$.
\item A \textit{structure group} $G$ of linear operators acting on $T$ such that, for every $\Gamma \in G$, every $\alpha \in A$,
and every $a \in T_\alpha$, one has
\begin{equ}[e:coundGroup]
\Gamma a - a \in \bigoplus_{\beta < \alpha} T_\beta\;.
\end{equ}
Furthermore, $\Gamma \one = \one$ for every $\Gamma \in G$.
\end{mylist}
\end{definition}

\begin{remark}
It will sometimes be an advantage to consider $G$ as an abstract group, together with
a representation $\Gamma$ of $G$ on $T$. This point of view will be very natural 
in the construction of Section~\ref{sec:solSPDE} below. We will then sometimes use the notation
$g \in G$ for the abstract group element, and $\Gamma_g$ for the corresponding
linear operator. For the moment however, we identify
elements of $G$ directly with linear operators on $T$ in order to reduce the notational overhead.
\end{remark}

\begin{remark}
Recall that the elements of
$T = \bigoplus_{\alpha \in A} T_\alpha$ are \textit{finite} series of the type
$a = \sum_{\alpha \in A} a_\alpha$ with $a_\alpha \in T_\alpha$. All the operations that we will
construct in the sequel will then make sense component by component. 
\end{remark}

\begin{remark}
A good analogy to have in mind is the space of all polynomials, which will be explored in 
detail in Section~\ref{sec:canonical} below. In line with this analogy, we say that $T_\alpha$ consists of elements that
are \textit{homogeneous} of order $\alpha$.
In the particular case of polynomials in commuting indeterminates
our theory boils down to the very familiar theory of Taylor expansions on $\R^d$,
so that the reader might find it helpful to read the present section 
and Section~\ref{sec:canonical} in parallel to help build an intuition.
The reader familiar with the theory of rough paths \cite{Lyons} will also find it helpful to simultaneously
read Section~\ref{sec:RP} which shows how the theory of rough paths (as well as the theory of ``branched rough paths''
\cite{Trees}) fits within our framework.
\end{remark}

The idea behind this definition is that $T$ is a space whose elements describe 
the ``jet'' or ``local expansion'' of a function
(or distribution!) $f$ at any given point. One should then think of $T_\alpha$ as encoding the 
information required to describe $f$ locally ``at order $\alpha$''  in the sense that, at scale $\eps$, 
elements of $T_\alpha$ describe fluctuations of size $\eps^\alpha$. 
This interpretation will be made much clearer below, but at an intuitive level it already
shows that a regularity structure with $A \subset \R_+$ will describe functions, while a regularity
structure with $A \not\subset \R_+$ will also be able to describe distributions. 

The role of the structure group $G$ will be to translate coefficients from a local expansion around
a given point into coefficients for an expansion around a different point. Keeping in line with the 
analogy of Taylor expansions, the coefficients of a Taylor polynomial are just given by the partial 
derivatives of the underlying function $\phi$ at some point $x$. However, in order to compare the Taylor polynomial
at $x$ with the Taylor polynomial at $y$, it is not such a good idea to compare the coefficients themselves.
Instead, it is much more natural to first translate the first polynomial by the quantity $y-x$. In the case of polynomials
on $\R^d$, the structure group $G$ will therefore simply be given by $\R^d$ with addition as its group property, 
but we will see that non-abelian structure groups arise naturally in more general situations.
(For example, the structure group is non-Abelian in the theory of rough paths.)

Before we proceed to a study of some basic properties of regularity structures, let us introduce a
few notations.
For an element $a \in T$, we write $\CQ_\alpha a$ for the component of $a$ in $T_\alpha$
and $\|a\|_\alpha = \|\CQ_\alpha a\|$ for its norm.
We also use the shorthand notations
\begin{equ}[e:defTalpha+]
T_\alpha^+ = \bigoplus_{\gamma \ge \alpha} T_\gamma\;,\qquad
T_\alpha^-  = \bigoplus_{\gamma < \alpha} T_\gamma\;,
\end{equ}
with the conventions that $T_\alpha^+ = \{0\}$ if $\alpha > \max A$ and $T_\alpha^- = \{0\}$ if $\alpha \le \min A$.
We furthermore denote by $L_0^-(T)$ the space of all operators $L$ on $T$
such that $L a \in T_\alpha^-$ for $a \in T_\alpha$ and by $L^-$ the set of operators $L$ such that
$L-1 \in L_0^-$, so that $G \subset L^-$.

The condition that $\Gamma a - a \in T_\alpha^-$ for $a \in T_\alpha$, together with the fact that the index set 
$A$ is bounded from below, implies that, for every $\alpha \in A$ there exists $n>0$ such that 
$(\Gamma - 1)^n T_\alpha = 0$ for every $\Gamma \in G$.
In other words, $G$ is necessarily nilpotent. 
In particular, one can define a
function $\log\colon G \to L_0^-$ by
\begin{equ}[e:defLog]
\log \Gamma = \sum_{k=1}^n {(-1)^{k+1}\over k} (\Gamma -1)^k\;.
\end{equ}
Conversely, one can define an exponential map $\exp\colon L_0^- \to L^-$ 
by its  Taylor series, and one has the rather unsurprising
identity $\Gamma = \exp(\log \Gamma)$. As usual in the theory of Lie groups, we write
$\g = \log G$ as a shorthand.

A useful definition will be the following:

\begin{definition}\label{def:sector}
Given a regularity structure as above and some $\alpha \le 0$, 
a \textit{sector} $V$ of regularity $\alpha$ is a graded subspace
$V = \bigoplus_{\beta \in A} V_\beta$ with $V_\beta \subset T_\beta$ having the following properties.
\begin{mylist}
\item One has $V_\beta = \{0\}$ for every $\beta < \alpha$.
\item The space $V$ is invariant under $G$, i.e.\ $\Gamma V \subset V$ for every $\Gamma \in G$.
\item For every $\beta \in A$, there exists a complement $\bar V_\beta \subset T_\beta$ such that
$T_\beta$ is given by the direct sum $T_\beta = V_\beta \oplus \bar V_\beta$.
\end{mylist}
A sector of regularity $0$ is also called \textit{function-like} for reasons that will become clear
in Section~\ref{sec:reconFcn}.
\end{definition}

\begin{remark}
The regularity of a sector will always be less or equal to zero. In the case
of the regularity structure generated by polynomials for example, any non-trivial
sector has regularity $0$ since it always has to contain the element $\one$.
See Corollary~\ref{cor:regularity} below for a justification of this terminology.
\end{remark}

\begin{remark}\label{rem:sectorReg}
Given a sector $V$, we can define $A_V \subset A$ as the set of indices $\alpha$ such that
$V_\alpha \neq \{0\}$. If $\alpha > 0$, our definitions then ensure that $\TT_V = (V,A_V, G)$ is again a regularity 
structure with $\TT_V \subset \TT$. (See below for the meaning of such an inclusion.) 
It is then natural to talk about a subsector $W \subset V$ if
$W$ is a sector for $\TT_V$.
\end{remark}

\begin{remark}
Two natural non-empty sectors are given by $T_0 = \lspan\{\one\}$ 
and by $T_\alpha$ with $\alpha = \min A$. In both cases, $G$ automatically acts on them in
a trivial way.
Furthermore, as an immediate consequence of the definitions, given a sector $V$ of regularity $\alpha$
and a real number $\gamma > \alpha$, the space $V \cap T_\gamma^-$ is again a sector of regularity $\alpha$.
\end{remark}

In the case of polynomials on $\R^d$, typical examples of sectors would be given by the
set of polynomials depending only on some subset of the variables or by the set of polynomials
of some fixed degree.

\subsection{Basic properties of regularity structures}
\label{sec:regBasic}

The smallest possible regularity structure is given by $\TT_0 = (\{0\}, \R, \{1\})$,
where $\{1\}$ is the trivial group consisting only of the identity operator, and with $\one = 1$. 
This ``trivial'' regularity structure is the smallest possible structure that accommodates
the local information required to describe an arbitrary continuous function, i.e.\ simply the value
of the function at each point. 

The set of all regularity structures comes with a natural partial order.
Given two regularity structures $\TT = (A, T, G)$ and $\bar \TT = (\bar A, \bar T, \bar G)$
we say that $\TT$ contains $\bar \TT$ and write $\bar \TT \subset \TT$ if the following holds.
\begin{claim}
\item One has $\bar A \subset A$.
\item There is an injection $\iota \colon \bar T \to T$ such that, for every $\alpha \in \bar A$, one has $\iota(\bar T_\alpha) \subset T_\alpha$.
\item The space $\iota(\bar T)$ is invariant under $G$ and the map 
$j\colon  G \to L(\bar T,\bar T)$ defined by the identity $j \Gamma  = \iota^{-1} \Gamma \iota$
is a surjective group homomorphism from $G$ to $\bar G$.
\end{claim}
With this definition, one has $\TT_0 \subset \TT$ for every regularity structure $\TT$, with $\iota 1 = \one$ and
$j$ given by the trivial homomorphism.

One can also define the product $\hat \TT = \TT \otimes \bar \TT$ of two regularity structures
 $\TT = (A, T, G)$ and $\bar \TT = (\bar A, \bar T, \bar G)$
by  $\hat \TT = (\hat A, \hat T, \hat G)$ with
\begin{claim}
\item $\hat A = A + \bar A$,
\item $\hat T = \bigoplus_{(\alpha, \beta)} T_\alpha \otimes \bar T_\beta$ and 
$\hat T_\gamma = \bigoplus_{\alpha + \beta = \gamma} T_\alpha \otimes \bar T_\beta$, where
both sums run over pairs $(\alpha, \beta) \in A \times \bar A$,
\item $\hat G = G \otimes \bar G$,
\end{claim}
Setting $\hat \one = \one \otimes \bar \one$ (where $\one$ and $\bar \one$ are the unit elements
of $\TT$ and $\bar \TT$ respectively), it is easy to verify that this definition satisfies all the required axioms
for a regularity structure. If the individual components of $T$ and / or $\bar T$ are infinite-dimensional, 
this construction does of course rely on choices of tensor products for $T_\alpha \otimes \bar T_\beta$.

\begin{remark}
One has both $\TT \subset \TT \otimes \bar \TT$ and $\bar \TT \subset \TT \otimes \bar \TT$
with obvious inclusion maps. Furthermore, one has $\TT \otimes \TT_0 \approx \TT$ for the trivial regularity
structure $\TT_0$.
\end{remark}

\subsection{The polynomial regularity structure}
\label{sec:canonical}

One very important example to keep in mind for the abstract theory of regularity structures
presented in the main part
of this article is that generated by polynomials in $d$ commuting variables.
In this case, we simply recover the usual theory of Taylor expansions / regular functions in $\R^d$.
However, it is still of interest since it helps building our intuition and provides a nicely unified way of
treating regular functions with different scalings.

In this case, the model space $T$ consists of all abstract polynomials in $d$ indeterminates.
More precisely, we have $d$ ``dummy variables'' $\{X_i\}_{i=1}^d$ and $T$ consists of polynomials in $X$.
Given a multiindex $k = (k_1,\ldots,k_d)$, we will  use throughout this article the shorthand notation
\begin{equ}
X^k \eqdef X_1^{k_1}\cdots X_d^{k_d}\;.
\end{equ}
Finally, we denote by $\one = X^0$ the ``empty'' monomial. 

In general, we will be interested in situations where different variables come with different degrees of
homogeneity. A good example to keep in mind is that of parabolic equations, where the linear operator
is given by $\d_t - \Delta$, with the Laplacian acting on the spatial coordinates. By homogeneity, it 
is then natural to make powers of $t$ ``count double''.
In order to implement this classical idea, 
we assume from now on that we fix a \textit{scaling}\label{lab:scaling} $\s\in \N^d$ of $\R^d$, which is 
simply a vector of strictly positive relatively prime integers. The Euclidean scaling
is simply given by $\s_c = (1,\ldots,1)$. 

Given such a scaling, we defined the ``scaled degree'' of a multiindex $k$ by
\begin{equ}[e:scaledk]
|k|_\s = \sum_{i=1}^d \s_i\, k_i \;.
\end{equ}
With this notation we define, for every $n \in \N$, the subspace $T_n \subset T$ by
\begin{equ}
T_n = \lspan \{X^k\,:\, |k|_\s = n\}\;.
\end{equ}
For a monomial $P$ of the type $P(X) = X^k$, we then refer to $|k|_\s$ as the scaled degree of $P$.
Setting $A = \N$, we have thus constructed the first two components of a regularity structure. 

Our structure comes with a natural model, which is given by the concrete realisation of
an abstract polynomial as a function on $\R^d$. More precisely, for every $x\in \R^d$,
we have a natural linear map $\Pi_x\colon T \to \CC^\infty(\R^d)$ given by  
\begin{equ}[e:canonicalPi]
\bigl(\Pi_x X^k\bigr)(y) = (y-x)^k \;.
\end{equ}
In other words, given any ``abstract polynomial'' $P(X)$, $\Pi_x$ realises it as a concrete polynomial
on $\R^d$ based at the point $x$.

This suggests that there is a natural action of $\R^d$ on $T$ which
simply shifts the base point $x$. This is precisely the action that is described by the group $G$
which is the last ingredient missing to obtain a regularity structure.
As an abstract group, $G$ will simply be a copy of $\R^d$ endowed with addition as its group
operation. For any $h \in \R^d\approx G$, the action of $\Gamma_h$ on an abstract polynomial is then given by
\begin{equ}
\bigl(\Gamma_h P\bigr)(X) = P(X+h)\;.
\end{equ}
It is obvious from our notation that one has the identities
\begin{equ}
\Gamma_h \circ \Gamma_{\bar h} = \Gamma_{h+\bar h}\;,\qquad  \Pi_{x+h} \Gamma_h = \Pi_{x}\;,
\end{equ}
which will play a fundamental role in the sequel.

The triple $(\N,T,G)$ constructed in this way thus defines a regularity structure, 
which we call $\TT_{d,\s}$. (It depends on the scaling $\s$ only in the way that $T$
is split into subspaces, so $\s$ does not explicitly appear in the definition of $\TT_{d,\s}$.)

In this construction, the space $T$ comes with more structure than just that of a regularity structure.
Indeed, it comes with a natural multiplication $\star$ given by
\begin{equ}
\bigl(P\star Q\bigr)(X) = P(X) Q(X)\;.
\end{equ}
It is then straightforward to verify that this representation satisfies the properties that
\begin{claim}
\item For $P \in T_m$ and $Q\in T_n$, one has $P\star Q \in T_{m+n}$.
\item The element $\one$ is neutral for $\star$.
\item For every $h \in \R^d$ and $P,Q \in T$, one has $\Gamma_h(P\star Q) = \Gamma_h P \star \Gamma_h Q$.
\end{claim}
Furthermore, there exists a natural element $\scal{\one,\cdot\,}$ in the dual of $T$ which consists of
formally evaluating the corresponding polynomial at the origin. More precisely, one sets $\scal{\one,X^k} = \delta_{k,0}$.

As a space of polynomials, $T$ arises naturally as the space in which the Taylor expansion
of a function $\phi \colon \R^d \to \R$ takes values. Given a smooth
function $\phi\colon \R^d \to \R$ and an integer $\ell \ge 0$, we 
can ``lift'' $\phi$ in a natural way to $T$ by computing its Taylor expansion of order 
less than $\ell$ at each point. 
More precisely, we set
\begin{equ}[e:TaylorNormal]
\bigl(\CT_\ell \phi\bigr)(x) = \sum_{|k|_\s < \ell}{X^k\over k !} D^k\phi(x)\;,
\end{equ}
where, for a given multiindex $k = (k_1,\ldots,k_d)$, $D^k \phi$ stands as usual 
for the partial derivative $\d_1^{k_1}\cdots \d_d^{k_d} \phi(x)$. 
It then follows immediately from the general
Leibniz rule that for $\CC^\ell$ functions, $\CT_\ell$ is ``almost'' an algebra morphism, in the sense that
in addition to being linear, one has
\begin{equ}[e:Leibniz]
 \CT_\ell (\phi\cdot \psi)(x) = \CT_\ell \phi(x) \star \CT_\ell \psi(x) + R(x)\;,
\end{equ}
where the remainder $R(x)$ is a sum of homogeneous terms of scaled 
degree greater or equal to $\ell$.

We conclude this subsection by defining the classes $\CC^\alpha_\s$ of functions that are $\CC^\alpha$
with respect to a given scaling $\s$.
Recall that, for $\alpha \in (0,1]$,
the class $\CC^\alpha$ of ``usual'' $\alpha$-H\"older continuous functions is given by those functions
$f$ such that $|f(x) - f(y)| \lesssim |x-y|^\alpha$, uniformly over $x$ and $y$ in any compact set.
For any $\alpha > 1$, we can then define $\CC^\alpha$ recursively as consisting of functions
that are continuously differentiable and such that each directional derivative belongs to $\CC^{\alpha-1}$.

\begin{remark}\mhpastefig{warning}
In order to keep our notations consistent,
we have slightly strayed from the usual conventions by declaring a function to be of class $\CC^1$ even if it
is only Lipschitz continuous. A similar abuse of notation will be repeated for all positive integers,
and this will be the case throughout this article.
\end{remark}

\begin{remark}
We could have defined the spaces 
$\CC^\alpha$ for $\alpha \in [0,1)$ (note the missing point $1$!) 
similarly as above, but replacing the bound on $f(x) - f(y)$ by
\begin{equ}[e:boundAlpha2]
\lim_{|h| \to 0} |f(x+h) - f(x)|/|h|^\alpha = 0\;,
\end{equ}
imposing uniformity of the convergence for $x$ in any compact set.
If we extended this definition to $\alpha \ge 1$ recursively as above, this 
would coincide with the usual spaces $\CC^k$ for integer $k$, 
but the resulting spaces would be slightly smaller than the H\"older spaces 
for non-integer values.
(In fact, they would then coincide with the closure of smooth functions under the $\alpha$-H\"older 
norm.) Since the bound \eref{e:boundAlpha2} includes a supremum and a limit rather than
just a supremum, we prefer to stick with the definition given above. 
\end{remark}

Keeping this characterisation in mind, one nice feature of the regularity structure just described 
is that it provides a very natural ``direct'' characterisation of $\CC^\alpha$ for any $\alpha > 0$
without having to resort to an inductive construction. 
Indeed, in the case of the classical Euclidean scaling $\s = (1,\ldots,1)$, we have the following result, where
for $a \in T$, we denote by $\|a\|_m$ the norm of the component of $a$ in $T_m$.

\begin{lemma}\label{lem:Lipschitz}
A function $\phi \colon \R^d \to \R$ is of class $\Lip^{\alpha}$ with $\alpha > 0$ if and only if 
there exists a function $\hat \phi \colon \R^d \to T_\alpha^-$ such that
$\scal{\one, \hat\phi(x)} = \phi(x)$ and such that
\begin{equ}[e:charactCalpha]
\|\hat \phi(x+h) - \Gamma_h \hat \phi(x)\|_m \lesssim  |h|^{\alpha-m}\;,
\end{equ}
uniformly over $m < \alpha$, $|h| \le 1$ and $x$ in any compact set.
\end{lemma}

\begin{proof}
For $\alpha \in (0,1]$, \eref{e:charactCalpha} is just a rewriting of the definition of $\Lip^{\alpha}$.
For the general case, denote by $\CD^\alpha$ the space of $T$-valued functions such that \eref{e:charactCalpha}
holds. Denote furthermore by $\CD_i\colon T \to T$ the linear map defined by 
$\CD_i X_j = \delta_{ij} \one$ and extended to higher powers of $X$ by the Leibniz rule.
For $\hat \phi \in \CD^\alpha$ with $\alpha > 1$, we then have that:
\begin{mylist}
\item The bound \eref{e:charactCalpha} for $m = 0$ implies that $\phi = \scal{\one,\hat \phi}$ 
is differentiable at $x$ with $i$th directional derivative given by $\d_i \phi(x) = \scal{\one,\CD_i \hat \phi(x)}$. 
\item The case $m = 1$ implies that the derivative $\d_i \phi$ is itself continuous. 
\item Since the operators $\CD_i$ commute with $\Gamma_h$ for every $h$, one has $\CD_i \phi \in \CD^{\alpha-1}$ for every $i \in \{1,\ldots,d\}$.
\end{mylist}
The claim then follows at once from the fact that this is precisely the recursive characterisation
of the spaces $\CC^\alpha$.
\end{proof}

This now provides a very natural generalisation of H\"older spaces of arbitrary order
to non-Euclidean scalings.
Indeed, to a scaling $\s$ of $\R^d$, we can naturally associate the metric $d_\s$ on $\R^d$ given by
\begin{equ}[e:defds]
d_\s(x,y) \eqdef \sum_{i=1}^d |x_i - y_i|^{1/\s_i}\;.
\end{equ}
We will also use in the sequel the notation $|\s| = \s_1+\ldots+\s_d$, 
which plays the role of a dimension. Indeed, with respect to the metric $d_\s$, the unit ball in $\R^d$ is
easily seen to have Hausdorff dimension $|\s|$ rather than $d$.
Even though the right hand side of \eref{e:defds} does not define a norm (it is not 
$1$-homogeneous, at least not in the usual sense), we will usually use the notation
$d_\s(x,y) = \|x-y\|_\s$.

\begin{remark}\label{rem:normSmooth}
It may occasionally be more convenient to use a metric with the same scaling
properties as $d_\s$ which is smooth away from the origin. In this case, one can for example
take $p = 2\,\mathop{\mathrm{lcm}}(\s_1,\ldots,\s_d)$ and set
\begin{equ}
\tilde d_\s(x,y) \eqdef \Bigl(\sum_{i=1}^d |x_i - y_i|^{p/\s_i}\Bigr)^{1/p}\;.
\end{equ}
It is easy to see that $\tilde d_\s$ and $d_\s$ are equivalent in the sense that they are bounded by fixed 
multiples of each other. In the Euclidean setting, $d_\s$ would be the $\ell^1$ distance, 
while $\tilde d_\s$ would be the $\ell^2$ distance.
\end{remark}

With this notation at hand, and in view of Lemma~\ref{lem:Lipschitz},
the following definition is very natural:

\begin{definition}\label{def:Calphas}
Given a scaling $\s$ on $\R^d$ and $\alpha > 0$, we say that a function 
$\phi \colon \R^d \to \R$ is of class $\Lip^{\alpha}_\s$ if there exists a function $\hat \phi \colon \R^d \to T_\alpha^-$
with $\scal{\one, \hat \phi(x)} = \phi(x)$ for every $x$ and
such that, for every compact set $\K \subset \R^d$, one has
\begin{equ}[e:boundAlpha]
\|\hat \phi(x+h) - \Gamma_h \hat \phi(x)\|_m  \lesssim \|h\|_\s^{\alpha-m}\;,
\end{equ}
uniformly over $m < \alpha$, $\|h\|_\s \le 1$ and $x\in \K$.
\end{definition}

\begin{remark}
One can verify that the map $x \mapsto \|x\|_\s^\alpha$ is in $\CC^\alpha_\s$ for 
$\alpha \in (0,1]$. Another well-known example \cite{MR876085,SPDENotes} is that the solutions
to the additive stochastic heat equation on the real line belong to
$\CC^\alpha_\s(\R^2)$ for every $\alpha < {1\over 2}$, provided that
the scaling $\s$ is the parabolic scaling $\s = (2,1)$. (Here, the first component is
the time direction.) 
\end{remark}

\begin{remark}\label{rem:uniquehatphi}
The choice of $\hat \phi$ in Definition~\ref{def:Calphas} is essentially unique in the sense that any two choices $\hat \phi_1$ and $\hat \phi_2$
satisfy $\CQ_\ell \hat \phi_1(x) = \CQ_\ell \hat \phi_2(x)$ for every $x$ and every $\ell < \alpha$.
(Recall that $\CQ_\ell$ is the projection onto $T_\ell$.)
This is because, similarly to the proof of Lemma~\ref{lem:Lipschitz}, one can show that the components
in $T_\ell$ have to coincide with the corresponding directional derivatives of $\phi$ at $x$, and that, 
if \eref{e:boundAlpha} is satisfied locally uniformly in $x$, these
directional derivatives exist and are continuous.
\end{remark}

\subsection{Models for regularity structures}
\label{sec:realisation}

In this section, we introduce the key notion of a ``model''
for a regularity structure, which was already alluded to several times in the introduction. 
Essentially, a model associates to each ``abstract'' element in $T$ a ``concrete'' 
function or distribution on $\R^d$. In the above example, such a model was given by
an interplay of the maps $\Pi_x$ that would associate to $a \in T$ a polynomial on $\R^d$ centred around $x$, 
and the maps $\Gamma_h$ that allow to translate the polynomial in 
question to any other point in $\R^d$. 

This is the structure that we are now going to generalise and this is where our theory departs
significantly from the theory of jets, as our model will typically contain elements that are
extremely irregular.
If we take again the case of the polynomial regularity structures as our 
guiding principle, we note that the index $\alpha \in A$ describes the speed at which functions of the form 
$\Pi_x a$ with $a \in T_\alpha$ vanish near $x$. The action of $\Gamma$ is then necessary 
in order to ensure
that this behaviour is the same at every point. In general, elements in the image of $\Pi_x$ 
are distributions and
not functions and the index $\alpha$ can be negative, so how do we describe the behaviour near a point?

One natural answer to this question is to test the distribution in question against approximations to a delta function and
to quantify this behaviour. 
Given a scaling $\s$, we thus define scaling maps
\begin{equ}[e:scalingMap]
\CS_\s^\delta \colon \R^d \to \R^d\;,\quad \CS_\s^\delta(x_1,\ldots,x_d) = (\delta^{-\s_1}x_1,\ldots,\delta^{-\s_d}x_d)\;.
\end{equ}
These scaling maps yield in a natural way a family of isometries on $L^1(\R^d)$ by
\begin{equ}[e:scaling]
\bigl(\CS_{\s,x}^\delta \phi\bigr)(y) \eqdef \delta^{-{|\s|}} \phi\bigl(\CS_\s^\delta (y-x)\bigr)\;.
\end{equ}
They are also the natural scalings under which $\|\cdot\|_\s$ behaves like a norm
in the sense that $\|\CS_\s^\delta x\|_\s = \delta^{-1} \|x\|_\s$.
Note now that if $P$ is a monomial of scaled degree $\ell \ge 0$ over $\R^d$ 
(where the scaled degree simply means that the monomial $x_i$ has degree $s_i$ rather than $1$)
and $\phi\colon \R^d \to \R$ is a compactly supported
function, then we have the identity
\begin{equs}
\int P(y-x) \bigl(\CS_{\s,x}^\delta \phi\bigr)(y) \,dy &= \int P(\delta^{\s_1}z_1,\ldots,\delta^{\s_d}z_d) \phi(z) \,dz \\
 &= \delta^{\ell}\int P(z) \phi(z) \,dz\;.\label{e:testPoly}
\end{equs}
Following the philosophy of taking the case of polynomials / Taylor expansions as our source of inspiration,
this simple calculation motivates the following definition.

\begin{definition}\label{def:model}
A \textit{model} for a given regularity structure $\TT = (A,T,G)$ on $\R^d$ 
with scaling $\s$ consists of the following elements:
\begin{mylist}
\item A map $\Gamma \colon \R^d\times \R^d \to G$ such that $\Gamma_{xx} = 1$, the identity operator, and 
such that $\Gamma_{xy}\, \Gamma_{yz} = \Gamma_{xz}$ for every $x,y,z$ in $\R^d$.
\item A collection of continuous linear maps $\Pi_x \colon T \to \CS'(\R^d)$ such that $\Pi_y = \Pi_x \circ \Gamma_{xy}$
for every $x,y \in \R^d$.
\end{mylist}
Furthermore, for every $\gamma > 0$ and every compact set $\K \subset \R^d$, there exists
a constant $C_{\gamma,\K}$ such that the bounds
\begin{equ}[e:boundPi]
|(\Pi_x a)\bigl(\CS_{\s,x}^\delta \phi\bigr)| \le C_{\gamma,\K} \|a\| \delta^{\ell}\;, \quad
\|\Gamma_{xy} a\|_m \le C_{\gamma,\K} \|a\| \, \|x-y\|_\s^{\ell-m}\;, 
\end{equ}
hold uniformly over all $x,y \in \K$, all $\delta \in (0,1]$, all smooth test functions
$\phi \colon B_\s(0,1) \to \R$ with $\|\phi\|_{\CC^{ r}}\le 1$, all $\ell \in A$ with $\ell < \gamma$, 
all $m < \ell$, and all $a \in T_\ell$.
Here, $r$ is the smallest integer such that $\ell > -r$ for every $\ell \in A$.
(Note that $\|\Gamma_{xy} a\|_m = \|\Gamma_{xy} a - a\|_m$ since $a \in T_\ell$ and $m < \ell$.)
\end{definition}

\begin{remark}
We will also sometimes call the pair $(\Pi,\Gamma)$ a \textit{model} for the regularity structure $\TT$.
\end{remark}

The following figure illustrates a typical example of model for 
a simple regularity structure
where $A = \{0,{1\over 2}, 1, {3\over 2}\}$ and each $T_\alpha$ is one-dimensional:
\begin{equ}
\mhpastefig[2/3]{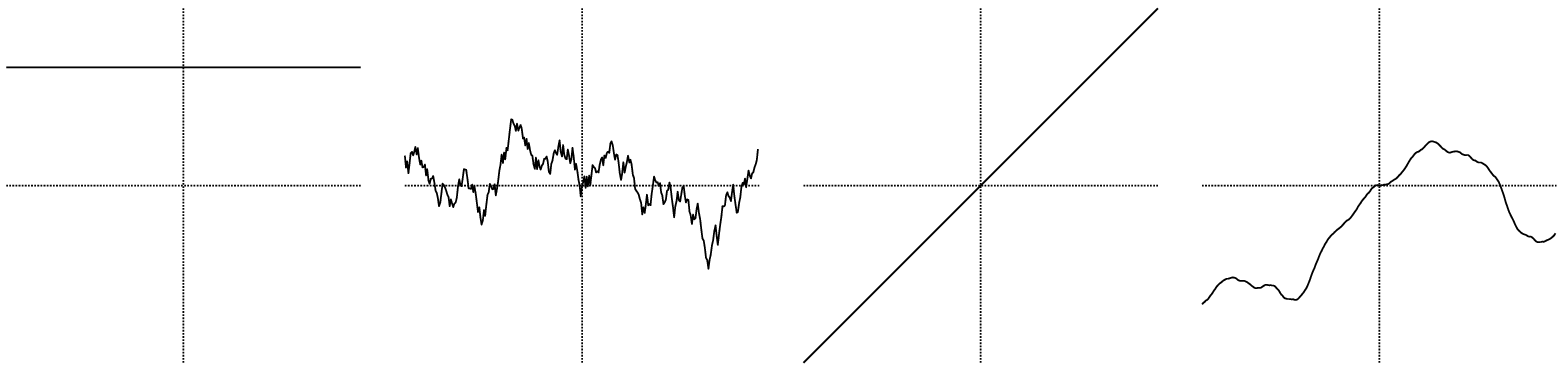}
\end{equ}
Write  $\tau_\alpha$ for the unit vector in $T_\alpha$. 
Given a ${1\over2}$-H\"older continuous function $f \colon \R \to \R$, the 
above picture has
\begin{equ}
\bigl(\Pi_x \tau_{1\over 2}\bigr)(y) = f(y) - f(x)\;,\quad
\bigl(\Pi_x \tau_{3\over 2}\bigr)(y) = \int_x^y \bigl(f(z) - f(x)\bigr)\,dz\;,
\end{equ}
while $\Pi_x \tau_0$ and $\Pi_x \tau_1$ are given by the canonical one-dimensional
model of polynomials.

A typical action of $\Gamma_{xy}$ is illustrated below:
\begin{equ}
\mhpastefig[2/3]{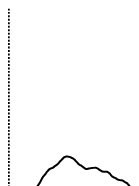}\qquad{\Gamma_{xy}\atop\Rightarrow}\qquad\mhpastefig[2/3]{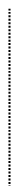}
\end{equ}
Here, the left figure
shows $\Pi_x \tau_{3\over 2}$, while the right figure shows $\Pi_y \tau_{3\over 2} = \Pi_x \Gamma_{xy} \tau_{3\over 2}$.
In this particular example, this 
is obtained from $\Pi_x \tau_{3\over 2}$ by adding a suitable affine function, i.e.\ 
a linear combination of $\Pi_x \tau_{0}$ and $\Pi_x \tau_{1}$.

\begin{remark}
Given a sector $V \subset T$, it will on
occasion be natural to consider models for $\TT_V$ rather than all of $\TT$.
In such a situation, we will say that $(\Pi, \Gamma)$ is a model for $\TT$
on $V$, or just a model for $V$.
\end{remark}

\begin{remark}\label{rem:normPiGamma}
Given a map $(x,y) \mapsto \Gamma_{xy}$ as above, the set of maps $x \mapsto \Pi_x$ as
above is actually a linear space. We can endow it with the natural system of seminorms 
$\|\Pi\|_{\gamma;\K}$
given by the smallest constant $C_{\gamma,\K}$ such that the first bound in \eref{e:boundPi} holds.
Similarly, we denote by $\|\Gamma\|_{\gamma;\K}$ the smallest constant $C_{\gamma,\K}$ such that 
the second bound in \eref{e:boundPi} holds.
Occasionally, it will be useful to have a notation for the combined bound, and we will then write
\begin{equ}[e:normModel]
\$Z\$_{\gamma;\K} = \|\Pi\|_{\gamma;\K} + \|\Gamma\|_{\gamma;\K}\;,
\end{equ}
where we set $Z = (\Pi,\Gamma)$.
\end{remark}

\begin{remark}\label{rem:altBound}
The first bound in \eref{e:boundPi} could alternatively have been formulated as 
$|(\Pi_x a)\bigl(\phi\bigr)| \le C \|a\| \delta^{\ell}$ for all smooth test functions
$\phi$ with support in a ball of radius $\delta$ around $x$ (in the $d_\s$-distance),
which are bounded by $\delta^{-|\s|}$ and such that their derivatives satisfy $\sup_x |D^\ell \phi(x)| \le \delta^{-|\s|-|\ell|_\s}$ for all 
multiindices $\ell$ of (usual) size less or equal to $r$.
\end{remark}

One important notion is that of an \textit{extension} of a model $(\Pi,\Gamma)$:

\begin{definition}\label{def:extension}
Let $\TT \subset \hat \TT$ be two regularity structures and let $(\Pi,\Gamma)$ be a model
for $\TT$. A model $(\hat \Pi,\hat \Gamma)$ is said to extend $(\Pi,\Gamma)$ for $\hat \TT$
if one has
\begin{equ}
\iota \Gamma_{xy} a = \hat \Gamma_{xy} \iota a\;,\qquad \Pi_x a = \hat \Pi_x \iota a\;,
\end{equ}
for every $a \in T$ and every $x,y$ in $\R^d$. Here, $\iota$ is as in Section~\ref{sec:regBasic}.
\end{definition}

We henceforth
denote by $\MM_\TT$ the set of all models of $\TT$,
which is a slight abuse of notation since one should
also fix the dimension $d$ and the scaling $\s$, but these are usually very clear from the context.
This space is endowed with a
natural system of pseudo-metrics by setting, 
for any two models $Z = (\Pi,\Gamma)$ and $\bar Z = (\bar \Pi, \bar \Gamma)$,
\begin{equs}[e:defDistModel]
\$Z ; \bar Z\$_{\gamma;\K} &\eqdef  \|\Pi - \bar \Pi\|_{\gamma;\K} + \|\Gamma - \bar \Gamma\|_{\gamma;\K}\;.
\end{equs}
While $\$\cdot ; \cdot \$_{\gamma;\K}$ defined in this way looks very much like
a seminorm, the space $\MM_\TT$ is \textit{not} a linear space due to the two nonlinear constraints
\begin{equ}[e:algebraic]
\Gamma_{xy}\Gamma_{yz} = \Gamma_{xz}\;,\quad\text{and}\quad \Pi_y = \Pi_x \circ \Gamma_{xy}\;,
\end{equ}
and due to the fact that $G$ is not necessarily a linear set of operators.
While $\MM_\TT$ is not linear,
it is however an algebraic variety in some infinite-dimensional Banach space.

\begin{remark}\label{rem:polynomStructures}
In most cases considered below, our regularity structure contains $\TT_{d,\s}$ for some
dimension $d$ and scaling $\s$.
In such a case, we denote by $\bar T \subset T$ the image of the model space of $\TT$ in $T$ under the 
inclusion map and we only consider models $(\Pi,\Gamma)$ that extend (in the sense of
Definition~\ref{def:extension}) the polynomial model
on $\bar T$. It is straightforward to verify that the polynomial model does indeed verify
the bounds and algebraic relations of Definition~\ref{def:model}, provided that we make
the identification $\Gamma_{xy} \sim \Gamma_h$ with $h = x-y$.
\end{remark}

\begin{remark}\label{rem:boundPi}
If, for every $a \in T_\ell$, $\Pi_x a$ happens to be a function such that 
$|\Pi_x a(y)| \le C \|x-y\|_\s^\ell$ for $y$ close to $x$, then the first bound in \eref{e:boundPi} holds for
$\ell \ge 0$. Informally, it thus states that $\Pi_x a$ behaves ``as if'' it were $\ell$-H\"older continuous
at $x$. The formulation given here has the very significant advantage that it also makes sense for negative values of $\ell$.
\end{remark}

\begin{remark}
Given a linear map $\bar \Pi \colon T \to \CS'(D)$, and a function $F \colon D \to G$, 
we can always set
\begin{equ}[e:genericFPi]
\Gamma_{xy} = F(x) \cdot F(y)^{-1}\;,\quad \Pi_x = \bar \Pi \circ F(x)^{-1}\;.
\end{equ}
Conversely, given a model $(\Pi,\Gamma)$ as above and a reference point $o$, we could set
\begin{equ}[e:refPoint]
F(x) = \Gamma_{xo}\;,\qquad \bar \Pi = \Pi_{o}\;,
\end{equ}
and $\Gamma$ and $\Pi$ could then be recovered from $F$ and $\bar \Pi$ by \eref{e:genericFPi}. 
The reason why we choose to keep our seemingly redundant formulation is that 
the definition \eref{e:defDistModel} and the bounds \eref{e:boundPi} are more natural in this formulation.
We will see in Section~\ref{sec:realAlg} below 
that in all the cases mentioned in the introduction, there are natural
maps $\bar \Pi$ and $F$ such that $(\Pi,\Gamma)$ are given by \eref{e:genericFPi}.
These are however \textit{not} of the form \eref{e:refPoint} for any reference point. 
\end{remark}

\begin{remark}
It follows from the definition
\eref{e:defLog} that the second bound in \eref{e:boundPi} is equivalent to the bound
\begin{equ}[e:boundGamma]
\|\log \Gamma_{xy} a\|_m \lesssim \|a\| \|x-y\|_\s^{\ell-m}\;,
\end{equ}
for all $a \in T_\ell$. Similarly, one can consider instead of \eref{e:defDistModel} the equivalent distance
obtained by replacing $\Gamma_{xy}$ by $\log\Gamma_{xy}$ and similarly for $\bar \Gamma_{xy}$.
\end{remark}

\begin{remark}
The reason for separating the notion of a regularity structure from the notion of a model is that, 
in the type of applications that we have in mind, the regularity structure will be fixed once and
for all. The model however will typically be random and there will be a different model
for the regularity structure for every realisation of the driving noise.
\end{remark}

\subsection{Automorphisms of regularity structures}
\label{sec:automorphisms}

There is a natural notion of ``automorphism'' of a given regularity 
structure. 
%
For this, we first define the set $L^+_0$ of linear maps $L\colon T\to T$ 
such that, for every $\alpha \in A$ there exists $\gamma \in A$ such that $L a \in \bigoplus_{\alpha<\beta \le \gamma}T_\beta$ for every  $a \in T_\alpha$.
We furthermore denote by $L^+_1$ the set of 
all linear operators $Q$ of the form
\begin{equ}
Qa - a = La\;,\qquad L \in L^+_0\;.
\end{equ}
Finally, we denote by 
$L^0$ the set of invertible ``block-diagonal'' operators $D$ such that $DT_\alpha \subset T_\alpha$ for every $\alpha\in A$.

With these notations at hand, denote by $L^+$ the set of all operators of the form
\begin{equ}
M = D\circ Q\;,\qquad D \in L^0\;,\quad Q \in L_1^+\;.
\end{equ}
This factorisation is unique since it suffices to define $D = \sum_{\alpha \in A} \CQ_\alpha M \CQ_\alpha$ and to
set $Q = D^{-1}M$, which yields an element of $L_1^+$. Note also that conjugation by block-diagonal
operators preserves $L_1^+$. Furthermore, elements in $L_1^+$ can be inverted by using the identity
\begin{equ}[e:inverse]
(1-L)^{-1} = 1+\sum_{n \ge 1} L^n\;,
\end{equ}
although this might map some elements of $T_\alpha$ into an infinite series.
With all of these notations at hand, we then give the following definition:

\begin{definition}
Given a regularity structure $\TT = (A,T,G)$, its group of automorphisms $\Aut \TT$ is given by
\begin{equ}
\Aut \TT = \{M \in L^+\,:\, M^{-1} \Gamma M \in G \quad\forall \Gamma \in G\}\;.
\end{equ}
\end{definition}

\begin{remark}
This is really an abuse of terminology since it might happen that $\Aut \TT$ contains some
elements in whose inverse maps finite series into infinite series and therefore does not belong to
$L^+$. In most cases of interest however, the index set $A$ is finite, in which case $\Aut \TT$ is always
an actual group.
\end{remark}

The reason why $\Aut \TT$ is important is that its elements induce an action on the
models for $\TT$ by
\begin{equ}
R_M\colon (\Pi,\Gamma)\mapsto (\bar \Pi, \bar \Gamma)\;,\qquad \bar \Pi_x = \Pi_x M\;,\quad \bar \Gamma_{xy} = M^{-1}\Gamma_{xy} M\;.
\end{equ}
One then has:

\begin{proposition}
For every $M \in \Aut \TT$, $R_M$ is a continuous map from $\MM_\TT$ into itself.
\end{proposition}

\begin{proof}
It is clear that the algebraic identities \eref{e:algebraic} are satisfied, so we only need to check that
the analytical bounds of Definition~\ref{def:model} hold for $(\bar \Pi, \bar \Gamma)$.

For $\Pi$, this is straightforward since, for $a \in T_\alpha$ and any $M \in L^+$, one has
\begin{equs}
\bar \Pi_x a(\psi_x^\lambda) &= \Pi_x M a(\psi_x^\lambda) =  \sum_{\beta \in A\cap[\alpha,\gamma]} \Pi_x \CQ_\beta Ma(\psi_x^\lambda) \\
&\le C \|a\|_\alpha \sum_{\beta \in A\cap[\alpha,\gamma]} \lambda^\beta \le \tilde C \lambda^\alpha \|a\|_\alpha\;,
\end{equs}
where, for a given test function $\psi$ we use the shorthand $\psi_x^\lambda = \CS_{\s,x}^\lambda \psi$
and where $\tilde C$ is a finite constant depending only on the norms of the components of $M$
and on the value $\gamma$ appearing in the definition of $L^+$.

For $\Gamma$, we similarly write, for $a\in T_\alpha$ and $\beta < \alpha$,
\begin{equs}
\|(\bar \Gamma_{xy}-1) a\|_\beta &= \|M^{-1} (\Gamma_{xy}-1) M a\|_\beta
\le C\sum_{\zeta \le \beta}\|(\Gamma_{xy}-1) M a\|_\zeta \\
&\le C \sum_{\zeta \le \beta}\sum_{\xi \ge \zeta} \|M a\|_\xi \|x-y\|_\s^{\xi-\zeta}
\le C \sum_{\zeta \le \beta}\sum_{\xi \ge (\zeta\vee \alpha)} \|a\|_\alpha \|x-y\|_\s^{\xi-\zeta}\;.
\end{equs}
Since one has on the one hand $\zeta \le \beta$ and on the other hand $\xi \ge \alpha$,
all terms appearing in this sum involve a power of $\|x-y\|_\s$ that is at least equal to $\alpha - \beta$.
Furthermore, the sum is finite by the definition of $L^+$, so that the claim follows at once.
\end{proof}

\section{Modelled distributions}
\label{sec:model}

Given a regularity structure $\TT$, as well as a model $(\Pi,\Gamma)$, 
we are now in a position to describe 
a class of distributions that locally ``look like'' the distributions in the model. 
Inspired by Definition~\ref{def:Calphas}, we define the space $\CD^\gamma$ 
(which depends in general not only on the regularity structure, but also on the model) in the following way.

\begin{definition}\label{def:Dgamma}
Fix a regularity structure $\TT$ and a model $(\Pi,\Gamma)$.
Then, for any $\gamma \in \R$,  the space $\CD^\gamma$ consists of all $T_\gamma^-$-valued functions 
$f$ such that, for every compact set $\K \subset \R^d$, one has
\begin{equ}[e:boundDgamma]
\$f\$_{\gamma;\K} = \sup_{x\in \K} \sup_{\beta < \gamma} {\|f(x)\|_\beta}  + \sup_{(x,y) \in \K \atop \|x-y\|_\s \le 1} \sup_{\beta < \gamma} {\|f(x) - \Gamma_{xy} f(y)\|_\beta \over \|x-y\|_\s^{\gamma - \beta} } < \infty\;.
\end{equ}
Here, the supremum runs only over elements $\beta \in A$. 
We call elements of $\CD^\gamma$ \textit{modelled distributions} for reasons that will
become clear in Theorem~\ref{theo:reconstruction} below.
\end{definition}

\begin{remark}\label{rem:discard}
One could alternatively think of 
$\CD^\gamma$ as consisting of equivalence classes of functions
where $f \sim g$ if $\CQ_\alpha f(x) = \CQ_\alpha g(x)$ for every $x\in \R^d$ and every $\alpha < \gamma$. 
However, any such equivalence class has one natural distinguished representative, which is the function
$f$ such that $\CQ_\alpha f(x) = 0$ for every $\alpha \ge\gamma$, and this is the representative
used in \eref{e:boundDgamma}. (In general, the norm $\$\cdot\$_{\gamma;\K}$ would depend
on the choice of representative because $\Gamma_{xy} \tau$ can have components in $T_\gamma^-$ even if
$\tau$ itself doesn't.) In the sequel, if we state that
$f \in \CD^\gamma$ for some $f$ which does not necessarily take values 
in $T_\gamma^-$, it is this representative
that we are talking about. This also allows to identify $\CD^{\bar\gamma}$ as a subspace of $\CD^{\gamma}$
for any $\bar \gamma >  \gamma$. (Verifying that this is indeed the case is a useful exercise!)
\end{remark}

\begin{remark}
The choice of notation $\CD^\gamma$ is intentionally close to the notation $\CC^\gamma$ for the space
of $\gamma$-H\"older continuous functions since, in the case of the ``canonical'' regularity structures built 
from polynomials, the two spaces essentially agree,
as we saw in Section~\ref{sec:canonical}.
\end{remark}

\begin{remark}
The spaces $\CD^\gamma$, as well as the norms $\$\cdot\$_{\gamma;\K}$ 
do depend on the choice of $\Gamma$, but not on the choice of $\Pi$. 
However, Definition~\ref{def:model} strongly interweaves 
$\Gamma$ and $\Pi$, so that a given choice of $\Gamma$ typically restricts the choice of $\Pi$ very severely.
As we will see in Proposition~\ref{prop:extension} below, there are actually situations in 
which the choice of $\Gamma$ completely determines $\Pi$.
In order to compare elements of spaces $\CD^\gamma$ corresponding to different choices of $\Gamma$,
say $f \in \CD^\gamma(\Gamma)$ and $\bar f \in \CD^\gamma(\bar \Gamma)$,
it will be convenient to introduce the norm
\begin{equ}
\|f-\bar f\|_{\gamma;\K} = \sup_{x\in \K} \sup_{\beta < \gamma} {\|f(x)-\bar f(x)\|_\beta} \;,
\end{equ} 
which is independent of the choice of $\Gamma$.
Measuring the distance between elements of $\CD^\gamma$ in the norm $\|\cdot\|_{\gamma;\K}$
will be sufficient to obtain some convergence properties, as 
long as this is supplemented by uniform
bounds in $\$\cdot\$_{\gamma;\K}$. 
\end{remark}

\begin{remark}\label{rem:notationReg}
It will often be advantageous to consider elements of $\CD^\gamma$ that only take values
in a given sector $V$ of $T$. In this case, we use the notation $\CD^\gamma(V)$ instead.
In cases where $V$ is of regularity $\alpha$ for some $\alpha \ge \min  A$, we will also occasionally 
use instead the notation
$\CD_\alpha^\gamma$ to emphasise this additional regularity. 
Occasionally, we will also write $\CD^\gamma(\Gamma)$ or $\CD^\gamma(\Gamma;V)$
to emphasise the dependence of these spaces on the particular choice of $\Gamma$.
\end{remark}

\begin{remark}
A more efficient way of comparing elements $f \in \CD^\gamma(\Gamma)$ and  $\bar f\in\CD^\gamma(\bar \Gamma)$
for two different models $(\Pi,\Gamma)$ and $(\bar \Pi, \bar \Gamma)$ is to introduce the
quantity
\begin{equ}
 \$f;\bar f\$_{\gamma;\K} = \|f-\bar f\|_{\gamma;\K}  + \sup_{(x,y) \in \K \atop \|x-y\|_\s \le 1} \sup_{\beta< \gamma} {\|f(x) - \bar f(x) - \Gamma_{xy} f(y) + \bar \Gamma_{xy} \bar f(y)\|_\beta \over \|x-y\|_\s^{\gamma - \beta} } \;.
\end{equ}
Note that this quantity is \textit{not} a function of $f - \bar f$, which is the reason for the 
slightly unusual notation $\$f;\bar f\$_{\gamma;\K}$.
\end{remark}

It turns out that the spaces $\CD^\gamma$ encode a very useful notion of 
regularity. The idea is that functions $f \in \CD^\gamma$ should be interpreted as ``jets'' of
distributions that locally, around any given point $x \in \R^d$, ``look like'' the model distribution 
$\Pi_x f(x) \in \CS'$.
The results of this section justify this point of view by showing that it is indeed possible to ``reconstruct''  
all elements of $\CD^\gamma$ as distributions in $\R^d$. Furthermore, the corresponding reconstruction
map $\CR$ is continuous as a function of both the element in $f \in \CD^\gamma$ and the model 
$(\Pi,\Gamma)$ realising the regularity structure under consideration. 

To this end, we further extend the definition of the H\"older spaces $\Lip_\s^{\alpha}$ 
to include exponents $\alpha < 0$,
consisting of distributions that are suitable for our purpose. Informally speaking, elements of $\Lip_\s^\alpha$ 
have scaling properties akin to $\|x-y\|_\s^\alpha$ when tested against a test function localised around some
$x \in \R^d$. In the following definition, we write $\CC^r_0$ for the space of
compactly supported $\CC^r$ functions.
For further properties of the spaces $\Lip_\s^\alpha$, see Section~\ref{sec:spaceC} below. We set:

\begin{definition}\label{def:C-alpha}
Let $\alpha < 0$ and let $r = -\lfloor \alpha\rfloor$. 
We say that $\xi \in \CS'$ belongs to 
$\Lip_\s^{\alpha}$ if it belongs to the dual of $\CC^r_0$ and, for every compact set $\K$,
there exists a constant $C$ such that the bound
\begin{equ}
\scal{\xi,\CS_{\s,x}^{\delta} \eta} \le C \delta^{\alpha}\;,
\end{equ}
holds for all $\eta\in \CC^r$ with $\|\eta\|_{\CC^r} \le 1$ and $\supp \eta \subset B_\s(0,1)$, all $\delta \le 1$,
and all $x \in \K$.
Here, $B_\s(0,1)$ denotes the ball of radius $1$ in the distance $d_\s$, centred at the origin.
\end{definition}

From now on, we will denote by $\CB^r_{\s,0}$ the set of all test functions $\eta$ as in 
Definition~\ref{def:C-alpha}.
For $\xi \in \CC_\s^\alpha$ and $\K$ a compact set, we will henceforth denote by $\|\xi\|_{\alpha;\K}$
the seminorm given by
\begin{equ}[e:seminormAlpha]
\|\xi\|_{\alpha;\K} \eqdef \sup_{x\in \K} \sup_{\eta\in \CB^r_{\s,0}} \sup_{\delta \le 1} \delta^{-\alpha} |\scal{\xi,\CS_{\s,x}^{\delta} \eta}|\;.
\end{equ}
We also write $\|\cdot\|_\alpha$ for the same expression with $\K=\R^d$.

\begin{remark}
The space $\CC^\alpha_\s$ is essentially the Besov space $B^\alpha_{\infty,\infty}$ (see e.g.~\cite{MR1228209}), 
with the slight difference that our definition is local rather than global and, more importantly, that it allows 
for non-Euclidean scalings.
\end{remark}

\begin{remark}
The seminorm \eref{e:seminormAlpha} depends of course not only on $\alpha$, but also on
the choice of scaling $\s$. This scaling will however always be clear from the context,
so we do not emphasise this in the notation. 
\end{remark}

%
%

The following ``reconstruction theorem'' is one of the main workhorses of this theory.

\begin{theorem}[Reconstruction theorem]\label{theo:reconstruction}
Let $\TT = (A,T,G)$ be a regularity structure, let $(\Pi,\Gamma)$ be a 
model for $\TT$ on $\R^d$ with scaling $\s$, let $\alpha = \min A$,
and let $r > |\alpha|$. 

Then, for every $\gamma \in \R$, there
exists a continuous linear map $\CR \colon \CD^\gamma \to  \Lip_\s^{\alpha}$ with the property that,
for every compact set $\K \subset \R^d$,
\begin{equ}[e:boundRf]
\bigl|\bigl(\CR f - \Pi_x f(x)\bigr)(\CS_{\s,x}^\delta \eta)\bigr| \lesssim \delta^{\gamma} \|\Pi\|_{\gamma; \bar \K} \$f\$_{\gamma;\bar \K} \;,
\end{equ}
uniformly over all test functions $\eta\in \CB^r_{\s,0}$, 
all $\delta \in (0,1]$, all $f \in \CD^\gamma$, and all $x \in \K$. 
If $\gamma > 0$, then the bound \eref{e:boundRf} defines $\CR f$ uniquely. 
Here, we denoted by $\bar \K$ the $1$-fattening of $\K$, and the proportionality 
constant depends 
only on $\gamma$ and the structure of $\TT$.

Furthermore, if $(\bar \Pi, \bar \Gamma)$ is a second model for $\TT$ with associated reconstruction
operator $\bar \CR$, then one has the bound
\begin{equ}[e:boundRfDiff]
\bigl|\bigl(\CR f - \bar \CR \bar f- \Pi_x f(x) + \bar \Pi_x \bar f(x)\bigr)(\CS_{\s,x}^\delta \eta)\bigr| \lesssim \delta^{\gamma} \bigl(\|\bar \Pi\|_{\gamma; \bar \K} \$f;\bar f\$_{\gamma;\bar \K} + \|\Pi-\bar \Pi\|_{\gamma; \bar \K} \$f\$_{\gamma;\bar \K}\bigr)\;,
\end{equ}
uniformly over $x$ and $\eta$ as above. Finally, for $0<\kappa<\gamma/(\gamma-\alpha)$ and for every $C>0$,
one has the bound 
\begin{equs}
\bigl|\bigl(\CR f - \bar \CR \bar f &- \Pi_x f(x) + \bar \Pi_x \bar f(x)\bigr)(\CS_{\s,x}^\delta \eta)\bigr| \label{e:boundRfDiff2} \\
&\quad \lesssim \delta^{\bar \gamma}  \bigl( \|f-\bar f\|_{\gamma;\bar \K}^\kappa  + \|\Pi-\bar \Pi\|_{\gamma; \bar \K}^\kappa
+ \|\Gamma - \bar \Gamma\|_{\gamma;\bar \K}^\kappa\bigr)\;,
\end{equs}
where we set $\bar \gamma = \gamma - \kappa(\gamma-\alpha)$, and where we assume that 
$\$f\$_{\gamma;\bar \K}$, $\|\Pi\|_{\gamma;\bar \K}$ and $\|\Gamma\|_{\gamma;\bar \K}$
are bounded by $C$, and similarly for $\bar f$, $\bar \Pi$ and $\bar \Gamma$. 
\end{theorem}

\begin{remark}
At first sight, it might seem surprising that $\Gamma$ does not appear
in the bound \eref{e:boundRf}. 
It does however appear in a hidden way through the definition of the spaces $\CD^\gamma$
and thus of the norm $\$f\$_{\gamma;\bar \K}$.
Furthermore, \eref{e:boundRf} is quite reasonable since, for $\Gamma$ fixed, 
the map $\CR$ is actually bilinear in $f$ and  $\Pi$. 
However, the mere existence of $\CR$ depends crucially on the nonlinear
structure encoded in Definition~\ref{def:model}, and the spaces $\CD^\gamma$ do depend on
the choice of $\Gamma$. Occasionally, when the particular model $(\Pi,\Gamma)$ plays a role,
we will denote $\CR$ by $\CR_\Gamma$ in order to emphasise its dependence on $\Gamma$.
\end{remark}

\begin{remark}
Setting $\tilde f(y) = f(y) - \Gamma_{yx} f(x)$, we note that one has 
\begin{equ}
\CR f - \Pi_x f(x) = \CR \tilde f - \Pi_x \tilde f(x) = \CR \tilde f\;.
\end{equ}
As a consequence, the bound \eref{e:boundRf} actually depends only on the  
second term in the right hand side of \eref{e:boundDgamma}.
\end{remark}

\begin{remark}
In the particular case when $(\bar \Pi, \bar \Gamma)=(\Pi, \Gamma)$, the
 bound \eref{e:boundRfDiff} is a trivial consequence of \eref{e:boundRf} and the bilinearity of $\CR$
 in $f$ and $\Pi$. As it stands however, this bound needs to be stated and proved separately.
 The bound \eref{e:boundRfDiff2} can be interpreted as an interpolation theorem between \eref{e:boundRf}
 and \eref{e:boundRfDiff}.
\end{remark}

\begin{proof}[(uniqueness only)]
The uniqueness of the map $\CR$ in the case $\gamma > 0$
is quite easy to prove. Take $f \in \CD^\gamma$ as in the 
statement and assume that the two distributions $\xi_1$ and $\xi_2$ are
candidates for $\CR f$ that both satisfy the bound \eref{e:boundRf}. Our aim is to show that one then necessarily has
$\xi_1 = \xi_2$. Take any smooth compactly supported test function $\psi\colon \R^d \to \R$, 
and choose
an even smooth function $\eta \colon B_1 \to \R_+$ with $\int \eta(x)\,dx = 1$. Define
\begin{equ}
\psi_\delta (y) = \scal{\CS_{\s,y}^\delta \eta, \psi}=  \int \psi(x)\, \bigl(\CS_{\s,x}^\delta \eta\bigr)(y)\,dx\;,
\end{equ}
so that, for any distribution $\xi$, one has the identity
\begin{equ}[e:idenXif]
\xi(\psi_\delta) =  \int \psi(x)\, \scal{\xi, \CS_{\s,x}^\delta \eta}\,dx\;.
\end{equ}
Choosing $\xi = \xi_2 - \xi_1$, it then follows from \eref{e:boundRf} that
\begin{equ}
|\xi(\psi_\delta)| \lesssim \delta^\gamma \int_D \psi(x)\, \bar \rho(x)\,dx\;,
\end{equ}
which converges to $0$ as $\delta \to 0$. On the other hand, one has $\psi_\delta \to \psi$ in the $\CC^\infty$ 
topology, so that $\xi(\psi_\delta) \to \xi(\psi)$. This shows that $\xi(\psi) = 0$ for every smooth
compactly supported test function $\psi$, so that $\xi = 0$.

The existence of a map $\CR$ with the required properties 
is much more difficult to establish, and this is the content of the remainder of this section. 
\end{proof}

\begin{remark}
We call the map $\CR$ the ``reconstruction map'' as it allows to reconstruct a distribution in terms of
its local description via a model and regularity structure.
\end{remark}

\begin{remark}\label{rem:continuousModel}
One very important special case is when the model $(\Pi,\Gamma)$ happens to be such that
there exists $\alpha>0$ such that $\Pi_x a \in \CC^\alpha_\s(\R^d)$ for every $a \in T$, even though the homogeneity 
of $a$ might be negative. In this case, for $f \in \CD^\gamma$ with $\gamma > 0$, $\CR f$ is a 
continuous function and one has the identity $(\CR f)(x) = \bigl(\Pi_x f(x)\bigr)(x)$.  Indeed, setting
$\tilde \CR f(x) = \bigl(\Pi_x f(x)\bigr)(x)$, one has
\begin{equ}
\bigl|\tilde \CR f(y) - \tilde \CR f(x)\bigr| \le \bigl|\bigl(\Pi_x f(x)\bigr)(x) - \bigl(\Pi_x f(x)\bigr)(y)\bigr| + 
\bigl|\Pi_y \bigl(\Gamma_{yx} f(x) - f(y)\bigr)(y)\bigr|\;.
\end{equ}
By assumption, the first term is bounded by $C \|x-y\|_\s^\alpha$ for some constant $C$. The second
term on the other hand is bounded by $C \|x-y\|_\s^\gamma$ by the definition of $\CD^\gamma$, combined
with the fact that our assumption on the model implies that  $\bigl(\Pi_x a\bigr)(x) = 0$ whenever 
$a$ is homogeneous of positive degree.
\end{remark}

A straightforward corollary of this result is given by the following statement, which is the 
\textit{a posteriori} justification for the terminology ``regularity'' in Definition~\ref{def:sector}:

\begin{corollary}\label{cor:regularity}
In the context of the statement of Theorem~\ref{theo:reconstruction}, 
if $f$ takes values in a sector $V$ of regularity $\beta \in [\alpha,0)$,
then one has $\CR f \in \CC^\beta_\s$ and, for every compact set $\K$ and $\gamma > 0$, there exists a 
constant $C$ such that
\begin{equ}
\|\CR f\|_{\beta;\K} \le C \|\Pi\|_{\gamma; \bar \K} \$f\$_{\gamma;\bar \K}\;.
\end{equ}
\end{corollary}

\begin{proof}
Immediate from \eref{e:boundRf}, Remark~\ref{rem:normPiGamma}, and the definition of $\|\cdot\|_{\beta;\K}$.
\end{proof}

Before we proceed to the remainder of the proof of Theorem~\ref{theo:reconstruction}, we introduce some of the 
basic notions of wavelet analysis required for its proof. For a more detailed introduction
to the subject, see for example \cite{MR1162107,MR1228209}.

\subsection{Elements of wavelet analysis}
\label{sec:wavelets}

Recall that a multiresolution analysis of $\R$ is based on
a real-valued ``scaling function'' $\phi \in L^2(\R)$ with the following two properties:
\begin{claim}
\item[1.] One has $\int \phi(x) \phi(x+k)\,dx = \delta_{k,0}$ for every $k \in \Z$.
\item[2.] There exist ``structure constants'' $a_k$ such that
\begin{equ}[e:structure]
\phi(x) = \sum_{k \in \Z} a_k \phi(2x-k)\;.
\end{equ}
\end{claim}
One classical example of such a function $\phi$ is given by the indicator function $\phi(x) = \one_{[0,1)}(x)$,
but this has the substantial drawback that it is not even continuous. A celebrated result by 
Daubechies (see the original article \cite{MR951745} or for example the monograph \cite{MR1162107}) 
ensures the existence of functions $\phi$ as above that are compactly supported
but still regular:

\begin{theorem}[Daubechies]
For every $r > 0$ there exists a compactly supported function $\phi$ with the two properties above
and such that $\phi \in \CC^r(\R)$.\qed
\end{theorem}

From now on, we will always assume that the scaling function $\phi$ is compactly supported.
Denote now $\Lambda_n = \{2^{-n}k\,:\, k \in \Z\}$ and, for $n \in \Z$ and $x \in \Lambda_n$, set
\begin{equ}[e:defphin]
\phi_x^n(y) = 2^{n/2}\phi\bigl(2^n (y-x)\bigr)\;.
\end{equ}
One furthermore denotes by $V_n \subset L^2(\R)$ the subspace generated by $\{\phi_x^n\,:\, x \in \Lambda_n\}$.
Property~2 above then ensures that these spaces satisfy the inclusion $V_n \subset V_{n+1}$ for every $n$. 
Furthermore, it turns out that there is a simple description of the orthogonal complement $V_n^\perp$ of $V_n$ in $V_{n+1}$.
It turns out that it is possible to find finitely many coefficients $b_k$ such that, setting  
\begin{equ}[e:defPsi]
\psi(x) =  \sum_{k \in \Z} b_k \phi(2x-k)\;,
\end{equ}
and defining $\psi_x^n$ similarly to \eref{e:defphin}, the space $V_n^\perp$ is given by the linear 
span of $\{\psi_x^n\,:\, x \in \Lambda_n\}$, see for example \cite[Chap.~6.4.5]{MR2100936}.
(One  has actually $b_k = (-1)^k a_{1-k}$ but this isn't important for us.)
The following result is taken from \cite{MR1228209}:

\begin{theorem}\label{thm:wavelet}
One has $\scal{\psi_x^n, \psi_y^m} = \delta_{n,m}\delta_{x,y}$ for every $n,m \in \Z$ and every $x \in \Lambda_n$, $y \in \Lambda_m$. Furthermore,  $\scal{\phi_x^n, \psi_y^m} = 0$ for every $m \ge n$ and every $x \in \Lambda_n$, $y \in \Lambda_m$. Finally, for every $n \in \Z$, the set
\begin{equ}
\{\phi_x^n\,:\, x \in \Lambda_n\}\cup \{\psi_x^m\,:\, m \ge n\,,\; x \in \Lambda_m\}\;,
\end{equ}
forms an orthonormal basis of $L^2(\R)$.\qed
\end{theorem}

Intuitively, one should think of the $\phi_x^n$ as providing a description of a function at scales down to $2^{-n}$ and
the $\psi_x^m$ as ``filling in the details'' at even smaller scales. In particular, for every function $f \in L^2$, one has
\begin{equ}[e:defPf]
\lim_{n \to \infty} \CP_n f \eqdef \lim_{n \to \infty} \sum_{x \in \Lambda_n} \scal{f, \phi_x^n}\phi_x^n = f\;,
\end{equ}
and this relation actually holds for much larger classes of $f$, including sufficiently regular tempered distributions \cite{MR1228209}.

One very useful properties of wavelets, which can be found for example
in \cite[Chap.~3.2]{MR1228209}, is that the functions $\psi_x^m$ automatically have vanishing moments:

\begin{lemma}
Let $\phi$ be a compactly supported scaling function as above which is $\CC^r$ for $r \ge 0$ and let $\psi$ be defined by \eref{e:defPsi}. Then, $\int_\R \psi(x)\, x^m\,dx = 0$
for every integer $m \le r$.\qed
\end{lemma}

For our purpose, we need to extend this construction to $\R^d$. Classically, such an extension can be performed
by simply taking products of the $\phi_x^n$ for each coordinate. In our case however, we want to take into account 
the fact that we consider non-trivial scalings. For any given scaling $\s$ of $\R^d$ and any $n \in \Z$, we thus define\label{lab:Lambdan}
\begin{equ}
\Lambda_n^\s = \Bigl\{\sum_{j=1}^d 2^{- n\s_j} k_j e_j\,:\, k_j \in \Z\Bigr\}\subset \R^d\;,
\end{equ}
where we denote by $e_j$ the $j$th element of the canonical basis of $\R^d$. For every $x \in \Lambda_n^\s$,
we then set
\begin{equ}[e:phins]
\phi_x^{n,\s}(y) \eqdef \prod_{j=1}^d \phi_{x_j}^{n\s_j}(y_j)\;.
\end{equ}
Since we assume that $\phi$ is compactly supported, it follows from \eref{e:structure} that there
exists a \textit{finite} collection of vectors $\CK \subset \Lambda_1^\s$ and structure constants $\{a_k\,:\, k \in \CK\}$
such that the identity
\begin{equ}[e:decompPhi]
\phi_x^{0,\s}(y) = \sum_{k \in \CK} a_k \phi_{x+k}^{1,\s}(y)\;,
\end{equ}
holds. In order to simplify notations, we will henceforth use the notation
\begin{equ}
2^{-n\s}k = \bigl(2^{-n\s_1}k_1,\ldots,2^{-n\s_d}k_d\bigr)\;,
\end{equ}
so that the scaling properties of the $\phi_x^{n,\s}$ combined with \eref{e:decompPhi} imply that
\begin{equ}[e:decompPhin]
\phi_x^{n,\s}(y) = \sum_{k \in \CK} a_k \phi_{x+2^{-n\s}k}^{n+1,\s}(y)\;.
\end{equ}

Similarly, there exists a finite collection $\Psi$ of orthonormal compactly supported functions such that,
if we define $V_n$ similarly as before, 
$V_n^\perp$ is given by 
\begin{equ}
V_n^\perp = \mathop{\mathrm{span}} \{\psi_x^{n,\s}\,:\, \psi \in \Psi\quad x \in \Lambda_n^\s\}\;.
\end{equ}
In this expression, given a  function $\psi \in \Psi$, we have set 
$\psi_x^{n,\s} = 2^{-n|\s|/2}\CS_{\s,x}^{2^{-n}} \psi$, where the scaling map was defined in \eref{e:scaling}.
(The additional factor makes sure that the scaling leaves the $L^2$ norm invariant instead of the $L^1$
norm, which is more convenient in this context.)
Furthermore, this collection forms an orthonormal basis of $V_n^\perp$.
Actually, the set $\Psi$ is given by all functions obtained by products of the form $\Pi_{i =1}^d \psi_{\pm}(x_i)$,
where $\psi_- = \psi$ and $\psi_+ = \phi$, and where at least one factor consists of an instance of $\psi$.

\subsection{A convergence criterion in \texorpdfstring{$\CC_\s^\alpha$}{Ca}}
\label{sec:spaceC}

The spaces $\Lip_\s^\alpha$ with $\alpha < 0$ given in Definition~\ref{def:C-alpha} 
enjoy a number of remarkable properties that will be very useful in the sequel. In particular, it turns out
that distributions in $\CC_\s^{\alpha}$ can be completely characterised by the magnitude of the 
coefficients in their wavelet expansion. This is true independently of the particular choice of the scaling
function $\phi$, provided that it has sufficient regularity. 

In this sense, the interplay between the wavelet expansion and the spaces $\CC_\s^\alpha$ is very
similar to the classical interplay between Fourier expansion and fractional Sobolev spaces.
The feature of wavelet expansions that makes it much more suitable for our purpose is that its 
basis functions are compactly supported with supports that are more and more localised
for larger values of $n$.
The announced characterisation is given by the following.

\begin{proposition}\label{prop:charSpaces}
Let $\alpha < 0$ and $\xi \in \CS'(\R^d)$.
Consider a wavelet analysis as above with a compactly supported scaling function $\phi \in \CC^r$ for some $r > |\alpha|$. Then $\xi \in \Lip_\s^{\alpha}$ if and only if $\xi$ belongs to the dual of $\CC^r_0$ and, for every compact set $\K \subset \R^d$, the bounds
\begin{equ}[e:boundScalingAlpha]
|\scal{\xi,\psi_x^{n,\s}}| \lesssim 2^{-{n|\s|\over 2}-n\alpha}\;, \qquad |\scal{\xi, \phi_y^{0}}| \lesssim 1\;, 
\end{equ}
hold uniformly over $n \ge 0$, every $\psi\in \Psi$,
every $x \in \Lambda_n^\s\cap \K$, and every $y \in \Lambda_0^\s\cap \K$.
\end{proposition}

The proof of Proposition~\ref{prop:charSpaces} relies on classical arguments very similar
to those found for example in the monograph \cite{MR1228209}. Since the spaces with inhomogeneous 
scaling do not seem to be standard in the literature and since we consider localised versions
of the spaces, we prefer to provide a proof. 
Before we proceed, we state the following elementary fact:

\begin{lemma}\label{lem:sumExp}
Let $a\in\R$ and let $b_-, b_+ \in \R$. Then, the bound
\begin{equ}
\sum_{n =0}^{n_0} 2^{an} 2^{-b_-(n_0-n)} + \sum_{n =n_0}^{\infty} 2^{an} 2^{-b_+(n-n_0)} \lesssim 2^{a n_0}\;,
\end{equ}
holds provided that $b_+ > a$ and $b_- > -a$.\qed
\end{lemma}

\begin{proof}[of~Proposition~\ref{prop:charSpaces}]
It is clear that the condition \eref{e:boundScalingAlpha} is necessary, 
since it boils down to taking $\eta \in \Psi$ and $\delta = 2^{-n}$ in Definition~\ref{def:C-alpha}. 
In order to show that it is also sufficient, we take an arbitrary test function $\eta \in \CC^r$ with support in $B_1$
and we rewrite $\scal{\xi,\CS_{\s,x}^{\delta} \eta}$ as
\begin{equ}
\scal{\xi,\CS_{\s,x}^{\delta} \eta}  = \sum_{n\ge 0}\sum_{y\in\Lambda_n^\s} \scal{\xi, \psi_y^{n,\s}}\scal{\psi_y^{n,\s}, \CS_{\s,x}^{\delta} \eta}   + \sum_{y\in\Lambda_0^\s} \scal{\xi, \phi_y^{0,\s}}\scal{ \phi_y^{0,\s}, \CS_{\s,x}^{\delta} \eta}\;. \label{e:decompProd}
\end{equ}
Let furthermore $n_0$ be the smallest integer such that $2^{-n_0} \le \delta$. 
For the situations where the supports of $\psi_y^{n,\s}$ and $\CS_{\s,x}^{\delta} \eta$ overlap, we then have the following bounds.

First, we note that if $(x,y)$ contributes to \eref{e:decompProd}, then $\|x-y\|_\s \le C$
for some fixed constant $C$. 
As a consequence of this, it follows that one  has the bound
\begin{equ}[e:boundScal]
|\scal{\xi,\psi_y^{n,\s}}|  \lesssim 2^{-{n|\s|\over 2}-n\alpha}\;,
\end{equ}
uniformly over all pairs $(x,y)$ yielding a non-vanishing contribution to \eref{e:decompProd}.

For $n \ge n_0$, and $\|x-y\|_\s \le C\delta$, we furthermore have the bound 
\begin{equ}[e:boundscalPhi]
|\scal{\psi_y^{n,\s}, \CS_{\s,x}^{\delta} \eta}| \lesssim 2^{-(n-n_0)\bigl(r + {|\s|\over 2}\bigr)}2^{n_0|\s|\over 2}\;,
\end{equ}
so that
\begin{equ}
\sum_{y \in \Lambda_n^\s} |\scal{\psi_y^{n,\s}, \CS_{\s,x}^{\delta} \eta}| \lesssim 2^{-(n-n_0)\bigl(r - {|\s|\over 2}\bigr)}2^{n_0|\s|\over 2}\;.
\end{equ}
Here and below, the proportionality constants are uniform over all $\eta$ with 
$\|\eta\|_{\CC^r} \le 1$ with $\supp \eta \subset B_1$.
On the other hand, for $n \le n_0$, and $\|x-y\|_\s \le C 2^{-n_0}$, we have the bound
\begin{equ}[e:boundfpsi]
|\scal{\psi_y^{n,\s}, \CS_{\s,x}^{\delta} \eta}| \lesssim 2^{n{|\s|\over 2}}\;,
\end{equ}
so that, since only finitely many terms contribute to the sum,
\begin{equ}
\sum_{y \in \Lambda_n^\s} |\scal{\psi_y^{n,\s}, \CS_{\s,x}^{\delta} \eta}| \lesssim 2^{n{|\s|\over 2}}\;.
\end{equ}
Since, by the assumptions on $r$ and $\alpha$, one has indeed
$r + {|\s|\over 2} > \alpha + {|\s|\over 2}$ and ${|\s|\over 2} > {|\s|\over 2} - \alpha$, we can apply Lemma~\ref{lem:sumExp}
to conclude that the first sum in \eref{e:decompProd} is indeed bounded by a multiple of $\delta^\alpha$,
which is precisely the required bound.
The second term on the other hand 
satisfies a bound similar to \eref{e:boundfpsi} with $n=0$, so that the claim follows.
\end{proof}

\begin{remark}
For $\alpha \ge 0$, it is not so straightforward to characterise the H\"older regularity of a function
by the magnitude of its wavelet coefficients due to special behaviour
at integer values, but for non-integer values the characterisation given above still holds, 
see \cite{MR1228209}.
\end{remark}

Another nice property of the spaces $\Lip_\s^{\alpha}$ is that, using Proposition~\ref{prop:charSpaces}, 
one can give a very useful and sharp condition for a sequence
of elements in $V_n$ to converge to an element in $\Lip_\s^{\alpha}$. Once again, we fix a multiresolution
analysis of sufficiently high regularity (i.e.\ $r > |\alpha|$) and the spaces $V_n$ are given in terms
of that particular analysis.
For this characterisation, 
we use the fact that 
a sequence
$\{f_n\}_{n \ge 0}$ with $f_n \in V_n$ for every $n$ can always be written as
\begin{equ}[e:defAnfn]
f_n = \sum_{x \in \Lambda_n^\s} A_x^n \phi_x^{n,\s}\;,\qquad A_x^n = \scal{\phi_x^{n,\s}, f_n}\;.
\end{equ}
Given a sequence of coefficients $A_x^n$,  we then define $\delta A_x^n$ by
\begin{equ}
\delta A_x^n = A_x^n - \sum_{k \in \CK} a_k A_{x+2^{-n\s}k}^{n+1}\;,
\end{equ}
where the set $\CK$ and the structure constants $a_k$ are as in \eref{e:decompPhi}. We then have
the following result, which can be seen as a generalisation of the ``sewing lemma''  
(see \cite[Prop.~1]{MR2091358} or \cite[Lem.~2.1]{MR2261056}),
which can itself be viewed as a generalisation of Young's original theory of integration \cite{Young}.
In order to make the link to these theories, consider the case where $\R^d$ is replaced by an interval
and take for $\phi$ the Haar wavelets.

\begin{theorem}\label{theo:sowingLemma}
Let $\s$ be a scaling of $\R^d$, let $\alpha < 0 < \gamma$, and fix a wavelet basis with regularity $r > |\alpha|$.
For every $n\ge 0$, let $x \mapsto A_x^n$ be a function on $\R^d$ satisfying the bounds
\begin{equ}[e:boundA]
|A_x^n| \le \|A\| 2^{-{n\s\over 2} - \alpha n} \;,\qquad
|\delta A_x^n| \lesssim \|A\|2^{-{n\s\over 2} - \gamma n}\;,
\end{equ}
for some constant $\|A\|$, uniformly over $n \ge 0$ and $x\in \R^d$.

Then, the sequence $\{f_n\}_{n \ge 0}$ given by $f_n = \sum_{x \in \Lambda_n^\s} A_x^n\, \phi_x^{n,\s}$ converges
in $\Lip_\s^{\bar \alpha}$ for every $\bar \alpha < \alpha$ and its limit $f$ belongs to $\Lip_\s^\alpha$. Furthermore, the bounds
\begin{equ}[e:boundCauchy]
\|f - f_n\|_{\bar \alpha} \lesssim \|A\| 2^{-(\alpha - \bar \alpha) n}\;,
\qquad \|\CP_n f - f_n\|_{\alpha} \lesssim \|A\| 2^{-\gamma n}\;,
\end{equ}
hold for $\bar \alpha \in (\alpha-\gamma,\alpha)$, where $\CP_n$ is as in \eref{e:defPf}.
\end{theorem}

\begin{proof}
By linearity, it is sufficient to restrict ourselves to the case $\|A\|= 1$.
By construction, we have $f_{n+1} - f_n \in V_{n+1}$, so that we can decompose this difference as
\begin{equ}[e:defgdf]
f_{n+1} - f_n = g_n + \delta f_n\;,
\end{equ}
where $\delta f_n \in V_n^\perp$ and $g_n \in V_n$. By Proposition~\ref{prop:charSpaces}, we 
note that there exists a constant $C$ such that, for every $n\ge 0$ and $m \ge n$, and for every $\beta < 0$, one has
\begin{equ}
\Bigl\|\sum_{k=n}^m \delta f_k\Bigr\|_{\beta} \le C \sup_{k \in \{n,\ldots,m\}} \bigl\|\delta f_k\bigr\|_{\beta}\;,
\end{equ}
so that a sufficient condition for the sequence $\{\sum_{k=0}^n\delta f_k\}_{n \ge 0}$ to have the required properties
is given by
\begin{equ}[e:convdeltaf]
\lim_{n \to \infty}  \|\delta f_n\|_{\bar \alpha} = 0\;,\qquad 
\sup_{n}  \|\delta f_n\|_{\alpha} < \infty\;.
\end{equ}
Regarding the bounds on $\delta f_n$, we have
\begin{equ}
\scal{\delta f_n, \psi_x^{n,\s}} = \scal{f_{n+1} - f_n, \psi_x^{n,\s}} = \sum_{\|x-y\|_\s \le K 2^{-n|\s|}} a_{xy}A_{y}^{n+1}\;,
\end{equ}
where the $a_{xy} = \scal{\phi_y^{n+1,\s},\psi_x^{n,\s}}$ 
are a finite number of uniformly bounded coefficients 
and $K>0$ is some fixed constant. 
It then follows from the assumption on the coefficients $A_{y}^n$ that 
\begin{equ}
|\scal{\delta f_n, \psi_x^{n,\s}}| \lesssim  2^{-{n |\s|\over 2} - \alpha n}\;.
\end{equ}
Combining this with the characterisation of $\Lip_\s^{\bar \alpha}$ given in Proposition~\ref{prop:charSpaces}, 
we conclude that
\begin{equ}[e:boundSumdf]
\|\delta f_n\|_{\bar \alpha} \lesssim 2^{-(\alpha-\bar \alpha)n}\;,\qquad
\|\delta f_n\|_{\alpha} \lesssim 1\;,
\end{equ}
so that the condition \eref{e:convdeltaf} is indeed satisfied.

It remains to show that the sequence of partial sums of the $g_k$ from \eref{e:defgdf}
also satisfies the requested properties. 
Using again the characterisation given by 
Proposition~\ref{prop:charSpaces}, we see that 
\begin{equ}[e:boundRemg]
\Bigl\|\sum_{k=n}^m g_k\Bigr\|_{\alpha} \lesssim \sup_{N\ge 0} \sum_{k =n}^m \|\CQ_N g_k\|_{\alpha}\;.
\end{equ}
From the definition of $g_n$, we furthermore have the identity
\begin{equs}
\scal{g_n, \phi_x^{n,\s}} &= \scal{f_{n+1}-f_n, \phi_x^{n,\s}} =  \Bigl(\sum_{k \in \CK} a_k \scal{f_{n+1},  \phi_{x+2^{-n\s}k}^{n+1,\s}}\Bigr) -  \scal{f_n, \phi_x^{n,\s}}\\
&= - \delta A_x^n\;, \label{e:gnAn}
\end{equs}
so that one can decompose $g_n$ as
\begin{equ}[e:decompgn]
g_n = -\sum_{x \in \Lambda_n^\s} \delta A_x^n\,\phi_x^{n,\s}\;.
\end{equ}
It follows in a straightforward way from the definitions that, for $m \le n$, there exists a constant $C$ such that we have the bound
\begin{equ}[e:boundpsiphi]
|\scal{\psi_y^{m,\s}, \phi_x^{n,\s}}| \le C2^{(m-n){|\s|\over 2}}\one_{\|x-y\|_\s \le C 2^{-m}}\;.
\end{equ}
Since on the other hand, one has
\begin{equ}
\bigl|\bigl\{x \in \Lambda_n^\s\,:\, \|x-y\|_\s \le C 2^{-m}\bigr\}\bigr| \lesssim 2^{(n-m)|\s|}\;,
\end{equ}
we obtain from this and \eref{e:decompgn} the bound 
\begin{equs}
|\scal{\psi_y^{m,\s}, g_n}| &\lesssim 2^{(n-m){|\s|\over 2}}\sup\bigl\{ |\delta A_x^n| \,:\, \|x-y\|_\s \le C 2^{-m}\bigr\}  \\
&\lesssim 2^{-m{|\s|\over 2} - \gamma n}\;,\label{e:boundGn1}
\end{equs}
where we used again the fact that $\|x-y\|_\s \lesssim d_\s(y,\d D)$ by the definition of the functions $\psi_y^{m,\s}$.
Combining this
with the characterisation of $\Lip_\s^{\alpha}$ given in Proposition~\ref{prop:charSpaces}, we conclude that
\begin{equ}
\|\CQ_m g_n\|_{\alpha} \lesssim 2^{\alpha m -\gamma n} \one_{m \le n}\;,
\end{equ}
so that 
\begin{equ}
\sum_{k=n}^m \|\CQ_N g_k\|_{\alpha} \lesssim  \sum_{k=n \vee N}^m 2^{\alpha N - \gamma k}
\lesssim 2^{\alpha N - \gamma (N \vee n)}\;.
\end{equ}
This expression is maximised at $N=0$, so that the bound 
$\|\sum_{k=n}^m g_k\|_{\alpha} \lesssim 2^{- \gamma n}$
follows from \eref{e:boundRemg}. Combining this with \eref{e:boundSumdf}, we thus obtain
\eref{e:boundCauchy}, as stated.
\end{proof}

A simple but important corollary of the proof is given by
\begin{corollary}
In the situation of Theorem~\ref{theo:sowingLemma}, let $\K\subset \R^d$ be a compact set and let $\bar \K$ be its $1$-fattening. Then, provided that \eref{e:boundA} holds uniformly over $\bar \K$, the bound
\eref{e:boundCauchy} still holds with $\|\cdot\|_{\alpha}$ replaced by $\|\cdot\|_{\alpha;\K}$.
\end{corollary}

\begin{proof}
Follow step by step the argument given above noting that, since all the arguments in the proof
of Proposition~\ref{prop:charSpaces} are local,
one can bound the norm $\|\cdot\|_{\alpha;\K}$ by the smallest constant such that the
bounds \eref{e:boundScalingAlpha} hold uniformly over $x,y \in \bar \K$.
\end{proof}

\subsection{The reconstruction theorem for distributions}

One very important special case of Theorem~\ref{theo:sowingLemma} is given by 
the situation where there exists a family $x \mapsto \zeta_x \in \CS'(\R^d)$
of distributions such that the sequence $f_n$ is given by \eref{e:defAnfn}
with $A_x^n = \scal{\phi_x^{n,\s}, \zeta_x}$. 
Once this is established, the reconstruction theorem will be straightforward.
In the situation just described, we have the following result which, as we will
see shortly, can really be interpreted as a generalisation of the reconstruction theorem.

\begin{proposition}\label{prop:reconstrGen}
In the above situation, assume that the family $\zeta_x$ is such that, for some
constants $K_1$ and $K_2$ and exponents $\alpha < 0 < \gamma$, the bounds
\begin{equ}[e:assGenReconst]
|\scal{\phi_x^{n,\s}, \zeta_x - \zeta_y}| \le K_1 \|x-y\|_\s^{\gamma-\alpha} 2^{- {n |\s|\over 2}-\alpha n}\;,
\quad |\scal{\phi_x^{n,\s}, \zeta_x}| \le K_2 2^{-\alpha n - {n |\s|\over 2}}\;,
\end{equ}
hold uniformly over all $x,y$ such that $2^{-n} \le \|x-y\|_\s \le 1$. Here, as before,
$\phi$ is the scaling function 
for a wavelet basis of regularity $r > |\alpha|$.
Then, the assumptions of Theorem~\ref{theo:sowingLemma} are satisfied.
Furthermore, the limit distribution $f \in \CC_\s^\alpha$ satisfies the bound
\begin{equ}[e:localBehaviour]
|(f - \zeta_x)(\CS_{\s,x}^{\delta} \eta)| \lesssim  K_1 \delta^\gamma\;,
\end{equ}
uniformly over $\eta \in \CB^r_{\s,0}$.
Here, the proportionality constant only depends on the choice of wavelet basis, but not on $K_2$.
\end{proposition}

\begin{proof}
We are in the situation of Theorem~\ref{theo:sowingLemma} with $A_x^n = \zeta_x(\phi_x^{n,\s})$,
so that one has the identity
\begin{equ}[e:exprdeltaA]
\delta A_x^n = \sum_{k \in \CK}a_k \scal[b]{\zeta_x - \zeta_y,\phi_{y}^{(n+1),\s}}\;,
\end{equ}
where we used the shortcut $y = x+2^{-n\s} k$ in the right hand side.
It then follows immediately from \eref{e:assGenReconst} that  
the assumptions of Theorem~\ref{theo:sowingLemma} are indeed satisfied, so that the sequence
$f_n$ converges to some limit $f$. It remains to show that the local behaviour of $f$
around every point $x$ is given by \eref{e:localBehaviour}.

For this, we write 
\begin{equ}
f - \zeta_x = \bigl(f_{n_0} - \CP_{n_0} \zeta_x\bigr)\label{e:decompWanted} 
 + \sum_{n \ge n_0} \bigl(f_{n+1} - f_n - (\CP_{n+1} - \CP_n) \zeta_x\bigr)\;,
\end{equ}
for some $n_0 > 0$.
We choose $n_0$ to be the smallest integer such that $2^{-n_0} \le \delta$.
Note that, as in \eref{e:boundscalPhi}, one has for $n \ge n_0$ the bounds
\begin{equ}[e:boundsPhi]
|\scal{\psi_y^{n,\s}, \CS_{\s,x}^{\delta} \eta}| \lesssim 2^{n_0|\s|\over 2}2^{-(n-n_0)\bigl(r + {|\s|\over 2}\bigr)}\;,\qquad
|\scal{\phi_y^{n,\s}, \CS_{\s,x}^{\delta} \eta}| \lesssim 2^{n_0|\s|\over 2}2^{-(n-n_0){|\s|\over 2}}\;.
\end{equ}
Since, by construction, the first term in \eref{e:decompWanted} belongs to $V_{n_0}$, we can rewrite it as
\begin{equ}
\bigl(f_{n_0} - \CP_{n_0} \zeta_x\bigr)(\CS_{\s,x}^{\delta} \eta) = \sum_{y \in \Lambda_{n_0}^\s} \bigl(\zeta_y - \zeta_x\bigr)(\phi_y^{n_0,\s})\,\scal{\phi_y^{n_0,\s}, \CS_{\s,x}^{\delta} \eta}\;.
\end{equ}
Since terms appearing in the above sum with 
$\|x-y\|_\s \ge \delta$ are identically $0$, we can use the bound
\begin{equ}
|\bigl(\zeta_y - \zeta_x\bigr)(\phi_y^{n_0,\s})| \lesssim K_12^{-\gamma n_0 - {n_0 |\s|\over 2}}\;.
\end{equ}
Combining this with \eref{e:boundsPhi} and the fact that there
are only finitely many non-vanishing terms in the sum, we obtain the bound
\begin{equ}[e:firstBound]
\bigl|\bigl(f_{n_0} - \CP_{n_0} \zeta_x\bigr)(\CS_{\s,x}^{\delta} \eta)\bigr| \lesssim K_1 2^{- n_0 \gamma} \approx K_1\delta^\gamma\;,
\end{equ}
which is of the required order.

Regarding the second term in \eref{e:decompWanted}, we decompose $f_{n+1} - f_n$ as
in the proof of Theorem~\ref{theo:sowingLemma} as $f_{n+1} - f_n = g_n + \delta f_n$ with
$g_n \in V_n$ and $\delta f_n \in V_n^\perp$.
As a consequence of \eref{e:gnAn} and of the bounds \eref{e:assGenReconst} and \eref{e:boundsPhi}, 
we have the bound
\begin{equs}
\bigl|\scal{g_n , \CS_{\s,x}^{\delta} \eta}\bigr| &\le
\sum_{y \in \Lambda_n^\s} \bigl|\scal{g_n, \phi_y^{n,\s}}\bigr|\,\bigl|\scal{\phi_y^{n,\s}, \CS_{\s,x}^{\delta} \eta}\bigr| \\ &\le
\sum_{y \in \Lambda_n^\s} \bigl|\delta A_y^n\bigr|\,\bigl|\scal{\phi_y^{n,\s}, \CS_{\s,x}^{\delta} \eta}\bigr| \lesssim
K_1 2^{-(n-n_0)(|\s|+\alpha)-\gamma n}\;,
\end{equs}
where we made use of \eref{e:exprdeltaA} for the last bound.
Summing this bound over all $n \ge n_0$, we obtain again a bound of order $K_1 \delta^\gamma$, as required.
It remains to obtain a similar bound for the quantity
\begin{equ}
\sum_{n \ge n_0} \bigl(\delta f_n - (\CP_{n+1} - \CP_n)\zeta_x\bigr) \bigl(\CS_{\s,x}^{\delta} \eta\bigr)\;.
\end{equ}
Note that $\delta f_n$ is nothing but the projection of $f_{n+1}$ onto the
space $V_n^\perp$. Similarly, $(\CP_{n+1} - \CP_n)\zeta_x$ is the projection of $\zeta_x$
onto that same space. As a consequence, we have the identity
\begin{equs}
\bigl(\delta f_n - &(\CP_{n+1} - \CP_n)\zeta_x\bigr) \bigl(\CS_{\s,x}^{\delta} \eta\bigr) \\
&= \sum_{z\in \Lambda_\s^{n+1}}\sum_{y\in \Lambda_\s^{n}}\sum_{\psi \in \Psi} \scal{\zeta_z - \zeta_x,\phi_z^{n+1,\s}} \scal{\phi_z^{n+1,\s}, \psi_y^{n,\s}}
\scal[b]{\psi_y^{n,\s}, \CS_{\s,x}^{\delta} \eta}\;.
\end{equs}
Note that this triple sum only contains of the order of $2^{(n-n_0)|\s|}$ terms since,
for any given value of $y$, the sum over $z$ only has a fixed finite number of non-vanishing terms.
At this stage, we make use of the first bound in \eref{e:boundsPhi}, together with the
assumption \eref{e:assGenReconst} and the fact that $2^{-n_0} \lesssim \|x-z\|_\s \lesssim \delta$ 
for every term in this sum. This yields for this
expression a bound of the order
\begin{equ}
K_12^{(n-n_0)|\s|} \delta^{\gamma-\alpha}  2^{- {n|\s|\over 2} - \alpha n} 2^{{n_0 |\s|\over 2}}
2^{-(n-n_0)(r+|\s|/2)} = K_1\delta^{\gamma-\alpha} 2^{-r(n-n_0)-\alpha n}  \;.
\end{equ}
Since, by assumption, $r$ is sufficiently large so that $r > |\alpha|$, this expression converges to $0$ as $n \to \infty$. Summing over $n \ge n_0$ and combining all of the above bounds, the claim follows at once.
\end{proof}

\begin{remark}
As before, the construction is completely local. As a consequence, the required bounds
hold over a compact $\K$, provided that the assumptions hold over its $1$-fattening
$\bar \K$.
\end{remark}

We now finally have all the elements in place to give the proof of Theorem~\ref{theo:reconstruction}.

\begin{proof}[of Theorem~\ref{theo:reconstruction}]
We first consider the case $\gamma > 0$, where the operator $\CR$ is unique.
In order to construct $\CR$, we will proceed by successive approximations, using
a multiresolution analysis. Again, we fix a wavelet basis as above associated with
a compactly supported scaling function $\phi$. We choose $\phi$ to be $\CC^{r}$ for 
$r > |\min A|$. (Which in particular also implies that the elements $\psi \in \Psi$
annihilate polynomials of degree $r$.)

Since, for any given $n>0$, the functions $\phi_x^{n,\s}$ are orthonormal and since, as $n \to \infty$, they
get closer and closer to forming a basis of very sharply localised functions of $L^2$, 
it appears natural to define a sequence of 
operators $\CR_n \colon \CD^\gamma \to \CC^{r}$ by
\begin{equ}
\CR_n f = \sum_{x \in \Lambda_n^\s} \bigl(\Pi_x  f(x)\bigr)(\phi_x^{n,\s})\, \phi_x^{n,\s}\;,
\end{equ}
and to define $\CR$ as the limit of $\CR_n$ as $n \to \infty$, if such a limit exists.

We are thus precisely in the situation of Proposition~\ref{prop:reconstrGen} with
$\zeta_x = \Pi_x f(x)$. Since we are interested in a local statement, we only need to
construct the distribution $\CR f$ acting on test functions supported on a fixed compact
domain $\K$. As a consequence, since all of our constructions involve some fixed wavelet
basis, it suffices to obtain bounds on the wavelet coefficients $\psi_x^n$ with
$x$ such that $\psi_x^n$ is supported in $\bar \K$, the $1$-fattening of $\K$.

It follows from the definitions of $\CD^\gamma$ and the space of models $\MM_\TT$ that,
for such values of $x$, one has
\begin{equ}
|\scal{\Pi_x f(x), \phi_x^{n,\s}}| \lesssim \|f\|_{\gamma;\bar\K} \|\Pi\|_{\gamma;\bar\K} 2^{-{n|\s|\over 2}- \alpha n}\;,
\end{equ}
where, as before, $\alpha = \min A$ is the smallest homogeneity arising in the description of
the regularity structure $\TT$. Similarly, we have
\begin{equs}
|\scal{\Pi_x f(x) - \Pi_y f(y), \phi_x^{n,\s}}|
&= \bigl|\scal[b]{\Pi_x \bigl(f(x) - \Gamma_{xy} f(y)\bigr), \phi_x^{n,\s}}\bigr| \label{e:boundDiffzeta1}\\
&\lesssim 
\sum_{\ell < \gamma} \$f\$_{\gamma;\bar\K} \|\Pi\|_{\gamma;\bar\K} \|x-y\|_\s^{\gamma-\ell} 2^{-{n|\s|\over 2}- \ell n}\;,
\end{equs}
where the sum runs over elements in $A$. Since, in the assumption of Proposition~\ref{prop:reconstrGen}, 
we only consider points $(x,y)$ such that $\|x-y\|_\s \gtrsim 2^{-n}$, the bound 
\eref{e:assGenReconst} follows.

As a consequence, we can apply Theorem~\ref{theo:sowingLemma} to construct a limiting distribution
$\CR f = \lim_{n \to \infty} \CR_n f$, where convergence takes place in $\CC_\s^{\bar \alpha}$ for
every $\bar \alpha < \alpha$. Furthermore, the limit does itself belong to $\CC_\s^\alpha$.
The bound \eref{e:boundRf} follows immediately from Proposition~\ref{prop:reconstrGen}.

In order to obtain the bound \eref{e:boundRfDiff}, we use again Proposition~\ref{prop:reconstrGen},
but this time with $\zeta_x = \Pi_x f(x) - \bar \Pi_x \bar f(x)$. We then have the identity
\begin{equ}
\zeta_x - \zeta_y = \Pi_x (f(x) - \Gamma_{xy} f(y) - \bar f(x) + \bar \Gamma_{xy} \bar f(y)\bigr)
+ (\Pi_x - \bar \Pi_x) \bigl(\bar f(x) - \bar \Gamma_{xy} \bar f(y)\bigr)\;.
\end{equ}
Similarly to above, it then follows from the definition of $\$f;\bar f\$_{\gamma;\K}$ that
\begin{equ}
\bigl|\scal{\zeta_x - \zeta_y, \phi_x^{n,\s}}\bigr| \lesssim  \bigl(\|\Pi\|_{\gamma;\K} \$f;\bar f\$_{\gamma;\K} + \|\Pi - \bar \Pi\|_{\gamma;\K} \$\bar f\$_{\gamma;\K}\bigr) \|x-y\|_\s^{\gamma-\alpha} 2^{- {n|\s|\over 2} - \alpha n}\;,
\end{equ}
from which the requested bound follows at once.

The bound \eref{e:boundRfDiff2} is obtained again from Proposition~\ref{prop:reconstrGen} 
with $\zeta_x = \Pi_x f(x) - \bar \Pi_x \bar f(x)$. This time however, we aim to obtain
bounds on this quantity by only making use of bounds on $\|f - \bar f\|_{\gamma;\K}$ rather than
$\$f ; \bar f\$_{\gamma;\K}$. Note first that, as a consequence of \eref{e:boundDiffzeta1},
we have the bound
\begin{equ}[e:firstBoundzeta]
\bigl|\scal{\zeta_x - \zeta_y, \phi_x^{n,\s}}\bigr| \lesssim \|x-y\|_\s^{\gamma-\alpha} 2^{-{n|\s|\over 2}- \alpha n}\;.
\end{equ}
On the other hand, we can rewrite $\zeta_x - \zeta_y$ as
\begin{equs}
\zeta_x - \zeta_y &= \Pi_x \bigl(f(x) -  \bar f(x)\bigr) + (\bar \Pi_x - \Pi_x)\bigl(\bar \Gamma_{xy} \bar f(y) - \bar f(x)\bigr) \\
&\qquad - \Pi_x \Gamma_{xy} \bigl(f(y) - \bar f(y)\bigr)
+ \Pi_x\bigl(\bar \Gamma_{xy}  -  \Gamma_{xy} \bigr)\bar f(x)\;.
\end{equs}
It follows at once that one has the bound
\begin{equ}
\bigl|\scal{\zeta_x - \zeta_y, \phi_x^{n,\s}}\bigr| \lesssim \bigl(\|f-\bar f\|_{\gamma;\K}
+ \|\Pi-\bar \Pi\|_{\gamma;\K} + \|\Gamma-\bar \Gamma\|_{\gamma;\K}\bigr) 2^{-{n|\s|\over 2}- \alpha n}\;.
\end{equ}
Combining this with \eref{e:firstBoundzeta} and making use of the bound
$a\wedge b \le a^\kappa b^{1-\kappa}$, which is valid for any two positive numbers $a$ and $b$,
we have
\begin{equ}
\bigl|\scal{\zeta_x - \zeta_y, \phi_x^{n,\s}}\bigr| \lesssim \bigl(\|f-\bar f\|_{\gamma;\K}
+ \|\Pi-\bar \Pi\|_{\gamma;\K} + \|\Gamma-\bar \Gamma\|_{\gamma;\K}\bigr)^\kappa \|x-y\|_\s^{\bar \gamma-\alpha} 2^{-{n|\s|\over 2}- \alpha n}\;,
\end{equ}
from which the claimed bound follows.

We now prove the claim for $\gamma \le 0$. It is clear that in this case $\CR$ cannot be unique since,
if $\CR f$ satisfies \eref{e:boundRf} and $\xi \in \CC_\s^\gamma$, then $\CR f + \xi$ does again satisfy \eref{e:boundRf}.
Still, the existence of $\CR f$ is not completely trivial in general since 
$\Pi_x f(x)$ itself only belongs to $\CC_\s^\alpha$ and one can have $\alpha < \gamma \le 0$ in general. 
It turns out that one very simple choice for $\CR f$ is given by
\begin{equ}[e:Rfgen]
\CR f = \sum_{n \ge 0} \sum_{x\in \Lambda_\s^n} \sum_{\psi \in \Psi} \scal{\Pi_x f(x),\psi_x^{n,\s}} \psi_x^{n,\s}
+ \sum_{x\in \Lambda_\s^0}\scal{\Pi_x f(x),\phi_x^{0,\s}} \phi_x^{0,\s}\;.
\end{equ}
This is obviously not canonical: different choices for our multiresolution analysis yield different
definitions for $\CR$. However, it has the  advantage of not relying at all on the axiom of choice,
which was used in \cite{Extension} to prove a similar result in the one-dimensional case.
Furthermore, it has the additional property that if $f$ is ``constant'' in the sense that
$f(x) = \Gamma_{xy}f(y)$ for any two points $x$ and $y$, then one has the identity
\begin{equ}[e:propRf]
\CR f = \Pi_x f(x)\;,
\end{equ}
where the right hand side is independent of $x$ by assumption. (This wouldn't be the case if the second term in
\eqref{e:Rfgen} were absent.)
Actually, our construction is related in spirit to the one given in \cite{Unterberger}, but it has the advantage of
being very straightforward to analyse.

For $\CR f$ as in \eref{e:Rfgen}, it remains to show that \eref{e:boundRf} holds. Note first that the second
part of \eqref{e:Rfgen} defines a smooth function, so that we can discard it.
To bound the remainder, let $\eta$
be a suitable test function and note that one has the bounds
\begin{equ}
|\scal{\CS_{\s,x}^\delta \eta, \psi_y^{n,\s}}| \lesssim 
\left\{\begin{array}{cl}
	2^{-n{|\s|\over 2}-rn}\delta^{-|\s|-r} & \text{if $2^{-n} \le \delta$,} \\
	2^{n{|\s|\over 2}} & \text{otherwise.}
\end{array}\right.
\end{equ}
Furthermore, one has of course $\scal{\CS_{\s,x}^\delta \eta, \psi_y^{n,\s}} = 0$
unless $\|x-y\|_\s \lesssim \delta + 2^{-n}$.
It also follows immediately from the definition \eref{e:Rfgen} that one has the bound
\begin{equs}
\bigl|\bigl(\CR f - \Pi_x f(x)\bigr)(\psi_y^{n,\s})\bigr| &= \bigl|\bigl(\Pi_y f(y) - \Pi_x f(x)\bigr)(\psi_y^{n,\s})\bigr|
= \bigl|\Pi_y \bigl(f(y) - \Gamma_{yx} f(x)\bigr)(\psi_y^{n,\s})\bigr| \\
&\lesssim \sum_{\beta < \gamma} \|x-y\|_\s^{\gamma-\beta} 2^{-n{|\s|\over 2} - \beta n}\;,
\end{equs}
where the proportionality constant is as in \eref{e:boundRf}.
These bounds are now inserted into the identity
\begin{equ}
\bigl(\CR f - \Pi_x f(x)\bigr)(\CS_{\s,x}^\delta \eta) = 
\sum_{n > 0} \sum_{y\in \Lambda_\s^n} \sum_{\psi \in \Psi} \bigl(\CR f - \Pi_x f(x)\bigr)(\psi_y^{n,\s})
\scal{\CS_{\s,x}^\delta \eta, \psi_y^{n,\s}} \;.
\end{equ}
For the terms with $2^{-n} \le \delta$, we thus obtain a contribution of the order
\begin{equ}
\delta^{|\s|} \sum_{2^{-n} \le \delta} 2^{n|\s|} 
\sum_{\beta < \gamma} \delta^{\gamma-\beta} 2^{-n{|\s|\over 2} - \beta n}
2^{-n{|\s|\over 2}-rn}\delta^{-|\s|-r} \lesssim \delta^{\gamma}\;.
\end{equ}
Here, the bound follows from the fact that we have chosen $r$ such that $r>|\gamma|$ and the factor
$\delta^{|\s|} 2^{n|\s|}$ counts the number of non-zero terms appearing in the sum over $y$. 
For the terms with $2^{-n} > \delta$, we similarly obtain a contribution of
\begin{equ}
\sum_{2^{-n} > \delta} \sum_{\beta < \gamma} \delta^{\gamma-\beta} 2^{-n{|\s|\over 2} - \beta n}
2^{n{|\s|\over 2}}\lesssim \delta^{\gamma}\;,
\end{equ}
where we used the fact that $\beta < \gamma \le 0$. The claim then follows at once.
\end{proof}

\begin{remark}\label{rem:boundReconstr}
Recall that in Proposition~\ref{prop:reconstrGen}, the bound on $f - \zeta_x$ depends on
$K_1$ but not on $K_2$. This shows that in the reconstruction theorem, 
the bound on $\CR f - \Pi_x f(x)$ only depends on the second part of the definition of
$\$f\$_{\gamma;\K}$. This remark will be important when dealing with
singular modelled distributions in Section~\ref{sec:singular} below.
\end{remark}

\subsection{The reconstruction theorem for functions}
\label{sec:reconFcn}

A very important special case is given by the situation in which $\TT$ contains 
a copy of the canonical regularity structure $\TT_{d,\s}$ 
(write $\bar T \subset T$ for the model space associated to the abstract polynomials) 
as in Remark~\ref{rem:polynomStructures},
and where the model $(\Pi, \Gamma)$ we consider yields the canonical polynomial model
when restricted to $\bar T$. 
We consider the particular case of the reconstruction theorem
applied to elements $f \in \CD^\gamma(V)$, where $V$ is a sector 
of regularity $0$, but such that
\begin{equ}[e:propV]
V \subset \bar T + T_\alpha^+\;,
\end{equ}
for some $\alpha \in (0,\gamma)$. Loosely speaking, this states that the 
elements of the model $\Pi$
used to describe $\CR f$ consist only of polynomials and of functions
that are H\"older regular of order $\alpha$ or more.

This is made more precise by the following result:

\begin{proposition}\label{prop:recFcn}
Let $f \in \CD^\gamma(V)$, where $V$ is a sector as in \eref{e:propV}. Then, 
$\CR f$ coincides with the function given by 
\begin{equ}[e:wantedRf]
\CR f(x) = \scal{\one, f(x)}\;,
\end{equ}
and one has $\CR f \in \CC_\s^\alpha$.
\end{proposition}

\begin{proof}
The fact that the function $x\mapsto \scal{\one, f(x)}$ belongs to $\CC_\s^\alpha$ is an immediate consequence
of the definitions and the fact that the projection of $f$ onto $\bar T$ belongs to $\CD^\alpha$.
It follows immediately that one has
\begin{equ}
\int_{\R^d} \bigl(\CR f(x) - \scal{\one, f(x)}\bigr)\, \psi_y^\lambda(x)\,dx \lesssim \lambda^\alpha\;,
\end{equ}
from which, by the uniqueness of the reconstruction operator,
we deduce that one does indeed have the identity \eref{e:wantedRf}.
\end{proof}

Another useful fact is the following result showing that once we know 
that $f \in \CD^\gamma$ for some $\gamma > 0$, the components of $f$ 
in $\bar T_k$ for $0 < k < \gamma$ are uniquely determined by the knowledge of
the remaining components. More precisely, we have

\begin{proposition}\label{prop:uniqueDeriv}
If $f, g \in \CD^\gamma$ with $\gamma > 0$ are such that 
$f(x) - g(x) \in \bigoplus_{0 < k < \gamma} \bar T_k$, then $f = g$.
\end{proposition}

\begin{proof}
Setting $h = f-g$, one has $\CR h = 0$ from the uniqueness of the reconstruction operator. 
The fact that this implies that $h = 0$ was already shown in Remark~\ref{rem:uniquehatphi}.
\end{proof}

\begin{remark}
In full generality, it is not true that $h$ is completely determined by the knowledge
of $\CR h$. Actually, whether such a determinacy holds or not depends on the intricate details
of the particular model $(\Pi,\Gamma)$ that is being considered. However, for
models that are built in a ``natural'' way from a sufficiently non-degenerate Gaussian
process, it does tend to be the case that $\CR h$ fully determines $h$. See \cite{Natesh} for
more details in the particular case of rough paths.
\end{remark}

\subsection{Consequences of the reconstruction theorem}

To conclude this section, we provide a few very useful consequences of the reconstruction theorem
which shed some light on the interplay between $\Pi$ and $\Gamma$.
First, we show that for $\alpha > 0$, the action of $\Pi_x$ on $T_\alpha$ is
completely determined by $\Gamma$. In  a way, one can interpret this result as a 
generalisation of \cite[Theorem~2.2.1]{Lyons}.

\begin{proposition}\label{prop:extension}
Let $\TT$ be a regularity structure, let $\alpha > 0$, and let $(\Pi,\Gamma)$ be a model for $\TT$
over $\R^d$ with scaling $\s$. Then, the action of $\Pi$ on $T_\alpha$ is 
completely determined by the action of $\Pi$ on $T_\alpha^-$ and the action of $\Gamma$ on
$T_\alpha$. Furthermore, one has the bound 
\begin{equ}[e:boundExt]
\sup_{x \in \K}\sup_{\delta < 1}\sup_{\phi \in\CB^r_{\s,0}}\sup_{a \in T_\alpha \atop \|a\| \le 1} \delta^{-\alpha}|(\Pi_x a)(\CS_{\s,x}^\delta \phi)| \le \|\Pi\|_{\alpha;\bar \K} \|\Gamma\|_{\alpha;\bar \K}\;, 
\end{equ}
where $\bar \K$ denotes the $1$-fattening of $\K$ as before and $r > |\min A|$.
If $(\bar \Pi, \bar \Gamma)$ is a second model for the same regularity structure,
one furthermore has the bound
\begin{equs}[e:boundExtDiff]
\sup_{x\in \K}\sup_{\delta;\phi;a}\delta^{-\alpha}|(\Pi_x a-\bar \Pi_x a)(\CS_{\s,x}^\delta \phi)| &\le \|\Pi - \bar \Pi\|_{\alpha;\bar \K} \bigl(\|\Gamma\|_{\alpha;\bar \K} + \|\Gamma\|_{\alpha;\bar \K}\bigr) \\
&\quad + \|\Gamma - \bar \Gamma\|_{\alpha;\bar \K} \bigl(\|\Pi\|_{\alpha;\bar \K} + \|\bar \Pi\|_{\alpha;\bar \K}\bigr)\;,
\end{equs}
where the supremum runs over the same set as in \eref{e:boundExt}.
\end{proposition}

\begin{proof}
For any $a \in T_\alpha$ and $x \in \R^d$, we define a function $f_{a,x}\colon \R^d \to T_\alpha^-$ by
\begin{equ}[e:fax]
f_{a,x}(y) = \Gamma_{yx} a - a\;.
\end{equ}
It follows immediately from the definitions that $f_{a,x} \in \CD^\alpha$
and that, uniformly over all $a$ with $\|a\|\le 1$,
its norm over any domain $\K$ is bounded by the corresponding 
norm of $\Gamma$. Indeed, we have the identity
\begin{equs}
\Gamma_{yz} f_{a,x}(z) - f_{a,x}(y) &= \bigl(\Gamma_{yx} a - \Gamma_{yz} a\bigr) - \bigl(\Gamma_{yx} a - a\bigr)\\
&= a - \Gamma_{yz} a\;,
\end{equs}
so that the required bound follows from Definition~\ref{def:model}.

We claim that one then has $\Pi_x a = \CR f_{a,x}$, which depends only on the action of $\Pi$ on $T_\alpha^-$. 
This follows from the fact that, for every $y \in \R^d$, one has $\Pi_x a = \Pi_y \Gamma_{yx} a$, so that
\begin{equs}
\bigl(\Pi_x a - \Pi_y f_{a,x}(y)\bigr)(\CS_{\s,y}^\lambda \eta)
&= \bigl(\Pi_y a\bigr)(\CS_{\s,y}^\lambda \eta) \lesssim \lambda^\alpha
\|\Pi\|_{\alpha;\bar \K} \$f_{a,x}\$_{\alpha;\bar \K} \\
&\le \lambda^\alpha
\|\Pi\|_{\alpha;\bar \K} \|\Gamma\|_{\alpha;\bar \K}\;, \label{e:goodBound}
\end{equs} 
for all suitable test functions $\eta$. The claim now follows from the uniqueness part of the reconstruction theorem. Furthermore, the bound \eref{e:boundExt} is a consequence of
\eref{e:goodBound} with $y=x$, noting that $f_{a,x}(x) = 0$.

It remains to obtain the bound \eref{e:boundExtDiff}.
For this, we consider two models as in the statement, and we set
$\bar f_{a,x}(y) = \bar \Gamma_{yx} a - a$,
We then apply the generalised version of the reconstruction theorem,
Proposition~\ref{prop:reconstrGen}, noting that we are exactly in the situation that it covers,
with $\zeta_y = \Pi_y f_{a,x}(y) - \bar \Pi_y \bar f_{a,x}(y)$.
We then have the identity
\begin{equs}
\zeta_y - \zeta_z &= \bigl(\Pi_y (\Gamma_{yx} - I)  - \bar \Pi_y (\bar \Gamma_{yx}-I)\bigr) a - \bigl(\Pi_z (\Gamma_{zx}-I) - \bar \Pi_z (\bar\Gamma_{zx}-I)\bigr)a\\
&= \Pi_y (\Gamma_{yz} - I)a   - \bar \Pi_y (\bar \Gamma_{yz}-I) a \\
&= (\Pi_y - \bar \Pi_y) (\Gamma_{yz} - I)a   + \bar \Pi_y (\Gamma_{yz} - \bar \Gamma_{yz}) a \;.
\end{equs}
It follows that one has the bound
\begin{equs}
2^{n |\s|\over 2}\scal{\zeta_y - \zeta_z, \phi_y^{n,\s}} &\le \|\Pi- \bar \Pi\|_{\alpha;\bar \K}\|\Gamma\|_{\alpha;\bar \K} \sum_{\beta < \alpha}
\|y-z\|_\s^{\alpha - \beta}2^{- \beta n}\\
&\qquad + \|\Gamma - \bar\Gamma\|_{\alpha;\bar \K}\|\bar \Pi\|_{\alpha;\bar \K} \sum_{\beta < \alpha}
\|y-z\|_\s^{\alpha - \beta}2^{- \beta n}\;,
\end{equs}
where, in both instances, the sum runs over elements in $A$. 
Since we only need to consider pairs $(y,z)$ such that
$\|y-z\|_\s \ge 2^{-n}$, this does imply the bound \eref{e:assGenReconst} with the desired constants,
so that the claim follows from Proposition~\ref{prop:reconstrGen}.
\end{proof}

Another consequence of the reconstruction theorem is that, in order to characterise a model
$(\Pi,\Gamma)$ on some sector $V \subset T$, 
it suffices to know the action of $\Gamma_{xy}$ on $V$, 
as well as the values of $\bigl(\Pi_x a\bigr)(\phi_x^{n,\s})$ for $a \in V$, $x \in \Lambda_\s^n$ and 
$\phi$ the scaling function of some fixed sufficiently regular multiresolution analysis as 
in Section~\ref{sec:wavelets}.
More precisely, we have:

\begin{proposition}\label{prop:wavelet}
A model $(\Pi,\Gamma)$ for a given regularity structure is completely determined by 
the knowledge of $\bigl(\Pi_x a\bigr)(\phi_x^{n,\s})$ for $x \in \Lambda_\s^n$ and $n \ge 0$,
as well as $\Gamma_{xy} a$ for $x,y \in \R^d$.

Furthermore, for every compact set $\K\subset \R^d$ and every sector $V$,
one has the bound
\begin{equ}[e:boundNormPi]
\|\Pi\|_{V;\K} \lesssim  \bigl(1+\|\Gamma\|_{V;\K}\bigr) \sup_{\alpha \in A_V} \sup_{a \in V_\alpha} \sup_{n \ge 0} \sup_{x \in \Lambda_\s^n(\bar \K)} 2^{\alpha n + {n|\s|\over 2}}  {\bigl|\bigl(\Pi_x a\bigr)(\phi_x^{n,\s})\bigr|\over \|a\|}\;.
\end{equ}
Here, we denote by $\|\Pi\|_{V;\K}$ the norm given as in Definition~\ref{def:model}, but where we
restrict ourselves to vectors $a \in V$.
Finally, for any two models $(\Pi, \Gamma)$ and $(\bar \Pi, \bar \Gamma)$, one has
\begin{equ}
\|\Pi-\bar \Pi\|_{V;\K} \lesssim  \bigl(1+\|\Gamma\|_{V;\K}\bigr) \sup_{\alpha \in A_V} \sup_{a \in V_\alpha} \sup_{n \ge 0} \sup_{x \in \Lambda_\s^n(\bar \K)} 2^{\alpha n + {n|\s|\over 2}}  {\bigl|\bigl(\Pi_x a - \bar \Pi_x a\bigr)(\phi_x^{n,\s})\bigr|\over \|a\|}\;.
\end{equ}

\end{proposition}

\begin{proof}
Given $a\in V_\alpha$ and $x \in \R^d$, we define similarly to above a function
$f_x^a\colon \R^d \to V$ by $f_x^a(y) = \Gamma_{yx}a$. (This time $\alpha$ can be
arbitrary though.)
One then has $\Pi_y f_x^a(y) = \Pi_y \Gamma_{yx} a = \Pi_x a$, so that 
$\CR f_x^a = \Pi_x a$.
On the other hand, the proof of the reconstruction theorem only makes
use of the values $\bigl(\Pi_x a\bigr)(\phi_x^{n,\s})$ and the function $(x,y) \mapsto \Gamma_{xy}$, 
so that the claim follows.

The bound \eref{e:boundNormPi}, as well as the corresponding bound on $\Pi - \bar \Pi$
are an immediate consequence of Theorem~\ref{theo:sowingLemma}, noting again that 
the coefficients $A_x^n$ only involve evaluations of $\bigl(\Pi_x a\bigr)(\phi_x^{n,\s})$
and the map $\Gamma_{xy}$.
\end{proof}

Although this result was very straightforward to prove, it is very important when
constructing random models for a regularity structure.
Indeed, provided that one has suitable moment estimates, it is in many cases possible
to show that the right hand side of \eref{e:boundNormPi} is bounded almost surely.
One can then make use of this knowledge to \textit{define} the distribution $\Pi_x a$
by $\CR f_x^a$ via the reconstruction theorem.
This is completely analogous to Kolmogorov's continuity criterion where the knowledge of
a random function on a dense countable subset of $\R^d$ is sufficient to define
a random variable on the space of continuous functions on $\R^d$ as a consequence 
of suitable moment bounds. 
Actually, the standard proof of Kolmogorov's continuity criterion is very similar
in spirit to the proof given here, since it also relies on the hierarchical 
approximation of points in $\R^d$ by points with dyadic coordinates, see for
example \cite{RevYor}.

\subsection{Symmetries}
\label{sec:symmetric}

It will often be useful to consider modelled distributions that, although they are
defined on all of $\R^d$, are known to obey certain symmetries. 
Although the extension of the framework to such a situation is completely straightforward,
we perform it here mostly in order to introduce the relevant notation which will be used
later.

Consider some discrete symmetry group $\SS$ which acts on $\R^d$ via isometries
$T_g$. In other words, for every $g \in \SS$, $T_g$ is an isometry of $\R^d$ and 
$T_{g \bar g} = T_g \circ T_{\bar g}$. Given a regularity structure $\TT$, we call a map
$M\colon \SS \to L^0$ (where $L^0$ is as in Section~\ref{sec:automorphisms}) 
an \textit{action} of $\SS$ on $\TT$ if $M_g \in \Aut \TT$ for every
$g \in \SS$ and furthermore one has the identity $M_{g\bar g} = M_{\bar g} \circ M_{g}$
for any two elements $g, \bar g \in \SS$.
Note that $\SS$ also acts naturally on any space of functions on $\R^d$ via
the identity
\begin{equ}
\bigl(T_g^\star \psi\bigr)(x) = \psi(T_g^{-1}x)\;.
\end{equ}
With these notations, the following definition is natural:

\begin{definition}\label{def:symmetric}
Let $\SS$ be a group of symmetries of $\R^d$ acting on some regularity structure $\TT$.
A model $(\Pi,\Gamma)$ for $\TT$ is said to be \textit{adapted} to the action of $\SS$
if the following two properties hold: 
\begin{mylist}
\item For every test function $\psi \colon \R^d \to \R$, every $x \in \R^d$, every $a \in T$,
and every $g \in \SS$, one has the identity 
$\bigl(\Pi_{T_g x} a\bigr)(T_g^\star \psi) = \bigl(\Pi_{x} M_g a\bigr)(\psi)$.
\item For every $x,y \in \R^d$ and every $g \in \SS$, one has the identity 
$M_g \Gamma_{T_g x T_g y} = \Gamma_{xy} M_g$.
\end{mylist}
A modelled distribution $f \colon \R^d \to T$ is said to be \textit{symmetric} if
$M_g f(T_g x) = f(x)$ for every $x \in \R^d$ and every $g \in \SS$.
\end{definition}

\begin{remark}
One could additionally impose that the norms on the spaces $T_\alpha$ are chosen in such a way
that the operators $M_g$ all have norm $1$. This is not essential but makes some expressions nicer.
\end{remark}

\begin{remark}\label{rem:symcanon}
In the particular case where $\TT$ contains the polynomial regularity structure $\TT_{d,\s}$
and $(\Pi,\Gamma)$ extends its canonical model, the action $M_g$ of $\SS$ on the 
abstract element $X$ is necessarily given by $M_g X = A_g X$, where $A_g$ is the $d\times d$
matrix such that $T_g$ acts on elements of $\R^d$ by $T_g x = A_g x + b_g$, for some vector
$b_g$. This can be checked by making use of the first identity in Definition~\ref{def:symmetric}.

The action on elements of the form $X^k$ for an arbitrary multiindex $k$ is then naturally
given by $M_g (X^k) = (A_g X)^k = \prod_i (\sum_j A_g^{ij} X_j)^{k_i}$.
\end{remark}

\begin{remark}
One could have relaxed the first property to the identity
$\bigl(\Pi_{T_g x} a\bigr)(T_g^\star \psi) = (-1)^{\eps(g)}\bigl(\Pi_{x} M_g a\bigr)(\psi)$, where $\eps \colon \SS \to \{\pm 1\}$ is any group morphism.
This would then also allow to treat Dirichlet boundary conditions in domains 
generated by reflections. We will not consider this for the sake of conciseness.
\end{remark}

\begin{remark}\label{rem:polynomSym}
While Definition~\ref{def:symmetric} ensures that the model $(\Pi,\Gamma)$ behaves ``nicely''
under the action of $\SS$, this does \textit{not} mean that the distributions $\Pi_x$ themselves
are symmetric in the sense that $\Pi_x(\psi) = \Pi_x(T_g^\star \psi)$.
The simplest possible example on which this is already visible is the case where
$\SS$ consists of a subgroup of the translations. If we take $\TT$ to be the
canonical polynomial structure and $M$ to be the trivial action, then it is straightforward
to verify that the canonical model $(\Pi,\Gamma)$ is indeed adapted to the action of $\SS$.
Furthermore, $f$ being ``symmetric'' in this case simply means that $f$ has a suitable
periodicity. However, polynomials themselves of course aren't periodic.
\end{remark}

Our definitions were chosen in such a way that one has the following result.

\begin{proposition}\label{prop:symmetric}
Let $\SS$ be as above, acting on $\TT$, let $(\Pi,\Gamma)$ be adapted to the action of $\SS$, 
and let $f \in \CD^\gamma$ (for some $\gamma > 0$)
be symmetric. Then, $\CR f$ satisfies $\bigl(\CR f\bigr)(T_g^\star \psi) = 
\bigl(\CR f\bigr)(\psi)$ for
every test function $\psi$ and every $g \in \SS$.
\end{proposition}

\begin{proof}
Take a smooth compactly supported test function $\phi$ that integrates to $1$ and fix
an element $g \in \SS$. Since $T_g$ is an isometry of $\R^d$, its action is given by
$T_g(x) = A_g x + b_g$ for some orthogonal matrix $A_g$ and a vector $b_g \in \R^d$.
We then define $\phi^g(x) = \phi(A_g^{-1} x)$, which is a test function having the same
properties as $\phi$ itself.

One then has the identity
\begin{equ}
\psi(x) = \lim_{\lambda \to 0} \int_{\R^d} \bigl(\CS_{\s,y}^\lambda \phi\bigr)(x)\,\psi(y)\,dy\;.
\end{equ}
Furthermore, this convergence holds not only pointwise, but in every space $\CC^k$. As a consequence
of this, combined with the reconstruction theorem, we have
\begin{equs}
\bigl(\CR f\bigr)(\psi) &= \lim_{\lambda \to 0}\int_{\R^d} \bigl(\CR f\bigr)\bigl(\CS_{\s,y}^\lambda \phi\bigr)\,\psi(y)\,dy
= \lim_{\lambda \to 0}\int_{\R^d} \bigl(\Pi_y f(y)\bigr)\bigl(\CS_{\s,y}^\lambda \phi\bigr)\,\psi(y)\,dy\\
&= \lim_{\lambda \to 0}\int_{\R^d} \bigl(\Pi_{T_g y} M_g^{-1} M_gf(T_g y)\bigr)\bigl(T_g^\star \CS_{\s,y}^\lambda \phi\bigr)\,\psi(y)\,dy\\
&= \lim_{\lambda \to 0}\int_{\R^d} \bigl(\Pi_{y} f(y)\bigr)\bigl(T_g^\star \CS_{\s,T_g^{-1}y}^\lambda \phi\bigr)\,\bigl(T_g^\star \psi\bigr)(y)\,dy\\
&= \lim_{\lambda \to 0}\int_{\R^d} \bigl(\Pi_{y} f(y)\bigr)\bigl(\CS_{\s,y}^\lambda \phi^g\bigr)\,\bigl(T_g^\star \psi\bigr)(y)\,dy = \bigl(\CR f\bigr)(T_g^\star \psi)\;,
\end{equs}
as claimed. Here, we used the symmetry of $f$ and the adaptedness of $(\Pi,\Gamma)$ 
to obtain the second line, while we performed a simple change of variables to obtain the 
third line.
\end{proof}

One particularly nice situation is that when the fundamental domain $\K$ of $\SS$ is compact in $\R^d$.
In this case, provided of course that $(\Pi,\Gamma)$ is adapted to the action of $\SS$,
the analytical bounds \eref{e:boundPi} automatically hold over all of $\R^d$. The same is true for the bounds
\eref{e:boundDgamma} if $f$ is a symmetric modelled distribution. 

\section{Multiplication}
\label{sec:multiplication}

So far, our theory was purely descriptive: 
we have shown that $T$-valued maps with a suitable regularity property can be used to provide a precise
local description of a class of distributions that locally look like a given family of ``model distributions''. 
We now proceed to show that one can perform a number of operations on these modelled distributions, 
while still retaining their description as elements in some $\CD^\gamma$. 

The most conceptually non-trivial of
such operations is of course the multiplication of distributions, which we address in this section. 
Surprisingly, even though elements in $\CD^\gamma$ describe distributions that can potentially
be extremely irregular, it is possible to work with them largely as if they consisted
of continuous functions. In particular, if we are given a product $\star$ on $T$
(see below for precise assumptions on $\star$), then
we can multiply modelled distributions by forming the pointwise product
\begin{equ}[e:defProd]
\bigl(f\star g\bigr)(x) = f(x)\star g(x)\;,
\end{equ}
and then projecting the result back to $T_\gamma^-$ for a suitable $\gamma$.

\begin{definition}
A continuous bilinear map $(a,b) \mapsto a\star b$ is a product on $T$ if
\begin{claim}
\item  For every $a \in T_\alpha$ and
$b \in T_\beta$, one has $a\star b \in T_{\alpha + \beta}$.
\item  One has $\one \star a = a\star \one = a$ for every $a \in T$.
\end{claim}
\end{definition}

\begin{remark}
In all of the situations considered later on, the product $\star$ will furthermore
be associative and commutative. However, these properties do not seem to be essential
as far as the abstract theory is concerned.
\end{remark}

\begin{remark}
What we mean by ``continuous'' here is that for any two indices $\alpha, \beta \in A$, the bilinear
map $\star\colon T_\alpha\times T_\beta \to T_{\alpha+\beta}$ is continuous.
\end{remark}

\begin{remark}
If $V_1$ and $V_2$ are two sectors of $\TT$ and $\star$ is defined as  a bilinear map on
$V_1\times V_2$, we can always extend it to $T$ by setting $a \star b = 0$ if either $a$
belongs to the complement of $V_1$ or $b$ belongs to the complement of $V_2$.
\end{remark}

\begin{remark}
We could have slightly relaxed the first assumption by allowing $a\star b \in T_{\alpha + \beta}^+$.
However, the current formulation appears more natural in the context of interpreting elements of 
the spaces $T_\alpha$ as ``homogeneous elements''. 
\end{remark}

Ideally, one would also like to impose the additional property that
$\Gamma (a\star b) = (\Gamma a)\star (\Gamma b)$ for every $\Gamma \in G$ and every $a,b \in T$.
Indeed, assume for a moment that $\Pi_x$ takes values in some function space and that the operation
$\star$ represents the actual pointwise product between two functions, namely
\begin{equ}[e:modelProduct]
\Pi_x (a\star b)(y) = (\Pi_x a)(y) \, (\Pi_x b)(y)\;.
\end{equ}
In this case, one has the
identity
\begin{equs}
\Pi_x \Gamma_{xy} (a\star b) &= \Pi_y (a\star b) = (\Pi_y a) \, (\Pi_y b) = (\Pi_x\Gamma_{xy} a) \, (\Pi_x\Gamma_{xy} b)\\
&= \Pi_x \bigl(\Gamma_{xy} a \star \Gamma_{xy} b\bigr)\;.
\end{equs}
In many cases considered in this article however, 
the model space $T$ is either finite-dimensional or, even though it is infinite-dimensional, 
some truncation still takes place and one 
cannot expect \eref{e:modelProduct} to hold exactly. Instead, the following
definition ensures that it holds up to an error which is ``of order $\gamma$''.

\begin{definition}\label{def:regularity}
Let $\TT$ be a regularity structure, let $V$ and $W$ be two sectors of $\TT$, and let $\star$ be a product
on $\TT$. The pair $(V,W)$ is said to be \textit{$\gamma$-regular} if $\Gamma (a\star b) = (\Gamma a)\star (\Gamma b)$ for every $\Gamma \in G$ and for every $a \in V_\alpha$ and $b \in W_\beta$ such that
$\alpha + \beta < \gamma$ and every $\Gamma \in G$. 

We say that $(V,W)$ is regular if it is $\gamma$-regular for every $\gamma$. 
In the case $V = W$, we say that $V$ is ($\gamma$-)regular if this is true for the pair $(V,V)$.
\end{definition}


The aim of this section is to demonstrate that, provided that a pair of sectors is $\gamma$-regular for some $\gamma > 0$,
the pointwise product between modelled distributions in these sectors yields again a modelled distribution.
Throughout this section, we assume that $V$ and $W$ are two sectors of regularities $\alpha_1$ and
$\alpha_2$ respectively. We then have the following:

\begin{theorem}\label{theo:mult}
Let $(V,W)$ be a pair of sectors with regularities 
$\alpha_1$ and $\alpha_2$ respectively, 
let $f_1 \in \CD^{\gamma_1}(V)$ and $f_2 \in \CD^{\gamma_2}(W)$,
and let $\gamma = (\gamma_1 + \alpha_2) \wedge (\gamma_2 + \alpha_1)$. Then,
provided that $(V,W)$ is $\gamma$-regular, one has $f_1\star f_2 \in \CD^\gamma(T)$ and, 
for every compact set $\K$, the bound
\begin{equ}
\$f_1 \star f_2 \$_{\gamma;\K} \lesssim \$f_1\$_{\gamma_1;\K}\$f_2\$_{\gamma_2;\K} (1+\|\Gamma\|_{\gamma_1+\gamma_2;\K})^2\;,
\end{equ}
holds for some proportionality constant only depending on the underlying structure $\TT$.
\end{theorem}

\begin{remark}\label{rem:multReg}
If we denote as before by $\CD_\alpha^\gamma$ an element of $\CD^\gamma(V)$ for some
sector $V$ of regularity $\alpha$, then Theorem~\ref{theo:mult} can loosely be stated as
\begin{equ}
f_1 \in \CD_{\alpha_1}^{\gamma_1}\quad\&\quad
f_2 \in \CD_{\alpha_2}^{\gamma_2}
\qquad\Rightarrow\qquad 
f_1 \star f_2 \in \CD_\alpha^\gamma\;,
\end{equ}
where $\alpha = \alpha_1 + \alpha_2$ and $\gamma = (\gamma_1 + \alpha_2) \wedge (\gamma_2 + \alpha_1)$.
This statement appears to be slightly misleading since it completely glosses over the assumption that
the pair $(V,W)$ be $\gamma$-regular. However, at the expense of possibly extending 
the regularity structure $\TT$ and the model $(\Pi,\Gamma)$,  we will see in Proposition~\ref{prop:extendMult} 
below that it is always possible to ensure that this assumption holds, albeit possibly in a 
non-canonical way.
\end{remark}

\begin{remark}
The proof of this result is a rather straightforward consequence of our 
definitions, combined with
standard algebraic manipulations. It has nontrivial consequences mostly when combined with 
the reconstruction theorem, Theorem~\ref{theo:reconstruction}.
\end{remark}
%
%
%

\begin{proof}[of Theorem~\ref{theo:mult}]
Note first that since we are only interested in showing that $f_1\star f_2 \in \CD^\gamma$, we
 discard all of the components in $T_\gamma^+$. (See also Remark~\ref{rem:discard}.) As a consequence, we actually consider the
function given by
\begin{equ}[e:defProdTrunc]
f(x) \eqdef \bigl(f_1\star_\gamma f_2\bigr)(x) \eqdef \sum_{m+n < \gamma} \CQ_m f_1(x)\star \CQ_n f_2(x)\;.
\end{equ}
It then follows immediately from the properties of the product that 
\begin{equ}
\|f_1\star_\gamma f_2\|_{\gamma;\K} \lesssim \|f_1\|_{V;\K}\|f_2\|_{W;\K} \;,
\end{equ}
where the proportionality constant depends only on $\gamma$ and $\TT$, but not on $\K$.

From now on we will assume that $\$f_1\$_{V;\K} \le 1$ and $\$f_2\$_{W;\K} \le 1$, which is not a restriction by bilinearity. 
It remains to obtain a bound on 
\begin{equ}
\Gamma_{xy} \bigl(f_1\star_\gamma f_2\bigr)(y) - \bigl(f_1\star_\gamma f_2\bigr)(x)\;.
\end{equ}
Using the triangle inequality and recalling that $\CQ_\ell(f_1 \star_\gamma f_2) = \CQ_\ell(f_1 \star f_2)$ for $\gamma < \ell$, we can write
\begin{equs}
\|\Gamma_{xy} f(y) - f(x)\|_\ell
&\le \|\Gamma_{xy} \bigl(f_1\star_\gamma f_2\bigr)(y) - \bigl(\Gamma_{xy} f_1(y)\bigr)\star \bigl(\Gamma_{xy} f_2(y)\bigr)\|_\ell \\
&\quad + \|\bigl(\Gamma_{xy} f_1(y) - f_1(x)\bigr)\star \bigl(\Gamma_{xy} f_2(y) - f_2(x)\bigr)\|_\ell \\
&\quad + \|\bigl(\Gamma_{xy} f_1(y) - f_1(x)\bigr)\star f_2(x)\|_\ell \\
&\quad + \|f_1(x)\star \bigl(\Gamma_{xy} f_2(y) - f_2(x)\bigr)\|_\ell\;. \label{e:splitProd}
\end{equs}
It follows from  \eref{e:defProdTrunc} and the definition of $(V,W)$ being $\gamma$-regular 
that for the first term, one has the identity
\begin{equ}[e:termDiffxy]
\Gamma_{xy} f(y) - \bigl(\Gamma_{xy} f_1(y)\bigr)\star \bigl(\Gamma_{xy} f_2(y)\bigr) = -\sum_{m+n \ge \gamma} \bigl(\Gamma_{xy} \CQ_m f_1(y)\bigr)\star \bigl(\Gamma_{xy} \CQ_n f_2(y)\bigr)\;.
\end{equ}
Furthermore, one has
\begin{equs}
\|\bigl(\Gamma_{xy} \CQ_m f_1(y)\bigr) &\star \bigl(\Gamma_{xy} \CQ_n f_2(y)\bigr)\|_\ell
\lesssim \sum_{\beta_1 + \beta_2 = \ell} \|\Gamma_{xy} \CQ_m f_1(y)\|_{\beta_1}
\|\Gamma_{xy} \CQ_n f_2(y)\|_{\beta_2} \\
&\lesssim \sum_{\beta_1 + \beta_2 = \ell} \|\Gamma\|_{\gamma_1+\gamma_2;\K}^2 \|x-y\|_\s^{m+n-\beta_1- \beta_2} \\
&\lesssim \|\Gamma\|_{\gamma_1+\gamma_2;\K}^2 \|x-y\|_\s^{\gamma - \ell} \label{e:boundProdGamma}
\end{equs}
where we have made use of the facts that $m+n \ge \gamma$ and that $\|x-y\|_\s \le 1$.

It follows from the properties of the product $\star$ that the second term in \eref{e:splitProd} 
is bounded by a constant times
\begin{equs}
\sum_{\beta_1 + \beta_2 =\ell} &\|\Gamma_{xy} f_1(y) - f_1(x)\|_{\beta_1} \|\Gamma_{xy} f_2(y) - f_2(x)\|_{\beta_2} \\
&\lesssim \sum_{\beta_1 + \beta_2 =\ell} \|x-y\|_\s^{\gamma_1-\beta_1}\|x-y\|_\s^{\gamma_2-\beta_2} 
\lesssim \|x-y\|_\s^{\gamma_1+\gamma_2-\ell}\;.
\end{equs}
The third term is bounded by a constant times
\begin{equ}
\sum_{\beta_1 + \beta_2 = \ell} \|\Gamma_{xy} f_1(y) - f_1(x)\|_{\beta_1} \|f_2(x)\|_{\beta_2} \lesssim \|x-y\|_\s^{\gamma_1-\beta_1} \,\one_{\beta_2 \ge \alpha_2} \lesssim
\|x-y\|_\s^{\gamma_1+\alpha_2-\ell}\;,
\end{equ}
where the second inequality uses the identity $\beta_1 + \beta_2 = \ell$. 
The last term is bounded similarly by reversing 
the roles played by $f_1$ and $f_2$.
\end{proof}

In applications, one would also like to have suitable continuity properties of the product as a function
of its factors. By bilinearity, it is of course straightforward to obtain bounds of the type
\begin{equs}
\|f_1 \star f_2 - g_1 \star g_2\|_{\gamma;\K} &\lesssim \|f_1-g_1\|_{\gamma_1;\K}\|f_2\|_{\gamma_2;\K}
+ \|f_2-g_2\|_{\gamma_2;\K}\|g_1\|_{\gamma_1;\K}\;,\\
\$f_1 \star f_2 - g_1 \star g_2\$_{\gamma;\K} &\lesssim \$f_1-g_1\$_{\gamma_2;\K}\$f_2\$_{\gamma_2;\K}
+ \$f_2-g_2\$_{\gamma_2;\K}\$g_1\$_{\gamma_1;\K}\;,
\end{equs}
provided that both $f_i$ and $g_i$ belong to $\CD^{\gamma_i}$ with respect to the same
model. Note also that as before the proportionality constants implicit in these bounds depend
on the size of $\Gamma$ in the domain $\K$.
However, one has also the following improved bound:

\begin{proposition}\label{prop:multDiff}
Let $(V,W)$ be as above, let $(\Pi, \Gamma)$ and $(\bar \Pi, \bar \Gamma)$ be two
models for $\TT$, and let $f_1 \in \CD^{\gamma_1}(V;\Gamma)$, $f_2 \in \CD^{\gamma_2}(W;\Gamma)$,
$g_1 \in \CD^{\gamma_1}(V;\bar \Gamma)$, and $g_2 \in \CD^{\gamma_2}(W;\bar \Gamma)$.

Then, for every $C > 0$, one has the bound
\begin{equ}
\$f_1 \star f_2 ; g_1 \star g_2\$_{\gamma;\K} \lesssim \$f_1;g_1\$_{\gamma_1;\K}
+ \$f_2;g_2\$_{\gamma_2;\K} + \|\Gamma-\bar \Gamma\|_{\gamma_1+\gamma_2;\K}\;,
\end{equ}
uniformly over all $f_i$ and $g_i$ with $\$f_i\$_{\gamma_i;\K} + \$g_i\$_{\gamma_i;\K} \le C$,
as well as models satisfying $\|\Gamma\|_{\gamma_1+\gamma_2;\K}+\|\bar \Gamma\|_{\gamma_1+\gamma_2;\K} \le C$. Here, the proportionality constant depends only on $C$.
\end{proposition}

\begin{proof}
As before, our aim is to bound the components in $T_\ell$ for $\ell < \gamma$ of the quantity
\begin{equ}
f_1(x) \star f_2(x) - g_1(x) \star g_2(x) - \Gamma_{xy} \bigl(f_1 \star_\gamma f_2\bigr)(y)
+ \bar \Gamma_{xy} \bigl(g_1 \star_\gamma g_2\bigr)(y)\;.
\end{equ}
First, as in the proof of Theorem~\ref{theo:mult}, we would like to replace 
$\Gamma_{xy} \bigl(f_1 \star_\gamma f_2\bigr)(y)$ by $\Gamma_{xy} f_1(y) \star \Gamma_{xy} f_2(y)$ and
similarly for the corresponding term involving the $g_i$.
This can be done just as in \eref{e:boundProdGamma}, which yields a bound of the order
\begin{equ}
\bigl(\|\Gamma-\bar \Gamma\|_{\gamma_1+\gamma_2;\K} + \|f_1 - g_1\|_{\gamma_1;\K} + \|f_2 - g_2\|_{\gamma_2;\K}\bigr) \|x-y\|_\s^{\gamma - \ell}\;,
\end{equ}
as required. We rewrite the remainder as
\begin{equs}
f_1(x) \star f_2(x) &- g_1(x) \star g_2(x) - \Gamma_{xy} f_1(y) \star \Gamma_{xy} f_2(y)
+ \bar \Gamma_{xy} g_1(y) \star \bar \Gamma_{xy} g_2(y)\\
&\quad = \bigl(f_1(x) - g_1(x)  - \Gamma_{xy} f_1(y) + \bar \Gamma_{xy} g_1(y)\bigr)\star f_2(x)\\
&\qquad + \Gamma_{xy}f_1(y)\star \bigl(f_2(x) - g_2(x)  - \Gamma_{xy} f_2(y) + \bar \Gamma_{xy} g_2(y)\bigr) \\
&\qquad + \bar \Gamma_{xy}\bigl( g_1(y) - f_1(y)\bigr)\star
\bigl(\bar \Gamma_{xy} g_2(y) - g_2(x)\bigr)\\
&\qquad + \bigl(\bar \Gamma_{xy} f_1(y) - \Gamma_{xy} f_1(y)\bigr)\star
\bigl(\bar \Gamma_{xy} g_2(y) - g_2(x)\bigr)\\
&\qquad + \bigl(g_1(y) - \bar \Gamma_{xy} g_1(y)\bigr)\star
\bigl(f_2(x) - g_2(x)\bigr)\\
&\quad \eqdef T_1 + T_2 + T_3 + T_4 + T_5\;. \label{e:termsTi}
\end{equs}
It follows from the definition of $\$\cdot;\cdot\$_{\gamma_1;\K}$ that we have the bound
\begin{equ}
\|T_1\|_\ell \lesssim \$f_1;g_1\$_{\gamma_1;\K}\sum_{m+n = \ell
\atop m \ge \alpha_1; n\ge \alpha_2} \|x-y\|_\s^{\gamma_1 - m}\;.
\end{equ}
(As usual, sums are performed over exponents in $A$.)
Since the largest possible value for $m$ is equal to $\ell - \alpha_2$, this is the required bound.
A similar bound on $T_2$ follows in virtually the same way.
The term $T_3$ is bounded by
\begin{equ}
\|T_3\|_\ell \lesssim \|f_1 - g_1\|_{\gamma_1;\K} \sum_{m+n = \ell \atop m \ge \alpha_1; n\ge \alpha_2} \|x-y\|_\s^{\gamma_2 - n}\;.
\end{equ}
Again, the largest possible value for $n$ is given by $\ell - \alpha_1$, so the required bound follows.
The bound on $T_4$ is obtained in a similar way, replacing $\|f_1 - g_1\|_{\gamma_1;\K}$
by $\|\Gamma - \bar \Gamma\|_{\gamma_1;\K}$. The last term $T_5$ is very similar to $T_3$ and
can be bounded in the same fashion, thus concluding the proof.
\end{proof}

As already announced earlier, 
the regularity condition on $(V,W)$ can always be satisfied by possibly extending
our regularity structure. 
However, at this level of generality, 
the way of extending $\TT$ and $(\Pi,\Gamma)$ can of course not be expected to be canonical! In practice, 
one would have to identify a ``natural'' extension, which can potentially require a great
deal of effort. Our abstract result however is:

\begin{proposition}\label{prop:extendMult}
Let $\TT$ be a regularity structure such that each of the $T_\alpha$ is finite-dimensional,
let $(V,W)$ be two sectors of $\TT$, let $(\Pi,\Gamma)$ be a model for $\TT$, and let $\gamma \in \R$. 
Then, it is always possible to find a regularity structure $\bar \TT$ containing $\TT$
and a model $(\bar \Pi, \bar \Gamma)$ for $\bar\TT$ extending $(\Pi,\Gamma)$, such that
the pair $(\iota V, \iota W)$ is $\gamma$-regular in $\bar \TT$.
\end{proposition}

\begin{proof}
It suffices to consider the situation where there exist $\alpha$ and $\beta$ in $A$ such that 
$(V,W)$ is $(\alpha+\beta)$-regular but $\star$ isn't yet defined on $V_\alpha$ and $W_\beta$.
In such a situation, we build the required extension as follows. 
First, extend the action of $G$ to $T \oplus (V_\alpha \otimes W_\beta)$ 
by setting 
\begin{equ}[e:defGamma]
\Gamma(a \otimes b) \eqdef \Gamma a \bstar \Gamma b\;,\quad a \in V_\alpha\;,\quad
b \in W_\beta\;,\quad \Gamma \in G\;,
\end{equ}
where $\bstar$ is defined on $V_\alpha \times W_\beta$ by
$a \bstar b =  a \otimes b$. (Outside of $V_\alpha \times W_\beta$, we simply set $\bstar = {\star}$.)
Then, consider \textit{some} linear equivalence relation $\sim$ on
$T_{\alpha + \beta} \oplus (V_\alpha \otimes W_\beta)$ such that
\begin{equ}[e:propRel]
a \sim b \;\Rightarrow\; \Gamma a - a = \Gamma b - b\quad \forall \Gamma \in G\;,
\end{equ}
and such that no two elements in $T_{\alpha+\beta}$ are equivalent.
(Note that the implication only goes from left to right. In particular, it is always possible
to take for $\sim$ the trivial relation under which no two distinct elements are equivalent. 
However, allowing for non-trivial equivalence relations allows to impose additional algebraic
properties, like the commutativity of $\bstar$ or Leibniz's rule.) Given such an equivalence relation,
we now define $\bar \TT = (\bar A, \bar T, \bar G)$ by setting
\begin{equ}
\bar A = A \cup \{\alpha + \beta\}\;,\qquad \bar T_{\alpha+\beta} = \bigl(T_{\alpha + \beta} \oplus (V_\alpha \otimes W_\beta)\bigr)/\sim \;.
\end{equ}
For $\gamma \neq \alpha+\beta$, we simply set $\bar T_\gamma = T_\gamma$. Furthermore,
we use $\bstar$ as the product in $\bar T$ which, by construction, coincides with $\star$, except
on $T_\alpha \otimes T_\beta$. 
Finally, the group $\bar G$ is identical to $G$
as an abstract group, but each element of $G$ is extended to $\bar T_{\alpha+\beta}$
in the way described above. Property \eref{e:propRel} ensures that this is well-defined in the sense that the 
action of $G$ on different elements of an equivalence class of $\sim$ is compatible. 

It remains to extend $(\Pi,\Gamma)$ to a model $(\bar \Pi,\bar \Gamma)$
for $\bar \TT$ as an abstract group element, with its action on $\bar T$ given by \eref{e:defGamma}. 
For $\bar \Gamma$, we simply set $\bar \Gamma_{xy} = \Gamma_{xy}$.
The definition \eref{e:propRel} then ensures that the bound \eref{e:boundPi} for $\Gamma$
also holds for elements in $\bar T_{\alpha+\beta}$.
Regarding $\bar \Pi$, since $\bar T_{\alpha+\beta}$ still contains $T_{\alpha+\beta}$ as a subspace,
it remains to define it on some basis of the complement of $T_{\alpha+\beta}$ in $\bar T_{\alpha+\beta}$.
For each such basis vector $a$, we can then proceed as in Proposition~\ref{prop:extension} to construct
$\Pi_x a$ for some (and therefore all) $x \in \R^d$. More precisely, we \textit{define} $\Pi_x a$ by 
$\Pi_x a = \CR f_{a,x}$ with $f_{a,x}$ as in \eqref{e:fax}, where $\CR$ is the reconstruction operator
given in the proof of Theorem~\ref{theo:reconstruction}. In case $\alpha+ \beta \le 0$, the choice
of $\CR$ is not unique and we explicitly make the choice given in \eqref{e:Rfgen} for a suitable wavelet basis.
This definition then implies for any two points $x$ and $z$ the identity
\begin{equ}
\Pi_z \Gamma_{zx} a - \Pi_x a = \Pi_z a - \Pi_x a + \Pi_z \bigl(\Gamma_{zx}a - a\bigr)
= \CR \bigl(f_{a,z} - f_{a,x}\bigr) + \Pi_z \bigl(\Gamma_{zx}a - a\bigr)\;,
\end{equ}
where we used the linearity of $\CR$.
Note now that $\bigl(f_{a,z} - f_{a,x}\bigr)(y) = \Gamma_{yz} \bigl(a - \Gamma_{zx}a\bigr)$, so that
we are precisely in the situation of \eqref{e:propRf}. This shows that our construction guarantees that
$\CR \bigl(f_{a,z} - f_{a,x}\bigr) = -\Pi_z \bigl(\Gamma_{zx}a - a\bigr)$, so that the algebraic identity
$\Pi_z \Gamma_{zx} a = \Pi_x a$ holds for any two points, as required. The required analytical bounds on 
$\Pi_x a$ on the other hand are an immediate consequence of Theorem~\ref{theo:reconstruction}.

As a byproduct of our construction and of Proposition~\ref{prop:extension}, 
we see that the extension is essentially unique if $\alpha + \beta > 0$,
but that there is considerable freedom whenever $\alpha + \beta \le 0$.
\end{proof}

\begin{remark}
At this stage one might wonder what the meaning of $\CR(f_1\star f_2)$ is in situations where the
distributions $\CR f_1$ and $\CR f_2$ cannot be multiplied in any ``classical'' sense. 
In general, this strongly depends on the choice of model and of regularity structure.
However, we will see below that in cases where the model was built using a natural renormalisation
procedure and the $f_i$ are obtained as solutions to some fixed point problem, it is usually
possible to interpret $\CR(f_1\star f_2)$ as the weak limit of some (possibly quite non-trivial) 
expression involving the $f_i$'s.
\end{remark}

\begin{remark}\label{rem:product}
In situations where a model happens to consist of continuous functions such that one has indeed
$\Pi_x (a\star b)(y) = \bigl(\Pi_x a\bigr)(y)\bigl(\Pi_x b\bigr)(y)$, it follows from Remark~\ref{rem:continuousModel}
that one has the identity $\CR (f_1 \star f_2) = \CR f_1\, \CR f_2$. In some situations, 
it may thus happen that there are natural approximating models and approximating functions such
that $\CR f_1 = \lim_{\eps\to 0} \CR_\eps f_{1;\eps}$ (and similarly for $f_2$) and 
$\CR(f_1\star f_2) = \lim_{\eps\to 0} (\CR_\eps f_{1;\eps})(\CR_\eps f_{2;\eps})$. 
See for example Section~\ref{sec:RP}, as well as \cite{MR1883719,MR2667703}.

However, this need not always be the case. As we have already seen in Section~\ref{sec:automorphisms}, the formalism
is sufficiently flexible to allow for products that encode some renormalisation procedure, which
is actually the main purpose of this theory.
\end{remark}

\subsection{Classical multiplication}

We are now able to give a rather straightforward application of this theory,
which can be seen as a multidimensional analogue of Young integration.
In the case of the Euclidean scaling, this result is of course well-known, see for 
example \cite{BookChemin}.

\begin{proposition}\label{prop:multDist}
For $\alpha, \beta \in \R $, the map $(f,g) \mapsto f\cdot g$ extends to a 
continuous bilinear map
from $\Lip^\alpha_\s(\R^d) \times \Lip^\beta_\s(\R^d)$ to $\Lip^{\alpha \wedge \beta}_\s(\R^d)$ if 
$\alpha + \beta > 0$. Furthermore, if $\alpha \not \in \N$, then this condition is also necessary.
\end{proposition}

\begin{remark}
More precisely, if $\K$ is a compact subset of $\R^d$ and $\bar \K$ its $1$-fattening,
then there exists a constant $C$ such that
\begin{equ}[e:boundProd]
\|f\cdot g\|_{(\alpha \wedge \beta);\K} \le C \|f\|_{\alpha;\bar \K}\,\|g\|_{\beta;\bar \K}\;,
\end{equ}
for any two smooth functions $f$ and $g$.
\end{remark}

\begin{proof}
The necessity of the condition $\alpha + \beta > 0$ 
is straightforward. Fixing a compact set $\K\subset \R^d$ and assuming that
$\alpha + \beta \le 0$ (or the corresponding strict inequality for integer values), 
it suffices to exhibit a sequence of $\CC^r$ functions $f_n, g_n \in \CC(\K)$ (with $r > \max\{|\alpha|,|\beta|\}$)
such that $\{f_n\}$ is bounded in $\CC^\alpha_\s(\K)$, $g_n$ is bounded in $\CC^\beta_\s(\K)$, and
$\scal{f_n, g_n} \to \infty$, where $\scal{\cdot,\cdot}$ denotes the usual $L^2$-scalar product. 
This is because, since $f_n$ and $g_n$ are supported  in $\K$,
one can easily find a smooth compactly supported test function $\phi$ such that $\scal{f_n,g_n} = \scal{\phi, f_n g_n}$.

A straightforward modification of \cite[Thm~6.5]{MR1228209} shows that the characterisation
of Proposition~\ref{prop:charSpaces} for $f \in \CC(\K)$ to belong to $\CC^\alpha_s$ is also valid
for $\alpha \in \R_+\setminus \N$ (since $f$ is compactly supported, there are no
boundary effects).
The required counterexample can then easily be constructed by setting for example
\begin{equ}
f_n = \sum_{k = 0}^n {1\over \sqrt k} \sum_{x \in \Lambda_k^\s \cap \bar \K} 2^{-k{|\s|\over 2} - \alpha k} \psi_x^{k,\s}\;,
\end{equ}
and similarly for $g_n$ with $\alpha$ replaced by $\beta$. Here, $\bar \K \subset \K$ is 
such that the support of each of the $\psi_x^{k,\s}$ is indeed in $\K$. (One may have to start the sum from some $k_0 > 0$.)
Noting that $\lim_{n \to \infty} \scal{f_n, g_n} = \infty$ as soon as $\alpha + \beta \le 0$, this is the required 
counterexample.

Combining Theorem~\ref{theo:mult} and the reconstruction theorem,
Theorem~\ref{theo:reconstruction}, we can give a short and elegant proof of the sufficiency of $\alpha + \beta > 0$
that no longer makes any reference to wavelet analysis. 
Assume from now on that $\xi \in \CC^\alpha_\s$ for some $\alpha < 0$ and that $f \in \CC^\beta_\s$
for some $\beta > |\alpha|$. 
By bilinearity, we can also assume without loss of generality that the norms appearing in the
right hand side of \eref{e:boundProd} are bounded by $1$.
We then build a regularity structure $\TT$ in the following way. For the set $A$,
we take $A = \N \cup (\N+\alpha)$.
For $T$, we set $T = V \oplus W$,
where each of the sectors $V$ and $W$ is a copy of $\TT_{d,\s}$,
 the canonical model. 
 We also choose $\Gamma$ as in the canonical
 model, acting simultaneously on each of the two instances.

As before, we denote by $X^k$ the canonical basis vectors in $V$. We also use the suggestive notation\label{lab:Xi}
``$\Xi X^k$'' for the corresponding basis vector in $W$, but we postulate that $\Xi X^k \in T_{\alpha + |k|_\s}$ rather than $\Xi X^k \in T_{|k|_\s}$.
With this notation at hand, we also define the product $\star$ between $V$ and $W$ by
the natural identity
\begin{equ}
\bigl(\Xi X^k\bigr)\star \bigl(X^\ell\bigr) =  \Xi X^{k+\ell}\;.
\end{equ}
It is straightforward to verify that, with this product, the pair $(V,W)$ is regular.

Finally, we define a map $J\colon \CC^\alpha_\s \to \MM_\TT$ given by $J\colon \xi \mapsto (\Pi^\xi,\Gamma)$,
where $\Gamma$ is as in the canonical model, while $\Pi^\xi$ acts as
\begin{equ}
\bigl(\Pi^\xi_x  X^k\bigr)(y) = (y-x)^k \;,\qquad
\bigl(\Pi^\xi_x  \Xi X^k\bigr)(y) = (y-x)^k \xi(y)\;,
\end{equ}
with the obvious abuse of notation in the second expression. 
It is then straightforward to verify that 
$\Pi_y = \Pi_x\circ \Gamma_{xy}$ and that 
the map $J$ is Lipschitz continuous.

Denote now by $\CR^\xi$ the reconstruction map associated to the model $J(\xi)$ and, for $u \in \CC^\beta_\s$,
denote by $\CT_\beta u$ as in \eref{e:TaylorNormal} the unique element in $\CD^\beta(V)$ such that $\scal{\one, (\CT_\beta u)(x)} = u(x)$.
Note that even though the space $\CD^\beta(V)$ does in principle depend on the choice of model, in our situation
it is independent of $\xi$ for every model $J(\xi)$. 
Since, when viewed as a $W$-valued function, one has $\Xi \in \CD^\infty(W)$,
one has $\CT_\beta u \star \Xi \in \CD^{\alpha+\beta}$ by Theorem~\ref{theo:mult}. 
We now consider the map
\begin{equ}
B(u,\xi) = \CR^\xi \bigl(\CT_\beta u \star \Xi\bigr)\;.
\end{equ}
By Theorem~\ref{theo:reconstruction}, combined with the continuity of $J$, 
this is a jointly continuous map from $\CC^\beta_\s \times \CC^\alpha_\s$ into $\CC^\alpha_\s$, provided that 
$\alpha + \beta > 0$.
If $\xi$ happens to be a smooth function, then it follows immediately from Remark~\ref{rem:continuousModel} that 
$B(u,\xi) = u(x)\xi(x)$, so that $B$ is indeed the requested continuous extension of the product.
\end{proof}

\subsection{Composition with smooth functions}
\label{sec:compSmooth}

In general, it makes no sense to compose elements $f \in \CD^\gamma$ with arbitrary 
smooth functions. 
In the particular case when $f \in \CD^\gamma(V)$ for a function-like sector $V$ however, 
this is possible. 
Throughout this subsection, we decompose elements $a \in V$ as 
$a = \bar a \one + \tilde a$, with $\tilde a \in T_0^+$ and $\bar a = \scal{\one, a}$.
(This notation is suggestive of the fact that $\tilde a$ encodes the small-scale
fluctuations of $\Pi_x a$ near $x$.) 
We denote by $\zeta > 0$ the smallest non-zero value such that $V_\zeta \neq 0$,
so that one actually has $\tilde a \in T_\zeta^+$.

Given a function-like sector $V$ and a smooth function $F \colon \R^n \to \R$,
we lift $F$ to a function $\hat F \colon V^n \to V$ by
setting
\begin{equ}[e:defFcirc]
\hat F (a) = \sum_{k} {D^k F(\bar a) \over k !} \tilde a^{\star k}\;,
\end{equ}
where the sum runs over all possible multiindices.
Here, $a = (a_1,\ldots,a_n)$ with $a_i \in V$
and, for an arbitrary multiindex $k = (k_1,\ldots, k_n)$, we used the shorthand notation
\begin{equ}
\tilde a^{\star k} = \tilde a_1^{\star k_1}\star \ldots \star \tilde a_d^{\star k_n}\;,
\end{equ}
with the convention that $\tilde a^{\star 0} = \one$.

In order for this definition to make any sense, the sector $V$ needs of course to be endowed
with a product $\star$ which also leaves $V$ invariant. 
In principle, the sum in \eref{e:defFcirc} looks infinite, but by the
properties of the product $\star$, we have $\tilde a^{\star k} \in T_{|k|\reg}^+$. Since $\reg$ is strictly 
positive, only finitely many terms in \eref{e:defFcirc} contribute at each order of homogeneity,
so that $\hat F(a)$ is well-defined as soon as $F \in \CC^\infty$. 
The main result in this subsection is given by:

\begin{theorem}\label{theo:smooth}
Let $V$ be a function-like sector of some regularity structure $\TT$,
let $\reg > 0$ be as above, let $\gamma > 0$, and let $F\in \CC^\kappa(\R^k, \R)$ for some 
$\kappa \ge \gamma/\zeta \vee 1$. Assume furthermore that $V$ is $\gamma$-regular.
Then, for any $f \in \CD^\gamma(V)$, the map $\hat F_\gamma(f)$ defined by
\begin{equ}
\hat F_\gamma (f)(x) = \CQ_\gamma^-\hat F(f(x))\;,
\end{equ}
again belongs to $\CD^\gamma(V)$.
If one furthermore has $F\in \CC^\kappa(\R^k, \R)$ for $\kappa \ge (\gamma / \reg \vee 1) + 1$, 
then the map $f \mapsto \hat F(f)$ is locally Lipschitz continuous 
in the sense that one has the bounds
\begin{equ}[e:locLip]
\|\hat F_\gamma(f) - \hat F_\gamma(g)\|_{\gamma;\K} \lesssim \|f - g\|_{\gamma;\K}\;,\qquad 
\$\hat F_\gamma(f) - \hat F_\gamma(g)\$_{\gamma;\K} \lesssim \$f - g\$_{\gamma;\K}\;,
\end{equ}
for any compact set $\K \subset \R^d$, where the proportionality constant in the first bound is uniform
over all $f$, $g$ with $\|f\|_{\gamma;\K} + \|g\|_{\gamma;\K} \le C$, while in the
second bound it is uniform over all $f$, $g$ with $\$f\$_{\gamma;\K} + \$g\$_{\gamma;\K} \le C$,
for any fixed constant $C$. We furthermore performed a slight abuse of notation by writing again
$\|f\|_{\gamma;\K}$ (for example) instead of $\sum_{i\le n} \|f_i\|_{\gamma;\K}$.
\end{theorem}

\begin{proof}
From now on we redefine $\zeta$ so that $\zeta = \gamma$ in the case when $A$ contains no 
index between $0$ and $\gamma$. In this case, our  original condition 
$\kappa \ge \gamma/\zeta \vee 1$ reads simply as $\kappa \ge \gamma/\zeta$.

Let $L = \lfloor\gamma / \reg\rfloor$,
which is the length of the largest multiindex 
appearing in \eref{e:defFcirc} which still yields a contribution to $T_\gamma^-$.
Writing $b(x) = \CQ_\gamma^- \hat F \bigl(f(x)\bigr)$, we aim to find a bound on $\Gamma_{yx} b(x) - b(y)$.
It follows from a straightforward generalisation of the computation from Theorem~\ref{theo:mult}
that
\begin{equs}
\Gamma_{yx} b(x) &= \sum_{|k| \le L} {D^k F(\bar f(x)) \over k !} \Gamma_{yx} \bigl(\CQ_\gamma^- \tilde f(x)^{\star k}\bigr)  \\
&= \sum_{|k|\le L} {D^k F(\bar f(x)) \over k !} \bigl(\Gamma_{yx} \tilde f(x)\bigr)^{\star k} + R_1(x,y)\;,
\end{equs}
with a remainder term $R_1$ such that 
$\|R_1(x,y)\|_\beta \lesssim \|x-y\|_\s^{\gamma - \beta}$,
for all $\beta < \gamma$. 
Since $\Gamma_{yx}\one = \one$, we can furthermore write
\begin{equ}
\Gamma_{yx} \tilde f(x) = \Gamma_{yx} f(x) - \bar f(x) \one = \tilde f(y) + (\bar f(y) - \bar f(x))\one + R_f(x,y)\;,
\end{equ}
where, by the assumption on $f$, the remainder term $R_f$ again satisfies the bound $\|R_f(x,y)\|_\beta \ls \|x-y\|_\s^{\gamma-\beta}$ for all $\beta < \gamma$. Combining this with the bound we already obtained, we
get
\begin{equ}[e:newbound]
\Gamma_{yx} b(x) = \sum_{|k| \le L} {D^k F(\bar f(x)) \over k !} \bigl(\tilde f(y) + (\bar f(y) - \bar f(x))\one\bigr)^{\star k} + R_2(x,y)\;,
\end{equ}
with
\begin{equ}
\|R_2(x,y)\|_\beta \lesssim \|x-y\|_\s^{\gamma-\beta}\;,
\end{equ}
for all $\beta < \gamma$ as above.
We now expand $D^kF$ around $\bar f(y)$, yielding
\begin{equ}[e:TaylorF]
D^kF(\bar f(x)) = \sum_{|k+\ell| \le L} {D^{k+\ell}F(\bar f(y)) \over \ell !} \bigl(\bar f(x) - \bar f(y)\bigr)^\ell + \CO \bigl(\|x-y\|_\s^{\gamma-|k|\zeta}\bigr)\;,
\end{equ}
where we made use of the fact that $|\bar f(x) - \bar f(y)| \lesssim \|x-y\|_\s^\reg$ by the definition
of $\CD^\gamma$, and the fact that $F$ is $\CC^{\gamma/\reg}$ by assumption.
Similarly, we have the bound
\begin{equ}[e:boundProdf]
\bigl\| \bigl(\tilde f(y) + (\bar f(y) - \bar f(x))\one\bigr)^{\star k}\bigr\|_\beta \lesssim \|x-y\|_\s^{\reg |k| - \beta}\;,
\end{equ}
so that, combining this with \eref{e:newbound} and \eref{e:TaylorF}, we obtain the identity
\begin{equ}[e:lastbx]
\Gamma_{yx} b(x) = \sum_{|k+\ell| \le L} {D^{k+\ell}F(\bar f(y)) \over k ! \ell !} \bigl(\tilde f(y) + (\bar f(y) - \bar f(x))\one\bigr)^{\star k} (\bar f(x) - \bar f(y))^\ell + R_3(x,y)\;,
\end{equ}
where $R_3$ is again a remainder term satisfying the bound
\begin{equ}[e:lastBoundRemainder]
\|R_3(x,y)\|_\beta \lesssim \|x-y\|_\s^{\gamma-\beta}\;.
\end{equ}
Using the generalised binomial identity, we have
\begin{equ}
\sum_{k + \ell = m} {1 \over k! \ell!} \bigl(\tilde f(y) + (\bar f(y) - \bar f(x))\one\bigr)^{\star k} (\bar f(x) - \bar f(y))^\ell = {\tilde f(y)^{\star m} \over m!}\;,
\end{equ}
so that the component in $T_\gamma^-$ of the first term in the right hand 
side of \eref{e:lastbx} is precisely equal to the component in $T_\gamma^-$ of $b(y)$. 
Since the remainder satisfies \eref{e:lastBoundRemainder}, this shows that one does
indeed have $b\in \CD^\gamma(V)$. 

The first bound in \eref{e:locLip} is immediate from the definition \eref{e:defFcirc},
as well as the fact that the assumption implies the local Lipschitz continuity of $D^k F$
for every $|k| \le L$.

The second bound is a little more involved. One way of obtaining it is to first define $h = f-g$
and to note that one then has the identity
\begin{equs}
\hat F (f(x)) - \hat F(g(x)) &= \sum_{k,i} \int_0^1 {D^{k+e_i} F(\bar g(x) + t\bar h(x)) \over k!} \bigl(\tilde g(x) + t \tilde h(x)\bigr)^{\star k} \bar h_i(x)\,dt \\
&\quad + \sum_{k,i} \int_0^1 {D^{k} F(\bar g(x) + t\bar h(x)) \over k!} k_i \bigl(\tilde g(x) + t \tilde h(x)\bigr)^{\star (k-e_i)} \tilde h_i(x)\,dt \\
&= \sum_{k,i} \int_0^1 {D^{k+e_i} F(\bar g(x) + t\bar h(x)) \over k!} \bigl(\tilde g(x) + t \tilde h(x)\bigr)^{\star k} h_i(x)\,dt\;.
\end{equs}
Here, $k$ runs over all possible multiindices and $i$ takes the values $1,\ldots,n$.
We used the notation $e_i$ for the $i$th canonical multiindex. Note also that our way of writing
the second
term makes sense since, whenever $k_i = 0$ so that $k-e_i$ isn't a multiindex anymore, 
it vanishes thanks to the prefactor $k_i$. 

From this point on, the calculation is virtually identical to the calculation already
performed previously. The main differences are that $F$ appears with one more derivative and
that every term always appears with a prefactor $h$, which is responsible for the 
bound proportional to $\$h\$_{\gamma;\K}$.
\end{proof}

\subsection{Relation to Hopf algebras}
\label{sec:Hopf}

Structures like the one of Definition~\ref{def:regularity}
must seem somewhat familiar to the reader used to the formalism 
of Hopf algebras \cite{Sweedler}. Indeed, there are several natural instances of regularity 
structures that are obtained from a Hopf algebra (see for example Section~\ref{sec:RP} below).
This will also be useful in the context of the kind of structures arising when solving semilinear PDEs,
so let us quickly outline this construction. 

Let $\CH$ be a connected, graded, commutative Hopf algebra with product $\star$ and a compatible
coproduct $\Delta$ so that $\Delta(f\star g) = \Delta f \star \Delta g$. 
We assume that the grading is indexed by $\Z_+^d$ for some $d \ge 1$, so that
$\CH = \bigoplus_{k \in \Z_+^d} \CH_k$, and that each of the $\CH_k$ is finite-dimensional.
The grading is assumed to be compatible with the product structures, meaning that
\begin{equ}[e:compat]
\star \colon \CH_k \otimes \CH_\ell \to \CH_{k+\ell}\;,\qquad \Delta\colon \CH_k \to \bigoplus_{\ell+m = k}\CH_\ell\otimes \CH_m\;.
\end{equ}
Furthermore, $\CH_0$ is spanned by the unit $\one$ (this is the definition of connectedness), 
the antipode $\CA$ maps $\CH_k$ to itself for every $k$, and the
counit $\one^*$ is normalised so that $\scal{\one^*, \one} = 1$.

The dual $\CH^\star = \bigoplus_{k \in \Z_+^d} \CH_k^*$ is then again a graded Hopf algebra with a product $\circ$ given by the 
adjoint of $\Delta$ and a coproduct $\Delta^\star$ given by the adjoint of $\star$. (Note that while $\star$ is assumed to be commutative,
$\circ$ is definitely not in general!) By \eref{e:compat}, both $\circ$ and $\Delta^\star$ respect the grading of $\CH^\star$.
There is a natural action $\Gamma$ of $\CH^\star$ onto $\CH$ given by
the identity
\begin{equ}[e:actionG]
\scal{\ell, \Gamma_g f} = \scal{\ell\circ g, f}\;,
\end{equ}
valid for all $\ell, g \in \CH^\star$ and all $f \in \CH$. An alternative way of writing this is
\begin{equ}[e:defAction]
\Gamma_g f = (1\otimes g)\Delta f\;,
\end{equ}
where we view $g$ as a linear operator from $\CH$ to $\R$.
It follows easily from \eref{e:compat}
that, if $g$  and $f$ are homogeneous of degrees $d_g$ and $d_f$ respectively, then $\Gamma_g f$ is
homogeneous of degree $d_f - d_g$, provided that $d_f - d_g \in \Z_+^d$. If not, then one necessarily
has $\Gamma_g f = 0$.

\begin{remark}\label{rem:leftAction}
Another natural action of $\CH^\star$ onto $\CH$ would be given by
\begin{equ}
\scal{\ell, \bar \Gamma_g f} = \scal{(\CA^\star g) \circ \ell, f}\;,
\end{equ}
where, $\CA^\star$, the adjoint of $\CA$, 
is the antipode for $\CH^\star$. Since it is an antihomomorphism, one has indeed the required identity
$\bar \Gamma_{g_1} \bar \Gamma_{g_2} = \bar \Gamma_{g_1\circ g_2}$.
\end{remark}

Since we assumed that $\star$ is commutative, it follows from the Milnor-Moore 
theorem \cite{MR0174052} that $\CH^\star$ is the universal enveloping algebra of $P(\CH^\star)$,
the set of primitive elements of $\CH^\star$ given by
\begin{equ}
P(\CH^\star) = \{g \in \CH^\star \,:\, \Delta^\star g = \one^\star\otimes g + g \otimes \one^\star\}\;.
\end{equ}
Using the fact that the coproduct $\Delta^\star$ is an algebra morphism, it is 
easy to check that $P(\CH^\star)$ is indeed a Lie algebra with bracket
given by $[g_1,g_2] = g_1 \circ g_2 - g_2 \circ g_1$.
This yields in a natural way a Lie group $G \subset \CH^\star$ given by $G = \exp(P(\CH^\star))$.
It turns out (see \cite{MR0214646}) that this Lie group has the very useful property that
\begin{equ}
\Delta^\star(g) = g \otimes g\;,\qquad \forall g \in G\;.
\end{equ}
As a consequence, it is straightforward to verify that one has the remarkable identity
\begin{equ}[e:realIdentity]
\Gamma_g (f_1\star f_2) = (\Gamma_g f_1)\star (\Gamma_g f_2)\;,
\end{equ}
valid for every $g \in G$. This is nothing but an exact version of the regularity requirement
of Definition~\ref{def:regularity}! Note also that \eref{e:realIdentity} is definitely \textit{not} true
for arbitrary elements $g \in \CH^\star$.

All this suggests that a very natural way of constructing a regularity structure is from a graded commutative
Hopf algebra. The typical set-up will then be to fix scaling exponents $\{\alpha_i\}_{i=1}^d$ and
to write $\scal{\alpha,k} = \sum_{i=1}^d \alpha_i k_1$ for any index $k \in \Z_+^d$. We then set
\begin{equ}
A = \{\scal{\alpha,k}\,:\, k \in \Z_+^d\} \;,\quad T_\gamma = \bigoplus_{\scal{\alpha,k} = \gamma} \CH_k\;.
\end{equ}
With this notation at hand, we have:

\begin{lemma}\label{lem:HopfRegular}
In the setting of this subsection, $(A,T,G)$ is a regularity structure, with $G$ acting on $T$ via
$\Gamma$. Furthermore, $T$ equipped with the product $\star$ is regular.
\end{lemma}

\begin{proof}
In view of \eref{e:realIdentity}, the only property that remains to be shown is that $\Gamma_g a - a \in T_\gamma^-$
for $a \in T_\gamma$.

It is easy to show that $P(\CH^\star)$ has a basis consisting of homogeneous elements
and that these belong to $\CH_k^\star$ for some $k \neq 0$. (Since $\Delta^\star \one^\star = \one^\star \otimes \one^\star$.) As a consequence, for $a \in T_\gamma$, $g \in P(\CH^\star)$, and $n > 0$, we 
have $\Gamma_{g^n} a \in T_\beta$ for some
$\beta < \gamma$. Since every element of $G$ is of the form $\exp(g)$ for some
$g \in P(\CH^\star)$ and since $g\mapsto \Gamma_g$ is linear, one has indeed $\Gamma_g a - a \in T_\gamma^-$.
\end{proof}

\begin{remark}
The canonical regularity structure is an example of a regularity structure that can be obtained
via this construction. Indeed, a natural dual to the space $\CH$ of polynomials in $d$ indeterminates 
is given by the space
$\CH^\star$ of differential operators over $\R^d$ with constant coefficients, 
which does itself come with a natural commutative product
given by the composition of operators. (Here, the word ``differential operator'' should be taken in a somewhat loose
sense since it consists in general of an infinite power series.)
Given such a differential operator $\CL$
and an (abstract) polynomial $P$, a natural duality pairing $\scal{\CL,P}$ is given by
applying $\CL$ to $P$ and evaluating the resulting polynomial at the origin. Somewhat informally, one sets
\begin{equ}
\scal{\CL,P} = \bigl(\CL P\bigr)(0)\;.
\end{equ}
The action $\Gamma$ described in \eref{e:actionG} is then given by simply applying $\CL$ to $P$:
\begin{equ}
\Gamma_\CL P = \CL P\;.
\end{equ}
It is indeed obvious that \eref{e:actionG} holds in this case. 
The space of primitives of $\CH^\star$ then consists of those differential operators that satisfy
Leibniz's rule, which are of course precisely the first-order differential operators.
The group-like elements consist of their exponentials, which act on polynomials indeed precisely 
as the group of translations on $\R^d$.
\end{remark}

\subsection{Rough paths}
\label{sec:RP}

A prime example of a regularity structure on $\R$ that is quite different from the canonical
structure of polynomials is the structure associated to $E$-valued 
geometric rough paths of class $\CC^\gamma$ for some $\gamma \in (0,1]$, and some Banach space $E$.
For an introduction to the theory of rough paths, see for example the monographs \cite{MR2036784,MR2314753,MR2604669}
or the original article \cite{Lyons}. 
We will see in this section that, given a Banach space $E$, we can associate to it in a natural way a
regularity structure $\RR_E^\gamma$ which describes the space of $E$-valued rough paths. The regularity index $\gamma$
will only appear in the definition of the index set $A$.
Given such a structure, the space of rough paths with regularity $\gamma$ turns out to be nothing but the 
space of models for $\RR_E^\gamma$.

Setting $A = \gamma \N$, we take for $T$ the tensor algebra
built upon $E^*$, the topological dual of $E$:
\begin{equ}[e:defT]
T = \bigoplus_{k=0}^\infty T_{k\gamma} \;,\qquad T_{k\gamma}=\bigl(E^*\bigr)^{\otimes k}\;,
\end{equ}
where $(E^*)^{\otimes 0} = \R$. 
The choice of tensor product on $E$ and $E^*$ does not matter in principle,
as long as we are consistent in the sense that $\bigl(E^{\otimes k}\bigr)^* =(E^*)^{\otimes k}$ for every $k$. 
We also introduce the space $T_\star$ (which is the predual of $T$) as the tensor algebra 
built from $E$, namely $T_\star = T((E))$.

\begin{remark}
One would like to write again $T_\star = \bigoplus_{k=0}^\infty E^{\otimes k}$.
However, while we consider for $T$ \textit{finite} linear combinations of 
elements in the spaces $T_{k\gamma}$,
for $T_\star$, it will be useful to allow for infinite linear combinations.
\end{remark}

Both $T$ and $T_\star$ come equipped with a natural product. 
On $T_\star$, it will be natural to consider the tensor product $\otimes$,
which will be used to define $G$ and its action on $T$. 
The space $T$ also comes equipped with a natural product, the \textit{shuffle product}, which plays in this context
the role that polynomial multiplication played for the canonical regularity structures. 
Recall that, for any alphabet $\CW$, the shuffle product $\shuffle$ is defined on the free algebra over $\CW$ 
by considering all possible ways of interleaving two words in ways that preserve the original order of the letters. 
In our context, if $a$, $b$ and $c$ are elements of $E^*$, we set for example
\begin{equ}
(a\otimes b) \shuffle (a\otimes c) = a \otimes b \otimes a \otimes c + 2 a \otimes a \otimes b \otimes c + 2 a \otimes a \otimes c \otimes b + a \otimes c \otimes a \otimes b\;.
\end{equ}

Regarding the group $G$, we then perform the following construction. 
For any two elements $a,b \in T_\star$, we define their ``Lie bracket'' by 
\begin{equ}
{}[a,b] = a \otimes b - b \otimes a\;.
\end{equ}
We then define $\LL \subset T_\star$ as the (possibly infinite) linear combinations of all such brackets,
and we set $G= \exp(\LL) \subset T_\star$, with the group operation given by the tensor product $\otimes$.
Here, for any element $a \in T_\star$, we write
\begin{equ}
\exp(a) = \sum_{k=0}^\infty {a^{\otimes k} \over k!}\;,
\end{equ}
with the convention that $a^{\otimes 0} = \one \in T_0$. Note that this sum
makes sense for every element in $T_\star$, and that $\exp(-a) = \bigl(\exp(a)\bigr)^{-1}$.
For every $a \in G$, the corresponding linear map $\Gamma_a$ acting on $T$ 
is then obtained by duality, via the identity 
\begin{equ}[e:defGammaHopf]
\scal{c,\Gamma_a b} = \scal{a^{-1}\otimes c, b}\;,
\end{equ}
where $\scal{\cdot,\cdot}$ denotes the pairing between $T$ and $T_\star$.
Let us denote by $\RR_E^\gamma$ the regularity structure $(A,T,G)$ constructed in this way.

\begin{remark}
The regularity structure $\RR_E^\gamma$ is yet another example of a regularity structure 
that can be obtained via the general construction of Section~\ref{sec:Hopf}.
In this case, our Hopf algebra is given by $T$, equipped with the commutative product $\shuffle$
and the non-commutative coproduct obtained from $\otimes$ by duality.
The required morphism property then just reflects the fact that the shuffle product 
is indeed a morphism for the deconcatenation coproduct.
The choice of action is then the one given by Remark~\ref{rem:leftAction}.
\end{remark}

What are the models $(\Pi,\Gamma)$ for the regularity structure $\RR_E^\gamma$?
It turns out that the elements $\Gamma_{st}$ (which we identify with an element $\XX_{st}$
in $T_\star$ acting via \eref{e:defGammaHopf}) are nothing but what is generally referred to as 
geometric rough paths. Indeed, the identity $\Gamma_{st} \circ \Gamma_{tu} = \Gamma_{su}$,
translates into the identity
\begin{equ}[e:Chen]
\XX_{su} = \XX_{st}\otimes \XX_{tu} \;,
\end{equ}
which is nothing but Chen's relations \cite{MR0073174}. 
The bound \eref{e:boundGamma} on the other hand precisely states that the rough path $\XX$
is $\gamma$-H\"older continuous in the sense of \cite{MR2604669} for example. Finally, it is well-known 
(see \eref{e:realIdentity} or \cite{reutenauer93}) that,
for $a \in T_{k\gamma}$ and $b \in T_{\ell\gamma}$ with $k + \ell \le p$, and any $\Gamma \in G$,
one has the shuffle identity,
\begin{equ}
\Gamma(a\shuffle b) = \bigl(\Gamma a\bigr)\shuffle \bigl(\Gamma b\bigr)\;,
\end{equ}
which can be interpreted as a way of encoding the chain rule.
This should again be compared to Definition~\ref{def:regularity}, which shows 
that the shuffle product is indeed
the natural product for $T$ in this context and that $T$ is regular for $\shuffle$.

By Proposition~\ref{prop:extension}, since our regularity structure only contains elements of positive
homogeneity, the model $\Pi$ is uniquely determined by $\Gamma$.
It is straightforward to check that if we set
\begin{equ}
\bigl(\Pi_s a\bigr)(t) = \scal{\XX_{st},a}\;,
\end{equ}
then the relations and bounds of Definition~\ref{def:model} 
are indeed satisfied, so that this is the
unique model $\Pi$ compatible with a given choice of $\Gamma$ (or equivalently $\XX$).

The interpretation of such a rough path is as follows. 
Denote by $X_t$ the projection of $\XX_{0t}$ onto $E$, the predual of $T_\gamma$.
Then, for every $a \in T_{k\gamma}$ with $k \in \N$, we interpret $\scal{\XX_{st}, a}$ as providing a 
value for the corresponding $k$-fold iterated integral, i.e.,
\begin{equ}[e:iterInt]
\scal{\XX_{st}, a} \;\text{``$=$''}\; \int_s^t \int_s^{t_k}\ldots \int_s^{t_{2}} \scal{dX_{t_1}\otimes\ldots \otimes dX_{t_{k-1}} \otimes dX_{t_k},a}\;.
\end{equ}
A celebrated result by Chen \cite{MR0073174} then shows that indeed, if $t \mapsto X_t \in E$ is a continuous function of bounded variation,
and if $\XX$ is defined by the right hand side of \eref{e:iterInt}, then it is the case that $\XX_{st} \in G$ for every $s,t$
and \eref{e:Chen} holds.

Now that we have identified geometric rough paths with the space of models realising $\RR_E^\gamma$,
it is natural to ask what is the interpretation of the spaces $\CD^\beta$ introduced in 
Section~\ref{sec:model}. An element $f$ of $\CD^\beta$ should then be thought of as
describing a function whose increments can locally (at scale $\eps$) be approximated by linear combinations of components of $\XX$,
up to errors of order $\eps^\beta$. Setting $p = \lfloor 1/\gamma \rfloor$, it
can be checked that elements of $\CD^\beta$ with $\beta = p\gamma$
are nothing but the controlled rough paths in the sense of \cite{MR2091358}.

Writing $f_0(t)$ for the component of $f(t)$ in $T_0 = \R$, it does indeed follow
from the definition of $\CD^\beta$ that
\begin{equ}
|f_0(t) - \scal{\XX_{st},f(s)}| \lesssim |t-s|^\beta\;.
\end{equ}
Since, on the other hand, $\scal{\XX_{st},\one} = 1$, we see that one has indeed
\begin{equ}
f_0(t) - f_0(s) =  \scal{\XX_{st}, \CQ_0^\perp f(s)} + \CO(|t-s|^\beta)\;,
\end{equ}
where $\CQ_0^\perp$ is the projection onto the orthogonal complement to $\one$.

The power of the theory is then that, even though $f_0$ itself is typically only $\gamma$-H\"older continuous, 
it does in many
respects behave ``as if'' it was actually $\beta$-H\"older continuous, and one can have $\beta > \gamma$.
In particular, it is now quite straightforward to define ``integration maps'' $\CI_a$ for $a \in E^*$ such that $F = \CI_a f$ 
should be thought of
as describing the integral $F_0(t) = \int_0^t f_0(s)\,d\scal{X_s,a}$, provided that $\beta + \gamma > 1$.

It follows from the interpretation \eref{e:iterInt} that if $f_0(t) = \scal{\XX_t, b}$ for some element $b \in T$, then 
it is natural to have $F_0(t) =   \scal{\XX_t, b \otimes a}$.
At first sight, this suggests that one should simply set $F(t) = \bigl(\CI_a f\bigr)(t) = f(t) \otimes a$.
However, since $\scal{\one, f(t) \otimes a} = 0$, this would not define an element of $\CD_\gamma^\beta$ for
any $\beta > \gamma$ so one still needs to find the correct value for $\scal{\one, F(t)}$.
The following result, which is essentially a reformulation of \cite[Thm~8.5]{Trees} in the geometric context,
states that there is a unique natural way of constructing this missing component. 

\begin{theorem}\label{theo:IntMap}
For every $\beta > 1-\gamma$ and every $a \in E^*$ there exists a unique linear map $I_a \colon \CD^\beta \to \CC^\gamma$ such 
that $\bigl(I_a f\bigr)(0) = 0$ and such that the map $\CI_a$ defined by
\begin{equ}
\bigl(\CI_a f\bigr)(t) = f(t) \otimes a + \bigl(I_a f\bigr)(t)\,\one\;,
\end{equ}
maps $\CD^\beta$ into $\CD^{\bar \beta}$ with $\bar \beta = (\beta \wedge \gamma p) + \gamma$.
\end{theorem}

\begin{remark}
Even in the context of the classical theory of rough paths, 
one advantage of the framework presented here is that it is straightforward to accommodate the
case of driving processes with different orders of regularity for different components. 
\end{remark}

\begin{remark}
Using Theorem~\ref{theo:IntMap}, it is straightforward to combine it with Theorem~\ref{theo:smooth} 
in order to solve ``rough differential equations'' of the form $dY = F(Y)\,dX$. It does indeed suffice to 
formulate them as fixed point problems
\begin{equ}
Y = y_0 + \CI \bigl(\hat F(Y)\bigr)\;.
\end{equ}
As a map from $\CD^\beta([0,T])$ into itself, $\CI$ then has norm $\CO(T^{\bar \beta - \beta})$,
which tends to $0$ as $T \to 0$ and the composition with $F$ is (locally) 
Lipschitz continuous for sufficiently regular $F$, so that this map is indeed a contraction for small enough $T$.
\end{remark}

\begin{remark}
In general, one can imagine theories of integration in which the chain rule fails, which is very natural in 
the context of numerical approximations.
In this case, it makes sense to replace the tensor algebra by the Connes-Kreimer Hopf
algebra of rooted trees \cite{MR2106008}, which plays in this context the role of the ``free'' 
algebra generated by the multiplication and
integration maps. This is precisely what was done in \cite{Trees}, and one can verify that the construction
given there is again equivalent to the construction of Section~\ref{sec:Hopf}. See also
\cite{MR0305608,hairer74} for more details on the role of the Connes-Kreimer algebra (whose group-like
elements are also called the ``Butcher group'' in the numerical analysis literature) in the
context of the numerical approximation of solutions to ODEs with smooth coefficients.
See also \cite{DavidK} for an analysis of this type of structure from a different angle
more closely related to the present work.
\end{remark}

%
%
%

\section{Integration against singular kernels}
\label{sec:integral}

In this section, we show how to integrate a modelled distribution against a kernel (think of the Green's function
for the linear part of the stochastic PDE under consideration) with a well-behaved
singularity on the diagonal in order to obtain another modelled distribution. In other words, given a
modelled distribution $f$, we would like to build another modelled distribution $\CK f$ with the property
that
\begin{equ}[e:wanted]
\bigl(\CR \CK f\bigr)(x) = \bigl(K * \CR f\bigr)(x) \eqdef \int_{\R^d} K(x,y) \CR f(y)\,dy\;,
\end{equ}
for a given kernel $K \colon \R^d\times \R^d \to \R$, which is singular on the diagonal. 
Here, $\CR$ denotes the reconstruction operator as before.
Of course, this way of writing is rather formal since neither $\CR f$ nor $\CR \CK f$ need to be functions, 
but it is more suggestive than the actual property we are interested in, namely
\begin{equ}[e:realWanted]
\bigl(\CR \CK f\bigr)(\psi) = \bigl(K * \CR f\bigr)(\psi) \eqdef \bigl(\CR f\bigr)(K^\star \psi)\;,\qquad
K^\star \psi(y) \eqdef \int_{\R^d} K(x,y) \psi(x)\,dx\;,
\end{equ}
for all sufficiently smooth test functions $\psi$.
In the remainder of this section, we will always use a notation of the type \eref{e:wanted}
instead of \eref{e:realWanted} in order to state our assumptions and results. It is always
straightforward to translate it into an expression that makes sense rigorously, but this would clutter the
exposition of the results, so we only use the more cumbersome notation in the proofs.
Furthermore, we would like to encode the fact that 
the kernel $K$ ``improves regularity by $\beta$''
in the sense that, in the notation of Remark~\ref{rem:multReg}, $\CK$ is 
bounded from $\CD_\alpha^\gamma$ into $\CD_{(\alpha+\beta)\wedge 0}^{\gamma+\beta}$
for some $\beta > 0$. For example, in the case of the convolution with the heat kernel, one would like
to obtain such a bound with $\beta = 2$, which would be a form of Schauder estimate in our context.

In the case when the right hand side of \eref{e:wanted} actually defines a function (which is the case for many
examples of interest), it may appear that it is straightforward to define $\CK$: simply encode it into the canonical
part of the regularity structure by \eref{e:wanted} and possibly some of its derivatives. The problem with this is
that since, for $f \in \CD_\alpha^\gamma$, one has  $\CR f \in \CC^\alpha$, the best one can expect is to have
$\CR \CK f \in \CC^{\alpha + \beta}$. Encoding this into the canonical regularity structure would then yield
an element of $\CD_{0}^{\alpha+\beta}$, provided that one even has
$\alpha + \beta > 0$. 
In cases where $\gamma > \alpha$, which is the generic situation considered in this article,
this can be substantially
short of the result announced above. As a consequence, $\CK f$ should in general also have non-zero components
in parts of $T$ that do \textit{not} encode the canonical regularity structure, which is why the 
construction of $\CK$ is highly non-trivial.

Let us first state exactly what we mean by the fact that the kernel $K\colon \R^d \times \R^d \to \R$ ``improves regularity by order $\beta$'':

\begin{assumption}\label{def:regK}
The function $K$ can be decomposed as
\begin{equ}[e:defKbar]
K(x,y) = \sum_{n \ge 0} K_n(x,y)\;,
\end{equ}
where the functions $K_n$ have the following properties:
\begin{claim}
\item For all $n \ge 0$, the map $K_n$ is supported in the set $\{(x,y)\,:\, \|x-y\|_\s \le 2^{-n}\}$.
\item For any two multiindices $k$ and $\ell$, there exists a constant $C$ such that
the bound
\begin{equ}[e:propK1]
\bigl|D_1^k D_2^\ell K_n(x,y)\bigr| \le C 2^{(|\s| - \beta + |\ell|_\s + |k|_\s)n}\;,
\end{equ}
holds uniformly over all $n \ge 0$ and all $x,y \in \R^d$.
\item For any two multiindices $k$ and $\ell$, there exists a constant $C$ such that
the bounds
\begin{equs}[e:propK2]
\Bigl| \int_{\R^d} (x-y)^\ell D_2^k K_n(x,y) \,dx\Bigr| &\le C2^{-\beta n}\;,\\
\Bigl| \int_{\R^d} (y-x)^\ell D_1^k K_n(x,y) \,dy\Bigr| &\le C2^{-\beta n}\;,
\end{equs}
hold uniformly over all $n \ge 0$ and all $x,y \in \R^d$.
\end{claim}
In these expressions, we write $D_1$ for the derivative with respect to the first argument
and $D_2$ for the derivative with respect to the second argument.
\end{assumption}

\begin{remark}
In principle, we typically only need \eref{e:propK1} and \eref{e:propK2} to hold for
multiindices $k$ and $\ell$ that are smaller than some fixed number, which depends on the 
particular ``Schauder estimate'' we wish to obtain. In practice however
these bounds tend to hold for all multiindices, so we assume this in order to simplify
notations. 
\end{remark}

A very important insight is that polynomials are going to play a distinguished role in this section.
As a consequence, we work with a fixed regularity structure $\TT = (A, T, G)$ and we assume
that one has $\TT_{d,\s} \subset \TT$ for the same scaling $\s$ and dimension $d$ as appearing in
Definition~\ref{def:regK}. As already mentioned in Remark~\ref{rem:polynomStructures}, 
we will use the notation $\bar T \subset T$ for
the subspace spanned by the ``abstract polynomials''. 
Furthermore, as in Section~\ref{sec:canonical}, we will denote by $X^k$ the canonical basis vectors of $\bar T$,
where $k$ is a multiindex in $\N^d$. We furthermore assume that, except for
polynomials, integer homogeneities are avoided:

\begin{assumption}\label{ass:integer}
For every integer value $n \ge 0$, $T_n = \bar T_n$ consists of the linear span of
elements of the form $X^k$ with $|k|_\s = n$. Furthermore, one considers models that are 
compatible with this structure in the sense that
$\bigl(\Pi_x X^k\bigr)(y) = (y-x)^k$.
\end{assumption}

In order to interplay nicely with our structure,
we will make the following additional assumption on the decomposition of the kernel $K$:

\begin{assumption}\label{ass:polynom}
There exists $r > 0$ such that
\begin{equ}[e:assKpoly]
\int_{\R^d} K_n(x,y)\, P(y)\,dy = 0\;,
\end{equ}
for every $n \ge 0$, every $x \in \R^d$, and every polynomial $P$ of scaled degree less than or equal to $r$.
\end{assumption}

All of these three assumptions will be standing throughout this whole section.
We will therefore not restate this explicitly, except in the statements of the main
theorems.
Even though Assumption~\ref{ass:polynom} seems quite restrictive, it turns out not to matter at all. Indeed, a kernel
$K$ that is regularity improving in the sense of Definition~\ref{def:regK} can typically 
be rewritten as
$K = K_0 + K_1$ such that $K_0$ is smooth and $K_1$ 
additionally satisfies both Assumptions~\ref{def:regK} and \ref{ass:polynom}.
Essentially, it suffices to ``excise the singularity'' with the help of a compactly supported smooth
cut-off function and to then add and subtract some smooth function supported away from the origin
which ensures that the required
number of moments vanish.

In many cases of interest, one can take $K$ to depend only on the difference between its
two arguments. In this case, one has the following result, which shows that our
assumptions typically do cover the Green's functions of differential operators with
constant coefficients.

\begin{lemma}\label{lem:diffSing}
Let $\bar K \colon \R^d \setminus \{0\} \to \R$ be a smooth function which is homogeneous
under the scaling $\s$ in the sense that there exists a $\beta > 0$ such that the identity
\begin{equ}[e:scalingK]
\bar K(\CS_\s^\delta x) = \delta^{|\s|-\beta}\bar K(x)\;,
\end{equ}
holds for all $x \neq 0$ and all $\delta \in (0,1]$. 
Then, it is possible to decompose $\bar K$ as $\bar K(x) = K(x) + R(x)$ in such a 
way that the ``remainder'' $R$ is $\CC^\infty$ on all of $\R^d$ and such that the
map $(x,y) \mapsto K(x-y)$ satisfies Assumptions~\ref{def:regK} and \ref{ass:polynom}. 
\end{lemma}

\begin{proof}
Note first that if each of the $K_n$ is a function of $x-y$, then the bounds 
\eref{e:propK2} follow from \eref{e:propK1} by integration by parts. We therefore
only need to exhibit a decomposition $K_n$ such that \eref{e:propK1} is satisfied and
such that \eref{e:assKpoly} holds for every polynomial $P$ of some fixed but
arbitrary degree.

Let $N \colon \R^d \setminus \{0\} \to \R_+$ be a smooth ``norm'' for the scaling $\s$ in the sense that
$N$ is smooth, convex, strictly positive, and $N(\CS_\s^\delta x) = \delta N(x)$. (See for example Remark~\ref{rem:normSmooth}.) Then, we can introduce
``spherical coordinates'' $(r,\theta)$ with $r \in \R_+$ and $\theta \in S \eqdef N^{-1}(1)$ by
$r(x) = N(x)$, and $\theta(x) = \CS_\s^{r(x)} x$. With these notations, \eref{e:scalingK}
is another way of stating that $\bar K$ can be factored as
\begin{equ}[e:factorK]
\bar K(x) = r^{\beta-|\s|} \Theta (\theta)\;,
\end{equ}
for some smooth function $\Theta$ on $S$. Here and below, we suppress the implicit dependency
of $r$ and $\theta$ on $x$.

Our main ingredient is then the existence of a smooth ``cutoff function''
$\phi \colon \R_+ \to [0,1]$ such that $\phi(r) = 0$ for $r \not \in [1/2,2]$, and such that
\begin{equ}[e:propPhi]
\sum_{n \in \Z} \phi(2^n r) = 1\;,
\end{equ}
for all $r > 0$ (see for example the construction of Paley-Littlewood blocks in \cite{BookChemin}). 
We also set $\phi_R(r) = \sum_{n < 0} \phi(2^n r)$ and,
for $n \ge 0$, $\phi_n(r) = \phi(2^n r)$.
With these functions at hand, we define
\begin{equ}
\bar K_n(x) = \phi_n(r)\bar K(x)\;,\qquad \bar R(x) = \phi_R(r)\bar K(x)\;.
\end{equ}
Since $\phi_R$ is supported away from the origin, the function $\bar R$ is 
globally smooth.
Furthermore, each of the $\bar K_n$ is supported in the ball of radius $2^{-n}$, 
provided that the ``norm'' $N$ was chosen such that $N(x) \ge 2\|x\|_\s$.

It is straightforward to verify that \eref{e:propK1} also holds. 
Indeed, by the exact scaling property \eref{e:scalingK} of $\bar K$, one has the identity
\begin{equ}
\bar K_n(x) = 2^{-(\beta-|\s|)n} \bar K_0(\CS_\s^{2^{-n}}x)\;,
\end{equ}
and \eref{e:propK1} then follows immediately form the fact that $K_0$ is a compactly supported
smooth function. 

It remains to modify this construction in such a way that \eref{e:assKpoly} holds as well.
For this, choose any function $\psi$ which is smooth, supported in the unit ball around the origin,
and such that, for every multiindex $k$ with $|k|_\s \le r$, one has the identity
\begin{equ}
\bigl(1 - 2^{-\beta - |k|_\s}\bigr)\int x^k \psi(x) \,dx = \int  x^k \bar K_0(x)\,dx \;.
\end{equ}
It is of course straightforward to find such a function. We then set
\begin{equ}
K_0(x) = \bar K_0(x) - \psi (x) + 2^{|\s|-\beta} \psi(\CS_\s^{2}x)\;,
\end{equ}
as well as 
\begin{equ}
K_n(x) = 2^{-(\beta-|\s|)n} K_0(\CS_\s^{2^{n}}x)\;,\qquad R(x) = \bar R(x) + \psi(x)\;.
\end{equ}
Since $\psi$ is smooth and $K_n$ has the same scaling properties as before, 
it is clear that the required bounds are still satisfied. Furthermore, our construction
is such that one has the identity
\begin{equ}
\sum_{n=0}^{N-1} K_n(x) = \sum_{n=0}^{N-1} \bar K_n(x) - \psi(x) + 2^{-(\beta-|\s|)N} \psi(\CS_\s^{2^{N}}x)\;,
\end{equ}
so that it is still the case that $\bar K(x) = R(x) + \sum_{n \ge 0} K_n(x)$.
Finally, the exact scaling properties of these expressions imply that
\begin{equs}
\int x^k K_n(x)\,dx &= 2^{-(\beta + |k|_\s)n} \int x^k K_0(x)\,dx \\
&= 2^{-(\beta + |k|_\s)n} \int x^k \bigl(\bar K_0(x) - \psi(x) + 2^{|\s|-\beta}\psi(\CS_\s^{2} x)\bigr) \,dx  \\
&= 2^{-(\beta + |k|_\s)n} \int x^k \bigl(\bar K_0(x) - (1-2^{-\beta-|k|_\s})\psi(x)\bigr) \,dx 
= 0\;,
\end{equs}
as required.
\end{proof}

\begin{remark}
A slight modification of the argument given above 
also allows to cover the situation where \eref{e:factorK}
is replaced by $\bar K(x) = \Theta(\theta) \log r$. 
One can then set
\begin{equ}
\bar K_n(x) = -\Theta(\theta) \int_r^\infty {\phi_n(r) \over r} dr\;,
\end{equ}
and the rest of the argument is virtually identical to the one just given.
In such a situation, one then has $\beta = |\s|$,
thus covering for example the case of the Green's
function of the Laplacian in dimension $2$.
\end{remark}

Of course, in order to have any chance at all to obtain a Schauder-type bound as above, 
our model needs to be sufficiently
``rich'' to be able to describe $\CK f$ with sufficient amount of detail. For this, we need two ingredients.
First, we need the existence of a map $\CI \colon T \to T$ that provides an ``abstract'' representation
of $\CK$ operating at the level of the regularity structure, and second we need that the model $\Pi$ is
adapted to this representation in a suitable manner.

In our definition, we denote again by $\bar T$ the sector spanned by abstract monomials 
of the type $X^k$ for some multiindex $k$.

\begin{definition}\label{def:abstractI}
Given a sector $V$, a linear map $\CI \colon V \to T$ is an abstract integration map of order $\beta > 0$
if it satisfies the following properties:
\begin{claim}
\item One has $\CI \colon V_\alpha \to T_{\alpha + \beta}$ for every $\alpha \in A$.
\item One has $\CI a = 0$ for every $a \in V \cap \bar T$.
\item One has $\CI \Gamma a - \Gamma \CI a \in \bar T$ for every $a \in V$ and 
every $\Gamma \in G$.
\end{claim}
(The first property should be interpreted as $\CI a = 0$ if $a \in V_\alpha$ and $\alpha + \beta \not \in A$.)
\end{definition}

\begin{remark}
At first sight, the second and third conditions might seem strange. It would have
been aesthetically more pleasing to impose that $\CI$ commutes
with $G$, i.e.\ that $\CI \Gamma = \Gamma \CI$. This would indeed be very natural if
$\CI$ was a ``direct'' abstraction of our integration map in the sense that
\begin{equ}[e:wantedDef]
\Pi_x \CI a =  \int_{\R^d} K(\cdot ,z)\bigl(\Pi_x a\bigr)(dz) \;.
\end{equ}
The problem with such a definition is that if $a \in T_\alpha$ with $\alpha > -\beta$,
so that $\CI a \in T_{\bar \alpha}$ for some $\bar \alpha > 0$, then \eref{e:boundPi}
requires us to define $\Pi_x \CI a$ in such a way that it vanishes to some positive
order for localised test functions. This is simply not true in general, so that 
\eref{e:wantedDef} is \textit{not} the right requirement. 
Instead, we will see below that one should modify \eref{e:wantedDef} in a way to 
subtract a suitable polynomial that forces the $\Pi_x \CI a$ to vanish
at the correct order. It is this fact that leads to consider structures with
$\CI \Gamma a - \Gamma \CI a \in \bar T$ rather than $\CI \Gamma a - \Gamma \CI a =0$.
\end{remark}

Our second and main ingredient is that the model should be ``compatible'' with the fact 
that $\CI$ encodes the
integral kernel $K$. 
For this, given an integral kernel $K$ as above,
an important role will be played by the function 
$\CJ\colon \R^d \to L_T^\beta$ which, for every $a \in T_\alpha$ and every $\alpha \in A$, is given by
\begin{equ}[e:defJx]
\CJ(x) a = \sum_{|k|_\s < \alpha + \beta} {X^k\over k!} \int_{\R^d} D_1^{k}K(x,z) \bigl(\Pi_x a\bigr)(dz)\;,
\end{equ}
where we denote by $D_1$ the differentiation operator with respect to the first variable.
It is straightforward 
to verify that, writing $K = \sum K_n$ as before and swapping the sum over $n$ with the
integration, this expression does indeed make sense.

\begin{definition}\label{def:reprK}
Given a sector $V$ and an abstract integration operator $\CI$ on $V$, 
we say that a model $\Pi$ realises $K$ for $\CI$ if, for every $\alpha \in A$,
every $a \in V_\alpha$ and every $x \in \R^d$, one has the identity
\begin{equ}[e:defIa]
\Pi_x \CI a =  \int_{\R^d} K(\cdot ,z)\bigl(\Pi_x a\bigr)(dz) - \Pi_x \CJ(x) a \;,
\end{equ}
\end{definition}

\begin{remark}
The rigorous way of stating this definition is that, for all smooth and compactly supported test functions $\psi$
and for all $a \in T_\alpha$,
one has
\begin{equ}[e:defIrigor]
\bigl(\Pi_x \CI a\bigr)(\psi) = \sum_{n \ge 0} \int_{\R^d} \psi(y) \bigl(\Pi_x a\bigr)(K_{n;xy}^\alpha)\,dy\;,
\end{equ}
where the function $K_{n;xy}^{\alpha}$ is given by
\begin{equ}[e:TaylorKn]
K_{n;xy}^{\alpha}(z) = K_n(y,z) - \sum_{|k|_\s < \alpha+\beta} {(y-x)^k \over k!} D_1^kK_n(x,z)\;.
\end{equ}
The purpose of subtracting the term involving the truncated Taylor expansion of $K$ is to ensure that
$\Pi_x \CI a$ vanishes at $x$ at sufficiently high order.
We will see below that in our context, it is always guaranteed that the sum over $n$ appearing
in \eref{e:defIrigor} converges absolutely, see Lemma~\ref{lem:wellDefInt} below.
\end{remark}

\begin{remark}\label{rem:const}
The case of simple integration in one dimension is very special in this respect. Indeed, the 
role of the ``Green's function'' $K$ is then played by the Heaviside function. This has the
particular property of being \textit{constant} away from the origin, so that all of its derivatives vanish.
In particular, the quantity $\CJ(x)a$ then always takes values in $T_0$.
This is why it is possible to consider expansions of arbitrary order in the theory of rough
paths without ever having to incorporate the space of polynomials into the corresponding regularity
structure. 

Note however that the ``rough integral'' is not an immediate corollary of Theorem~\ref{theo:Int} below,
due in particular to the fact that Assumption~\ref{ass:polynom} does not hold for the Heaviside function. 
It is however straightforward to build the rough integral of any controlled path against
the underlying rough path using the formalism developed here. In order not to stray too far
from our main line of investigation we refrain from giving this construction. 
\end{remark}

With all of these definitions at hand, we are now in the position to provide the definition of 
the map $\CK$ on modelled distributions announced at the beginning of this section. 
Actually, it turns out that for different values of $\gamma$ one should use slightly
different definitions. Given $f \in \CD^\gamma$, we set
\begin{equ}[e:defI]
\bigl(\CK_\gamma f\bigr)(x) = \CI f(x) + \CJ(x) f(x) + \bigl(\CN_\gamma f\bigr)(x)\;,
\end{equ}
where $\CI$ is as above, acting pointwise, $\CJ$ is given in \eref{e:defJx},
and the operator $\CN_\gamma$ maps $f$ into a
$\bar T$-valued function by setting
\begin{equ}[e:defIbar]
\bigl(\CN_\gamma f\bigr)(x) = \sum_{|k|_\s < \gamma+\beta} {X^k\over k!}\int_{\R^d} D_1^k K(x,y) \bigl(\CR f - \Pi_x f(x)\bigr)(dy)\;.
\end{equ}
(We will show later that this expression is indeed well-defined for all $f \in \CD^\gamma$.)

With all of these definitions at hand, we can state the following two results, which are
the linchpin around which the whole theory developed in this work revolves. First,
we have the announced Schauder-type estimate:

\begin{theorem}\label{theo:Int}
Let $\TT = (A,T,G)$ be a regularity structure 
and $(\Pi,\Gamma)$ be a model for $\TT$ satisfying Assumption~\ref{ass:integer}.
Let $K$ be a $\beta$-regularising kernel for some $\beta > 0$,
let $\CI$ be an abstract integration map of order $\beta$ acting on some sector $V$, 
and let $\Pi$ be a model realising $K$ for $\CI$.
Let furthermore $\gamma > 0$,
assume that $K$ satisfies Assumption~\ref{ass:polynom} for $r = \gamma + \beta$, 
and define the operator $\CK_\gamma$ by \eref{e:defI}.

Then, provided that $\gamma + \beta \not \in \N$,
$\CK_\gamma$ maps $\CD^\gamma(V)$ into $\CD^{\gamma+\beta}$, and the identity
\begin{equ}[e:wantedIden]
\CR \CK_\gamma f = K * \CR f\;,
\end{equ}
holds for every $f \in \CD^\gamma(V)$. Furthermore, if $(\bar \Pi, \bar \Gamma)$ is
a second model realising $K$ and one has $\bar f \in \CD^\gamma(V;\bar \Gamma)$, 
then the bound
\begin{equ}
\$\CK_\gamma f ; \bar \CK_\gamma \bar f\$_{\gamma+\beta;\K} \lesssim 
\$f ; \bar f\$_{\gamma;\bar \K} + \|\Pi - \bar \Pi\|_{\gamma;\bar \K}
+ \|\Gamma - \bar \Gamma\|_{\gamma+\beta;\bar \K}\;,
\end{equ}
holds. Here, $\K$ is a compact and $\bar \K$ is its $1$-fattening. 
The proportionality constant implicit in the bound depends only on the norms
$\$f\$_{\gamma;\bar \K}$, $\$\bar f\$_{\gamma;\bar \K}$, as well as similar bounds on the
two models.
\end{theorem}

\begin{remark}
One surprising feature of Theorem~\ref{theo:Int} is that the only non-local term in $\CK_\gamma$ is the 
operator $\CN_\gamma$ which is a kind of ``remainder term''. In particular, the ``rough'' parts of $\CK_\gamma f$, 
i.e.\ the fluctuations that cannot be described by the canonical model consisting of polynomials, are always
obtained as the image of the ``rough'' parts of $f$ under a simple local linear map.
We will see in Section~\ref{sec:SPDE} below that, as a consequence of this fact, if $f \in \CD^\gamma$ 
is the solution to a stochastic PDE built from a local fixed point argument using this theory,
then the ``rough'' part in the description of $f$ is \textit{always} given by \textit{explicit
local functions} of the ``smooth part'', which can be interpreted as some kind of
renormalised Taylor series.
\end{remark}

The assumptions on the model $\Pi$ and on the regularity structure $\TT = (A,T,G)$ (in particular the existence
of a map $\CI$ with the right properties) may look quite stringent at first sight. However, it turns out that it is \textit{always} possible
to embed \textit{any} regularity structure $\TT$ into a 
larger regularity structure in such a way that these assumptions
are satisfied. This is our second main result, which can be stated in the following way.

\begin{theorem}[Extension theorem]\label{theo:extension}
Let $\TT = (A,T,G)$ be a regularity structure containing the canonical regularity structure 
$\TT_{d,\s}$ as stated in Assumption~\ref{ass:integer}, let $\beta > 0$, and
let $V \subset T$ be a sector of order $\bar \gamma$ with the property that for
every $\alpha \not \in \N$ with $V_\alpha \neq 0$, one has $\alpha + \beta \not \in \N$.
Let furthermore $W \subset V$ be a subsector of $V$ and
let $K$ be a kernel on $\R^d$ satisfying 
Assumptions~\ref{def:regK} and \ref{ass:polynom} for every $r \le \bar \gamma$.
 Let $(\Pi,\Gamma)$ be a model for $\TT$, and let $\CI \colon W \to T$ be 
 an abstract integration map of order $\beta$ such that $\Pi$ realises $K$ for $\CI$.
 
Then, there exists a regularity structure $\hat \TT$ containing $\TT$, a model
$(\hat \Pi,\hat \Gamma)$ for $\hat \TT$ extending $(\Pi,\Gamma)$, and an abstract integration map 
$\hat \CI$ of order $\beta$ acting on $\hat V = \iota V$ such that:
\begin{claim}
\item The model $\hat \Pi$ realises $K$ for $\hat \CI$.
\item The map $\hat \CI$ extends $\CI$ in the sense that $\hat \CI \iota a = \iota \CI a$ for every $a \in W$.
\end{claim}

Furthermore, the map $(\Pi,\Gamma) \mapsto (\hat \Pi, \hat \Gamma)$
is locally bounded and Lipschitz continuous in the sense that if $(\Pi, \Gamma)$ and $(\bar \Pi, \bar \Gamma)$
are two models for $\TT$ and $(\hat \Pi, \hat \Gamma)$ and $(\hat{\bar \Pi}, \hat {\bar \Gamma})$
are their respective extensions, then one has the bounds
\begin{equs}
\|\hat \Pi\|_{\hat V; \K} + \|\hat \Gamma\|_{\hat V; \K} &\lesssim \|\Pi\|_{V; \bar \K} (1+\|\Gamma\|_{V;\bar \K})\;, \label{e:boundDiffReal}\\
\|\hat \Pi - \hat{\bar \Pi}\|_{\hat V; \K} + \|\hat \Gamma - \hat{\bar \Gamma}\|_{\hat V; \K} &\lesssim \|\Pi - \bar \Pi\|_{V; \bar \K}(1+\|\Gamma\|_{V;\bar \K})
+  \|\bar\Pi\|_{V; \bar \K} \|\Gamma - \bar \Gamma\|_{V;\bar \K}\;,
\end{equs}
for any compact $\K \subset \R^d$ and its $2$-fattening $\bar \K$.
\end{theorem}

\begin{remark}
In this statement, the sector $W$ is also allowed to be empty.
See also Section~\ref{sec:realAlg} below for a general construction showing how one can
build a regularity structure from an abstract integration map.
\end{remark}

The remainder of this section is devoted to the proof of these two results. We start with the proof of the
extension theorem, which allows us to introduce all the objects that are then needed in the proof of the 
multi-level Schauder estimate, Theorem~\ref{theo:Int}.

\subsection{Proof of the extension theorem}
\label{sec:extension}

Before we turn to the proof, we prove the following lemma which will turn out to be very useful:

\begin{lemma}\label{lem:defCJ}
Let $\CJ\colon \R^d \to \bar T$ be as above, let $V \subset T$ be a sector,
and let $\CI\colon V \to T$ be adapted to the kernel $K$.
Then one has the identity
\begin{equ}[e:actionGammaI]
\Gamma_{xy} \bigl(\CI + \CJ(y)\bigr) = \bigl(\CI + \CJ(x)\bigr) \Gamma_{xy}\;,
\end{equ}
for every $x,y \in \R^d$.
\end{lemma}

\begin{proof}
Note first that $\CJ$ is well-defined in the sense that the following expression converges:
\begin{equ}[e:defJxreal]
\bigl(\CJ(x) a\bigr)_k = {1\over k!}\sum_{\gamma \in A \atop
|k|_\s < \gamma + \beta} \sum_{n \ge 0} \bigl(\Pi_x \CQ_\gamma a\bigr) \bigl(D_1^{k}K_n(x,\cdot)\bigr)\;.
\end{equ}
Indeed, applying the bound \eref{e:boundPixaKn}  which will be obtained
in the proof of Lemma~\ref{lem:wellDefInt} below, we see that the sum in \eref{e:defJxreal} is uniformly convergent for every $\gamma \in A$.

In order to show \eref{e:actionGammaI} we use the fact that, by the definition of an abstract integration map, 
we have $\Gamma_{xy} \CI a - \CI\Gamma_{xy} a \in \bar T$
for every $a \in T$ and every pair $x,y \in \R^d$. Since $\Pi_x$ is injective on $\bar T$ 
(it maps an abstract polynomial
into its concrete realisation based at $x$), it
therefore suffices to show that one has the identity
\begin{equ}
\Pi_y \bigl(\CI + \CJ(y)\bigr) = \Pi_x \bigl(\CI + \CJ(x)\bigr) \Gamma_{xy}\;.
\end{equ}
This however follows immediately from \eref{e:defIa}.
\end{proof}

\begin{proof}[of Theorem~\ref{theo:extension}]
We first argue that we can assume without loss of generality that we are in a situation where 
the sector $V$ is given by a finite sum
\begin{equ}[e:assV]
V = V_{\alpha_1} \oplus V_{\alpha_2} \oplus \ldots \oplus V_{\alpha_n}\;,
\end{equ}
where the $\alpha_i$ are an increasing sequence of elements in $A$,
and where furthermore $W_{\alpha_k} = V_{\alpha_k}$ for all $k < n$. Indeed, we can first
consider the case $V = V_{\alpha_1}$ and $W=W_{\alpha_1}$ and apply our result to build an extension
to all of $V_{\alpha_1}$. We then consider the case $V=V_{\alpha_1} \oplus V_{\alpha_2}$ and
$W = V_{\alpha_1} \oplus W_{\alpha_2}$, etc. 
We then denote by $\bar W$ the complement of $W_{\alpha_n}$
in $V_{\alpha_n}$ so that $V_{\alpha_n} = W_{\alpha_n} \oplus \bar W_{\alpha_n}$.

The proof then consists of two steps. First, we build the regularity structure $\hat \TT = (\hat A, \hat T, \hat G)$
and the map $\hat \CI$, and we show that they have the required properties.
In a second step, we will then build the required extension $(\hat \Pi,\hat \Gamma)$ and we will show that
it satisfies the identity given by Definition~\ref{def:reprK}, as well as the bounds of Definition~\ref{def:model} 
required to make it a bona fide model for $\hat \TT$. 

The only reason why $\TT$ needs to be extended is that we have no way a priori to define $\hat\CI$ to
$\bar W$, so we simply add a copy of it to $T$ and we postulate this copy to be image of $\bar W$ under
the extension $\hat\CI$ of $\CI$. We then extend $G$ in a way which is consistent with 
Definition~\ref{def:abstractI}.
More precisely, our construction goes as follows.
We first  define 
\begin{equ}
\hat A = A \cup \{\alpha_n + \beta\}\;,
\end{equ}
where $\alpha_n$ is as in \eref{e:assV},
and we define $\hat T$ to be the space given by
\begin{equ}
\hat T = T \oplus \bar W\;.
\end{equ}
We henceforth denote elements in $\hat T$ by $(a,b)$ with $a \in T$ and $b \in \bar W$,
and the injection map $\iota \colon T \to \hat T$ is simply given by $\iota a = (a,0)$.
Furthermore, we set
\begin{equ}
\hat T_\alpha = 
\left\{\begin{array}{ll}
	T_\alpha \oplus \bar W & \text{if $\alpha = \alpha_n + \beta$,} \\
	T_{\alpha} \oplus 0 & \text{otherwise.}
\end{array}\right.
\end{equ}
With these notations, one then indeed has the identity $\hat T = \bigoplus_{\alpha \in \hat A} \hat T_\alpha$
as required. 

In order to complete the construction of $\hat \TT$, it remains to extend $G$. As a set, we simply 
set $\hat G = G \times M_{\bar W}^{\alpha_n+\beta}$, where $M_{\bar W}^\alpha$ denotes the set of linear
maps from $\bar W$ into $\bar T_\alpha^-$ (i.e.\ the polynomials of scaled degree strictly less than $\alpha$).
The composition rule on $\hat G$ is then given by the following skew-product:
\begin{equ}[e:compGhat]
(\Gamma_1, M_1) \circ (\Gamma_2, M_2) = \bigl(\Gamma_1  \Gamma_2, \Gamma_1  M_2 + M_1 + \bigl(\Gamma_1 \CI - \CI \Gamma_1\bigr)(\Gamma_2 - 1)\bigr)\;.
\end{equ}
One can check that this composition rule yields an element of $\hat G$. Indeed, by assumption, 
$G$ leaves $\bar T$ invariant, so that
$\Gamma_1  M_2$ is indeed again an element of $M_{\bar W}^{\alpha_n+\beta}$.
Furthermore, $\Gamma_1 \CI - \CI \Gamma_1$ is an element of $L_V^\beta \subset M_{V}^{\alpha_n + \beta}$ by assumption, so that the last term also maps $\bar W$ into $\bar T_{\alpha_n+\beta}^-$ as required.
For any $(\Gamma, M) \in \hat G$, we then 
give its action on $\hat T$ by  setting 
\begin{equ}
(\Gamma, M)(a,b) = \bigl(\Gamma a + \CI (\Gamma b - b) + M b, b\bigr)\;.
\end{equ}
Observe that
\begin{equ}
(\Gamma, M)(a,b) - (a,b) = \bigl((\Gamma a - a) + \CI (\Gamma b - b) + M b, 0\bigr)\;,
\end{equ}
so that this definition does satisfy the condition \eref{e:coundGroup}.

Straightforward verification shows that one has indeed
\begin{equ}
\bigl((\Gamma_1, M_1) \circ (\Gamma_2, M_2)\bigr)(a,b) = 
(\Gamma_1, M_1) \bigl((\Gamma_2, M_2)(a,b)\bigr)\;.
\end{equ}
Since it is immediate that this action is also faithful, this does imply that the operation $\circ$ defined
in \eref{e:compGhat} is associative as required. Furthermore, one can verify that $(1,0)$ is neutral 
for the operation $\circ$ and that $(\Gamma, M)$ has an inverse given by
\begin{equ}
(\Gamma, M)^{-1} = \bigl(\Gamma^{-1}, - \Gamma^{-1}\bigl(M + (\Gamma \CI - \CI \Gamma)(\Gamma-1)\bigr)\bigr)\;,
\end{equ}
so that $(\hat G, \circ)$ is indeed a group. 
This shows that $\hat \TT = (\hat A, \hat T, \hat G)$ is indeed again a regularity structure.
Furthermore, the map $j\colon \hat G \to G$ given by
$j (\Gamma, M) = \Gamma$ is a group homomorphism which verifies that, 
for every $a \in T$ and $\Gamma \in G$, one has the identity
\begin{equ}
\bigl(j (\Gamma, M)\bigr)a = \Gamma a = \iota^{-1} (\Gamma a, 0) = \iota^{-1} (\Gamma, M)\iota a\;.
\end{equ}
This shows that $\iota$ and $j$ do indeed define a canonical inclusion $\TT \subset \hat \TT$, see Section~\ref{sec:regBasic}.

It is now very easy to extend $\CI$ to the image of all of $V$ in $\hat T$. Indeed, for any $a \in V$, we
have a unique decomposition $a = a_0 + a_1$ with $a_0 \in W$ and $a_1\in \bar W$. We then set
\begin{equ}
\hat \CI (a,0) = (\CI a_0, a_1)\;.
\end{equ}
Since $a_1 = 0$ for $a \in W$, one has indeed $\hat \CI \iota a = \hat \CI(a,0) = (\CI a, 0) = \iota \CI a$ in this
case, as claimed in the statement of the theorem. As far as the abstract part of our construction is concerned,
it therefore remains to verify that $\hat \CI$ defined in this way does verify our definition of an
abstract integration map.
The fact that $\hat \CI \colon \hat V_\alpha \to \hat T_{\alpha+\beta}$ is a direct consequence of the fact that
we have simply \textit{postulated} that $0 \oplus \bar W \subset \hat T_{\alpha_n+\beta}$.
Since the action of $\CI$ on $\bar T$ did not change in our construction, one still has $\hat \CI \bar T = 0$.
Regarding the third property, for any $(\Gamma,M) \in \hat G$ and every 
$a = a_1 + a_2 \in V$ as above, we have
\begin{equ}
\hat \CI (\Gamma, M)(a,0) = \hat \CI (\Gamma a, 0) = \bigl(\CI \Gamma a_1 + \CI (\Gamma a_2 - a_2), a_2\bigr)\;,
\end{equ}
where we use the fact that $\Gamma a_2 - a_2 \in V$ by the structural assumption \eref{e:assV} we made
at the beginning of this proof. On the other hand, we have
\begin{equ}
(\Gamma, M) \hat \CI (a,0) = (\Gamma, M) (\CI a_1, a_2)
= \bigl(\Gamma \CI a_1 + \CI(\Gamma a_2 - a_2) + M a_2, a_2\bigr)\;,
\end{equ}
so that the last  property of an abstract integration map is also satisfied.

It remains to provide an explicit formula for the extended model $(\hat \Pi, \hat \Gamma)$. 
Regarding $\hat \Pi$, for $b \in \bar W$ and $x \in \R^d$, we simply \textit{define} it to be given by
\begin{equ}[e:defPihat]
\hat \Pi_x (a,b) =  \Pi_x a + \int_{\R^d} K(\cdot ,z)\bigl(\Pi_x b\bigr)(dz) - \Pi_x \CJ(x) b\;,
\end{equ}
where $\CJ$ is given by \eref{e:defJx}, which guarantees that the model $\hat \Pi$ realises $K$ for $\hat \CI$ on $V$.
Again, this expression is only formal and should really be interpreted as in \eref{e:defIrigor}.
It follows from Lemma~\ref{lem:wellDefInt} below that the sum in \eref{e:defIrigor} converges
and that it furthermore satisfies the required bounds when tested against smooth test functions that are
localised near $x$. Note that the map $\Pi_x \mapsto \hat \Pi_x$ is linear and does not
depend at all on the realisation of $\Gamma$. As a consequence, the bound
on the difference between the extensions of different regularity structures follows at once.
It remains to define $\hat \Gamma_{xy} \in \hat G$ and to show that 
it satisfies both the algebraic and the analytical conditions given by Definition~\ref{def:model}.

We set
\begin{equ}[e:defMxy]
\hat \Gamma_{xy} = (\Gamma_{xy}, M_{xy})\;,\qquad M_{xy}b =  \CJ(x) \Gamma_{xy}b - \Gamma_{xy} \CJ(y)b \;.
\end{equ}
By the definition of $\CJ$, the linear map $M_{xy}$ defined in this way does indeed belong to
$M_{\bar W}^{\alpha_n+\beta}$.
Making use of Lemma~\ref{lem:defCJ}, we 
then have the identity
\begin{equs}
\hat \Gamma_{xy} \circ \hat \Gamma_{yz} &= \bigl(\Gamma_{xy} \Gamma_{yz}, \Gamma_{xy} (\CJ(y) \Gamma_{yz} - \Gamma_{yz} \CJ(z)) + \CJ(x) \Gamma_{xy} - \Gamma_{xy} \CJ(y) \\
&\qquad + (\Gamma_{xy} \CI - \CI \Gamma_{xy})(\Gamma_{yz} - 1)\bigr) \\
& = \bigl(\Gamma_{xz}, -\Gamma_{xz} \CJ(z) + \Gamma_{xy} \CJ(y) \Gamma_{yz} +  \CJ(x) \Gamma_{xy} - \Gamma_{xy} \CJ(y)\\
&\qquad + (\CJ(x) \Gamma_{xy} - \Gamma_{xy} \CJ(y))(\Gamma_{yz} - 1)\bigr)\\
& = \bigl(\Gamma_{xz}, \CJ(x) \Gamma_{xz}-\Gamma_{xz} \CJ(z)\bigr)\;,
\end{equs}
which is the first required algebraic identity. Regarding the second identity, we have
\begin{equs}
\hat \Pi_x \hat \Gamma_{xy} (a,b) &= \hat \Pi_x \bigl(\Gamma_{xy}a + \CI(\Gamma_{xy} b - b) + \CJ(x) \Gamma_{xy}b - \Gamma_{xy} \CJ(y)b , b\bigr) \\
& = \Pi_x a + \int_{\R^d} K(\cdot, z) \Pi_x(\Gamma_{xy} b - b)(dz) - \Pi_x \CJ(x) (\Gamma_{xy}b-b)\\
&\quad + \Pi_x \CJ(x) \Gamma_{xy} b - \Pi_y \CJ(y) b + \int_{\R^d} K(\cdot, z) \Pi_x b(dz) -  \Pi_x \CJ(x) b\\
&= \Pi_x a + \int_{\R^d} K(\cdot, z) \Pi_y b(dz) - \Pi_y \CJ(y) b\\
& = \hat \Pi_y (a,b)\;. \label{e:checkIdenPiGamma}
\end{equs}
Here, in order to go from the first to the second line, we used the fact that $\CI$ realises $K$ for $\CI$ on $W$ by assumption.

It then only remains to check the bound on $\hat \Gamma_{xy}$ stated in \eref{e:boundPi}. 
Since $\hat \Gamma_{xy}(a,0) = (\Gamma_{xy} a, 0)$, we only need to check that
the required bound holds for elements of the form $(0,b)$. Note here that $(0,b) \in \hat T_{\alpha_n + \beta}$, 
but that $(b,0) \in \hat T_{\alpha_n}$. As a consequence, 
\begin{equ}
\|\CI (\Gamma_{xy} b - b)\|_{\gamma} = \|\Gamma_{xy} b - b\|_{\gamma-\beta}
\lesssim \|x-y\|_\s^{\alpha_n - (\gamma - \beta)} = \|x-y\|_\s^{(\alpha_n + \beta) - \gamma}\;,
\end{equ}
as required. It therefore remains to obtain a similar bound on the term $\|M_{xy}b\|_\gamma$.
In view of \eref{e:defMxy}, this on the other hand
is precisely the content of Lemma~\ref{lem:boundGammaxy} below, which concludes the proof.
\end{proof}

\begin{remark}
It is clear from the construction that $\hat \TT$ is the ``smallest possible''
extension of $\TT$ which is guaranteed to have all the required properties.
In some particular cases it might however happen that there exists an even smaller
extension, due to the fact that the matrices $M_{xy}$ appearing in \eref{e:defMxy}
may have additional structure.
\end{remark}

The remainder of this subsection is devoted to the proof of the quantitative estimates
given in Lemma~\ref{lem:wellDefInt} and Lemma~\ref{lem:boundGammaxy}.
We will assume without further restating it that some regularity structure $\TT = (A,T,G)$ is given and that $K$
is a kernel satisfying Assumptions~\ref{def:regK} and \ref{ass:polynom} for some $\beta > 0$.
The test functions $K_{n;xy}^{\alpha}$ introduced in \eref{e:TaylorKn}
will play an important role in these bounds. Actually, we will encounter
the following variant: for any multiindex $k$ and for $\alpha \in \R$, set
\begin{equ}
K_{n,xy}^{k,\alpha}(z) = D^k_1 K_n(y,z) - \sum_{|k+\ell|_\s < \alpha+\beta} {(y-x)^\ell\over \ell!} D_1^{k+\ell} K_n(x,z)\;,
\end{equ}
so that $K_{n,xy}^{\alpha} = K_{n,xy}^{0,\alpha}$.
We then have the following bound:

\begin{lemma}\label{lem:boundPibar}
Let $K_{n,xy}^{k,\alpha}$ be as above, $a \in T_{\alpha}$ for some $\alpha \in A$, and assume that $\alpha + \beta \not \in \N$. 
Then, one has the bound
\begin{equ}
\bigl|\bigl(\Pi_y a\bigr)\bigl(K_{n,xy}^{k,\alpha}\bigr)\bigr| \lesssim \|\Pi\|_{\alpha;\K_x} \bigl(1+\|\Gamma\|_{\alpha;\K_x}\bigr)\sum_{\delta > 0} 2^{\delta n} \|x-y\|_\s^{\delta +  \alpha+\beta -|k|_\s}\;,\label{e:boundKn2}
\end{equ}
and similarly for $\bigl|\bigl(\Pi_x a\bigr)\bigl(K_{n,xy}^{k,\alpha}\bigr)\bigr|$.
Here, the sum runs over finitely many strictly positive values and we used the shorthand $\K_x$ for
the ball of radius $2$ centred around $x$. 
Furthermore, one has the bound
\begin{equs}
\bigl|\bigl(\Pi_y - \bar \Pi_y a\bigr)\bigl(K_{n,xy}^{k,\alpha}\bigr)\bigr| &\lesssim \bigl(\|\Pi - \bar \Pi\|_{\alpha;\K_x} \bigl(1+\|\Gamma\|_{\alpha;\K_x}\bigr) + \|\bar \Pi\|_{\alpha;\K_x} \|\Gamma - \bar \Gamma\|_{\alpha;\K_x}\bigr)\\
&\qquad \times \sum_{\delta > 0} 2^{\delta n} \|x-y\|_\s^{\delta +  \alpha+\beta -|k|_\s}\;,\label{e:boundKn2diff}
\end{equs}
(and similarly for $\Pi_x - \bar \Pi_x$) for any two models $(\Pi,\Gamma)$ and $(\bar \Pi,\bar \Gamma)$.
\end{lemma}

\begin{proof}
It turns out that the cases $\alpha + \beta > |k|_\s$ and $\alpha + \beta < |k|_\s$ are treated slightly differently. 
(The case $\alpha + \beta = |k|_\s$ is ruled out by assumption.)
In the case $\alpha + \beta > |k|_\s$, it 
follows from Proposition~\ref{prop:Taylor} that we can express $K_{n;xy}^{k,\alpha}$ as
\begin{equ}[e:reprKalphan]
K_{n;xy}^{k,\alpha}(z) = \sum_{\ell \in \d A_{\alpha}} \int_{\R^d} D_1^{k+\ell} K_n(y+h,z) \CQ^\ell(x-y,dh)\;,
\end{equ}
where $A_\alpha$ is the set of multiindices given by $A_\alpha = \{\ell\,:\, |k+\ell|_\s < \alpha + \beta\}$
and the objects $\d A_\alpha$ and $\CQ^\ell$ are as in Proposition~\ref{prop:Taylor}. In particular,
note that $|\ell|_\s \ge \alpha + \beta - |k|_\s$ for every term appearing in the above sum.

At this point, we note that, thanks to the first two properties in Definition~\ref{def:regK}, we have the
bound
\begin{equ}[e:boundPixaKn]
\bigl|\bigl( \Pi_y a\bigr) \bigl(D_1^{k+\ell} K_n(y,\cdot)\bigr)\bigr| \lesssim 2^{|k+\ell|_\s n -  \alpha n - \beta n}\|\Pi\|_{\alpha;\K_x}\;,
\end{equ}
uniformly over all $y$ with $\|y-x\|_\s \le 1$ and for all $a \in T_{\alpha}$. Unfortunately, the function $D^{k+\ell} K_n$ is evaluated at $(y+h,z)$
in our case, but this can easily be remedied by shifting the model:
\begin{equs}
\bigl( \Pi_y a\bigr)(D_1^{k+\ell} K_n(y+h,\cdot)) &= \bigl( \Pi_{y+h} \Gamma_{y+h,x} a\bigr)(D_1^{k+\ell} K_n(y+h,\cdot)) \\
&\lesssim  \sum_{\zeta \le \alpha} \|h\|_\s^{\alpha - \zeta} 2^{|k+\ell|_\s n - \zeta n - \beta n}\;,\label{e:boundPiKn}
\end{equs}
where the sum runs over elements in $A$ (in particular, it is a finite sum).
In order to obtain the bound on the second line, we made use of the properties \eref{e:boundPi} of the model.
We now use the fact that $\CQ^\ell(y-x,\cdot)$ is supported on values $h$ such that $\|h\|_\s \le \|x-y\|_\s$
and that
\begin{equ}[e:boundMassQl]
\CQ^\ell(y-x,\R^d) \lesssim \prod_{i-1}^d |y_i-x_i|^{\ell_i} \lesssim \|x-y\|_\s^{|\ell|_\s}\;.
\end{equ}
Combining these bounds, it follows that one has indeed 
\begin{equ}
\bigl|\bigl( \Pi_y a\bigr) \bigl(K_{n;xy}^{k,\alpha}\bigr)\bigr| \lesssim \sum_{\zeta; \ell} \|x-y\|_\s^{\alpha - \zeta+|\ell|_\s} 2^{|k+\ell|_\s n - \zeta n - \beta n}\;,
\end{equ}
where the sum runs over finitely many values of $\zeta$ and $\ell$ with $\zeta \le  \alpha$ and $|\ell|_\s \ge \alpha + \beta - |k|_\s$. Since, by assumption, one has $\alpha +\beta \not \in \N$, it follows that one actually has $|\ell|_\s > \alpha + \beta - |k|_\s$ for each of these terms, so that the required bound follows at once. The bound with $\Pi_y$ replaced by $\Pi_x$ follows
in exactly the same way as above.

In the case $\alpha + \beta < |k|_\s$,
we have $K_{n;yx}^{k,\alpha}(z) = D_1^k K_n(x,z)$ and, proceeding almost exactly as above, one obtains
\begin{equs}
\bigl|\bigl( \Pi_x a\bigr)(D_1^k K_n(x,\cdot))\bigr| &\lesssim  2^{|k|_\s n - \alpha n - \beta n}\;,\\
\bigl|\bigl( \Pi_y a\bigr)(D_1^k K_n(x,\cdot))\bigr| &\lesssim  \sum_{\zeta \le  \alpha} \|x-y\|_\s^{\alpha - \zeta} 2^{|k|_\s n - \zeta n - \beta n}\;.
\end{equs}
with proportionality constants of the required order.

Regarding the bound on the differences between two models, the proof is again
virtually identical, so we do not repeat it.
\end{proof}

Definition~\ref{def:reprK} makes sense thanks to the following lemma:

\begin{lemma}\label{lem:wellDefInt}
In the same setting as above, for any $\alpha \in A$
with $\alpha + \beta \not \in \N$,
the right hand side in \eref{e:defIrigor} with $a \in T_\alpha$ converges absolutely. Furthermore, one has the bound
\begin{equ}[e:wantedBoundInt]
\sum_{n \ge 0} \int_{\R^d} (\Pi_x a)(K_{n;yx}^{\alpha})\,\psi_x^\lambda(y)\,dy \lesssim \lambda^{\alpha + \beta} \|\Pi\|_{\alpha; \K_x}\bigl(1+\|\Gamma\|_{\alpha;\K_x}\bigr)\;,
\end{equ}
uniformly over all $x \in \R^d$, all $\lambda \in (0,1]$,
and all smooth functions supported in $B_\s(1)$ with $\|\psi\|_{\CC^r} \le 1$.
Here, we used the shorthand notation $\psi_x^\lambda = \CS_{\s,x}^\lambda\psi$,
and $\K_x$ is as above. As in Lemma~\ref{lem:boundPibar}, a similar bound holds for 
$\Pi_x - \bar \Pi_x$, but with the expression from the right hand side of the first line of \eref{e:boundDiffReal}
replaced by the expression appearing on the second line.
\end{lemma}

\begin{remark}
The condition that $\alpha + \beta \not \in \N$ is actually known to be necessary in general.
Indeed, it is possible to construct examples of functions $f \in \CC(\R^2)$ such that
$K * f \not \in \CC^2(\R^2)$, where $K$ denotes the Green's function of the Laplacian \cite{MR1228209}. 
\end{remark}

\begin{proof}
We treat various regimes separately. For this, we obtain separately the bounds
\minilab{e:boundsKnpsi}
\begin{equs}
\bigl(\Pi_x a\bigr)(K_{n;yx}^{\alpha})   &\lesssim \|\Pi\|_{\alpha; \K_x}\bigl(1+\|\Gamma\|_{\alpha;\K_x}\bigr)\sum_{\delta > 0} \|x-y\|_\s^{\alpha + \beta + \delta} 2^{\delta n }\;,\qquad\quad \label{e:boundKnpsi1}\\
\int_{\R^d} \bigl(\Pi_x a\bigr)(K_{n;yx}^{\alpha})\, &\psi_{x}^\lambda(y)\,dy  \lesssim \|\Pi\|_{\alpha; \K_x}\sum_{\delta>0}  \lambda^{\alpha+\beta-\delta}2^{-\delta n}\;,\label{e:boundKnpsi2}
\end{equs} 
for $\|x-y\|_\s \le 1$. 
Both sums run over some finite set of strictly positive indices $\delta$.
Furthermore, \eref{e:boundKnpsi1} holds whenever $\|x-y\|_\s \le 2^{-n}$, 
while \eref{e:boundKnpsi2} holds whenever $2^{-n} \le \lambda$.
Using the expression \eref{e:defIrigor}, it is then straightforward to show 
that \eref{e:boundsKnpsi}  implies \eref{e:wantedBoundInt} by using the bound
\begin{equ}
\int_{\R^d} \|x-y\|_\s^\gamma \psi_x^\lambda(y)\,dy \lesssim \lambda^\gamma\;,
\end{equ}
and summing the resulting expressions over $n$.

The bound \eref{e:boundKnpsi1} (as well as the corresponding version for the difference between 
two different models for our regularity structure) is a particular case of Lemma~\ref{lem:boundPibar}, so
we only need to consider the second bound. 
This bound is only useful in the regime $2^{-n} \le \lambda$, so that 
we assume this  from now on.
It turns out that in this case, the bound \eref{e:boundKnpsi2} does not require the
use of the identity $\Pi_z = \Pi_x \Gamma_{xz}$, so that the corresponding bound on the difference
between two models follows by linearity.
For fixed $n$, it follows from the linearity of $\Pi_x a$ that
\begin{equ}
\int_{\R^d} \bigl(\Pi_x a\bigr)(K_{n;yx}^{\alpha}) \psi_x^\lambda(y)\,dy =
 \bigl(\Pi_x a\bigr)\Bigl(\int_{\R^d}K_{n;yx}^{\alpha}(\,\cdot\,)\,\psi_x^\lambda(y)\,dy\Bigr) \;.
\end{equ}
We decompose $K_{n;yx}^{\alpha}$ according to \eref{e:TaylorKn} and consider the first term.
It follows from the first property in Definition~\ref{def:regK} that the function
\begin{equ}[e:defYn]
Y_n^\lambda(z) = \int_{\R^d}K_n(y,z)\, \psi_x^\lambda(y)\,dy
\end{equ}
is supported in a ball of radius $2\lambda$ around $x$, and bounded by $C 2^{-\beta n} \lambda^{-|\s|}$ for some constant $C$.
In order to bound its derivatives, we use the fact that  
\begin{equ}
D^\ell Y_n^\lambda(z) = 
 \sum_{k < \ell} {\bigl(D^k \psi_x^\lambda\bigr)(x) \over k!} \int_{\R^d} D_2^\ell K_n(y,z)\,  (y-x)^k \,dy
+ \int_{\R^d} D_2^\ell K_n(y,z)\, R_x(y)\,dy\;,
\end{equ}
where the remainder $R_x(y)$ satisfies the bound $|R_x(y)| \lesssim \lambda^{-|\s|-|\ell|_\s} \|x-y\|_\s^{|\ell|_\s}$.
Making use of \eref{e:propK1} and \eref{e:propK2}, we thus obtain  the bound
\begin{equs}
\sup_{z\in \R^d}\bigl|D^\ell Y_n^\lambda(z)\bigr| &\lesssim
\sum_{k < \ell} 2^{- \beta n} \lambda^{-|\s| - |k|_\s} +  2^{-\beta n} \lambda^{-|\s|-|\ell|_\s}\\
&\lesssim 2^{-\beta n} \lambda^{-|\s|-|\ell|_\s}\;.\label{e:boundYn}
\end{equs}
Combining these bounds with Remark~\ref{rem:altBound}, we obtain the estimate
\begin{equ}
\bigl|\bigl(\Pi_x a\bigr)(Y_n^\lambda)\bigr| \lesssim \lambda^\alpha 2^{-\beta n}\;.
\end{equ}
It remains to obtain a similar bound on the remaining terms in the decomposition of  $K_{n;yx}^{\alpha}$.
This follows if we obtain a bound analogous to \eref{e:boundYn}, but for the test functions
\begin{equ}
Z_{n,\ell}^\lambda(z) = D_1^\ell K_n(x,z) \int_{\R^d} (y-x)^\ell\, \psi_x^\lambda(y)\,dy\;.
\end{equ}
These are supported in a ball of radius $2^{-n}$ around $x$ and bounded by 
a constant multiple of $2^{(|\ell|_\s + |\s| -\beta) n} \lambda^{|\ell|_\s}$. 
Regarding their derivatives, the bound \eref{e:propK1} immediately yields
\begin{equ}
\sup_{z\in \R^d}\bigl|D^k Z_{n,\ell}^\lambda(z)\bigr| \lesssim
2^{(|\ell|_\s +|k|_\s+ |\s| -\beta) n} \lambda^{|\ell|_\s} \;.
\end{equ}
Combining these bounds again with Remark~\ref{rem:altBound} yields the estimate
\begin{equ}
\bigl|\bigl(\Pi_x a\bigr)(Z_{n,\ell}^\lambda)\bigr| \lesssim 2^{(|\ell|_\s - \alpha-\beta) n} \lambda^{|\ell|_\s}\;.
\end{equ}
Since the indices $\ell$ appearing in \eref{e:TaylorKn} all
satisfy $|\ell|_\s < \alpha + \beta$, the bound \eref{e:boundKnpsi2} does indeed hold for some finite collection
 of strictly positive indices $\delta$.
\end{proof}

The following lemma is the last ingredient required for the proof of the extension theorem.
In order to state it, we make use of the shorthand notation
\begin{equ}[e:defJxy]
\CJ_{xy} \eqdef \CJ(x) \Gamma_{xy} - \Gamma_{xy} \CJ(y)\;,
\end{equ}
where, given a regularity structure $\TT$ and a model $(\Pi,\Gamma)$, the map $\CJ$
was defined in \eref{e:defJx}.

\begin{lemma}\label{lem:boundGammaxy}
Let $V\subset T$ be a sector satisfying the same assumptions
as in Theorem~\ref{theo:extension}. Then, for every $\alpha \in A$, $a \in V_\alpha$, every 
multiindex $k$ with $|k|_\s < \alpha + \beta$, and every 
pair $(x,y)$ with $\|x-y\|_\s \le 1$, 
one has the bound
\begin{equ}[e:wantedCJ]
\bigl|\bigl(\CJ_{xy} a\bigr)_k\bigr| \lesssim \|\Pi\|_{\alpha;\K_x} \bigl(1+\|\Gamma\|_{\alpha;\K_x}\bigr) \|x-y\|_\s^{\alpha +\beta-|k|_\s}\;,
\end{equ}
where $\K_x$ is as before. 
Furthermore, if we denote by $\bar \CJ_{xy}$ the function defined like \eref{e:defJxy},
but with respect to a second model $(\bar \Pi, \bar \Gamma)$,
then we obtain a bound similar to \eref{e:wantedCJ} on the difference $\CJ_{xy} a - \bar \CJ_{xy} a$,
again with the expression from the right hand side of the first line of \eref{e:boundDiffReal}
replaced by the expression appearing on the second line.
\end{lemma}

\begin{proof}
For any multiindex $k$ with $|k|_\s < \alpha + \beta$,
we can rewrite the $k$th component of $\CJ_{xy} a$ as
\begin{equs}
\bigl(\CJ_{xy} a\bigr)_k &= {1\over k!}\sum_{n \ge 0} 
\Bigl(\sum_{ |k|_\s - \beta < \gamma \le \alpha} \bigl( \Pi_x \CQ_\gamma \Gamma_{xy} a\bigr)\bigl(D^{k}_1 K_n(x,\cdot)\bigr) \label{e:termJxy}\\
&\qquad - \sum_{|\ell|_\s < \alpha +\beta-|k|_\s} {(x-y)^\ell\over \ell!}\bigl( \Pi_y a\bigr)\bigl(D^{k+\ell}_1 K_n(y,\cdot)\bigr)\Bigr)\\
&\eqdef {1\over k!} \sum_{n\ge 0}  \CJ_{xy}^{n,k} a\;.
\end{equs}
As usual, we treat separately the cases $\|x-y\|_\s \le 2^{-n}$ and $\|x-y\|_\s \ge 2^{-n}$.
In the case $\|x-y\|_\s \le 2^{-n}$, we rewrite $\CJ_{xy}^{n,k} a$ as
\begin{equ}
\CJ_{xy}^{n,k} a = \bigl(\Pi_y a\bigr) \bigl(K_{n;xy}^{k,\alpha}\bigr)
 - \sum_{\gamma \le  |k|_\s - \beta} \bigl(\Pi_x \CQ_\gamma \Gamma_{xy} a\bigr)\bigl(D^{k}_1 K_n(x,\cdot)\bigr)\;.\label{e:decompCJxy}
\end{equ}
The first term has already been bounded in Lemma~\ref{lem:boundPibar}, yielding a bound of the
type \eref{e:wantedCJ} when summing over the relevant values of $n$. 
Regarding the second term, we make use of the fact that, for $\gamma < \alpha$ (which is 
satisfied since $|k|_\s < \alpha + \beta$),
one has the bound $\|\Gamma_{xy} a\|_\gamma \lesssim \|x-y\|_\s^{\alpha-\gamma}$. Furthermore,
for any $b \in T_\gamma$, one has 
\begin{equ}[e:boundDKn]
\bigl(\Pi_x b\bigr)\bigl(D^{k}_x K_n(x,\cdot)\bigr)
\lesssim \|b\| 2^{(|k|_\s - \beta - \gamma)n}\;.
\end{equ}
In principle, the exponent appearing in this term might vanish.
As a consequence of our assumptions, this however cannot happen. Indeed,
if $\gamma$ is such that $\gamma + \beta = |k|_\s$, then we necessarily have that 
$\gamma$ itself is an integer. By Assumptions~\ref{ass:integer}
and \ref{ass:polynom} however, we  have the identity
\begin{equ}
\bigl(\Pi_x b\bigr)\bigl(D^{k}_x K_n(x,\cdot)\bigr) = 0\;,
\end{equ}
for every $b$ with integer homogeneity.

Combining all these bounds, we thus obtain similarly to before the bound
\begin{equ}[e:boundJxynk]
\bigl|\CJ_{xy}^{n,k} a \bigr| \lesssim  \|\Pi\|_{\alpha;\K_x} \bigl(1+\|\Gamma\|_{\alpha;\K_x}\bigr) \sum_{\delta > 0} \|x-y\|_\s^{\alpha+\beta - |k|_\s +\delta} 2^{\delta n}\;,
\end{equ}
where the sum runs over a finite number of exponents.
This expression is valid for all $n\ge 0$ with $\|x-y\|_\s \le 2^{-n}$.
Furthermore, if we consider two different models $(\Pi,\Gamma)$ and $(\bar \Pi, \bar \Gamma)$,
we obtain a similar bound on the difference $\CJ_{xy}^{n,k} a - \bar \CJ_{xy}^{n,k} a$.

In the case $\|x-y\|_\s \ge 2^{-n}$, we treat the two terms in \eref{e:termJxy} separately
and, for both cases, we make use of the bound \eref{e:boundDKn}. As a consequence, we obtain
\begin{equs}
\bigl|\CJ_{xy}^{n,k} a \bigr| &\lesssim  \sum_{ |k|_\s - \beta < \gamma \le \alpha} \|x-y\|_\s^{\alpha-\gamma}2^{(|k|_\s-\beta-\gamma)n} \\
&\qquad+  \sum_{|\ell|_\s < \alpha +\beta-|k|_\s} \|x-y\|_\s^{|\ell|_\s} 2^{(|k|_\s+|\ell|_\s-\beta-\alpha)n}\;,
\end{equs}
with a proportionality constant as before.
Thanks to our assumptions, the exponent of $2^n$ appearing in each of these terms is always strictly
\textit{negative}. We thus obtain a bound like \eref{e:boundJxynk}, but where the sum now runs over a finite number
of exponents $\delta$ with $\delta < 0$. Summing both bounds over $n$, we see that
\eref{e:wantedCJ} does indeed hold for $\CJ_{xy}$.
In this case, the bound on the difference again simply holds by linearity.
\end{proof}

\subsection{Multi-level Schauder estimate}
\label{sec:Schauder}

We now have all the ingredients in place to prove the ``multi-level Schauder estimate''
announced at the beginning of this section. Our proof has a similar flavour to proofs of the classical
(elliptic or parabolic) Schauder estimates using scale-invariance, like for example \cite{MR1459795}.

\begin{proof}[of Theorem~\ref{theo:Int}]
We first note that \eref{e:defIbar} is well-defined for every $k$ with $|k|_\s < \gamma+\beta$.
Indeed, it follows from the reconstruction theorem and the assumptions on $K$ that 
\begin{equ}[e:boundRf2]
\bigl(\CR f - \Pi_x f(x)\bigr)\bigl(D_1^k K_n(x,\cdot)\bigr) \lesssim 2^{(|k|_\s -\beta-\gamma)n}\;,
\end{equ}
which is summable since the exponent appearing in this expression is strictly negative. 
Regarding $\CK_\gamma f - \bar \CK_\gamma \bar f$, we use \eref{e:boundRfDiff}, which yields
\begin{equs}
\bigl|\bigl(\CR f - \bar \CR \bar f &- \Pi_x f(x) + \bar \Pi_x \bar f(x)\bigr)\bigl(D_1^k K_n(x,\cdot)\bigr)\bigr| 
\label{e:boundRf2Diff}\\ 
&\quad \lesssim 2^{(|k|_\s -\beta-\gamma)n} \bigl(\$f;\bar f\$_{\gamma; \bar \K} + \|\Pi - \bar \Pi\|_{\gamma;\bar \K}\bigr)\;,
\end{equs}
where the proportionality constants depend on the bounds on $f$, $\bar f$, and the two
models. 
In particular, this already shows that one has the bounds
\begin{equ}
\|\CK_\gamma f\|_{\gamma+\beta;\K} \lesssim 
\$f\$_{\gamma; \bar \K} \;,\qquad 
\|\CK_\gamma f - \bar \CK_\gamma \bar f\|_{\gamma+\beta;\K} \lesssim 
\$f;\bar f\$_{\gamma; \bar \K} + \|\Pi - \bar \Pi\|_{\gamma;\bar \K} \;,
\end{equ}
so that it remains to obtain suitable bounds on differences between two points.

We also note that by the definition of $\CK_\gamma$ and the properties of $\CI$, one has for $\ell \not \in \N$ the
bound
\begin{equs}
\|\CK_\gamma f(x) - \Gamma_{xy} \CK_\gamma f(y)\|_\ell &= \|\CI \bigl(f(x) - \Gamma_{xy} f(y)\bigr)\|_\ell
\lesssim \|f(x) - \Gamma_{xy} f(y)\|_{\ell-\beta} \\
&\lesssim \|x-y\|_\s^{\gamma+\beta-\ell}\;,
\end{equs}
which is precisely the required bound. 
A similar calculation allows to bound the terms involved in the definition
of $\$\CK_\gamma f;\bar \CK_\gamma \bar f\$_{\gamma+\beta;\K}$, so that it 
remains to show a similar bound for $\ell \in \N$.

It follows from \eref{e:actionGammaI}, combined with the fact that $\CI$ does not produce any component
in $\bar T$ by assumption, that one has the identity
\begin{equs}
\bigl(\Gamma_{xy} \CK_\gamma f(y)\bigr)_k - \bigl(\CK_\gamma f(x)\bigr)_k &= 
\bigl(\Gamma_{xy} \CN_\gamma f(y)\bigr)_k - \bigl(\CN_\gamma f(x)\bigr)_k \\
&\qquad + \bigl(\CJ(x) \bigl(\Gamma_{xy} f(y) - f(x)\bigr)\bigr)_k\;,
\end{equs}
so our aim is to bound this expression. We decompose $\CJ$ as $\CJ = \sum_{n \ge 0} \CJ^{(n)}$
and $\CN_\gamma = \sum_{n \ge 0} \CN_\gamma^{(n)}$, where the $n$th term in each 
sum is obtained by replacing
$K$ by $K_n$ in the expressions for $\CJ$ and $\CN_\gamma$ respectively.
It follows from the definition of $\CN_\gamma$, as well as the action of $\Gamma$ on the
space of elementary polynomials that one has the identities
\begin{equs}
\bigl(\Gamma_{xy}  \CN_\gamma^{(n)} f(y)\bigr)_k &= {1\over k!}\sum_{|k+\ell|_\s < \gamma+\beta}{(x-y)^\ell\over \ell!} \bigl(\CR f -  \Pi_y f(y)\bigr)\bigl(D_1^{k+\ell} K_n(y,\cdot)\bigr)\;,\\
\bigl(\CJ^{(n)}(x) \Gamma_{xy} f(y)\bigr)_k &= {1\over k!} \sum_{\delta \in B_k} \bigl( \Pi_x \CQ_\delta \Gamma_{xy} f(y)\bigr)\bigl(D_1^{k} K_n(x,\cdot)\bigr)\;,\label{e:expressions}\\
\bigl(\CJ^{(n)}(x) f(x)\bigr)_k &= {1\over k!} \sum_{\delta \in B_k} \bigl(\Pi_x \CQ_\delta  f(x)\bigr)\bigl(D_1^{k} K_n(x,\cdot)\bigr)\;,
\end{equs}
where the set $B_k$ is given by
\begin{equ}
B_k = \{\delta \in A\,:\, |k|_\s - \beta < \delta < \gamma\}\;.
\end{equ}
(The upper bound $\gamma$ appearing in $B_k$ actually has no effect since, by assumption, $f$ has no component
in $T_\delta$ for $\delta \ge \gamma$.) 
As previously, we use different strategies for small scales and for large scales. 

We first bound the terms at small scales, i.e.\ when $2^{-n} \le \|x-y\|_\s$. In this case,
we bound separately the terms $\CN_\gamma^{(n)} f$, $\Gamma_{xy} \CN_\gamma^{(n)} f$, and
$\CJ^{(n)}(x) \bigl(\Gamma_{xy} f(y) - f(x)\bigr)$.
In order to bound the distance between $\CK_\gamma f$ and 
$\bar \CK_\gamma \bar f$, we also need to obtain similar bounds on
$\CN_\gamma^{(n)} f - \bar \CN_\gamma^{(n)} \bar f$, $\Gamma_{xy} \CN_\gamma^{(n)} f - \bar \Gamma_{xy} \bar \CN_\gamma^{(n)} \bar f$, as well as 
$\CJ^{(n)}(x) \bigl(\Gamma_{xy} f(y) - f(x)\bigr) - \bar \CJ^{(n)}(x) \bigl(\bar \Gamma_{xy} \bar f(y) - \bar f(x)\bigr)$. Here, we denote by $\bar \CJ$ the same function as $\CJ$, but defined from
the model $(\bar \Pi, \bar \Gamma)$. The same holds for $\bar \CN_\gamma$.

Recall from \eref{e:boundRf2} that
we have for $\CN_\gamma^{(n)} f$ the bound
\begin{equ}[e:boundNgamma]
\bigl|\bigl(\CN_\gamma^{(n)} f(x)\bigr)_k\bigr| \lesssim 2^{(|k|_\s - \beta - \gamma)n}\;,
\end{equ}
so that, since we only consider indices $k$ such that $|k|_\s - \beta - \gamma < 0$, one obtains
\begin{equ}
\sum_{n \,:\, 2^{-n} \le \|x-y\|_\s}\bigl|\bigl(\CN_\gamma^{(n)} f(x)\bigr)_k\bigr| \lesssim \|x-y\|_\s^{\beta + \gamma-|k|_\s}\;,
\end{equ}
as required. In the same way, we obtain the bound
\begin{equ}
\sum_{n \,:\, 2^{-n} \le \|x-y\|_\s}\bigl|\bigl(\CN_\gamma^{(n)} f(x) - \bar \CN_\gamma^{(n)} \bar f(x)\bigr)_k\bigr| \lesssim \|x-y\|_\s^{\beta + \gamma-|k|_\s} \bigl(\$f;\bar f\$_{\gamma;\bar \K}
+ \|\Pi - \bar \Pi\|_{\gamma;\bar \K}\bigr)\;,
\end{equ}
where we made use of \eref{e:boundRf2Diff} instead of \eref{e:boundRf2}.

Similarly, we obtain for $\bigl(\Gamma_{xy} \CN_\gamma^{(n)} f(y)\bigr)_k$ the bound
\begin{equ}
\bigl|\bigl(\Gamma_{xy} \CN_\gamma^{(n)} f(y)\bigr)_k\bigr| \lesssim \sum_{|k+\ell|_\s < \gamma+\beta} 
\|x-y\|_\s^{|\ell|_\s} 2^{(|k+\ell|_\s - \beta - \gamma)n}\;.
\end{equ}
Summing
over values of $n$ with $2^{-n} \le \|x-y\|_\s$, we can bound this term again by 
a multiple of $\|x-y\|_\s^{\beta + \gamma-|k|_\s}$. 
In virtually the same way, we obtain the bound
\begin{equs}
\bigl|\bigl(\Gamma_{xy} \CN_\gamma^{(n)} f &- \bar \Gamma_{xy} \bar \CN_\gamma^{(n)} \bar f\bigr)_k\bigr|
\\
&\quad \lesssim \|x-y\|_\s^{\beta + \gamma-|k|_\s} \bigl(\$f;\bar f\$_{\gamma;\bar \K}
+ \|\Pi - \bar \Pi\|_{\gamma;\bar \K} + \|\Gamma - \bar \Gamma\|_{\gamma+\beta;\bar \K}\bigr)\;,
\end{equs}
where rewrote the left hand side as $(\Gamma_{xy} - \bar \Gamma_{xy}) \CN_\gamma^{(n)} f + \bar \Gamma_{xy} \bigl(\bar \CN_\gamma^{(n)} \bar f - \CN_\gamma^{(n)} f\bigr)$
and then proceeded to bound both terms as above.

We now turn to the term involving $\CJ^{(n)}$. From the definition of $\CJ^{(n)}$, we then obtain the bound
\begin{equs}
\bigl|\bigl(\CJ^{(n)}(x) (\Gamma_{xy}f(y) - f(x))\bigr)_k\bigr| &= 
\sum_{\delta \in B_k} \bigl(\Pi_x \CQ_\delta \bigl( \Gamma_{xy} f(y) - f(x)\bigr)\bigr)\bigl(D_1^{k} K_n(x,\cdot)\bigr) \\
&\lesssim \sum_{\delta\in B_k} \|x-y\|_\s^{\gamma-\delta} 2^{(|k|_\s - \beta - \delta)n}\;.\label{e:boundJGf}
\end{equs}
It follows from the definition of $B_k$ that
$|k|_\s - \beta - \delta < 0$ for every term appearing in this sum. As a consequence, summing
over all $n$ such that $2^{-n} \le \|x-y\|_\s$, we obtain a bound of the order $\|x-y\|_\s^{\gamma+\beta -|k|_\s}$
as required. Regarding the corresponding term arising in $\CK_\gamma f - \bar \CK_\gamma \bar f$,
we use the identity
\begin{equs}
\Pi_x \CQ_\delta \bigl( \Gamma_{xy} f(y) &- f(x)\bigr) - 
\bar \Pi_x \CQ_\delta \bigl( \bar \Gamma_{xy} \bar f(y) - \bar f(x)\bigr) \label{e:idenDiffPiGamma}\\
&\quad = \bigl(\Pi_x -\bar \Pi_x\bigr) \CQ_\delta \bigl( \Gamma_{xy} f(y) - f(x)\bigr) \\
&\qquad + \bar \Pi_x \CQ_\delta \bigl(\bar f(x) - f(x) - \bar \Gamma_{xy} \bar f(y) + \Gamma_{xy} f(y)\bigr)\;,
\end{equs}
and we bound both terms separately in the same way as above, making use of the definition
of $\$f;\bar f\$_{\gamma;\bar \K}$ in order to control the second term.

It remains to obtain similar bounds on large scales, i.e.\ in the regime $2^{-n} \ge \|x-y\|_\s$.
We define
\begin{equs}
\CT_1^k &\eqdef -k! \bigl(\bigl(\CN_\gamma^{(n)} f\bigr)(x) + \CJ^{(n)}(x) f(x)\bigr)_k \;,\\
\CT_2^k &\eqdef k! \bigl(\bigl(\Gamma_{xy} \CN_\gamma^{(n)} f\bigr)(y) + \CJ^{(n)}(x) \Gamma_{xy}f(y)\bigr)_k\;.
\end{equs}
Inspecting the definitions of these terms, we then obtain the identities
\begin{equs}
\CT_1^k &= \Bigl(\sum_{\zeta \le |k|_\s - \beta} \Pi_x \CQ_\zeta f(x) - \CR f\Bigr)\bigl(D_1^k K_n(x,\cdot)\bigr)\;,\\
\CT_2^k &= \sum_{\zeta > |k|_\s-\beta} \bigl(\Pi_x \CQ_\zeta \Gamma_{xy} f(y)\bigr)\bigl(D_1^{k} K_n(x,\cdot)\bigr)\\
&\quad -\sum_{|k+\ell|_\s < \gamma+\beta}{(x-y)^\ell\over \ell!} \bigl(\Pi_y f(y)-\CR f\bigr)\bigl(D_1^{k+\ell} K_n(y,\cdot)\bigr)\;.
\end{equs}
Adding these two terms, we have 
\begin{equs}
\CT_2^k + \CT_1^k &= \bigl(\Pi_y f(y) - \CR f\bigr)\bigl(K_{n;xy}^{k,\gamma}\bigr) \label{e:largeScaleIden}\\
&\quad -\sum_{\zeta \le |k|_\s - \beta} \bigl(\Pi_x \CQ_\zeta \bigl(\Gamma_{xy} f(y) - f(x)\bigr)\bigr)\bigl(D_1^k K_n(x,\cdot)\bigr)\;.
\end{equs}
In order to bound the first term, we proceed similarly to the proof of the second part of Lemma~\ref{lem:boundPibar}.
The only difference is that the analogue to the left hand side of \eref{e:boundPiKn} is now given by
\begin{equs}
\bigl(\Pi_y f(y) - \CR f\bigr)\bigl(D_1^{k+\ell}K_n(\bar y,\cdot)\bigr)
&= \bigl(\Pi_{\bar y} f(\bar y) - \CR f\bigr)\bigl(D_1^{k+\ell}K_n(\bar y,\cdot)\bigr) \label{e:splitExprPiR}\\ 
&\qquad + \bigl(\Pi_{\bar y} \bigl(\Gamma_{\bar y y} f(y) - f(\bar y)\bigr)\bigr)\bigl(D_1^{k+\ell}K_n(\bar y,\cdot)\bigr)\;,
\end{equs}
where we set $\bar y = y+h$. Regarding the first term in this expression, recall from
\eref{e:boundRf2} that
\begin{equ}
\bigl|\bigl(\Pi_{\bar y} f(\bar y) - \CR f\bigr)\bigl(D_1^{k+\ell}K_n(\bar y,\cdot)\bigr)\bigr|
\lesssim 2^{(|k+\ell|_\s-\beta-\gamma)n}\;.
\end{equ}
Since $\beta +\gamma \not \in \N$ by assumption, the exponent appearing in this
expression is always strictly positive, thus yielding the required bound.
The corresponding bound on $\CK_\gamma f - \bar \CK_\gamma \bar f$ is obtained
in the same way, but making use of \eref{e:boundRf2Diff} instead of \eref{e:boundRf2}.

To bound the second term 
in \eref{e:splitExprPiR}, we use the fact that $f \in \CD^\gamma$ which yields
\begin{equ}
\bigl|\bigl(\Pi_{\bar y} \bigl(\Gamma_{\bar y y} f(y) - f(\bar y)\bigr)\bigr)\bigl(D_1^{k+\ell}K_n(\bar y,\cdot)\bigr)\bigr|
\lesssim \sum_{\zeta \le \gamma} \|x-y\|_\s^{\gamma-\zeta} 2^{(|k+\ell|_\s-\zeta-\beta)n}\;.
\end{equ}
We thus obtain a bound analogous to \eref{e:boundPiKn}, with $\alpha$ replaced by $\gamma$.
Proceeding analogously to \eref{e:idenDiffPiGamma}, we obtain a similar bound
(but with a prefactor $\|\Gamma - \bar \Gamma\|_{\gamma+\beta;\bar \K} + \$f;\bar f\$_{\gamma;\bar \K}$)
for the corresponding term appearing in the difference between $\CK_\gamma f$ and $\bar \CK_\gamma\bar f$.
Proceeding as in the remainder of the proof of Lemma~\ref{lem:boundPibar}, we then obtain the bound
\begin{equ}[e:boundRfPi]
\bigl|\bigl(\Pi_y f(y) - \CR f\bigr)\bigl(K_{n;xy}^{k,\gamma}\bigr) \bigr| \lesssim \sum_{\delta > 0} 2^{\delta n} \|x-y\|_\s^{\delta + \gamma+\beta - |k|_\s}\;,
\end{equ}
where the sum runs only over finitely many values of $\delta$.
The corresponding bound for the difference is obtained in the same way.

Regarding the second term in \eref{e:largeScaleIden}, we obtain the bound
\begin{equ}
\bigl|\bigl(\Pi_x \CQ_\zeta \bigl(\Gamma_{xy} f(y) - f(x)\bigr)\bigr)\bigl(D_1^k K_n(x,\cdot)\bigr)\bigr|
\lesssim \|x-y\|_\s^{\gamma-\zeta} 2^{(|k|_\s-\beta-\zeta)n}\;.
\end{equ}
At this stage, one might again have summability 
problems if $\zeta = |k|_\s - \beta$. 
However, just as in the proof of Lemma~\ref{lem:boundGammaxy}, our 
assumptions guarantee that such
terms do not contribute.
Summing both of these bounds over the relevant values of $n$, the requested bound follows at once.
Again, the corresponding term involved in the difference can be bounded in the same way,
by making use of the decomposition \eref{e:idenDiffPiGamma}.

It remains to show that the identity \eref{e:wantedIden} holds. Actually, by the uniqueness part of the reconstruction
theorem, it suffices to show that, for any suitable test function $\psi$ and any $x \in D$, one has
\begin{equ}
\bigl(\Pi_x \CK f(x) - K * \CR f\bigr)(\CS_{\s,x}^\lambda \psi) \lesssim \lambda^\delta\;,
\end{equ}
for some strictly positive exponent $\delta$. Writing $\psi_x^\lambda = \CS_{\s,x}^\lambda \psi$ as a shorthand, we obtain the identity
\begin{equs}
\bigl(\Pi_x \CK &f(x) - K * \CR f\bigr)(\psi_x^\lambda) \\
&= \sum_{n\ge 0} \int \Bigg(\sum_{\zeta\in A} \bigl( \Pi_x \CQ_\zeta f(x)\bigr)
\Bigl(K_n(y,\cdot) - \sum_{|\ell|_\s < \zeta + \beta} {(y-x)^\ell\over \ell!} D_1^\ell K_n(x,\cdot)\Bigr) \\
&\quad +  \sum_{\zeta\in A}\sum_{|\ell|_\s < \zeta + \beta} {(y-x)^\ell\over \ell!} \bigl( \Pi_x \CQ_\zeta f(x)\bigr)
\bigl(D_1^\ell K_n(x,\cdot)\bigr) \\
&\quad+ \sum_{|k|_\s < \gamma+\beta} {(y-x)^k\over k!} \bigl(\CR f -  \Pi_x f(x)\bigr)\bigl(D_1^k K_n(x,\cdot)\bigr) \\
&\quad - \bigl(\CR f\bigr)\bigl(K_n(y,\cdot)\bigr)\Bigg)\,\psi_x^\lambda(y)\,dy\\
&= \sum_{n \ge 0} \int \bigl( \Pi_x f(x)-\CR f \bigr)(K_{n;yx}^\gamma) \psi_x^\lambda(y)\,dy\;.
\end{equs}
It thus remains to obtain
a suitable bound on $\bigl(\Pi_x f(x)-\CR f \bigr)(K_{n;yx}^\gamma)$. As is by now usual, we
treat separately the cases $2^{-n} \lessgtr \lambda$.

In the case $2^{-n} \ge \lambda$, we already obtained the bound \eref{e:boundRfPi} (with $k=0$),
which yields a bound of the order of $\lambda^{\gamma+\beta}$ when summed over $n$ and integrated
against $\psi_x^\lambda$. In the case $2^{-n} \le \lambda$, we rewrite $K_{n;yx}^\gamma$
as 
\begin{equ}[e:decompKn]
K_{n;yx}^\gamma = K_n(y,\cdot) - \sum_{|\ell|_\s < \gamma+\beta} {(y-x)^\ell\over \ell!} D^\ell_x K_n(x,\cdot)\;, 
\end{equ}
and we bound the resulting terms separately. To bound the terms involving derivatives of $K_n$,
we note that, as a consequence of the reconstruction theorem, we have the bound
\begin{equ}
\bigl|\bigl(\Pi_y f(y)-\CR f \bigr) \bigl(D^\ell_x K_n(x,\cdot)\bigr)\bigr| \lesssim 2^{(|\ell|_\s-\beta-\gamma)n}\;.
\end{equ}
Since this exponent is always strictly negative (because $\gamma+\beta \not\in \N$ by assumption),
this term is summable for large $n$. After summation and integration against $\psi_x^\lambda$, we 
indeed obtain a bound of the order of $\lambda^{\gamma+\beta}$ as required.

To bound the expression arising from the first term in \eref{e:decompKn}, we rewrite it as
\begin{equ}
\int \bigl(\Pi_x f(x)-\CR f \bigr) \bigl(K_n(y,\cdot)\bigr) \,\psi_x^\lambda(y)\,dy =
\bigl(\Pi_x f(x)-\CR f \bigr) \bigl(Y_n^\lambda\bigr)\;,
\end{equ}
where $Y_n^\lambda$ is as in \eref{e:defYn}. It then follows from \eref{e:boundYn}, combined with
the reconstruction theorem, that
\begin{equs}
\bigl|\bigl(\Pi_x f(x)-\CR f \bigr) \bigl(Y_n^\lambda\bigr)\bigr| \lesssim 2^{-\beta n} \lambda^{\gamma}\;.
\end{equs}
Summing over all $n$ with $2^{-n} \le \lambda$, we obtain again a bound of the order $\lambda^{\gamma+\beta}$,
which concludes the proof. 
\end{proof}

\begin{remark}
Alternatively, it is also possible to prove the multi-level Schauder estimate as a
consequence of the extension and the reconstruction theorems. 
The argument goes as follows: first, we add to $T$ one additional ``abstract''
element $b$ which we decree to be of homogeneity $\gamma$. We then extend the
representation $(\Pi,\Gamma)$ to $b$ by setting
\begin{equ}
\Pi_x b \eqdef \CR f - \Pi_x f(x)\;,\qquad \Gamma_{xy} b - b \eqdef f(x) - \Gamma_{xy} f(y) \;.
\end{equ}
(Of course the group $G$ has to be suitable extended to ensure the second identity.)
It is an easy exercise to verify that this satisfies the required algebraic
identities. Furthermore, the required analytical bounds on $\Pi$ are satisfied
as a consequence of the reconstruction theorem, while the bounds on $\Gamma$
are satisfied by the definition of $\CD^\gamma$.

Setting $\hat f(x) = f(x) + b$, it then follows immediately from the definitions
that $\Pi_x \hat f(x) = \CR f$ for every $x$. One can then apply the extension
theorem to construct an element $\CI b$ such that \eref{e:defIa} holds. 
In particular, this shows that the function $\hat F$ given by
\begin{equ}
\hat F(x) = \CI \hat f(x) + \CJ(x) \hat f(x)\;,
\end{equ}
satisfies $\Pi_x \hat F(x) = K * \CR f$ for every $x$.
Noting that $\Gamma_{xy} \hat F(x) = \hat F(y)$, it is then 
possible to show that on the one hand the map $x \mapsto \hat F(x) - \CI b$
belongs to $\CD^{\gamma+\beta}$, and that on the other hand one has
 $\hat F(x) - \CI b = \bigl(\CK_\gamma f\bigr)(x)$, so the claim follows.
 
The reason for providing the longer proof is twofold. First, it is more direct
and therefore gives a ``reality check''
of the rather abstract construction performed in the extension theorem.
Second, the direct proof extends to the case of singular modelled distributions
considered in Section~\ref{sec:singular} below, while the short argument given
above does not.  
\end{remark}

\subsection{The symmetric case}

If we are in the situation of some symmetry group $\SS$ acting on $\TT$ as in 
Section~\ref{sec:symmetric}, then it is natural to impose that $K$ is also symmetric in the sense
that $K(T_g x, T_g y) = K(x,y)$, and that the abstract integration map $\CI$ commutes with
the action of $\SS$ in the sense that $M_g \CI = \CI M_g$ for every $g \in \SS$.

One then has the following result:

\begin{proposition}\label{prop:intSym}
In the setting of Theorem~\ref{theo:Int}, assume furthermore that a discrete symmetry group
$\SS$ acts on $\R^d$ and on $\TT$, that $K$ is symmetric under this action, that 
$(\Pi,\Gamma)$ is adapted to it, and that $\CI$ commutes with it. 
Then, if $f \in \CD^\gamma$ is symmetric, so is $\CK_\gamma f$.
\end{proposition}

\begin{proof}
For $g \in \SS$, we write again its action on $\R^d$ as
$T_g x = A_g x + b_g$.
We want to verify that $M_g \bigl(\CK_\gamma f\bigr)(T_g x) = \bigl(\CK_\gamma f\bigr)(x)$.
Actually, this identity holds true separately for the three terms that make up
$\CK_\gamma f$ in \eref{e:defI}. 

For the first term, this holds by our assumption on $\CI$. To treat the second term, recall 
Remark~\ref{rem:polynomSym}. With the notation used there, we have the identity
\begin{equs}
M_g \CJ(T_g x) a &= \sum_{|k|_\s \le \alpha} {(A_g X)^k \over k!} \int D_1^k K(T_g x, z)\,\bigl(\Pi_{T_g x} a\bigr)(dz)\\
&= \sum_{|k|_\s \le \alpha} {(A_g X)^k \over k!} \int D_1^k K(T_g x, T_g z)\,\bigl(\Pi_{x} M_g a\bigr)(dz)\\
&= \sum_{|k|_\s \le \alpha} {X^k \over k!} \int D_1^k K(x, z)\,\bigl(\Pi_{x} M_g a\bigr)(dz)
= \CJ(x) M_g a\;,
\end{equs}
as required. Here, we made use of the symmetry of $K$, combined with the
fact that $A_g$ is an orthogonal matrix, to go from the second line to the third.
The last term is treated similarly by exploiting the symmetry of $\CR f$ given
by Proposition~\ref{prop:symmetric}.
\end{proof}

Finally, one has

\begin{lemma}\label{lem:decompSym}
In the setting of Lemma~\ref{lem:diffSing}, if $\bar K$ is symmetric, then it is possible
to choose the decomposition $\bar K = K+R$ in such a way that both $K$ and $R$ are symmetric.
\end{lemma}

\begin{proof}
Denote by $\GG$ the crystallographic point group associated to $\SS$. Then, given any decomposition
$\bar K = K_0 + R_0$ given by Lemma~\ref{lem:diffSing}, it suffices to set
\begin{equ}
K(x) = {1\over |\GG|}\sum_{A \in \GG} K(Ax)\;,\qquad
R(x) = {1\over |\GG|}\sum_{A \in \GG} R(Ax)\;.
\end{equ}
The required properties then follow at once.
\end{proof}

\subsection{Differentiation}
\label{sec:diff}

Being a local operation, differentiating a modelled distribution is straightforward, 
provided again that the model one works with
is sufficiently rich. Denote by $D_i$ the (usual) derivative of a distribution on $\R^d$ with respect to
the $i$th coordinate. We then have the following natural definition:

\begin{definition}\label{def:Di}
Given a sector $V$ of a regularity structure $\TT$,
a family of operators $\DD_i \colon V \to T$ is an abstract gradient for $\R^d$ with scaling $\s$
if
\begin{claim}
\item one has $\DD_i a \in T_{\alpha-\s_i}$ for every $a \in V_\alpha$,
\item one has $\Gamma \DD_i a = \DD_i \Gamma a$ for every $a \in V$ and every $i$.
\end{claim}
\end{definition}

Regarding the realisation of the actual derivations $D_i$, we use the following definition:

\begin{definition}\label{def:compatDiff}
Given an abstract gradient $\DD$ as above, a model $(\Pi,\Gamma)$ on $\R^d$ with scaling $\s$ 
is compatible with $\DD$ if the identity
\begin{equ}
D_i \Pi_x a = \Pi_x \DD_i a\;,
\end{equ}
holds for every $a \in V$ and every $x \in \R^d$.
\end{definition}

\begin{remark}
Note that we do not make any assumption on the interplay between the abstract gradient 
$\DD$ and the product $\star$. In particular, unless one happens to have the identity
$\DD_i(a\star b) = a\star \DD_i b + \DD_i a \star b$,
there is absolutely no \textit{a priori} reason forcing the Leibniz rule to hold.
This is not surprising since our framework can accommodate It\^o integration, where the
chain rule (and thus the Leibniz rule) fails.
See \cite{DavidK} for a more thorough investigation of this fact.
\end{remark}

\begin{proposition}
Let $\DD$ be an abstract gradient as above and let $f \in \CD_\alpha^\beta(V)$ for some $\beta > \s_i$
and some model $(\Pi,\Gamma)$ compatible with $\DD$. 
Then, $\DD_i f \in \CD_{\alpha-\s_i}^{\beta-\s_i}$ and the identity $\CR \DD_i f = D_i \CR f$ holds. 
\end{proposition}

\begin{proof}
The fact that $\DD_i f \in \CD_{\alpha-\s_i}^{\beta-\s_i}$ is an immediate consequence of the definitions, so we 
only need to show that $\CR \DD_i f = D_i \CR f$.

By the ``uniqueness'' part of the reconstruction theorem, this on the other hand follows immediately
if we can show that, for every fixed test function $\psi$ and every $x \in D$, one has
\begin{equ}
\bigl(\Pi_x \DD_i f(x) - D_i \CR f\bigr)(\psi_x^\lambda) \lesssim \lambda^\delta\;,
\end{equ}
for some $\delta > 0$. Here, we defined $\psi_x^\lambda = \CS_{\s,x}^\lambda \psi$ as before. 
By the assumption on the model $\Pi$, we have the identity
\begin{equ}
\bigl(\Pi_x \DD_i f(x) - D_i \CR f\bigr)(\psi_x^\lambda) = \bigl( D_i \Pi_x f(x) - D_i \CR f\bigr)(\psi_x^\lambda)
= -\bigl(\Pi_x f(x) - \CR f\bigr)(D_i \psi_x^\lambda)\;.
\end{equ}
Since $D_i \psi_x^\lambda = \lambda^{-\s_i} \CD_{\s,x}^\lambda D_i \psi$, it then follows immediately from the reconstruction theorem that the right hand side of this expression is of order $\lambda^{\beta-\s_i}$, as required.
\end{proof}

\begin{remark}
The polynomial regularity structures $\TT_{d,\s}$ do of course come equipped with a natural
gradient operator, obtained by setting $\DD_i X_j = \delta_{ij}\one$ and extending this to all of
$T$ by imposing the Leibniz rule. 
\end{remark}

\begin{remark}\label{rem:gradCov}
In cases where a symmetry $\SS$ acts on $\TT$, it is natural to impose that the abstract gradient
is covariant in the sense that if $g \in \SS$ acts on $\R^d$ as $T_g x = A_g x + b_g$ and $M_g$ denotes the
corresponding action on $T$, then one imposes that
\begin{equ}
M_g \DD_i \tau = \sum_{j=1}^d A_g^{ij} \DD_j \tau\;,
\end{equ}
for every $\tau$ in the domain of $\DD$. This is consistent with 
the fact that
\begin{equs}
\bigl(\Pi_x M_g \DD_i \tau\bigr)(\psi) &= 
\bigl(\Pi_{T_gx} \DD_i \tau\bigr)(T_g^\sharp \psi) = 
\bigl(D_i \Pi_{T_gx} \tau\bigr)(T_g^\sharp \psi) \\ &= 
- \bigl(\Pi_{T_gx} \tau\bigr)(D_i T_g^\sharp \psi) = 
- A_g^{ij} \bigl(\Pi_{T_gx} \tau\bigr)( T_g^\sharp D_j\psi) \\ &= 
- A_g^{ij} \bigl(\Pi_{x} M_g \tau\bigr)(D_j\psi) = 
A_g^{ij} \bigl(\Pi_{x} \DD_j M_g \tau\bigr)(\psi)\;,
\end{equs}
where summation over $j$ is implicit.
It is also consistent with Remark~\ref{rem:symcanon}.
\end{remark}

\section{Singular modelled distributions}
\label{sec:singular}

In all of the previous section, we have considered situations where 
our modelled distributions belong to some space $\CD^\gamma$, which ensures that the bounds
\eref{e:boundDgamma} hold locally uniformly in $\R^d$. One very important
situation for the treatment of initial conditions and / or boundary values is that of functions
$f \colon \R^d \to T$ which are of the class $\CD^\gamma$ away from some fixed 
sufficiently regular submanifold $P$ (think of the hyperplane formed by ``time $0$'', which will be
our main example), but may exhibit a singularity on $P$.

In order to streamline the exposition, we only consider the case where $P$ is given by a hyperplane 
that is furthermore parallel to some of the canonical basis elements of $\R^d$.
The extension to general submanifolds is almost immediate.
Throughout this section, we fix again the ambient space $\R^d$ and its scaling $\s$, and we
fix a hyperplane $P$ which we assume for simplicity to be given by
\begin{equ}
P = \{x \in \R^d\,:\, x_i = 0\;,\quad i=1,\ldots,\bar d\}\;.
\end{equ}
An important role will be played by the ``effective codimension'' of $P$, which we denote by
\begin{equ}[e:defm]
\m = \s_1+\ldots+\s_{\bar d}\;.
\end{equ}

\begin{remark}
In the case where $P$ is a 
smooth submanifold, it is important for our analysis that it has a product structure with 
each factor belonging to a 
subspace with all components having the same scaling. 
More precisely, we consider a partition $\PP$
of the set $\{1,\ldots,d\}$ into $J$ disjoint non-empty subsets
with cardinalities $\{d_j\}_{j=1}^J$ such that 
$\s_i = \s_j$ if and only if $i$ and $j$ belong to the same element of $\PP$.
This yields a decomposition
\begin{equ}
\R^d \sim \R^{d_1} \times \cdots \times \R^{d_J}\;.
\end{equ}
With this notation, we impose that $P$ is of the form
$\CM_1\times\ldots \times \CM_J$, with each of the $\CM_j$ being a smooth 
(or at least Lipschitz) submanifold of $\R^{d_j}$.
The effective codimension $\m$ is then given by $\m = \sum_{j=1}^J \m_j$, where 
$\m_j$ is the codimension of $\CM_j$ in $\R^{d_j}$, multiplied by the corresponding
scaling factor.
\end{remark}

We also introduce the notations
\begin{equ}
\|x\|_P = 1\wedge d_\s(x,P)\;,\qquad \|x,y\|_P = \|x\|_P \wedge \|y\|_P\;.
\end{equ}
Given a subset $\K \subset \R^d$, we also denote by $\K_P$ the set
\begin{equ}
\K_P = \{(x,y) \in (\K\setminus P)^2\,:\, x \neq y\quad\text{and}\quad \|x-y\|_\s \le  \|x,y\|_P\}\;.
\end{equ}
With these notations at hand, we define the spaces $\CD^{\gamma,\eta}_P$ similarly
to $\CD^\gamma$, but we introduce an additional exponent $\eta$ controlling
the behaviour of the coefficients near $P$. Our precise definition goes as follows:

\begin{definition}\label{def:singDist}
Fix a regularity structure $\TT$ and a model $(\Pi,\Gamma)$, as well as a hyperplane $P$ as above.
Then, for any $\gamma > 0$ and $\eta \in \R$,  we set
\begin{equ}
\|f\|_{\gamma,\eta;\K} \eqdef \sup_{x\in \K \setminus P} \sup_{\ell < \gamma} {\|f(x)\|_\ell \over \|x\|_P^{(\eta-\ell)\wedge0}} \;,\qquad
\n{f}_{\gamma,\eta;\K} \eqdef \sup_{x\in \K \setminus P} \sup_{\ell < \gamma} {\|f(x)\|_\ell \over \|x\|_P^{\eta-\ell}} \;. 
\end{equ}
The space $\CD^{\gamma,\eta}_P(V)$ then consists of all functions 
$f\colon \R^d \setminus P \to T_\gamma^-$ such that, for every compact set $\K \subset \R^d$, one has
\begin{equ}[e:boundDgammaSing]
\$f\$_{\gamma,\eta;\K} \eqdef \|f\|_{\gamma,\eta;\K}   + \sup_{(x,y) \in \K_P} \sup_{\ell < \gamma} {\|f(x) - \Gamma_{xy} f(y)\|_\ell \over \|x-y\|_\s^{\gamma - \ell} \|x,y\|_P^{\eta-\gamma}} < \infty\;.
\end{equ}
Similarly to before, we also set
\begin{equ}
\$f;\bar f\$_{\gamma,\eta;\K} \eqdef \|f-\bar f\|_{\gamma,\eta;\K}  + \sup_{(x,y) \in \K_P} \sup_{\ell < \gamma} {\|f(x) - \bar f(x) - \Gamma_{xy} f(y) + \bar \Gamma_{xy} \bar f(y)\|_\ell \over \|x-y\|_\s^{\gamma - \ell} \|x,y\|_P^{\eta-\gamma}}\;.
\end{equ}
\end{definition}

\begin{remark}
In the particular case of $\TT = \TT_{d,\s}$ and $(\Pi,\Gamma)$ being the canonical model
consisting of polynomials, we use the notation $\CC^{\gamma,\eta}_P(V)$ instead of $\CD^{\gamma,\eta}_P(V)$.
\end{remark}

At distances of order $1$ from $P$, we see that 
the spaces $\CD^{\gamma,\eta}_P$ and $\CD^\gamma$ coincide.
However, if $\K$ is such that $d_\s(x,P) \sim \lambda$ for all $x \in \K$,
then one has, roughly speaking,
\begin{equ}[e:approxScaling]
\$f\$_{\gamma,\eta;\K} \sim \lambda^{\gamma - \eta} \$f\$_{\gamma;\K} \;.
\end{equ}
In fact, this is not quite true: the components appearing in the first term
in \eref{e:boundDgammaSing} scale slightly differently. However, it turns out that the 
first bound actually follows from the second, provided that one has an order one bound on
$f$ somewhere at order one distance from $P$, so that \eref{e:approxScaling} does 
convey the right intuition in most situations.

The spaces $\CD^{\gamma,\eta}_P$ will be particularly useful when setting
up fixed point arguments to solve semilinear parabolic problems, where the solution exhibits a singularity
(or at least some form of discontinuity) at $t = 0$. In particular, in all of the concrete examples
treated in this article, we will have $P = \{(t,x)\,:\, t = 0\}$.

\begin{remark}
The space $\CD^{\gamma,0}_P$ does \textit{not} coincide with $\CD^\gamma$.
This is due to the fact that our definition still allows for some discontinuity at $P$. 
However, $\CD^{\gamma,\gamma}_P$ essentially coincides with
$\CD^\gamma$, the difference being that the supremum in \eref{e:boundDgammaSing}
only runs over elements in $\K_P$. If $P$ is a hyperplane of codimension $1$,
then $f(x)$ can have different limits whether $x$ approaches $P$ from one side or the other.
\end{remark}

Definition~\ref{def:singDist} is tailored in such a way that if $\K$ is of bounded diameter
and we know that 
\begin{equ}
\sup_{\ell < \gamma} \|f(x)\|_\ell < \infty
\end{equ}
for \textit{some} $x\in \K \setminus P$, then the bound on the first term in
\eref{e:boundDgammaSing} follows from the bound on the second term. 
The following statement is a slightly different 
version of this fact which will be particularly useful
when setting up local fixed point arguments, since it yields good control on $f(x)$
for $x$ near $P$.

For $x \in \R^d$ and $\delta > 0$, we write $\CS_P^\delta x$ for the value\label{lab:defSPdelta}
\begin{equ}
\CS_P^\delta x = (\delta x_1,\ldots,\delta x_{\bar d},x_{\bar d+1},\ldots ,x_d)\;.
\end{equ}
With this notation at hand, we then have:

\begin{lemma}\label{lem:vanishPower}
Let $\K$ be a domain such that for every $x = (x_1,\ldots,x_d) \in \K$, one has
$\CS_P^\delta x \in \K$ for every $\delta \in [0,1]$.
Let $f \in \CD^{\gamma,\eta}_P$ for some $\gamma > 0$ and assume that, for every  $\ell < \eta$, 
the map $x \mapsto \CQ_\ell f(x)$ extends continuously to all of $\K$ 
in such a way that $\CQ_\ell f(x)=0$ for $x\in P$.
Then, one has the bound
\begin{equ}
\n{f}_{\gamma,\eta;\K} \lesssim \$f\$_{\gamma,\eta;\K}\;,
\end{equ}
with a proportionality constant depending affinely on $\|\Gamma\|_{\gamma;\K}$.
Similarly, let $\bar f \in \CD^{\gamma,\eta}_P$ with respect to a second model $(\bar \Pi, \bar \Gamma)$
and assume this time that $\lim_{x \to P}  \CQ_\ell (f(x) - \bar f(x)) = 0$ for every $\ell < \eta$.
Then, one has the bound
\begin{equs}
\n{f - \bar f}_{\gamma,\eta;\K} &\lesssim \$f;\bar f\$_{\gamma,\eta;\K} + \|\Gamma - \bar \Gamma\|_{\gamma;\K} \bigl(\$f\$_{\gamma,\eta;\K} + \$\bar f\$_{\gamma,\eta;\K}\bigr) \;,
\end{equs}
with a proportionality constant depending again affinely on $\|\Gamma\|_{\gamma;\K}$
and $\|\bar \Gamma\|_{\gamma;\K}$.
\end{lemma}

\begin{proof}
For $d_\s(x,P) \ge 1$ or $\ell \ge \eta$, the bounds follow trivially from the 
definitions, so we only need to consider the
case $d_\s(x,P) < 1$ and $\ell < \eta$. We then set
$x_n = \CS_P^{2^{-n}} x$ and $x_\infty = \CS_P^0 x$.
We also use the shorthand $\Gamma_n = \Gamma_{x_{n+1}x_n}$,
and we assume without loss of generality that $\$f\$_{\gamma,\eta;\K} \le 1$.
Note that the sequence $x_n$ converges to $x_\infty$ and that 
\begin{equ}[e:boundDist]
\|x_{n+1} - x_n\|_\s = \|x_{n+1} - x_\infty\|_\s = \|x_{n+1}\|_P = 2^{-(n+1)}\|x\|_P\;.
\end{equ}

The argument now goes by ``reverse induction'' on $\ell$. Assume that the 
bound $\|f(x)\|_m \lesssim \|x\|_P^{\eta-m}$
holds for all $m > \ell$, which we certainly know to be the case when
$\ell$ is the largest element in $A$ smaller than $\eta$ since then this bound is already controlled
by $\$f\$_{\gamma,\eta;\K}$.
One then has
\begin{equs}
\|f(x_{n+1}) - f(x_n)\|_\ell &\le \|f(x_{n+1}) - \Gamma_n f(x_n)\|_\ell + \|(1-\Gamma_n)f(x_n)\|_\ell \label{e:firstLine}\\
&\lesssim 2^{-n(\eta - \ell)}\|x\|_P^{\eta - \ell} + \sum_{m > \ell} 2^{-n(m-\ell)}\|x\|_P^{m-\ell} 2^{-n(\eta - m)} \|x\|_P^{\eta-m} \\
&\lesssim 2^{-n(\eta - \ell)}\|x\|_P^{\eta-\ell}\;,
\end{equs}
where we made use of the definition of $\$f\$_{\gamma,\eta;\K}$ and \eref{e:boundDist}
to bound the first term and
of the inductive hypothesis, combined with \eref{e:boundDist} and
the bounds on $\Gamma$ for the second term.
It immediately follows that 
\begin{equ}
\|f(x)\|_\ell =
\|f(x) - f(x_\infty)\|_\ell 
\le \sum_{n \ge 0} \|f(x_{n+1}) - f(x_n)\|_\ell \lesssim  \sum_{n \ge 0} 2^{-n(\eta - \ell)}\|x\|_P^{\eta-\ell}\;,
\end{equ}
which is precisely what is required for the first bound to hold.
Here, the induction argument on $\ell$ works because $A$ is locally finite
by assumption.

The second bound follows in a very similar way. Setting $\delta f = f - \bar f$, we write
\begin{equs}
\|\delta f(x_{n+1}) - \delta f(x_n)\|_\ell &\le \|f(x_{n+1}) - \bar f(x_{n+1}) - \Gamma_n f(x_n) + \bar \Gamma_n \bar f(x_n)\|_\ell \\
&\qquad + \|(1-\Gamma_n)f(x_n)- (1-\bar \Gamma_n)\bar f(x_n)\|_\ell\;.
\end{equs}
The first term in this expression is bounded in the same way as above. The 
second term is bounded by
\begin{equ}
\|(1-\Gamma_n)f(x_n) - (1-\bar \Gamma_n)\bar f(x_n)\|_\ell \lesssim 
2^{-n(\eta - \ell)}\|x\|_P^{\eta-\ell} \bigl(\$f, \bar f\$_{\gamma,\eta;\K} + \|\Gamma - \bar \Gamma\|_{\gamma;\K}\bigr)\;,
\end{equ}
from which the stated bound then also follows in the same way as above.
\end{proof}

%
The following kind of interpolation inequality will also be useful:

\begin{lemma}\label{lem:interpolation}
Let $\gamma > 0$ and $\kappa \in (0,1)$ and let $f$ and $\bar f$ satisfy the assumptions of Lemma~\ref{lem:vanishPower}. Then, for every compact set $\K$, one has the bound 
\begin{equ}
\$f;\bar f\$_{(1-\kappa)\gamma,\eta;\K} \lesssim \n{f-\bar f}_{\gamma,\eta;\K}^{\kappa}  \bigl(\$f\$_{\gamma,\eta;\K} + \$\bar f\$_{\gamma,\eta;\K}\bigr)^{1-\kappa}\;,
\end{equ}
where the proportionality constant depends on $\|\Gamma\|_{\gamma;\K} + \|\bar \Gamma\|_{\gamma;\K}$.
\end{lemma}

\begin{proof}
All the operations are local, so we can just as well take $\K = \R^d$.
First, one then has the obvious bound
\begin{equ}
\|f(x) - \Gamma_{xy} f(y) - \bar f(x) + \bar \Gamma_{xy} \bar f(y)\|_{\ell} 
\le \bigl(\$f\$_{\gamma,\eta} + \$\bar f\$_{\gamma,\eta}\bigr) \|x-y\|_\s^{\gamma - \ell} \|x,y\|_P^{\eta-\gamma}\;.
\end{equ}
On the other hand, one also has the bound
\begin{equ}
\|f(x) - \Gamma_{xy} f(y) - \bar f(x) + \bar \Gamma_{xy} \bar f(y)\|_{\ell} 
\lesssim \n{f-\bar f}_{\gamma,\eta} \|x,y\|_P^{\eta-\ell}\;,
\end{equ}
where the proportionality depends on the sizes of $\Gamma$ and $\bar \Gamma$. 
As a consequence of these two bounds, we obtain
\begin{equs}
\|f(x) &- \Gamma_{xy} f(y) - \bar f(x) + \bar \Gamma_{xy} \bar f(y)\|_{\ell}
\lesssim \n{f-\bar f}_{\gamma,\eta}^\kappa \bigl(\$f\$_{\gamma,\eta} + \$\bar f\$_{\gamma,\eta}\bigr)^{1-\kappa}\\
&\qquad \times \|x-y\|_\s^{\gamma - \ell - \kappa(\gamma - \ell)}
\|x,y\|_P^{\eta-\kappa \ell - (1-\kappa)\gamma} \\
&\lesssim \n{f-\bar f}_{\gamma,\eta}^\kappa \bigl(\$f\$_{\gamma,\eta} + \$\bar f\$_{\gamma,\eta}\bigr)^{1-\kappa} \|x-y\|_\s^{(1-\kappa) \gamma - \ell}
\|x,y\|_P^{\eta - (1-\kappa)\gamma}\;,
\end{equs}
which is precisely the required bound. Here, we made use of the fact that 
we only consider points with $\|x-y\|_\s \lesssim \|x,y\|_P$ to obtain the last inequality.

Regarding the bound on $\|f(x) - \bar f(x)\|_\ell$, one immediately obtains the required bound
\begin{equ}
\|f(x) - \bar f(x)\|_\ell \lesssim \n{f-\bar f}_{\gamma,\eta}^\kappa \bigl(\$f\$_{\gamma,\eta} + \$\bar f\$_{\gamma,\eta}\bigr)^{1-\kappa} \|x\|_P^{(\eta-\ell)\wedge 1}\;,
\end{equ}
simply because both $\n{\cdot}_{\gamma,\eta}$ and $\$\cdot\$_{\gamma,\eta}$ dominate that term.
\end{proof}

In this section, we show that all of the calculus developed in the previous sections still carries over to
these weighted spaces, provided that the exponents $\eta$ are chosen in a suitable way. 
The proofs are mostly based on relatively straightforward but tedious modifications 
of the existing proofs in the uniform case, so we will try to focus mainly on those aspects that 
do actually differ.

\subsection{Reconstruction theorem}

We first obtain a modified version of the reconstruction theorem for elements 
$f \in \CD^{\gamma,\eta}_P$. Since the reconstruction operator $\CR$ is local and since 
$f$ belongs to $\CD^\gamma$ away from $P$,  there exists a unique element 
$\tilde \CR f$ in the dual of smooth functions that are compactly supported away from $P$
which is such that
\begin{equ}
\bigl(\tilde \CR f - \Pi_x f(x)\bigr)(\psi_x^\lambda) \lesssim \lambda^\gamma\;,
\end{equ}
for all $x \not \in P$ and $\lambda \ll d(x,P)$.
The aim of this subsection is to show that, under suitable assumptions,  $\tilde \CR f$ extends
in a natural way to an actual distribution $\CR f$ on $\R^d$. 

In order to prepare for this result, the following result will be useful.
\begin{lemma}\label{lem:restriction}
Let $\TT = (A,T,G)$ be a regularity structure and let $(\Pi,\Gamma)$ be a model for $\TT$ over $\R^d$
with scaling $\s$. 
Let $\psi \in \CB_{\s,0}^r$ with $r > |\min A|$ and $\lambda > 0$. Then, for $f \in \CD^\gamma$,
one has the bound
\begin{equ}
\bigl|\bigl(\CR f - \Pi_x f(x)\bigr)(\psi_x^\lambda)\bigr| \lesssim \lambda^\gamma
\sup_{y,z\in B_{2\lambda}(x)}\sup_{\ell < \gamma} {\|f(z) - \Gamma_{zy} f(y)\|_\ell \over \|z-y\|_\s^{\gamma-\ell}}\;,
\end{equ}
where the proportionality constant is of order $1+\|\Gamma\|_{\gamma; B_{2\lambda}(x)} \$\Pi\$_{\gamma; B_{2\lambda}(x)}$.
\end{lemma}

\begin{remark}
This is essentially a refinement of the reconstruction theorem. The difference is that 
the bound only uses information about $f$ in a small area around the support of 
$\phi_x^\lambda$.
\end{remark}

\begin{proof}
Inspecting the proof of Proposition~\ref{prop:reconstrGen}, we note that one really 
only uses the bounds \eref{e:assGenReconst} only for pairs $x$ and $y$ with $\|x-y\|_\s \le C \lambda$
for some fixed $C > 0$. By choosing $n_0$ sufficiently large, one can furthermore 
easily ensure that $C \le 2$.
\end{proof}
\begin{proposition}\label{prop:recSing}
Let $f \in \CD^{\gamma,\eta}_P(V)$ for some sector $V$ of regularity $\alpha \le 0$, some $\gamma > 0$, 
and some $\eta \le \gamma$.
Then, provided that $\alpha \wedge \eta  > -\m$ where $\m$ is as in \eref{e:defm},
there exists a unique distribution $\CR f \in \CC_\s^{\alpha \wedge \eta}$ such that
$\bigl(\CR f\bigr)(\phi) = \bigl(\tilde \CR f\bigr)(\phi)$ for smooth test functions
that are compactly supported away from $P$.
If $f$ and $\bar f$ are modelled after two models $Z$ and $\bar Z$, then one has the bound
\begin{equ}
\|\CR f - \bar \CR \bar f\|_{\alpha\wedge \eta;\K} \lesssim \$f;\bar f\$_{\gamma,\eta;\bar \K} + \$Z;\bar Z\$_{\gamma;\bar \K}\;,
\end{equ}
where the proportionality constant depends on the norms of $f$, $\bar f$, $Z$ and $\bar Z$.
Here, $\K$ is any compact set and $\bar \K$ is its $1$-fattening.
\end{proposition}

\begin{remark}
The condition $\alpha \wedge \eta  > -\m$ rules out the possibility of creating a non-integrable
singularity on $P$, which would prevent $\tilde \CR f$ from defining a distribution on all of $\R^d$.
(Unless one ``cancels out'' the singularity by a diverging term located on $P$, but 
this would then lead to $\CR f$ being well-posed only up to some finite distribution localised
on $P$.)
\end{remark}

\begin{remark}
If $\alpha = 0$ and $\eta \ge 0$, then due to our definition of $\CC^\alpha_\s$, 
Proposition~\ref{prop:recSing} only
implies that $\CR f$ is a bounded function, not that it is actually continuous.
\end{remark}

\begin{proof}
Since the reconstruction operator is linear and local, it suffices to consider the case where
$\$f\$_{\gamma,\eta;\K} \le 1$, which we will assume from now on.

Our main tool in the proof of this result is a suitable partition of the identity in the complement of
$P$. Let $\phi \colon \R_+ \to [0,1]$ be as in Lemma~\ref{lem:diffSing} and let $\tilde \phi\colon \R \to [0,1]$
be a smooth function such that $\supp \tilde \phi \subset [-1,1]$ and
\begin{equ}
\sum_{k \in \Z} \tilde \phi(x+k) = 1\;.
\end{equ}
For $n \in \Z$, we then define the countable sets $\Xi_P^n$ by
\begin{equ}
\Xi_P^n = \{x \in \R^d \,:\, \text{$x_i = 0$ for $i \le \bar d$ and $x_i \in 2^{-n\s_i}\Z$ for  $i > \bar d$}\}\;.
\end{equ}
This is very similar to the definition of the sets $\Lambda_{n}^{\s}$ in Section~\ref{sec:wavelets},
except that the points in $\Xi_P^n$ are all located in a small ``boundary layer'' around $P$.
For $n \in \Z$ and $x \in \Xi_P^n$, we define the cutoff function $\phi_{x,n}$ by
\begin{equ}
\phi_{x,n}(y) = \phi(2^{n} N_P(y)) \tilde \phi\bigl(2^{n\s_{\bar d+1}}(y_{\bar d+1}-x_{\bar d+1})\bigr)
\cdots \tilde \phi\bigl(2^{n\s_{d}}(y_d-x_{d})\bigr)\;,
\end{equ}
where $N_P$ is a smooth function on $\R^d \setminus P$
which depends only on $(y_1,\ldots,y_{\bar d})$, and which is ``$1$-homogeneous''
in the sense that $N_P(D_\s^\delta y) = \delta N_P(y)$.

One can verify that this construction yields a partition of the unity in the sense that
\begin{equ}
\sum_{n \in \Z} \sum_{x\in \Xi_P^n} \phi_{x,n}(y) = 1\;,
\end{equ}
for every $y \in \R^d \setminus P$. 

Let furthermore $\hat \phi_N$ be given by
$\hat \phi_N = \sum_{n \le N}\sum_{x\in \Xi_P^n} \phi_{x,n}$.
One can then show that, for every distribution $\xi \in \CC_\s^{\bar \alpha}$ with
$\bar \alpha > -\m$ and every smooth test function $\psi$, one has
\begin{equ}
\lim_{N \to \infty} \xi\bigl(\psi (1-\hat \phi_N)\bigr) = 0\;.
\end{equ}
As a consequence, it suffices to show that, for every smooth compactly supported test function $\psi$,
the sequence $\bigl(\tilde\CR f\bigr)\bigl(\psi \hat \phi_N\bigr)$ is Cauchy and
that its limit, which we denote by $\bigl(\CR f\bigr)(\psi)$, satisfies the bound of Definition~\ref{def:C-alpha}.

Take now a smooth test function $\psi$ supported in $B(0,1)$ and define the 
translated and rescaled versions $\psi_x^\lambda$ as before with $\lambda \in (0,1]$. 
If $d_\s(x,P) \ge 2\lambda$, then it follows from Lemma~\ref{lem:restriction}
that
\begin{equ}[e:boundRfPi2]
\bigl(\tilde \CR f - \Pi_x f(x)\bigr)(\psi_x^\lambda) \lesssim  d_\s(x,P)^{\eta-\gamma} \lambda^\gamma
\lesssim   \lambda^\eta \;,
\end{equ}
where the last bound follows from the fact that $\gamma \ge \eta$ by assumption.
Since furthermore
\begin{equ}[e:boundPifSing]
\bigl(\Pi_x f(x)\bigr)(\psi_x^\lambda) \lesssim \sum_{\alpha \le \ell < \gamma} \|x\|_P^{(\eta-\ell)\wedge 0}  \lambda^\ell
\lesssim \lambda^{\alpha \wedge \eta}\;,
\end{equ}
we do have the required bound in this case.

In the case $d_\s(x,P) \le 2\lambda$, we rewrite $\psi_x^\lambda$ as
\begin{equ}
\psi_x^\lambda = \sum_{n \ge n_0} \sum_{y \in \Xi^n_P} \psi_x^\lambda \phi_{y,n}\;,
\end{equ}
where $n_0$ is the greatest integer such that $2^{-n_0} \ge 3\lambda$. 
Setting
\begin{equ}
\chi_{n,xy} = \lambda^{|\s|} 2^{n|\s|}\psi_x^\lambda \phi_{y,n}\;,
\end{equ}
it is straightforward to verify that $\chi_{n,xy}$ satisfies the bounds
\begin{equ}
\sup_{z \in \R^d} |D^k \chi_{n,xy}(z)| \lesssim 2^{-(|\s|+ |k|_\s)n}\;,
\end{equ}
for any multiindex $k$. Furthermore, just as in the case of the bound \eref{e:boundRfPi2},
every point in the support of $\chi_{n,xy}$ is located at a distance of $P$ that is of the same order.
Using a suitable partition of unity, one can therefore rewrite it as
\begin{equ}
\chi_{n,xy} = \sum_{j=1}^M \chi_{n,xy}^{(j)}\;,
\end{equ}
where $M$ is a fixed constant and where each of the $\chi_{n,xy}^{(j)}$ has its support centred in
a ball of radius ${1\over 2}d_\s(z_j, P)$ around some point $z_j$.
As a consequence, by the same argument as before, we obtain the bound
\begin{equ}[e:boundRfChi]
\bigl(\tilde \CR f - \Pi_{z_j} f(z_j)\bigr)(\chi_{n,xy}) \lesssim  \sum_{j=1}^M d_\s(z_j,P)^{\eta-\gamma} 2^{-\gamma n}
\lesssim  2^{-\eta n} \;.
\end{equ}
Using the same argument as in \eref{e:boundPifSing}, it then follows at once that
\begin{equ}[e:boundRfchi]
\bigl|\bigl(\tilde \CR f\bigr)(\chi_{n,xy})\bigr| \lesssim 2^{-(\alpha \wedge \eta)n}\;.
\end{equ}
Note now that we have the identity 
\begin{equ}
\bigl(\tilde\CR f\bigr)\bigl(\psi_x^\lambda \hat \phi_N\bigr) = \sum_{n = n_0}^N \lambda^{-|\s|} 2^{-n|\s|}\sum_{y \in \Xi_P^n} \bigl(\tilde\CR f\bigr)\bigl(\chi_{n,xy}\bigr)\;.
\end{equ}
At this stage, we make use of the fact that $\chi_{n,xy} = 0$, unless $\|x-y\|_\s \lesssim \lambda$.
As a consequence, for $n \ge n_0$, the number of terms contributing in the above sum is
bounded by $(2^n \lambda)^{|\s| - \m}$. Combining
this remark with \eref{e:boundRfchi} yields the bound
\begin{equ}
\Bigl|\lambda^{-|\s|} 2^{-n|\s|}\sum_{y \in \Xi_P^n} \bigl(\tilde\CR f\bigr)\bigl(\chi_{n,xy}\bigr)\Bigr|
\lesssim \lambda^{-\m} 2^{-((\alpha\wedge\eta)+\m)n}\;,
\end{equ}
from which the claim follows at once, provided that $\alpha \wedge \eta  > - \m$, which is
true by assumption.
The bound on $\CR f - \bar \CR f$ then follows in exactly the same way.
\end{proof}

In the remainder of this section, we extend the calculus developed in the previous sections 
to the case of singular modelled distributions. 

\subsection{Multiplication}

We now show that the product of two singular modelled distributions yields again a singular modelled
distribution under suitable assumptions. The precise workings of the exponents is as follows:

\begin{proposition}\label{prop:multSing}
Let $P$ be as above and let
$f_1 \in \CD^{\gamma_1,\eta_1}_P(V^{(1)})$ and $f_2 \in \CD^{\gamma_2,\eta_2}_P(V^{(2)})$ for two sectors
$V^{(1)}$ and $V^{(2)}$ with respective regularities $\alpha_1$ and $\alpha_2$.
Let furthermore $\star$ be a product on $T$ such that $(V^{(1)},V^{(2)})$ is $\gamma$-regular with
$\gamma = (\gamma_1 + \alpha_2) \wedge (\gamma_2 + \alpha_1)$.
Then, the function $f = f_1\star_\gamma f_2$ belongs to
$\CD^{\gamma, \eta}_P$ with $\eta = (\eta_1 + \alpha_2)\wedge (\eta_2+\alpha_1) \wedge (\eta_1+\eta_2)$.
(Here, $\star_\gamma$ is the projection of the product $\star$ onto $T_\gamma^-$ as before.)

Furthermore, in the situation analogous to Proposition~\ref{prop:multDiff}, 
writing $f = f_1 \star f_2$ and $g = g_1 \star g_2$, one has the bound
\begin{equ}
\$f;g\$_{\gamma,\eta;\K}
\lesssim \$f_1;g_1\$_{\gamma_1,\eta_1;\K} + \$f_2;g_2\$_{\gamma_2,\eta_2;\K}
+ \|\Gamma - \bar \Gamma\|_{\gamma_1+\gamma_2;\K}\;,
\end{equ}
uniformly over any bounded set.
\end{proposition}

\begin{proof}
We first show that $f = f_1 \star_\gamma f_2$ does indeed satisfy the claimed bounds. 
By Theorem~\ref{theo:mult}, we only need to consider points $x,y$ which are both at 
distance less than $1$ from $P$.
Also, by bilinearity and locality, it suffices to consider the case when both $f_1$ and $f_2$ are of 
norm $1$ on the fixed compact $\K$.
Regarding the supremum bound on $f$, we have
\begin{equs}
\|f(x)\|_\ell &\le \sum_{\ell_1 + \ell_2 = k} \|f_1(x)\|_{\ell_1}\|f_2(x)\|_{\ell_2} \le  \sum_{\ell_1 + \ell_2 = \ell} \|x\|_P^{(\eta_1 - \ell_1)\wedge 0}\|x\|_P^{(\eta_2 - \ell_2)\wedge 0}\\
 &\lesssim \|x\|_P^{(\eta - \ell)\wedge 0}\;,
\end{equs}
which is precisely as required.

It remains to obtain a suitable bound on $f(x) - \Gamma_{xy} f(y)$. For this, it follows from 
Definition~\ref{def:singDist}
that it suffices to consider pairs $(x,y)$ such that $2\|x-y\|_\s \le d_\s(x,P) \wedge d_\s(y,P) \le 1$. 
For such pairs $(x,y)$, it follows immediately from the triangle inequality that
\begin{equ}[e:simDist]
d_\s(x,P) = \|x\|_P \sim \|y\|_P \sim \|x,y\|_P\;,
\end{equ}
in the sense that any of these quantities is bounded by a multiple of any other quantity, with  
some universal proportionality constants. For $\ell < \gamma_i$, one then has the bounds
\begin{equs}[e:boundfi]
\|f_i(x) - \Gamma_{xy} f_i(y)\|_\ell &\lesssim \|x-y\|_\s^{\gamma_i - \ell_i} \|x,y\|_P^{\eta_i-\gamma_i} \;,\\
\|f_i(x)\|_\ell &\lesssim \|x,y\|_P^{(\eta_i - \ell_i)\wedge 0}\;,
\end{equs}
for $i \in \{1,2\}$.

As in \eref{e:boundProdGamma}, one then has
\begin{equs}
 \|\Gamma_{xy} f(y) &- (\Gamma_{xy} f_1(y))\star (\Gamma_{xy} f_2(y))\|_\ell
 \lesssim \sum_{m+n \ge \gamma} \|x-y\|_\s^{m+n-\ell} \|f_1(y)\|_m\|f_2(y)\|_n \\
 &\lesssim \|x-y\|_\s^{\gamma-\ell} \sum_{m+n \ge \gamma} \|x,y\|_P^{m+n-\gamma} \|x,y\|_P^{(\eta_1 -m)\wedge 0}\|x,y\|_P^{(\eta_2 -n)\wedge 0}\\
 &= \|x-y\|_\s^{\gamma-\ell} \sum_{m+n \ge \gamma} \|x,y\|_P^{-\gamma} \|x,y\|_P^{\eta_1 \wedge m}\|x,y\|_P^{\eta_2 \wedge n}\\
 &\lesssim \|x-y\|_\s^{\gamma-\ell} \|x,y\|_P^{-\gamma} \|x,y\|_P^{\eta_1 \wedge \alpha_2}\|x,y\|_P^{\eta_2 \wedge \alpha_1}\\
&= \|x-y\|_\s^{\gamma-\ell} \|x,y\|_P^{\eta-\gamma}\;. \label{e:prodSing}
\end{equs}
Here, in order to obtain the second line, 
we made use of \eref{e:simDist}, as well as
the fact that we are only considering points 
$(x,y)$ such that $\|x-y\|_\s \le \|x,y\|_P$.
Combining this with the bound \eref{e:splitProd} from the proof of Theorem~\ref{theo:mult}
and using again the bounds \eref{e:boundfi}, the requested bound then follows at once.

It remains to obtain a bound on $\$f;g\$_{\gamma,\eta;\K}$. For this, we proceed
almost exactly as in Proposition~\ref{prop:multDiff}. First note that, proceeding as
above, one obtains the estimate
\begin{equs}
\|\Gamma_{xy} f(y) &- \bar \Gamma_{xy} g(y) - \Gamma_{xy} f_1(y)\star \Gamma_{xy} f_2(y)
+  \bar\Gamma_{xy} g_1(y)\star \bar\Gamma_{xy} g_2(y)\|_\ell \\
&\le \|\Gamma - \bar \Gamma\|_{\gamma_1+\gamma_2;\K} \sum_{m+n \ge \gamma} \|x-y\|_\s^{m+n-\ell} \|f_1(y)\|_m\|f_2(y)\|_n \\
&\quad + \sum_{m+n \ge \gamma} \|x-y\|_\s^{m+n-\ell} \|f_1(y) - g_1(y)\|_m\|f_2(y)\|_n\\
&\quad + \sum_{m+n \ge \gamma} \|x-y\|_\s^{m+n-\ell} \|g_1(y)\|_m\|f_2(y)-g_2(y)\|_n\;,
\end{equs}
which then yields a bound of the desired type by proceeding as in \eref{e:prodSing}. 
The remainder is then decomposed exactly as in \eref{e:termsTi}. Denoting by $T_1,\ldots,T_5$
the terms appearing there, we proceed to bound them again separately.

For the term $T_1$, we obtain this time the bound
\begin{equ}
\|T_1\|_\ell \lesssim \$f_1;g_1\$_{\gamma_1,\eta_1;\K} \sum_{m+n = \ell\atop
n \ge \alpha_2; m \ge \alpha_1} \|x-y\|_\s^{\gamma_1 - m} \|x\|_P^{(\eta_2-n)\wedge 0} \|x,y\|_P^{\eta_1-\gamma_1}\;.
\end{equ}
Since, as in the proof of Proposition~\ref{prop:multDiff},  
all the terms in this sum satisfy $\gamma_1 - m > \gamma-\ell$, 
we can bound $\|x-y\|_\s^{\gamma_1 - m}$ by $\|x-y\|_\s^{\gamma - \ell} \|x,y\|_P^{n + \gamma_1 - \gamma}$. We thus obtain the bound
\begin{equ}
\|T_1\|_\ell \lesssim \$f_1;g_1\$_{\gamma_1,\eta_1;\K} \|x-y\|_\s^{\gamma - \ell} \|x,y\|_P^{(\eta_2 \wedge \alpha_2) +\eta_1 - \gamma}\;.
\end{equ}
Since $\eta \le (\eta_2 \wedge \alpha_2) +\eta_1$, this bound is precisely as required.

The bound on $T_2$ follows in a similar way, once we note that for the pairs $(x,y)$ under
consideration one has
\begin{equs}
\|\Gamma_{xy} f_1(y)\|_\ell &\ls \sum_{m \ge \ell}\|x-y\|_\s^{m-\ell}\|f_1(y)\|_m
\ls \sum_{m \ge \ell}\rho_\s^{m-\ell}(x, y)\|f_1(y)\|_m \\
&\ls \sum_{m \ge \ell}\rho_\s^{m-\ell}(x, y)\rho_\s^{(\eta_1-m)\wedge 0}(x, y)
\lesssim \rho_\s^{(\eta_1-\ell)\wedge 0}(x, y)\;,\label{e:remBoundxy}
\end{equs}
where we used \eref{e:simDist} to obtain the penultimate bound,
so that $\Gamma_{xy} f_1(y)$ satisfies essentially the same bounds as $f_1(x)$.

Regarding the term $T_5$, we obtain
\begin{equ}
\|T_5\|_\ell \ls \|f_2-g_2\|_{\gamma_2,\eta_2;\K}\sum_{m+n = \ell \atop m \ge \alpha_1; n\ge \alpha_2}
\|x-y\|_\s^{\gamma_1-m}\|x,y\|_P^{\eta_1-\gamma_1} \|y\|_P^{(\eta_2-n)\wedge 0}\;,
\end{equ}
from which the required bound follows in the same way as for $T_1$. The term $T_3$ is treated
in the same way by making again use of the remark \eref{e:remBoundxy}, this time with 
$g_1(y) - f_1(y)$ playing the role of $f_1(y)$.

The remaining term $T_4$ can be bounded in virtually the same way as $T_5$, the main difference
being that the bounds on $\bigl(\bar \Gamma_{xy} - \Gamma_{xy}\bigr) f_1(y)$ are proportional
to $\|\Gamma-\bar \Gamma\|_{\gamma_1;\K}$, so that one has
\begin{equ}
\|T_4\|_\ell \ls \|\Gamma-\bar \Gamma\|_{\gamma_1;\K} \|x-y\|_\s^{\gamma - \ell} \|x,y\|_P^{(\eta_1 \wedge \alpha_1) +\eta_2 - \gamma}\;.
\end{equ}
Combining all of these bounds completes the proof.
\end{proof}

\subsection{Composition with smooth functions}

Similarly to the case of multiplication of two modelled distributions, we can compose them with 
smooth functions as in Section~\ref{sec:compSmooth}, provided that they belong to $\CD^{\gamma,\eta}_P(V)$
for some function-like sector $V$ stable under the product $\star$, and for some
$\eta \ge 0$.

\begin{proposition}\label{prop:compSmoothSing}
Let $P$ be as above, let $\gamma > 0$, and let
$f_i \in \CD^{\gamma,\eta}_P(V)$ be a collection of $n$ modelled distributions for some
function-like sector $V$ which is  stable under the product $\star$.
Assume furthermore that $V$ is $\gamma$-regular in the sense of Definition~\ref{def:regularity}.

Let furthermore $F\colon \R^n \to \R$ be a smooth function. Then, provided that $\eta \in [0,\gamma]$,
the modelled distribution
$\hat F_\gamma(f)$ defined as in Section~\ref{sec:compSmooth} also belongs to $\CD^{\gamma,\eta}_P(V)$.
Furthermore, the map $\hat F_\gamma\colon \CD^{\gamma,\eta}_P(V) \to\CD^{\gamma,\eta}_P(V)$ is locally Lipschitz
continuous in any of the seminorms $\|\cdot\|_{\gamma,\eta;\K}$ and $\$\cdot\$_{\gamma,\eta;\K}$.
\end{proposition}

\begin{remark}
In fact, we do not need $F$ to be $\CC^\infty$, but the same regularity requirements as in 
Section~\ref{sec:compSmooth} suffice. Also, it is likely that one could obtain 
continuity in the strong sense, but in the interest of brevity, we refrain from doing so.
\end{remark}

\begin{proof}
Write $b(x) = \hat F_\gamma(f(x))$ as before. 
We also set $\zeta \in [0,\gamma]$ as in the proof of Theorem~\ref{theo:smooth}.
Regarding the bound on $\|b\|_{\gamma,\eta;\K}$, we note first that since we assumed
that $\eta \ge 0$, \eref{e:boundDgammaSing} implies that the quantities $D^k F(\bar f(x))$
are locally uniformly bounded. It follows that one has the bound
\begin{equ}
\|b(x)\|_\ell \ls \sum_{\ell_1+\ldots+\ell_n = \ell} \|f(x)\|_{\ell_1}\ldots \|f(x)\|_{\ell_n}\;,
\end{equ}
where the sum runs over all possible ways of decomposing $\ell$ into 
finitely many strictly positive elements $\ell_i \in A$.
Note now that one necessarily has the bound
\begin{equ}[e:boundsumIndices]
\bigl((\eta-\ell_1)\wedge 0\bigr)+\ldots+\bigl((\eta-\ell_n)\wedge 0\bigr) \ge (\eta-\ell)\wedge 0\;.
\end{equ}
Indeed, if all of the terms on the left vanish, then the bound holds trivially. Otherwise,
at least one term is given by $\eta - \ell_i$ and, for all the other terms, we use the fact that
$(\eta - \ell_j)\wedge 0 \ge -\ell_j$.
Since $\|x\|_P \le 1$, it follows at once that 
\begin{equ}
\|b(x)\|_\ell \ls \|x\|_P^{(\eta-\ell)\wedge 0}\;,
\end{equ}
as required.

In order to bound $\Gamma_{xy}b(y) - b(x)$, we proceed exactly as in the proof of 
Theorem~\ref{theo:smooth}. All we need to show is that the various remainder terms
appearing in that proof satisfy bounds of the type
\begin{equ}[e:wantedBoundRi]
\|R_i(x,y)\|_\ell \ls \|x-y\|_\s^{\gamma-\ell} \|x,y\|_P^{\eta-\gamma}\;.
\end{equ}
Regarding the term $R_1(x,y)$, it follows from a calculation similar to \eref{e:termDiffxy}
that it consists of terms proportional to
\begin{equ}
\Gamma_{xy} \CQ_{\ell_1} \tilde f(y)\star\ldots\star \Gamma_{xy} \CQ_{\ell_n} \tilde f(y)\;,
\end{equ}
where $\sum \ell_i \ge \gamma$. Combining the bounds on $\Gamma$ with the definition
of the space $\CD^{\gamma,\eta}_P$, we know furthermore that 
each of these factors satisfies a bound of the type
\begin{equ}[e:boundGxyf]
\|\Gamma_{xy} \CQ_{\ell_i} \tilde f(y)\|_m \lesssim \|x-y\|_\s^{\ell_i-m} \|x\|_P^{(\eta-\ell_i)\wedge 0}\;.
\end{equ} 
Combining this with the fact that  $\sum \ell_i \ge \gamma$, that $\|x-y\|_\s \ls \|x\|_P$,
and the bound \eref{e:boundsumIndices}, the bound \eref{e:wantedBoundRi} follows for $R_1$.

Regarding $R_f$, it follows from the definitions that
\begin{equ}[e:bRf]
\|R_f(x,y)\|_m \lesssim \|x-y\|_\s^{\gamma-m} \|x,y\|_P^{\eta - \gamma}\;.
\end{equ}
Furthermore, as a consequence of the fact that $\eta \ge 0$ and $\|x-y\|_\s \ls \|x\|_P$,
it follows from \eref{e:boundGxyf}  and \eref{e:bRf} that 
\begin{equ}
\|\Gamma_{xy} \tilde f(y)\|_m \lesssim \|x-y\|_\s^{-m}\;,\qquad
\|\tilde f(y) + (\bar f(y) - \bar f(x))\one\|_m \ls \|x-y\|_\s^{-m}\;.
\end{equ}
Combining this with \eref{e:bRf} and the expression for $R_2$, we immediately conclude that
$R_2$ also satisfies \eref{e:wantedBoundRi}.

Note now that one has the bound
\begin{equs}
|\bar f(x) - \bar f(y)| &\ls \|\Gamma_{xy}\tilde f(y)\|_0 + \|x-y\|_\s^\gamma\, \|x,y\|_P^{\eta-\gamma}\label{e:boundfbardiff}\\
&\ls \sum_{\zeta \le \ell < \gamma} \|x-y\|_\s^\ell \,\|x,y\|_P^{(\eta-\ell)\wedge 0}
\ls \|x-y\|_\s^\zeta\,  \|x,y\|_P^{(\eta - \zeta)\wedge 0}\;,
\end{equs}
where we used the fact that $\zeta \le \gamma$.
Since we furthermore know that $\bar f(x)$ is uniformly bounded in $\K$ as a consequence of the
fact that $\eta \ge 0$, it follows that the bound equivalent to \eref{e:TaylorF} in this context is
given by
\begin{equ}[e:boundDkF]
D^kF(\bar f(x)) = \sum_{|k+\ell| \le L} {D^{k+\ell}F(\bar f(y)) \over \ell !} \bigl(\bar f(x) - \bar f(y)\bigr)^\ell + \CO \bigl(\|x-y\|_\s^{\gamma-|k|\zeta} \|x,y\|_P^{\mu_k}\bigr)\;,
\end{equ}
where $L = \lfloor \gamma/\zeta \rfloor$ and the exponent $\mu_k$ is given by $\mu_k = \bigl(|k|\zeta - \gamma - |k|\eta + (\gamma\eta / \zeta)\bigr)\wedge 0$.
We can furthermore assume without loss of generality that $\zeta \le 1$.
Furthermore, making use of \eref{e:boundfbardiff}, it follows as in \eref{e:boundProdf} that
\begin{equs}
\bigl\| \bigl(\tilde f(y) + (\bar f(y) - \bar f(x))\one\bigr)^{\star k}\bigr\|_\beta &\lesssim \sum_{m\ge 0}\sum_{\ell}
\|x-y\|_\s^{\reg (|k| - m)} \|x,y\|_P^{(|k|-m)((\eta-\zeta)\wedge 0)} \\
&\qquad \times \|x,y\|_P^{(\eta-\ell_1)\wedge 0}\cdots \|x,y\|_P^{(\eta-\ell_m)\wedge 0} \;,
\end{equs}
where the second sum runs over all indices $\ell_1,\ldots,\ell_m$ with $\sum \ell_i = \beta$
and $\ell_i \ge \zeta$ for every $i$. In particular, one has the bound
\begin{equs}
\bigl\| \bigl(\tilde f(y) &+ (\bar f(y) - \bar f(x))\one\bigr)^{\star k}\bigr\|_\beta \lesssim \|x-y\|_\s^{\reg |k| - \beta} \|x,y\|_P^{\beta - \reg m} \\
&\quad \times\sum_{m\ge 0}\sum_{\ell} \|x,y\|_P^{(|k|-m)((\eta-\zeta)\wedge 0)} 
 \|x,y\|_P^{(\eta-\ell_1)\wedge 0}\cdots \|x,y\|_P^{(\eta-\ell_m)\wedge 0} \;.
\end{equs}
Let us have a closer look at the exponents of $\|x,y\|_P$ appearing in this expression:
\begin{equ}
\mu_{m,\ell} \eqdef \beta - \zeta m + (|k|-m)((\eta-\zeta)\wedge 0) + \sum_{i=1}^m (\eta-\ell_i)\wedge 0\;.
\end{equ}
Note that, thanks to the distributivity of the infimum with respect to addition
and to the facts that $\sum \ell_i = \beta$ and $\ell_i \ge \zeta$, one has the bound
\begin{equ}
\sum_{i=1}^m (\eta-\ell_i)\wedge 0 \ge \inf_{n\le m} \bigl(n\eta - \beta + (m-n)\zeta\bigr)
= m\zeta - \beta + \inf_{n \le m} n(\eta - \zeta)\;.
\end{equ}
As a consequence, we have $\mu_{m,\ell} \ge 0$ if $\eta \ge \zeta$ and 
$\mu_{m,\ell} \ge |k|(\eta - \zeta)$ otherwise, so that
\begin{equ}
\bigl\| \bigl(\tilde f(y) + (\bar f(y) - \bar f(x))\one\bigr)^{\star k}\bigr\|_\beta \lesssim \|x-y\|_\s^{\reg |k| - \beta} \|x,y\|_P^{|k|(\eta-\zeta) \wedge 0}\;.
\end{equ}
Note furthermore that, by an argument similar to above, one has the bound
\begin{equ}
\mu_k + |k|(\eta-\zeta) \wedge 0 \ge (\eta-\zeta){\gamma \over \zeta}\wedge 0 \ge (\eta - \gamma)\wedge 0 = \eta - \gamma\;,
\end{equ}
where we used the fact that $\zeta \le \gamma$ and 
the last identity follows from the assumption that $\eta \le \gamma$.
Combining this with \eref{e:boundDkF} and the definition of $R_3$ from \eref{e:lastbx}, we 
obtain the bound \eref{e:wantedBoundRi} for $R_3$, which implies that $\hat F(f) \in \CD_P^{\gamma,\eta}$
as required.

The proof of the local Lipschitz continuity then follows in exactly the same way as in
the proof of Theorem~\ref{theo:smooth}.
\end{proof}

\subsection{Differentiation}

In the same context as Section~\ref{sec:diff}, one has the following result:

\begin{proposition}\label{prop:diffSing}
Let $\DD$ be an abstract gradient as in Section~\ref{sec:diff} and let 
$f \in \CD^{\gamma,\eta}_P(V)$ for some $\gamma > \s_i$ and $\eta \in \R$. 
Then, $\DD_i f \in \CD^{\gamma-\s_i, \eta-\s_i}_P$. 
\end{proposition}

\begin{proof}
This is an immediate consequence of the definition \eref{e:boundDgammaSing} and the properties 
of abstract gradients.
\end{proof}

\subsection{Integration against singular kernels}

In this section, we extend the results from Section~\ref{sec:integral} to spaces of singular
modelled distributions. Our main result can be stated as follows.

\begin{proposition}\label{prop:intSing}
Let $\TT$, $V$, $K$ and $\beta$ be as in Theorem~\ref{theo:Int} and let $f \in \CD^{\gamma,\eta}_P(V)$
with $\eta < \gamma$. Denote furthermore by $\alpha$ the regularity of the sector $V$ and assume
that $\eta \wedge \alpha > -\m$. 
Then, provided that $\gamma + \beta \not \in \N$ and $\eta + \beta \not \in \N$,
one has $\CK_\gamma f \in \CD^{\bar \gamma,\bar \eta}_P$
with 
$\bar \gamma = \gamma + \beta$ and $\bar \eta = (\eta\wedge\alpha)+\beta$.

Furthermore, in the situation analogous to that of the last part of Theorem~\ref{theo:Int},
one has the bound
\begin{equ}[e:diffIntSing]
\$\CK_\gamma f ; \bar \CK_\gamma \bar f\$_{\bar\gamma,\bar \eta;\K} \lesssim 
\$f ; \bar f\$_{\gamma,\eta;\bar \K} + \|\Pi - \bar \Pi\|_{\gamma;\bar \K}
+ \|\Gamma - \bar \Gamma\|_{\bar \gamma;\bar \K}\;,
\end{equ}
for all $f \in \CD^{\gamma,\eta}_P(V;\Gamma)$ and $\bar f \in \CD^{\gamma,\eta}_P(V;\bar \Gamma)$.
\end{proposition}

\begin{proof}
We first observe that $\CN_\gamma f$ is well-defined for a singular
modelled distribution as in the statement. Indeed, for every $x \not \in P$,
it suffices to decompose $K$ as $K = K^{(1)} + K^{(2)}$, where $K^{(1)}$
is given by $K^{(1)} = \sum_{n \ge n_0} K_n$, and $n_0$ is sufficiently large so 
that $2^{-n_0} \le d_\s(x,P)/2$, say. Then, the fact that \eref{e:defIbar} is well-posed
with $K$ replaced by $K^{(1)}$ follows from Theorem~\ref{theo:Int}.
The fact that it is well-posed with $K$ replaced by $K^{(2)}$ follows from the fact that
$K^{(2)}$ is globally smooth and compactly supported, combined with Proposition~\ref{prop:recSing}.

To prove that $\CK_\gamma f$ belongs to $\CD^{\gamma+\beta, (\eta\wedge \alpha)+\beta}_P(V)$, 
we proceed as in the proof of Theorem~\ref{theo:Int}. We first consider values of $\ell$ with 
$\ell \not \in \N$.
For such values, one has as before
$\CQ_\ell \bigl(\CK_\gamma f\bigr)(x) = \CQ_\ell \CI f(x)$
and $\CQ_\ell \Gamma_{xy}\bigl(\CK_\gamma f\bigr)(y) = \CQ_\ell \CI \Gamma_{xy}f(y)$, so that the
required bounds on $\|\CK_\gamma f(x)\|_\ell$, $\|\CK_\gamma f(x) - \Gamma_{xy} \CK_\gamma f(x)\|_\ell$, 
$\|\CK_\gamma f(x) - \bar \CK_\gamma \bar f(x)\|_\ell$, 
as well as $\|\CK_\gamma f(x) - \Gamma_{xy} \CK_\gamma f(x) - \bar \CK_\gamma \bar f(x) + \bar \Gamma_{xy} \bar \CK_\gamma \bar f(x)\|_\ell$
follow at once. (Here and below we use the fact that $\|x,y\|_P^{\eta - \gamma} \le \|x,y\|_P^{(\eta\wedge \alpha) - \gamma}$ since one only considers pairs $(x,y)$
such that $\rho_\s \le 1$.)

It remains to treat the integer values of $\ell$. First, we want to show that one has the bound
\begin{equ}
\|\CK_\gamma f(x)\|_\ell \lesssim \|x\|_P^{(\bar \eta - \ell)\wedge 0}\;,
\end{equ}
and similarly for $\|\CK_\gamma f-\bar \CK_\gamma \bar f\|_\ell$.
For this, we proceed similarly to Theorem~\ref{theo:Int}, noting that 
if $2^{-(n+1)} \le \|x\|_P$ then, by Remark~\ref{rem:boundReconstr}, one has the bound
\begin{equ}
\bigl|\bigl(\CR f - \Pi_x f(x)\bigr) \bigl(D_1^\ell K_n(x,\cdot)\bigr)\bigr| \lesssim 2^{(|\ell|_\s - \beta - \gamma)n} \|x\|_P^{\eta - \gamma}\;.
\end{equ}
(In this expression, $\ell$ is a multiindex.) Furthermore, regarding $\CJ^{(n)}(x) f(x)$, one has
\begin{equ}
\|\CJ^{(n)}(x) f(x)\|_\ell \lesssim \sum_{\zeta > \ell-\beta} \|x\|_P^{(\eta-\zeta)\wedge 0} 2^{(\ell - \beta - \zeta)n}\;.
\end{equ}
Combining these two bounds and summing over the relevant values of $n$ yields
\begin{equ}
\sum_{2^{-(n+1)} \le \|x\|_P} \|\CK_\gamma^{(n)} f\|_\ell \lesssim \sum_{\zeta > \ell-\beta} \|x\|_P^{\zeta + \beta - \ell + ((\eta-\zeta)\wedge 0)}\;,
\end{equ}
which is indeed bounded by $\|x\|_P^{(\bar \eta - \ell)\wedge 0}$ as required since
one always has $\zeta \ge \alpha$.
For $\|x\|_P < 2^{-(n+1)}$ on the other hand, we make use of the reconstruction theorem
for modelled distributions which yields 
\begin{equs}
\bigl|\bigl(\CR f - \Pi_x f(x)\bigr) &\bigl(D_1^\ell K_n(x,\cdot)\bigr) + \CQ_{|\ell|_\s} \CJ^{(n)}(x) f(x) \bigr| \\
&\lesssim 
\bigl|\bigl(\CR f \bigr) \bigl(D_1^\ell K_n(x,\cdot)\bigr)\bigr|
+ \sum_{\zeta \le |\ell|_\s - \beta} \bigl|\bigl(\Pi_x f(x)\bigr) \bigl(D_1^\ell K_n(x,\cdot)\bigr)\bigr| \\
&\lesssim 2^{(|\ell|_\s - \beta - (\eta \wedge \alpha))n} + \sum_{\zeta \le |\ell|_\s - \beta} 2^{(|\ell|_\s - \beta - \zeta)n} \|x\|_P^{(\eta - \zeta)\wedge 0}\;.
\end{equs}
Summing again over the relevant values of $n$ yields again
\begin{equ}
\sum_{2^{-(n+1)} > \|x\|_P} \|\CK_\gamma^{(n)} f\|_\ell \lesssim \sum_{\zeta \le \ell - \beta} \|x\|_P^{\zeta + \beta - \ell + ((\eta-\zeta)\wedge 0)}\;,
\end{equ}
which is bounded by $\|x\|_P^{(\bar \eta - \ell)\wedge 0}$ for the same reason as before.
The corresponding bounds on $\|\CK_\gamma f-\bar \CK_\gamma \bar f\|_\ell$ are obtained in virtually the
same way.

It therefore remains to obtain the bounds on $\|\CK_\gamma f(x) - \bar \CK_\gamma \bar f(x)\|_\ell$
and $\|\CK_\gamma f(x) - \Gamma_{xy} \CK_\gamma f(x) - \bar \CK_\gamma \bar f(x) + \bar \Gamma_{xy} \bar \CK_\gamma \bar f(x)\|_\ell$.
For this, we proceed exactly as in the proof
of Theorem~\ref{theo:Int}, but we keep track of the dependency on $x$ and $y$, rather than just the difference. 
Recall also that we only ever consider the case where $(x,y) \in \K_P$, so that $\|x,y\|_P > \|x-y\|_\s$.
This time, we consider separately the three cases 
$2^{-n} \le  \|x-y\|_\s$, $2^{-n} \in [\|x-y\|_\s, {1\over 2}\|x,y\|_P]$ and $2^{-n} \ge {1\over 2}\|x,y\|_P$.

When $2^{-n} \le \|x-y\|_\s$, 
we use Remark~\ref{rem:boundReconstr} which shows that,
when following the exact same considerations as in 
Theorem~\ref{theo:Int},
we always obtain the same bounds, but multiplied by
a factor $\|x,y\|_P^{\eta-\gamma}$. The case $2^{-n} \le \|x-y\|_\s$ therefore follows at once.

We now turn to the case $2^{-n} \in [\|x-y\|_\s, {1\over 2}\|x,y\|_P]$. As in the proof of Theorem~\ref{theo:Int} 
(see \eref{e:largeScaleIden} in particular), we can again reduce this case to obtaining the bounds
\begin{equs}
 \bigl|\bigl(\Pi_x \CQ_\zeta \bigl(\Gamma_{xy} f(y) - f(x)\bigr)\bigr)\bigl(D_1^k K_n(x,\cdot)\bigr)\bigl|&\lesssim \|x,y\|_P^{\eta-\gamma}\sum_{\delta > 0}\|x-y\|_\s^{\delta+\gamma+\beta - |k|_\s} 2^{\delta n}\\
\bigl|\bigl(\Pi_y f(y) - \CR f\bigr)\bigl(K_{n;xy}^{k,\gamma}\bigr)\bigr| &\lesssim \|x,y\|_P^{\eta-\gamma}\sum_{\delta > 0}\|x-y\|_\s^{\delta+\gamma+\beta - |k|_\s} 2^{\delta n} 
\;,
\end{equs}
for every $\zeta \le |k|_\s - \beta$ and where the sums over $\delta$ contain only finitely many terms.
The first line is obtained exactly as in the proof of Theorem~\ref{theo:Int}, so we focus on the second
line.
Following the same strategy as in the proof of Theorem~\ref{theo:Int}, we
 similarly reduce it to obtaining bounds of the form 
\begin{equ}
\bigl|\bigl(\Pi_y f(y) - \CR f\bigr)\bigl(D_1^{\ell} K_n(\bar y, \cdot)\bigr)\bigr|\lesssim
\|x,y\|_P^{\eta-\gamma} \sum_{\delta > 0}\|x-y\|_\s^{\delta+\gamma+\beta - |\ell|_\s} 2^{\delta n} \;,
\end{equ}
where $\bar y$ is such that $\|x - \bar y\|_\s \le \|x-y\|_\s$ and $\ell$ is a multiindex with 
$|\ell|_\s \ge |k|_\s +\bigl(0 \vee (\gamma +\beta)\bigr)$.
Since we only consider pairs $(x,y)$ such that $\|x-y\|_\s \le {1\over 2} \|x,y\|_P$,
one has $\|y,\bar y\|_P \sim \|x,y\|_P$. As a consequence, we obtain as in the 
proof of Theorem~\ref{theo:Int}
\begin{equ}
\bigl|\bigl(\Pi_{\bar y} \bigl(\Gamma_{\bar y y} f(y) - f(\bar y)\bigr)\bigr)\bigl(D_1^{\ell}K_n(\bar y,\cdot)\bigr)\bigr|
\lesssim \|x,y\|_P^{\eta-\gamma} \sum_{\zeta \le \gamma} \|x-y\|_\s^{\gamma-\zeta} 2^{(|\ell|_\s-\zeta-\beta)n}\;.
\end{equ}
Furthermore, since $2^{-n} \le \|x,y\|_P$,
we obtain as in \eref{e:boundRfPi2} the bound
\begin{equ}
\bigl|\bigl(\Pi_{\bar y} f(\bar y) - \CR f\bigr)\bigl(D_1^{\ell} K_n(\bar y, \cdot)\bigr)\bigr|
\lesssim 2^{(|\ell|_\s -\beta-\gamma)n} \|x,y\|_P^{\eta-\gamma}\;.
\end{equ}
The rest of the argument is then again exactly the same as for Theorem~\ref{theo:Int}.
The corresponding bounds on the distance between $\CK_\gamma f$ and $\bar \CK_\gamma \bar f$
follows analogously.

It remains to consider the case $2^{-n} \ge {1\over 2}\|x,y\|_P$. In this case, we 
proceed as before but, in order to bound the term involving $\Pi_y f(y) - \CR f$, we 
simply use the triangle inequality to rewrite it as
\begin{equ}
\bigl|\bigl(\Pi_y f(y) - \CR f\bigr)\bigl(K_{n;xy}^{k,\gamma}\bigr)\bigr|
\le \bigl|\bigl(\Pi_y f(y)\bigr)\bigl(K_{n;xy}^{k,\gamma}\bigr)\bigr|
+ \bigl|\bigl(\CR f\bigr)\bigl(K_{n;xy}^{k,\gamma}\bigr)\bigr|\;.
\end{equ}
We then use again the representation \eref{e:reprKalphan} for $K_{n;xy}^{k,\gamma}$,
together with the bounds
\begin{equs}
\bigl|\bigl(\CR f\bigr)\bigl(D_1^{k+\ell} K_n(\bar y,\cdot)\bigr)\bigr| &\lesssim
2^{(|k+\ell|_\s - \beta - (\alpha \wedge \eta))n} \;,\\
\bigl|\bigl(\Pi_y f(y)\bigr)\bigl(D_1^{k+\ell} K_n(\bar y,\cdot)\bigr)\bigr| &\lesssim
\sum_{\alpha \le \zeta < \gamma} \|y\|_P^{(\eta - \zeta)\wedge 0} 2^{(|k+\ell|_\s - \beta - \zeta)n}\;.
\end{equs}
Here, the first bound is a consequence of the reconstruction theorem for singular
modelled distributions, while the second bound follows from Definition~\ref{def:singDist}.
Since 
\begin{equ}
2^{(|k+\ell|_\s - \beta - (\alpha \wedge \eta))n} \le 2^{(|k+\ell|_\s - \beta - \alpha)n} + 2^{(|k+\ell|_\s - \beta - \eta)n}\;,
\end{equ}
and since $\eta \in [\alpha,\gamma)$ by assumption, we see that the first bound is actually of the same
form as the second, so that
\begin{equ}
\bigl|\bigl(\Pi_y f(y) - \CR f\bigr)\bigl(D_1^{k+\ell} K_n(\bar y,\cdot)\bigr)\bigr|
\lesssim \sum_{\alpha \le \zeta < \gamma} \|y\|_P^{(\eta - \zeta)\wedge 0} 2^{(|k+\ell|_\s - \beta - \zeta)n}\;,
\end{equ}
where the sum runs over finitely many terms. 
Performing the integration in \eref{e:reprKalphan} and using the bound \eref{e:boundMassQl}, we 
conclude that 
\begin{equ}
\bigl|\bigl(\Pi_y f(y) - \CR f\bigr)\bigl(K_{n;xy}^{k,\gamma}\bigr)\bigr|
\lesssim \sum_{\zeta ; \ell} \|x-y\|_\s^{|\ell|_\s} \|x,y\|_P^{(\eta - \zeta)\wedge 0} 2^{(|k+\ell|_\s - \beta - \zeta)n}\;,
\end{equ}
where we used the fact that $\|y\|_P \sim \|x,y\|_P$.
Here, the sum runs over exponents $\zeta$ as before and multiindices 
$\ell$ such that $|k+\ell|_\s > \beta + \gamma$. Summing this expression over the relevant range
of values for $n$, we have
\begin{equs}
\sum_{2^{-n} \ge \|x,y\|_P}
\bigl|\bigl(\Pi_y f(y) - \CR f\bigr)\bigl(K_{n;xy}^{k,\gamma}\bigr)\bigr|
&\lesssim \sum_{\zeta ; \ell} \|x-y\|_\s^{|\ell|_\s} \|x,y\|_P^{(\eta \wedge \zeta)+\beta - |k+\ell|_\s}\\
&\lesssim \|x-y\|_\s^{\gamma+\beta - |k|_\s} \|x,y\|_P^{(\eta \wedge \alpha)-\gamma}\;,
\end{equs}
where we used the fact that $\|x-y\|_\s \le {1\over 2} \|x,y\|_P$ to obtain the second
bound. 
Again, the corresponding bounds on the distance between $\CK_\gamma f$ and $\bar \CK_\gamma \bar f$
follow analogously, thus concluding the proof.
\end{proof}

\begin{remark}\label{rem:senseSingDist}
The condition $\alpha \wedge \eta > -\m$ is only required in order to be able to apply
Proposition~\ref{prop:recSing}. There are some situations in which, even though 
$\alpha \wedge \eta  < -\m$, there exists a canonical element $\CR f \in \CC^{\alpha \wedge \eta}_\s$
extending $\tilde \CR f$. In such a case, Proposition~\ref{prop:intSing} still holds
and the bound \eref{e:diffIntSing} holds provided that the corresponding bound holds for
$\CR f - \bar \CR \bar f$.
\end{remark}

\section{Solutions to semilinear (S)PDEs}
\label{sec:solSPDE}

In order to solve a typical semilinear PDE of the type
\begin{equ}
\d_t u = Au + F(u)\;,\qquad u(0) = u_0\;,
\end{equ}
a standard methodology is to rewrite it in its mild form as
\begin{equ}
u(t) = S(t) u_0 + \int_0^t S(t-s)F(u(s))\,ds\;,
\end{equ}
where $S(t) = e^{A t}$ is the semigroup generated by $A$. One then looks for
some family of spaces $\CX_T$ of space-time functions (with $\CX_T$ containing 
functions up to time $T$) such that the map given by
\begin{equ}
\bigl(\CM u\bigr)(t) = S(t) u_0 + \int_0^t S(t-s)F(u(s))\,ds\;,
\end{equ}
is a contraction in $\CX_T$, provided that the terminal time $T$ is sufficiently
small. (As soon as $F$ is nonlinear, the notion of ``sufficiently small'' typically 
depends on the choice of $u_0$, thus leading to a local solution theory.) 
The main step of such an argument is to show 
that the linear map $S$ given by
\begin{equ}
\bigl(S v\bigr)(t) = \int_0^t S(t-s)\, v(s)\,ds\;,
\end{equ}
can be made to have arbitrarily small norm as $T \to 0$ as a map from some suitable space $\CY_T$
into $\CX_T$, where $\CY_T$ is chosen such that $F$ is then 
locally Lipschitz continuous as a map from $\CX_T$ to $\CY_T$, with some uniformity in $T \in (0,1]$, 
say.

The aim of this section is to show that, in many cases, this methodology can still be applied
when looking for solutions in $\CD^{\gamma,\eta}_P$ for suitable exponents $\gamma$ and $\eta$,
and for suitable regularity structures allowing to formulate a fixed point map 
of the type of $\CM_F$.
At this stage, all of our arguments are purely deterministic. However, they rely on a choice
of model for the given regularity structure one works with, which in many interesting cases
can be built using probabilistic techniques.

\subsection{Short-time behaviour of convolution kernels}
\label{sec:applSPDE}

From now on, we assume that we work with $d-1$ spatial coordinates, so that the solution $u$
we are looking for is a function on $\R^d$. (Or rather a subset of it.) 
In order to be able to reuse the results of Section~\ref{sec:integral}, we also assume that 
$S(t)$ is given by an integral operator with kernel $G(t,\cdot)$.
For simplicity, assume that the scaling $\s$
and exponent $\beta$ are such that, as a space-time function, 
$G$ furthermore satisfies the assumptions of Section~\ref{sec:integral}.
(Typically, one would actually write $G = K+R$, where $R$ is smooth and a $K$ satisfies the 
assumptions of Section~\ref{sec:integral}. We will go into more details in Section~\ref{sec:SPDE} below.)
In this section, time plays a distinguished role. We will therefore denote points in $\R^d$ either
by $(t,x)$ with $t \in \R$ and $x\in \R^{d-1}$ or by $z \in \R^d$, depending on the context.

In our setting, we have so far been working solely with modelled distributions 
defined on all of $\R^d$, so it not clear a priori how a map like $S$ should be defined
when acting on (possibly singular) modelled distributions. One natural way of reformulating it
is by writing
\begin{equ}[e:defST]
S v = G * (\PR^{+} v)\;,
\end{equ}
where $\PR^+ \colon \R \times \R^{d-1} \to \R$ is given by $\PR^+(t,x) = 1$ for $t > 0$ and
$\PR^+(t,x) = 0$ otherwise. 

From now on, we always take $P \subset \R^d$ to be the hyperplane defined by ``time $0$'',
namely $P = \{(t,x)\;:\; t = 0\}$, which has effective codimension $\m = \s_1$. 
We then note that the obvious interpretation of $\PR^+$ as
a modelled distribution yields an element of $\CD^{\infty,\infty}_P$, whatever the 
details of the underlying regularity structure. Indeed, the second term in \eref{def:singDist} always
vanishes identically, while the first term is non-zero only for $\ell = 0$, in which case
it is bounded for every choice of $\eta$. It then follows immediately from
Proposition~\ref{prop:multSing} that the map $v \mapsto \PR^+\, v$ is always 
bounded as a map from $\CD^{\gamma,\eta}_P$ into $\CD^{\gamma,\eta}_P$.
Furthermore, this map does not even rely on a choice of product, since $\PR^+$ is
proportional to $\one$, which is always neutral for any product.


In order to avoid the problem of having to control the behaviour of functions at infinity,
we will from now on assume that we have a symmetry group $\SS$ acting on $\R^d$ in such a way that
\begin{claim}
\item The time variable is left unchanged in the sense that there is an action $\tilde T$ of $\SS$
on $\R^{d-1}$ such that $T_g(t,x) = (t, \tilde T_g x)$.
\item The fundamental domain $\K$ of the action $\tilde T$ is compact in $\R^{d-1}$.
\end{claim}
We furthermore assume that $\SS$ acts on our regularity structure $\TT$ 
and that the model $(\Pi,\Gamma)$ for $\TT$ is adapted to its
action. All the modelled distributions considered in the remainder of this section will
always be assumed to be symmetric, and when we write $\CD^\gamma$, $\CD^{\gamma,\eta}_P$, etc,
we always refer to the closed subspaces consisting of symmetric functions. 

One final ingredient used in this section will be that the kernels arising in the context of semilinear
PDEs are \textit{non-anticipative} in the sense that
\begin{equ}
t < s \quad\Rightarrow\quad K\bigl((t,x), (s,y)\bigr) = 0\;.
\end{equ}
We furthermore use the notations $O = [-1,2]\times \R^{d-1}$
and $O_T = (-\infty,T]\times \R^{d-1}$. Finally, we will use the shorthands
$\$\cdot\$_{\gamma,\eta;T}$ as a shorthand for $\$\cdot\$_{\gamma,\eta;O_T}$,
and similarly for $\$\cdot\$_{\gamma;T}$. The backbone of our argument
is then provided by Proposition~\ref{prop:extension} which guarantees that one can give bounds
on $\CK_\gamma f$ on $O_T$, solely in terms of the behaviour of $f$ on $O_T$.

With all of these preliminaries in place, the main result of this subsection is the following.

\begin{theorem}\label{theo:boundKT}
Let $\gamma > 0$
and let $K$ be a non-anticipative kernel satisfying Assumptions~\ref{def:regK} and \ref{ass:polynom}
for some $\beta > 0$ and $r > \gamma + \beta$.
Assume furthermore that the regularity structure $\TT$ comes with an integration
map $\CI$ of order $\beta$ acting on some sector $V$ of regularity $\alpha > -\s_1$
and assume that the models $Z=(\Pi,\Gamma)$ and $\bar Z=(\bar \Pi,\bar \Gamma)$
both realise $K$ for $\CI$ on $V$.
Then, there exists a constant $C$ such that, for every $T \in (0,1]$, the bounds
\begin{equs}
\$\CK_\gamma \PR^+ f\$_{\gamma+\beta,\bar \eta;T} &\le C T^{\kappa/\s_1} \$f\$_{\gamma,\eta;T}\;, \\
\$\CK_\gamma \PR^+ f; \bar \CK_\gamma \PR^+ \bar f\$_{\gamma+\beta,\bar \eta;T} &\le C T^{\kappa/\s_1} \bigl( \$f;\bar f\$_{\gamma,\eta;T} + \$Z,\bar Z\$_{\gamma;O}\bigr)\;,
\end{equs}
hold, provided that $f\in \CD^{\gamma,\eta}_P(V;\Gamma)$ and $\bar f\in \CD^{\gamma,\eta}_P(V;\bar \Gamma)$
 for some $\eta > -\s_1$. Here, $\bar \eta$ and $\kappa$ 
are such that 
$\bar \eta = (\eta \wedge \alpha) + \beta - \kappa$ and $\kappa > 0$.

In the first bound, the proportionality constant depends only
on $\$Z\$_{\gamma;O}$, while in the second bound it is also allowed to depend on $\$f\$_{\gamma,\eta;T} + \$\bar f\$_{\gamma,\eta;T}$.
\end{theorem}

One of main ingredients of the proof is the fact that $\bigl(\CK_\gamma \PR_+ f\bigr)(t,x)$ is 
well-defined using only the knowledge of $f$ up to time $t$. This is a consequence of the following
result, which is an improved version of Lemma~\ref{lem:restriction}.

\begin{proposition}\label{prop:restriction}
In the setting of Lemma~\ref{lem:restriction}, and assuming that $\phi(0) \neq 0$, one has the improved bound
\begin{equ}[e:boundSupp]
\bigl|\bigl(\CR f - \Pi_x f(x)\bigr)(\psi_x^\lambda)\bigr| \lesssim \lambda^\gamma
\sup_{y,z\in \supp \psi_x^\lambda}\sup_{\ell < \gamma} {\|f(x) - \Gamma_{xy} f(y)\|_\ell \over \|x-y\|_\s^{\gamma-\ell}}\;,
\end{equ}
where the proportionality constant is as in Lemma~\ref{lem:restriction}.
\end{proposition}

\begin{proof}
Since the statement is linear in $f$, we can assume without loss of generality that the right hand side
of \eref{e:boundSupp} is equal to $1$.
Let $\phi$ be the scaling function of a wavelet basis of $\R^d$ and let $\phi_y^n$ be defined by
\begin{equ}
\phi_y^n(z) = \phi\bigl(\CS_\s^{2^{-n}}(z-x)\bigr)\;.
\end{equ}
Note that this is slightly different from the definition of the $\phi_y^{n,\s}$ in Section~\ref{sec:wavelets}!
The reason for this particular scaling is that it  ensures that $\sum_{y \in \Lambda_n^\s} \phi_y^n(z) = 1$.
Again, we have coefficients $a_k$ such that, similarly to \eref{e:decompPhin},
\begin{equ}
\phi_y^{n-1}(z) = \sum_{k \in \CK} a_k \phi_{y + 2^{-n}k}^{n}(z)\;,
\end{equ}
for some finite set $\CK \subset \Z^d$, 
and this time our normalisation ensures that $\sum_{k \in \CK} a_k = 1$.

For every $n \ge 0$, define 
\begin{equ}
\Lambda_n^\psi = \{y \in \Lambda_n^\s\,:\, \supp \phi_y^n \cap \supp \psi_x^\lambda \neq \emptyset\}\;,
\end{equ}
and, for any $y \in \Lambda_n^\psi$, we denote by $y|_n$ some point
in the intersection of these two supports. There then exists some constant $C$ depending only on
our choice of scaling function such that
$\|y - y|_n\|_\s \le C 2^{-n}$. Let now $R_n$ be defined by
\begin{equ}
R_n \eqdef \sum_{y \in \Lambda_{n}^\psi} \bigl(\CR f - \Pi_{y|_{n}}f(y|_{n})\bigr)(\psi_x^\lambda \phi_y^{n})\;,
\end{equ}
and let $n_0$ be the smallest value such that $2^{-n_0} \le \lambda$. 
It is then straightforward to see that one has
\begin{equ}[e:boundFirstRn]
\bigl|\bigl(\CR f - \Pi_x f(x)\bigr)(\psi_x^\lambda) - R_{n_0}\bigr|
= \Bigl|\sum_{y \in \Lambda_{n_0}^\psi}\bigl(\Pi_{x}f(x) - \Pi_{y|_{n}}f(y|_{n})\bigr)(\psi_x^\lambda \phi_y^{n})\Bigr| \lesssim \lambda^\gamma\;.
\end{equ}
Furthermore, using as in Section~\ref{sec:wavelets} the shortcut
$z = y + 2^{-n\s} k$,
one then has for every $n \ge 1$ the identity
\begin{equs}
R_{n-1} &= \sum_{y \in \Lambda_{n-1}^\psi} \sum_{k \in \CK} a_k \bigl(\CR f - \Pi_{y|_{n-1}}f(y|_{n-1})\bigr)(\psi_x^\lambda \phi_{z}^{n})\\
&= \sum_{y \in \Lambda_{n-1}^\psi} \sum_{k \in \CK} a_k \bigl(\CR f - \Pi_{z|_n}f(z|_n)\bigr)(\psi_x^\lambda \phi_{z}^{n})\\
&\qquad + \sum_{y \in \Lambda_{n-1}^\psi} \sum_{k \in \CK} a_k \bigl(\Pi_{z|_n}f(z|_n) - \Pi_{y|_{n-1}}f(y|_{n-1})\bigr)(\psi_x^\lambda \phi_{z}^{n})\\
&= R_n + \sum_{y \in \Lambda_{n-1}^\psi} \sum_{k \in \CK} a_k \bigl(\Pi_{z|_n}f(z|_n) - \Pi_{y|_{n-1}}f(y|_{n-1})\bigr)(\psi_x^\lambda \phi_{z}^{n})\;. \label{e:lastTerm}
\end{equs}
Note now that, in \eref{e:lastTerm}, one has
$\|z|_n - y|_{n-1}\|_\s \le \tilde C 2^{-n}$ for some constant $\tilde C$.
It furthermore follows from the scaling properties of our functions that 
if $n \ge n_0$ and $\tau \in T_\ell$ with $\|\tau\| = 1$, one has
\begin{equ}
\bigl|\bigl(\Pi_{y} \tau\bigr)(\psi_x^\lambda \phi_{z}^{n})\bigr| \lesssim \lambda^{-|\s|} 2^{-\ell n - |\s| n}\;,
\end{equ}
with a proportionality constant that is uniform over all $y$ and $z$ such that  
$\|y-z\|_\s \le \tilde C 2^{-n}$. As a consequence, each summand in the last term of
\eref{e:lastTerm}
is bounded by some fixed multiple of $\lambda^{-|\s|} 2^{-\gamma n - |\s| n}$.
Since furthermore the number of terms in this sum is bounded by a fixed multiple of
$(2^n \lambda)^{|\s|}$, this yields the bound
\begin{equ}[e:boundDiffRn]
\bigl|R_{n-1}- R_n\bigr| \lesssim 2^{-\gamma n}\;.
\end{equ}
Finally, writing $S_n(\psi)$ for the $2^{-n}$-fattening of the support of $\psi_x^\lambda$,
we see that, as a consequence of Lemma~\ref{lem:restriction} and using a similar argument to what we have just
used to bound $R_{n-1}- R_n$, one has
\begin{equ}
|R_n| \lesssim 2^{-\gamma n} \$f\$_{\gamma; S_n(\psi)}\;.
\end{equ}
This is the only time that we use information on $f$ (slightly) away from the support of $\psi_x^\lambda$.
This however is only used to conclude that $\lim_{n \to \infty}|R_n| = 0$, and no explicit bound on this
rate of convergence is required. Combining this with \eref{e:boundDiffRn} and 
\eref{e:boundFirstRn}, the stated bound follows. 
\end{proof}

\begin{proof}[of Theorem~\ref{theo:boundKT}]
First of all, we see that, as a consequence of Proposition~\ref{prop:restriction}, we can exploit the
fact that $K$ is non-anticipative to strengthen \eref{e:diffIntSing} to
\begin{equ}[e:betterBound]
\$\CK_\gamma f; \bar \CK_\gamma \bar f\$_{\bar \gamma, \bar \eta; T}
\lesssim \$f; \bar f\$_{\gamma, \eta; T} + \|\Pi - \bar \Pi\|_{\gamma; O} + \|\Gamma - \bar \Gamma\|_{\bar \gamma;O}\;,
\end{equ}
in the particular case where furthermore $f(t,x) = 0$ for $t < 0$ and similarly for $\bar f$.
Of course, a similar bound also holds for $\$\CK_\gamma f\$_{\bar \gamma, \bar \eta; T}$.

The main ingredient of the proof is the following remark. Since, provided that $\eta > -\s_1$, we know that 
$\CR \PR^+ f \in \CC_\s^{\alpha \wedge \eta}$ by Proposition~\ref{prop:recSing}, 
it follows that the quantity
\begin{equ}
z \mapsto \int_{\R^d} D_1^k K(z,\bar z) \bigl(\CR \PR^+ f\bigr)(\bar z)\,d\bar z\;, \qquad
\end{equ}
is continuous as soon as $|k|_\s < (\alpha \wedge \eta) + \beta$. 
Furthermore, since $K$ is non-anticipative and $\CR \PR^+ f \equiv 0$ for negative times,
this quantity vanishes there.

As a consequence, we can apply Lemma~\ref{lem:vanishPower} which
shows that the bound \eref{e:betterBound} can in this case be strengthened to
the additional bounds
\begin{equs}
\sup_{z \in O_T} \sup_{\ell < \gamma +\beta} {\|\CK_\gamma \PR^+ f(z)\|_\ell \over \|z\|_P^{(\eta\wedge\alpha) + \beta - \ell}}
&\lesssim \$f\$_{\gamma,\eta;T}\;, \\
\sup_{z \in O_T} \sup_{\ell < \gamma +\beta} {\|\CK_\gamma \PR^+ f(z) - \bar \CK_\gamma \PR^+ \bar f(z)\|_\ell \over \|z\|_P^{(\eta\wedge\alpha) + \beta - \ell}}
&\lesssim \$f;\bar f\$_{\gamma,\eta;T} + \$Z,\bar Z\$_{\gamma;O}\;.
\end{equs} 
Since, for every $z, \bar z \in O_T$, one has $\|z\|_P \le T^{1/\s_1}$
as well as $\|z,\bar z\|_P \le T^{1/\s_1}$, we can combine these bounds with the
definition of the norm $\$\cdot\$_{\gamma+\beta,\bar \eta;T}$ to show that one has
\begin{equ}
\$\CK_\gamma \PR^+ f\$_{\gamma+\beta,\bar \eta;T} \lesssim T^{\kappa/\s_1} \$f\$_{\gamma,\eta;T}\;,
\end{equ}
and similarly for $\$\CK_\gamma \PR^+ f;\bar \CK_\gamma \PR^+ \bar f\$_{\gamma+\beta,\bar \eta;T}$,
thus concluding the proof.
\end{proof}

In all the problems we consider in this article, the Green's function of the linear
part of the equation, i.e.\ the kernel of $\CL^{-1}$ where $\CL$ is as in \eref{e:mainProb},
can be split into a sum of two terms, one of which satisfies the assumptions of 
Section~\ref{sec:integral}
and the other one of which is smooth (see Lemma~\ref{lem:diffSing}). 
Given a smooth kernel 
$R\colon \R^d \times \R^d \to \R$ that is supported in $\{(z,\bar z)\,:\, \|z -\bar z\|_\s \le L\}$
for some $L > 0$, and a regularity structure $\TT$ containing 
$\TT_{\s,d}$ as usual, we can define an operator
$R_\gamma \colon \CC^\alpha_\s \to \CD^\gamma$ by
\begin{equ}[e:defRgamma]
\bigl(R_\gamma \xi\bigr)(z) = \sum_{|k|_\s < \gamma} {X^k \over k!}\int_{\R^d} D_1^k R(z,\bar z)\,\xi(\bar z)\, d\bar z\;.
\end{equ}
(As usual, this integral should really be interpreted as $\xi(D_1^k R(z,\cdot))$, but the
above notation is much more suggestive.) The fact that this is indeed an element of $\CD^\gamma$
is a consequence of the fact that $R$ is smooth in both variables, so that it follows
from Lemma~\ref{lem:Lipschitz}. The following result is now straightforward:

\begin{lemma}\label{lem:boundT}
Let $R$ be a smooth kernel and consider a symmetric situation as above. 
If furthermore $R$ is non-anticipative,  then the bounds
\begin{equs}
\$R_\gamma \CR \PR^+ f\$_{\gamma+\beta,\bar \eta;T} &\le C T \$f\$_{\gamma,\eta;T}\;,\\
\$R_\gamma \CR \PR^+ f; R_\gamma \bar\CR \PR^+ \bar f\$_{\gamma+\beta,\bar \eta;T} &\le C T \bigl(\$f;\bar f\$_{\gamma,\eta;T} 
+ \$Z;\bar Z\$_{\gamma,O}\bigr)\;,
\end{equs}
holds uniformly over all $T \le 1$.
\end{lemma}

\begin{proof}
Since $R$ is assumed to be non-anticipative, one has $\bigl(R_\gamma \CR \PR^+ f\bigr)(t,x) = 0$
for every $t \le 0$. Furthermore, the map $(t,x) \mapsto \bigl(R_\gamma \CR \PR^+ f\bigr)(t,x)$ is
smooth (in the classical sense of a map taking values in a finite-dimensional vector space!), so that
the claim follows at once. Actually, it would even be true with $T$ replaced by an 
arbitrarily large power of $T$ in the bound on the right hand side.
\end{proof}

\subsection{The effect of the initial condition}

One of the obvious features of PDEs is that they usually have some boundary data. 
In this article, we restrict ourselves to spatially periodic situations, but even
such equations have some boundary data in the form of their initial condition. 
When they are considered in their mild formulation, the initial condition enters
the solution to a semilinear PDE through a term of the form $S(t)u_0$ for some
function (or distribution) $u_0$ on $\R^{d-1}$ and $S$ the semigroup generated by the
linear evolution.

All of the equations mentioned in the introduction are nonlinear perturbations
of the heat equation. More generally, their linear part is of the form
\begin{equ}
\CL = \d_t - Q(\nabla_x)\;,
\end{equ}
where $Q$ is a polynomial of even degree which is homogeneous of degree $2q$
for some scaling $\bar \s$ on $\R^{d-1}$ and some integer $q>0$. (In our case, this would always be the
Euclidean scaling and one has $q = 1$.)
In this case, the operator $\CL$ itself has the property that
\begin{equ}[e:scalingCL]
\CL \CS^\delta_\s \phi = \delta^{2q} \CS^\delta_\s \CL \phi\;.
\end{equ}
where $\s$ is the scaling on $\R^d = \R \times \R^{d-1}$ given by $\s = (2q,\bar \s)$.
Denote by $G$ the Green's function $G$ of $\CL$ which is a distribution satisfying 
$\CL G = \delta_0$ in the distributional sense and $G(x,t) = 0$ for $t \le 0$.
Assuming that $\CL$ is such that these properties define $G$ uniquely (which is the 
case if $\CL$ is hypoelliptic), 
it follows from \eref{e:scalingCL} and the scaling properties of the 
Dirac distribution that $G$ has exact scaling property
\begin{equ}[e:scalingG]
 G(\CS_\s^\delta z) = \delta^{|\bar \s|} G(z)\;,
\end{equ} 
which is precisely of the form \eref{e:scalingK}
with $\beta = 2q$. Under well-understood assumptions on $Q$, $\CL$ is known
to be hypoelliptic \cite{MR0076151}, so that
its Green's function
$G$ is smooth. In this case, the following lemma applies.

\begin{lemma}\label{lem:propG}
If $G$ satisfies \eref{e:scalingG}, is non-anticipative, and is smooth then
there exists a smooth 
function $\hat G\colon \R^d \to \R$ such that one has the identity 
\begin{equ}[e:defGhat]
G(x,t) = t^{-{|\bar \s| \over 2q}} \hat G(\CS_{\bar \s}^{t^{1/2q}} x)\;,
\end{equ}
and such that, for every $(d-1)$-dimensional multiindex $k$ and
every $n > 0$, there exists a constant $C$ such that the bound
\begin{equ}[e:boundGhat]
|D^k\hat G(y)| \le C (1+ |y|^2)^{-n}\;,
\end{equ}
holds uniformly over $y \in \R^{d-1}$.
\end{lemma}

\begin{proof}
The existence of $\hat G$ such that \eref{e:defGhat} holds follows immediately from 
the scaling property \eref{e:scalingG}.
The bound \eref{e:boundGhat} can be obtained by noting that, since $G$ is smooth
off the origin and satisfies $G(x,t) = 0$ for $t \le 0$, 
one has, for every $n>0$, a bound of the type
\begin{equ}[e:boundDG]
|D_x^k G(x,t)| \lesssim t^n\;,
\end{equ}
uniformly over all $x \in \R^{d-1}$ with $\|x\|_{\bar \s} = 1$.
It follows from \eref{e:defGhat} that
\begin{equ}
D^k G(x,t) = t^{-{|\bar\s| + |k|_{\bar \s}\over 2q}} \bigl(D^k \hat G\bigr)\bigl(\CS_{\bar \s}^{t^{1/2q}} x\bigr)\;.
\end{equ}
Setting $y = \CS_{\bar s}^{t^{1/2q}} x$ and noting that $\|y\|_{\bar \s} = 1/t^{1/2q}$ if $\|x\|_{\bar \s} = 1$,
it remains to combine this with \eref{e:boundDG} to obtain the required bound.
\end{proof}

Given a function (or distribution) $u_0$ on $\R^{d-1}$
with sufficiently nice behaviour at infinity, we now denote by $Gu_0$ its ``harmonic
extension'', given by
\begin{equ}[e:defPu0]
\bigl(G u_0\bigr)(x,t) = \int_{\R^{d-1}}G(x-y,t)\,u_0(y)\,dy\;.
\end{equ}
(Of course this is to be suitably interpreted when $u_0$ is a distribution.) 
This expression does define a function of $(t,x)$ which, thanks to Lemma~\ref{lem:propG}, 
is smooth everywhere except at $t = 0$. As in Section~\ref{sec:canonical}, 
we can lift $G u_0$ at every
point to an element of the model space $T$ (provided of course that $\TT_{d,\s} \subset \TT$ which
we always assume to be the case) by considering its truncated Taylor expansion.
We will from now on use this point of view without introducing a new notation.

We can say much more about the function $G u_0$, namely we can find out precisely to which
spaces $\CD^{\gamma,\eta}_P$ it belongs. This is the
content of the following Lemma, variants of which are commonplace in the PDE
literature. However, since our spaces are not completely standard and
since it is very easy to prove, we give a sketch of the proof here.

\begin{lemma}\label{lem:harmonicExt}
Let $u_0 \in \CC^{\alpha}_{\bar \s}(\R^{d-1})$ be periodic.
Then, for every $\alpha \not \in \N$, the function $v = G u_0$ defined in \eref{e:defPu0} belongs to $\CD^{\gamma,\alpha}_{P}$
for every $\gamma > (\alpha \vee 0)$.
\end{lemma}

\begin{proof}
We first aim to bound the various directional derivatives 
of $v$. In the case $\alpha < 0$, it follows immediately from the scaling 
and decay properties of $G$,
combined with the definition of $\CC^\alpha_\s$ that, for any fixed $(t,x)$, one has the bound
\begin{equ}
\bigl|\bigl(G u_0\bigr)(x,t)\bigr| \lesssim t^{\alpha \over 2}\;,
\end{equ}
valid uniformly over $x$ (by the periodicity of $u_0$) and over $t \in (0,1]$.
As a consequence (exploiting the fact that, as an operator, 
$G$ commutes with all spatial derivatives
and that one has the identity $\d_t Gu_0 = Q(\nabla_x) Gu_0$), one
also obtains the bound
\begin{equ}[e:boundDer]
\bigl|\bigl(D^k G u_0\bigr)(x,t)\bigr| \lesssim t^{{\alpha - |k|_\s \over 2}}\;,
\end{equ}
where $k$ is any $d$-dimensional multiindex (i.e.~we also admit time derivatives).

For $\alpha > 0$, we use the fact that elements in $\CC^\alpha_{\bar\s}$ can be characterised recursively
as those functions whose $k$th distributional derivatives 
belong to $\CC^{\alpha-|k|_{\bar \s}}_{\bar \s}$. It follows
that the bound \eref{e:boundDer} then still holds for $|k|_\s > \alpha$, while one has
$\bigl|\bigl(D^k G u_0\bigr)(x,t)\bigr| \lesssim 1$ for $|k|_\s < \alpha$.
This shows that the first bound in \eref{e:boundDgammaSing} does indeed hold 
for every integer value $\ell$ as required.

In this particular case, the second bound in \eref{e:boundDgammaSing}  
is then an immediate consequence of the first 
by making use of the generalised Taylor expansion from Proposition~\ref{prop:Taylor}. Since the
argument is very similar to the one already used for example in the proof of 
Lemma~\ref{lem:boundPibar}, we omit it here. 
\end{proof}

Starting from a Green's function $G$ as above, we would like to apply the theory developed
in Section~\ref{sec:integral}. 
From now on, we will assume that we are in the situation where we have a symmetry 
given by a discrete subgroup $\SS$ of the group of isometries of $\R^{d-1}$ 
with compact fundamental domain $\K$.
This covers the case of periodic boundary conditions, when $\SS$ is a subgroup of the
group of translations, but it also covers Neumann boundary conditions in the case where
$\SS$ is a reflection group.

\begin{remark}
One could even cover Dirichlet boundary conditions by reflection, but this would
require a slight modification of Definition~\ref{def:symmetric}.
In order to simplify the exposition, we refrain from doing so.
\end{remark}

To conclude this subsection, we show how, in the presence of a symmetry with compact fundamental
domain, a Green's function $G$ as above can
be decomposed in a way similar to Lemma~\ref{lem:diffSing}, but such that $R$
is also compactly supported.
We assume therefore that we are given a symmetry $\SS$ acting on $\R^{d-1}$ with compact fundamental
domain and that $G$ respects this symmetry in the sense that, for every $g \in \SS$
acting on $\R^{d-1}$ via an isometry $T_g\colon x \mapsto A_g x + b_g$, one has the identity
$G(t,x) = G(t, A_g x)$.
We then have the following result:

\begin{lemma}\label{lem:decompGreen}
Let $G$ and $\SS$ be as above. Then, there exist functions
$K$ and $R$ such that the identity
\begin{equ}[e:propertyGKR]
\bigl(G * u\bigr)(z) = \bigl(K * u\bigr)(z) + \bigl(R * u\bigr)(z)\;,
\end{equ}
holds for every symmetric function $u$ supported in $\R_+\times \R^{d-1}$
and every $z \in (-\infty,1]\times \R^{d-1}$.

Furthermore, $K$ is non-anticipative and symmetric, 
and satisfies Assumption~\ref{def:regK} with $\beta = 2q$,
as well as Assumption~\ref{ass:polynom} for some arbitrary (but fixed) value $r$.
The function $R$ is smooth, symmetric, non-anticipative, and compactly supported.
\end{lemma}

\begin{proof}
It follows immediately from Lemmas~\ref{lem:diffSing} and \ref{lem:decompSym} that one can write
\begin{equ}
G = K + \bar R\;,
\end{equ}
where $K$ has all the required properties, and $\bar R$ is smooth, non-anticipative, and symmetric.
Since $u$ is supported on positive times and we only consider \eref{e:propertyGKR}
for times $t \le 1$, we can replace $\bar R$ by any function $\tilde R$ which is supported
in $\{(t,x) \,:\, t \le 2\}$ say, and such that $\tilde R(t,x) = \bar R(t,x)$ for $t \le 1$.

It remains to replace $\tilde R$ by a kernel $R$ which is compactly supported. It is well-known \cite{MR1511623,MR1511704} that any crystallographic group $\SS$
can be written as the skew-product of a (finite) crystallographic point group $\GG$ 
with a lattice $\Gamma$ of
translations.
We then fix a function $\phi \colon \R^{d-1} \to [0,1]$ which is 
compactly supported in a ball of radius $C_\phi$ around the origin and such 
that $\sum_{k \in \Lambda} \phi(x+k) = 1$ for 
every $x$. Since elements in $\GG$ leave the lattice $\Lambda$ invariant, 
the same property holds true for the maps $x \mapsto \phi(Ax)$ for every $A \in \GG$.

It then suffices to set
\begin{equ}
R(t,x) = {1\over |\GG|} \sum_{A \in \GG} \sum_{k \in \Lambda} \tilde R(t, x + k) \phi(A x)\;.
\end{equ}
The fact that $R$ is compactly supported follows from the same property for $\phi$.
Furthermore, the above sum converges to a smooth function by Lemma~\ref{lem:propG}.
Also, using the fact that $u$ is invariant under translations by elements in $\Lambda$
by assumption, it is straightforward to verify that $\tilde R * u = R * u$ as required.
Finally, for any $A_0 \in \GG$, one has
\begin{equs}
R(t,A_0 x) &= {1\over |\GG|} \sum_{A \in \GG} \sum_{k \in \Lambda} \tilde R(t, A_0x + k) \phi(A A_0 x) \\
&= {1\over |\GG|} \sum_{A \in \GG} \sum_{k \in \Lambda} \tilde R(t, A_0 (x + k)) \phi(A x) \\
&= {1\over |\GG|} \sum_{A \in \GG} \sum_{k \in \Lambda} \tilde R(t, x + k) \phi(A x) = R(t,x)\;,
\end{equs}
so that $R$ is indeed symmetric for $\SS$.
Here, we first exploited the fact that elements of $\GG$ leave the lattice $\Lambda$ invariant,
and then used the symmetry of $\tilde R$.
\end{proof}

\subsection{A general fixed point map}
\label{sec:GenFP}

We have now collected all the ingredients necessary for the proof of the following result,
which can be viewed as one of the main abstract theorems of this article.
The setting for our result is the following. 
As before, we assume that we have a crystallographic group $\SS$ acting on $\R^{d-1}$.
We also write $\R^d = \R \times \R^{d-1}$, endow $\R^d$ with a scaling $\s$,
and extend the action of $\SS$ to $\R^{d}$ in the  obvious way.
Together with this data, we assume that we
are given a  non-anticipative kernel  $G\colon \R^d \setminus 0 \to \R$ that is smooth
away from the origin, preserves the symmetry $\SS$, and is scale-invariant with exponent
$\beta - |\s|$ for some fixed $\beta > 0$. 

Using Lemma~\ref{lem:decompGreen}, we then construct a singular kernel $K$ and a smooth
compactly supported function $R$ on $\R^d$ such that \eref{e:propertyGKR} holds for
symmetric functions $u$ that are supported on positive times.
Here, the kernel $K$ is assumed to be again non-anticipative and symmetric, and it is 
chosen in such a way that it annihilates all polynomials
of some arbitrary (but fixed) degree $r > 0$.
We then assume that we are given a regularity structure $\TT$ containing $\TT_{\s,d}$ 
such that $\SS$ acts on it, and which is endowed with an abstract integration map
$\CI$ of order $2q \in \N$. (The domain of $\CI$ will be specified later.)
We also assume that we have abstract differentiation maps $\DD_i$
which are covariant with respect to the symmetry $\SS$ as in Remark~\ref{rem:gradCov}.
We also denote by $\MM_\TT^r$ the set of 
all models for $\TT$ which realise $K$ on $T_r^-$.
As before, we denote by $\CK_\gamma$ the concrete integration map against $K$
acting on $\CD^\gamma$ and constructed in Section~\ref{sec:integral}, and by $R_\gamma$
the integration map against $R$ constructed in \eref{e:defRgamma}.

Finally, we denote by $P = \{(t,x) \in \R\times \R^{d-1} \,:\, t = 0\}$ the ``time $0$''
hyperplane and we consider the spaces $\CD_P^{\gamma,\eta}$ as in Section~\ref{sec:singular}.
Given $\gamma \ge \bar \gamma > 0$,
a map $F \colon \R^d \times T_\gamma \to T_{\bar \gamma}$, 
and a map $f \colon \R^d \to T_\gamma$, we denote by $F(f)$ the map 
given by
\begin{equ}[e:defFF]
\bigl(F(f)\bigr)(z) \eqdef F(z,f(z))\;.
\end{equ}
If it so happens that, via \eref{e:defFF}, $F$ maps $\CD^{\gamma,\eta}_P$ into
$\CD^{\bar \gamma, \bar \eta}_P$ for some $\eta, \bar \eta \in \R$, we say that $F$
is \textit{locally Lipschitz} if, for every compact set $\K \subset \R^d$ and every $R>0$,
there exists a constant $C>0$ such that the bound
\begin{equ}
\$F(f) - F(g)\$_{\bar \gamma, \bar \eta;\K} \le C \$f - g\$_{\gamma, \eta;\K}\;,
\end{equ}
holds for every $f,g \in \CD^{\gamma,\eta}_P$ with 
$\$f\$_{\gamma, \eta;\K} + \$g\$_{\gamma, \eta;\K} \le R$, as well as for all models
$Z$ with $\$Z\$_{\gamma;\K}\le R$. We also impose that the similar bound
\begin{equ}[e:boundNonSing]
\n{F(f) - F(g)}_{\bar \gamma, \bar \eta;\K} \le C \n{f - g}_{\gamma, \eta;\K}\;,
\end{equ}
holds.

We say that it is 
\textit{strongly locally Lipschitz} if furthermore
\begin{equ}
\$F(f) ; F(g)\$_{\bar \gamma, \bar \eta;\K} \le C\bigl( \$f ; g\$_{\gamma, \eta;\K} + \$Z-\bar Z\$_{\gamma;\bar \K}\bigr)\;,
\end{equ}
for any two models $Z, \bar Z$ with $\$Z\$_{\gamma;\bar \K} + \$\bar Z\$_{\gamma;\bar \K}\le R$, 
where this time $f \in \CD^{\gamma,\eta}_P(Z)$, $g \in \CD^{\gamma,\eta}_P(\bar Z)$, and $\bar \K$
denotes the $1$-fattening of $\K$.
Finally, given an open interval $I \subset \R$, we use the terminology
\begin{equ}
\text{``} \quad u = \CK_\gamma v \quad \text{on} \quad I \quad \text{''}
\end{equ}
to mean that the identity $u(t,x) = \bigl(\CK_\gamma v\bigr)(t,x)$ holds for every $t \in I$ and
$x \in \R^{d-1}$, and that for those values of $(t,x)$ the
quantity $\bigl(\CK_\gamma v\bigr)(t,x)$ only depends on the values $v(s,y)$
for $s \in I$ and $y \in \R^{d-1}$.

With all of this terminology in place, we then have the following general result.

\begin{theorem}\label{theo:fixedPointGen}
Let $V$ and $\bar V$ be two sectors of a regularity structure $\TT$
with respective regularities $\zeta, \bar \zeta \in \R$ with $\zeta \le \bar \zeta + 2q$.
In the situation described above, for some $\gamma \ge \bar \gamma > 0$ and some $\eta \in \R$, 
let $F \colon \R^d \times V_\gamma \to \bar V_{\bar \gamma}$ be a smooth function
such that, if $f \in \CD^{\gamma,\eta}_P$ is 
symmetric with respect to $\SS$, then $F(f)$, defined
by \eref{e:defFF}, belongs to $\CD^{\bar \gamma,\bar \eta}_P$
and is also symmetric with respect to $\SS$. Assume furthermore that we are
given an abstract integration map $\CI$ as above such that $\CQ_{\gamma}^- \CI \bar V_{\bar \gamma} \subset V_\gamma$.

If $\eta < (\bar \eta \wedge \bar \zeta) + 2q$,
 $\gamma < \bar \gamma + 2q$, $(\bar \eta \wedge \bar \zeta) > -2q$,
and $F$ is locally Lipschitz then,
for every $v \in \CD^{\gamma,\eta}_P$ which is symmetric with respect to $\SS$,
and for every symmetric model $Z=(\Pi,\Gamma)$ for the 
regularity structure $\TT$ such that $\CI$ is adapted to the kernel $K$,
there exists a time $T>0$ such that the equation
\begin{equ}[e:fixedPointGen]
u = (\CK_{\bar \gamma} + R_{\gamma}\CR) \PR^+ F(u) + v\;,
\end{equ}
admits a unique solution $u \in \CD^{\gamma,\eta}_P$ on $(0,T)$. The solution map
$\CS_T \colon (v, Z) \mapsto u$ is jointly continuous in a  
neighbourhood around $(v, Z)$ in the sense that, for every fixed $v$ and $Z$
as above, as well as any $\eps > 0$, there exists $\delta > 0$ such that,
denoting by $\bar u$ the solution to the fixed point map with data $\bar v$ and $\bar Z$,
one has the bound
\begin{equ}
\$u; \bar u\$_{\gamma,\eta;T} \le \eps\;, 
\end{equ}
provided that $\$Z;\bar Z\$_{\gamma;O} + \$v; \bar v\$_{\gamma,\eta;T} \le \delta$.

If furthermore $F$ is strongly locally Lipschitz 
then the map $(v, Z) \mapsto u$ is jointly Lipschitz continuous in a  
neighbourhood around $(v, Z)$ in the sense that $\delta$ can locally
be chosen proportionally to $\eps$ in the bound above.
\end{theorem}

\begin{proof}
We first consider the case of a fixed model $Z = (\Pi,\Gamma)$, so that the space
$\CD^{\gamma,\eta}_P$ (defined with respect to the given multiplicative map $\Gamma$) is a 
Banach space. In this case, 
denote by $\CM_F^Z(u)$ the right hand side of \eref{e:fixedPointGen}. 
Note that, even though $\CM_F^Z$ appears not to depend on $Z$ at first sight,
it does so through the definition of $\CK_{\bar \gamma}$.

It follows from Theorem~\ref{theo:boundKT} and Lemma~\ref{lem:boundT}, as well as our
assumptions on the exponents $\gamma$, $\bar \gamma$, $\eta$ and $\bar \eta$ that there
exists $\kappa > 0$ such that one has the bound
\begin{equ}
\$\CM_F^Z(u) - \CM_F^Z(\bar u)\$_{\gamma,\eta;T} \lesssim T^{\kappa}\$F(u) -F(\bar u)\$_{\bar \gamma, \bar \eta;T}\;.
\end{equ}
It follows from the local Lipschitz continuity of $F$ that, for every $R>0$, there exists 
a constant $C>0$ such that
\begin{equ}
\$\CM_F^Z(u) - \CM_F^Z(\bar u)\$_{\gamma,\eta;T} \le C T^{\kappa}\$u - \bar u\$_{\gamma,\eta;T}\;,
\end{equ}
uniformly over $T \in (0,1]$ and over all $u$ and $\bar u$ such that
$\$u\$_{\gamma,\eta;T} + \$\bar u\$_{\gamma,\eta;T} \le R$. Similarly, for every $R>0$,
there exists a constant $C>0$ such that one has the bound
\begin{equ}
\$\CM_F^Z(u)\$_{\gamma,\eta;T} \le C T^\kappa + \$v\$_{\gamma,\eta;T}\;.
\end{equ}
As a consequence, as soon as $\$v\$_{\gamma,\eta;T}$ is finite and provided that $T$ is 
small enough $\CM_F^Z$ maps the ball of radius $\$v\$_{\gamma,\eta;T} + 1$ in $\CD^{\gamma,\eta}_P$
into itself and is a contraction there, so that it admits a unique fixed point.
The fact that this is also the unique global fixed point for $\CM_F^Z$
follows from a simple continuity argument 
similar to the one given in the proof of Theorem~4.8 in \cite{KPZ}.

For a fixed model $Z$, the local Lipschitz continuity of the map $v \mapsto u$ for sufficiently small $T$
is immediate. Regarding the dependency on the model $Z$, we first consider the simpler case where
$F$ is assumed to be strongly Lipschitz continuous. In this case, the same
argument as above yields the bound
\begin{equ}
\$\CM_F^Z(u) ; \CM_F^{\bar Z}(\bar u)\$_{\gamma,\eta;T} \le C T^{\kappa}\bigl(\$u ; \bar u\$_{\gamma,\eta;T}
+ \$Z;\bar Z\$_{\gamma;O}\bigr)\;,
\end{equ}
so that the claim follows at once. 

It remains to show that the solution is also locally uniformly continuous as a function of the model
$Z$ in situations where $F$ is locally Lipschitz continuous, but not in the strong sense. 
Given a second model $\bar Z = (\bar \Pi, \bar \Gamma)$, we denote by
$\bar u$ the corresponding solution to \eref{e:fixedPointGen}. We assume that $\bar Z$
is sufficiently close to $Z$ so that both $\CM_F^{Z}$ and $\CM_F^{\bar Z}$ are strict contractions 
on the same ball.
We also use the shorthand notations
$u^{(n)} = \bigl(\CM_F^Z\bigr)^n(0)$ and $\bar u^{(n)} = \bigl(\CM_F^{\bar Z}\bigr) ^n(0)$.
Using the strict contraction property of the two fixed point maps, we have the bound
\begin{equs}
 \|u - \bar u\|_{\gamma,\eta;T} &\lesssim  \|u - u^{(n)}\|_{\gamma,\eta;T} +  \|u^{(n)} - \bar u^{(n)}\|_{\gamma,\eta;T} +  \|\bar u^{(n)} - \bar u\|_{\gamma,\eta;T} \\
&\lesssim \rho^n + \n{u^{(n)} - \bar u^{(n)}}_{\gamma,\eta;T}\;,
\end{equs}
for some constant $\rho < 1$. As a consequence of 
Lemma~\ref{lem:vanishPower}, Lemma~\ref{lem:interpolation}, \eref{e:boundNonSing}, 
Proposition~\ref{prop:intSing}, and using the fact
that there is a little bit of ``wiggle room'' between $\gamma$ and $\bar \gamma + 2q$, we 
obtain the existence of a constant $\kappa > 0$ such that one has the bound
\begin{equs}
\n{u^{(n)} - \bar u^{(n)}}_{\gamma,\eta;T} &\lesssim 
\$\CM_F^Z(u^{(n-1)}) ; \CM_F^{\bar Z}(\bar u^{(n-1)})\$_{\gamma,\eta;T}\\
&\lesssim \$F(u^{(n-1)}) ; F(\bar u^{(n-1)})\$_{\bar \gamma-\kappa,\bar \eta;T} + \$Z;\bar Z\$_{\gamma;O}\\
&\lesssim \n{F(u^{(n-1)}) - F(\bar u^{(n-1)})}_{\bar \gamma,\bar \eta;T}^\kappa + \$Z;\bar Z\$_{\gamma;O}\\
&\lesssim
\n{u^{(n-1)} - \bar u^{(n-1)}}_{\gamma,\eta;T}^\kappa + 
\$Z;\bar Z\$_{\gamma;O}\;,
\end{equs}
uniformly in $n$. By making $T$ sufficiently small, one can furthermore ensure that the proportionality
constant that in principle appears in this bound is bounded by $1$.
%
Since $u^0 = \bar u^0$, we can iterate this bound $n$ times to obtain
\begin{equ}
\|u^{(n)} - \bar u^{(n)}\|_{\gamma,\eta;T} \lesssim \$Z;\bar Z\$_{\gamma;O}^{\kappa^n}\;,
\end{equ}
with a proportionality constant that is bounded uniformly in $n$.
Setting $\eps = \$Z;\bar Z\$_{\gamma;O}$, a simple calculation shows that the term $\rho^n$ and
the term $\eps^{\kappa^n}$ are of (roughly) the same order when $n \sim \log \log \eps^{-1}$,
which eventually yields a bound of the type
\begin{equ}
\$u ; \bar u\$_{\gamma,\eta;T} \lesssim \bigl|\log \$Z;\bar Z\$_{\gamma;O}\bigr|^{-\nu}\;,
\end{equ}
for some exponent $\nu > 0$, uniformly in a small neighbourhood of any initial condition and
any model $Z$. While this bound is of course suboptimal in many situations, it is sufficient 
to yield the joint continuity of the solution map for a very large class of nonlinearities.
\end{proof}

\begin{remark}\label{rem:absorbSingular}
The condition $(\bar \eta \wedge \bar \zeta) > -2q$ is required in order to be able to apply
Proposition~\ref{prop:intSing}. Recall however that the assumptions of that theorem can on occasion be
slightly relaxed, see Remark~\ref{rem:senseSingDist}. The relevant situation in our context is when
$F$ can be rewritten as $F(z,u) = F_0(z,u) + F_1(z)$, where $F_0$ satisfies the assumption of 
our theorem, but $F_1$ does not. 
If we then make sense of $(\CK_{\bar \gamma} + R_{\gamma}\CR) \PR^+ F_1$ 
``by hand'' as an element of $\CD_P^{\gamma, \eta}$ and impose sufficient restrictions on our
model $Z$ such that this element is continuous as a function of $Z$, 
then we can absorb it into $v$ so that all of our conclusions still hold.
\end{remark}

\begin{remark}
In many situations, the map $F$ has the property that
\begin{equ}[e:assumF]
\CQ_{\zeta + 2q}^- \tau = \CQ_{\zeta + 2q}^- \bar \tau \quad\Rightarrow\quad 
\CQ_{\zeta + 2q}^- F(z,\tau) = \CQ_{\zeta + 2q}^- F(z,\bar \tau)\;.
\end{equ}
Denote as before by $\bar T \subset T$ the sector spanned by abstract polynomials. Then, provided
that \eref{e:assumF} holds, for every $z \in \R^d$ and every $v \in \bar T$, the equation
\begin{equ}
\tau = \CQ_\gamma^- \bigl(\CI F(z,\tau) + v\bigr)\;,
\end{equ}
admits a unique solution $\FF(z,v)$ in $V$. Indeed, it follows from the properties of the
abstract integration map $\CI$, combined with \eref{e:assumF}, that there exists $n > 0$ such that
the map $F_{z,v}\colon \tau \mapsto \CQ_\gamma^- \bigl(\CI F(z,\tau) + v\bigr)$ has the property that
$F_{z,v}^{n+1}(\tau) = F_{z,v}^n(\tau)$.

It then follows from the definitions of the operations appearing in \eref{e:fixedPointu}
that, if we denote by $\bar \CQ u$ the component of $u$ in $\bar T$, one has the identity
\begin{equ}[e:idenSol]
u(t,x) = \FF\bigl((t,x), \bar \CQ u(t,x)\bigr)\;,\qquad t \in (0,T]\;,
\end{equ}
for the solution to our fixed point equation \eref{e:fixedPointu}.
In other words, if we interpret the $\bar \CQ u(t,x)$ as a ``renormalised Taylor expansion''
for the solution $u$, then any of the components $\CQ_\zeta u(t,x)$ is given by some
explicit nonlinear function of the renormalised Taylor expansion up to some order depending
on $\zeta$. This fact will be used to great effect in Section~\ref{sec:PAMGenRigour} below.
\end{remark}

Before we proceed, we show that, in the situations of interest for us, the local solution
maps built in Theorem~\ref{theo:fixedPointGen} are consistent. In other words, we would like to be able to
construct a ``maximal solution'' by piecing together local solutions. In the context considered
here, it is \textit{a priori} not obvious that this is possible. In order to even
formulate what we mean by such a statement, we introduce the set $P_t = \{(s,y)\,:\, s = t\}$
and write $\PR_t^+$ for the indicator function of the set $\{(s,y)\,:\, s > t\}$, which
we interpret as before as a bounded operator from $\CD_{P_t}^{\gamma, \eta}$ into itself
for any $\gamma>0$ and $\eta \in \R$.

From now on, we assume that $G$ is the parabolic Green's function of a constant coefficient 
parabolic differential operator $\CL$ on $\R^{d-1}$. In this way, 
for any distribution $u_0$ on $\R^{d-1}$,
the function $v = G u_0$ defined as in Lemma~\ref{lem:harmonicExt}
is a classical solution to the equation $\d_t v = \CL v$ for $t > 0$. We then consider the
class of equations of the type \eref{e:fixedPointGen} with $v = G u_0$,
for some function (or possibly distribution) $u_0$ on $\R^{d-1}$.
We furthermore assume that the sector $V$ is function-like. Recall Proposition~\ref{prop:recFcn}, which implies
that any modelled distribution $u$ with values in $V$ is such that $\CR u$ is a continuous
function belonging to $\CC^{\beta}_\s$ for some $\beta > 0$. 
In particular, $\bigl(\CR u\bigr)(t,\cdot)$ is then
perfectly well-defined as a function on $\R^{d-1}$ belonging to $\CC^\beta_{\bar s}$.
We then have the following result:

\begin{proposition}\label{prop:solGen}
In the setting of Theorem~\ref{theo:fixedPointGen}, assume that $\zeta = 0$ and $-\s_1 < \eta < \beta$ with
$\eta \not \in \N$ and $\beta$ as above. 
Let $u_0 \in \CC^\eta_{\bar s}(\R^{d-1})$ be symmetric and let $T>0$ be sufficiently small
so that the equation
\begin{equ}[e:fixedPointu]
u = (\CK_{\bar \gamma} + R_{ \gamma}\CR) \PR^+ F(u) + G u_0\;,
\end{equ}
admits a unique solution $u \in \CD^{\gamma,\eta}_{P}$ on $(0,T)$.
Let furthermore $s \in (0,T)$ and $\bar T > T$ be such that
\begin{equ}
\bar u = (\CK_{\bar \gamma} + R_{ \gamma}\CR) \PR_s^+ F(\bar u) +  G u_s\;,
\end{equ}
where $u_s \eqdef \bigl(\CR u\bigr)(s,\cdot)$, admits a unique solution $\bar u \in \CD^{\gamma,\eta}_{P_s}$
on $(s,\bar T)$.

Then, one necessarily has $\bar u(t,x) = u(t,x)$ for every $x \in \R^{d-1}$ and every $t \in (s,T)$.
Furthermore, the element $\hat u \in \CD^{\gamma,\eta}_{P}$ defined by
$\hat u(t,x) = u(t,x)$ for $t \le s$ and $\hat u(t,x) = \bar u(t,x)$ for $t > s$
satisfies \eref{e:fixedPointu} on $(0,\bar T)$.
\end{proposition}

\begin{proof}
%
%
Setting $v = \PR_s^+ u \in \CD_{P_s}^{\gamma,\eta}$, it follows from the
definitions of $\CK_{\bar \gamma}$ and $R_\gamma$ that one has for $t \in (s,T]$ the identity
\begin{equs}
\scal{\one,v(t,x)} &= \int_0^t \int_{\R^{d-1}} G(t-r, x-y) \bigl(\CR F(u)\bigr)(r,y)\,dy\,dr\\
&\quad + \int_{\R^{d-1}} G(t, x-y) u_0(y)\,dy\\
&= \int_s^t \int_{\R^{d-1}} G(t-r, x-y) \bigl(\CR F(v)\bigr)(r,y)\,dy\,dr\\
&\quad + \int_{\R^{d-1}} G(t-s, x-y) u_s(y)\,dy\;.
\end{equs} 
Here, the fact that there appears no additional term is due to the fact that $\bar \zeta > -2q$,
so that the term $\scal{\one, \CJ(t,x) \bigl(F(u)(t,x)\bigr)}$ cancels exactly with the corresponding term
appearing in the definition of $\CN_{\bar \gamma}$.
This quantity on the other hand is precisely equal to
\begin{equ}
\scal{\one, \bigl((\CK_{\bar \gamma} + R_{ \gamma}\CR) \PR_s^+ F(v) + G u_s\bigr)(t,x)}\;.
\end{equ}
Setting
\begin{equ}
w = (\CK_{\bar \gamma} + R_{ \gamma}\CR) \PR_s^+ F(v) + G u_s\;,
\end{equ}
we deduce from the definitions of the various operators appearing above that, for $\ell \not \in \N$, 
one has $\CQ_\ell w(z) = \CQ_\ell \CI F(z,v(z))$. However, we also know that 
$v$ satisfies  $\CQ_\ell v(z) = \CQ_\ell \CI F(z,v(z))$.
We can therefore apply Proposition~\ref{prop:uniqueDeriv}, which yields the identity $w = v$,
from which it immediately follows that $v = \bar u$ on $(0,T)$.

The argument regarding $\hat u$ is virtually identical, so we do not reproduce it here.
\end{proof}

This shows that we can patch together local solutions in exactly the same way as for
``classical'' solutions to nonlinear evolution equations. Furthermore, it shows that 
the only way in which local solutions can fail to be global is by an explosion of the
$\CC_{\bar \s}^\eta$-norm of the quantity $\bigl(\CR u\bigr)(t,\cdot)$. 
Furthermore, since the reconstruction operator $\CR$ is continuous into $\CC_{\bar \s}^\eta$,
this norm is continuous as a function of time, so that for any cut-off value $L>0$, there exists 
a (possibly infinite) first time $t$ at which $\|u(t,\cdot)\|_{\eta} = L$.

Given a symmetric model
$Z = (\Pi,\Gamma)$ for $\TT$, a symmetric initial condition $u_0 \in \CC_{\bar \s}^\eta$, and 
some (typically large) cut-off value $L>0$, we denote by $u = \CS^L(u_0, Z) \in \CD_P^{\gamma,\eta}$ and 
$T = T^L(u_0, Z) \in \R_+ \cup \{+\infty\}$ 
the (unique) modelled distribution and time such that
\begin{equ}
u = (\CK_{\bar \gamma}+R_\gamma \CR) \PR^+ F\bigl(u\bigr)
+ G u_0\;,
\end{equ}
on $[0, T]$, 
such that $\|\bigl(\CR u\bigr)(t,\cdot)\|_{\eta} < L$ for $t < T$,
and such that $\|\bigl(\CR u\bigr)(t,\cdot)\|_{\eta} \ge L$ for $t \ge T$.
The following corollary is now straightforward:

\begin{corollary}\label{cor:genFixedPoint}
Let $L>0$ be fixed.
In the setting of Proposition~\ref{prop:solGen}, let $\CS^L$ and $T^L$ be defined 
as above and set $O = [-1,2] \times \R^{d-1}$.
Then, for every $\eps > 0$ and $C>0$ there exists $\delta > 0$ such that,
setting $T = 1 \wedge T^L(u_0, Z) \wedge T^L(\bar u_0, \bar Z)$, one has the bound
\begin{equ}
\|\CS^L(u_0, Z) - \CS^L(\bar u_0, \bar Z)\|_{\gamma,\eta;T} \le \eps\;,
\end{equ}
for all $u_0$, $\bar u_0$, $Z$, $\bar Z$ such that $\$Z\$_{\gamma;O} \le C$,
$\$\bar Z\$_{\gamma;O} \le C$, $\|u_0\|_{\eta} \le L/2$, $\|\bar u_0\|_{\eta} \le L/2$,
 $\|u_0- \bar u_0\|_\eta \le \delta$, and $\$Z;\bar Z\$_{\gamma;O} \le \delta$.
\end{corollary}

\begin{proof}
The argument is straightforward and works in exactly the same way as analogous statements
in the classical theory of semilinear PDEs. The main ingredient is the fact that
for every $t > 0$, one can obtain an \textit{a priori} bound on the number of iterations
required to reach the time $t \wedge T^L(u_0, Z)$. 
\end{proof}

\section{Regularity structures for semilinear (S)PDEs}
\label{sec:SPDE}

In this section, we show how to apply the theory developed in this article to construct an abstract
solution map to a very large class of semilinear PDEs driven by rough input data.
Given Theorem~\ref{theo:fixedPointGen}, the only task that remains is to build a sufficiently large
regularity structure allowing to formulate the equation. 

First, we give a relatively simple heuristic that allows one to very quickly decide whether a given
problem is at all amenable to the analysis presented in this article. 
For the sake of conciseness, we will assume that the problem of interest can be rewritten as a fixed point
problem of the type
\begin{equ}[e:genSPDE]
u = K * F(u,\nabla u, \xi)+ \tilde u_0\;,
\end{equ}
where $K$ is a singular integral operator that is $\beta$-regularising on $\R^d$ with respect to some
fixed scaling $\s$, $F$ is a smooth function, $\xi$ denotes the rough input data, and $\tilde u_0$
describes some initial condition (or possibly boundary data). 
In general, one might imagine that $F$ also depends on
derivatives of higher order (provided that $\beta$ is sufficiently large) and / or 
that $F$ itself involves some singular integral operators. 
We furthermore assume that $F$ is affine in $\xi$. (Accommodating the general case where $F$ is polynomial
in $\xi$ would also be possible with minor modifications, but we stick to the affine case
for ease of presentation.)

It is also straightforward to deal with the situation when $F$ is non-homogeneous
in the sense that it depends on the (space-time) location explicitly, as long as any such dependence is
sufficiently smooth.
For the sake of readability, we will refrain from presenting such extensions and we
will focus on a situation which is just general enough to be able to describe
all of the examples given in the introduction.

\begin{remark}
In all the examples we are considering, $K$ is the Green's function of some differential operator $\CL$.
In order to obtain optimal results, it is usually advisable to fix the scaling $\s$ in such 
a way that all the 
dominant terms in $\CL$ have the same homogeneity, when counting powers with the weights given by $\s$.
\end{remark}

\begin{remark}
We have seen in Section~\ref{sec:applSPDE} that in general, one would really want to consider
instead of \eref{e:genSPDE} fixed point problems of the type
\begin{equ}[e:genSPDEreal]
u =  \bigl((K + R) * \bigl(\PR^+ F(u,\nabla u, \xi)\bigr)\bigr) + \tilde u_0\;,
\end{equ}
where $\PR^+$ denotes again the characteristic function of the set of positive times
and $R$ is a smooth non-anticipative kernel.
However, if we are able to formulate \eref{e:genSPDE}, then it is always straightforward
to also formulate \eref{e:genSPDEreal} in our framework, so we concentrate on \eref{e:genSPDE} for
the moment in order not to clutter the presentation. 
\end{remark}


Denoting by $\alpha < 0$ the
regularity of $\xi$ and considering our multi-level Schauder estimate, 
Theorem~\ref{theo:Int}, we then expect the regularity of the 
solution $u$ to be of order at most  $\beta + \alpha$, the regularity of $\nabla u$ to of
order at most $\beta + \alpha - 1$, etc. We then make the following assumption:

\begin{assumption}[local subcriticality]\label{ass:powercount}
In the formal expression of $F$, replace $\xi$ by a dummy variable $\Xi$.
For any $i \in \{1,\ldots,d\}$, if
 $\beta + \alpha \le \s_i$, then replace furthermore any occurrence of $\d_i u$ by the dummy variable $P_i$.
Finally, if $\beta + \alpha \le 0$, replace any occurrence of $u$ by the dummy variable $U$. 

We then make the following two assumptions. First, we assume that the resulting expression 
is polynomial in the dummy variables. Second, we associate to each such monomial a homogeneity
by postulating that $\Xi$ has homogeneity $\alpha$, $U$ has homogeneity $\beta + \alpha$, and 
$P_i$ has homogeneity $\beta + \alpha - \s_i$. (The homogeneity of a monomial then being the sum
of the homogeneities of each factor.) 
With these notations, the assumption of local 
subcriticality is that terms containing $\Xi$ do not contain the dummy variables and that the
remaining monomials each have homogeneity strictly greater than $\alpha$. 
\end{assumption}

Whenever a problem of the type 
\eref{e:genSPDE} satisfies Assumption~\ref{ass:powercount}, we say that it is \textit{locally subcritical}.
The role of this assumption is to ensure that, using 
Theorems~\ref{theo:mult}, \ref{theo:smooth}, and \ref{theo:Int}, one 
can reformulate \eref{e:genSPDE} as a fixed point map in 
$\CD^\gamma$ for sufficiently high $\gamma$ (actually any $\gamma > |\alpha|$ would do) by
replacing the convolution $K*$ with $\CK_\gamma$ as in Theorem~\ref{theo:Int}, 
replacing all products by the abstract product $\star$, and interpreting compositions with
smooth functions as in Section~\ref{sec:compSmooth}. 

For such a formulation to make sense, we need of course to build a sufficiently 
rich regularity structure. This could in principle be done by repeatedly applying Proposition~\ref{prop:extendMult} 
and Theorem~\ref{theo:extension}, but we will actually make use of a more explicit 
construction given in this section, which will also have the advantage of coming automatically
with a ``renormalisation group'' that allows to understand the kind of convergence results
mentioned in Theorem~\ref{theo:mainConvPAM} and Theorem~\ref{thm:Phi4}. 
Our construction suggests the following ``metatheorem'', which is essentially a combination
of Theorem~\ref{theo:fixedPointGen}, Theorem~\ref{theo:mult}, Theorem~\ref{theo:smooth}, 
and Theorem~\ref{theo:genStruct} below.

\begin{metatheorem}
Whenever \eref{e:genSPDE} is locally subcritical, it is possible to build a regularity structure allowing
to reformulate it
as a fixed point problem in $\CD^\gamma$ for $\gamma$ large enough. Furthermore, if the problem is
parabolic on a bounded domain (say the torus), then the fixed point problem admits a unique local solution.
\end{metatheorem}

Before we proceed to building the family of regularity structures allowing to formulate these SPDEs, 
let us check that Assumption~\ref{ass:powercount}
is indeed verified for our examples \eref{e:Phi4}, \eref{e:PAM}, and \eref{e:KPZ}.
Note first that it is immediate from Proposition~\ref{prop:charSpaces} and the equivalence of moments for
Gaussian random variables 
that white noise on $\R^d$ with scaling $\s$
almost surely belongs to $\CC^\alpha_\s$ for every $\alpha < -{|\s|\over 2}$.
(See also Lemma~\ref{lem:approxGauss} below.)
Furthermore, the heat kernel is $2$-regularising, so that $\beta = 2$ in all of the problems considered
here.

In the case of \eref{e:Phi4} in dimension $d$, space-time is given by $\R^{d+1}$ with scaling $\s=(2,1,\ldots,1)$, so that $|\s| = d+2$.
This implies that $\xi$ belongs to $\CC^\alpha_\s$ for every 
$\alpha < -{d+2\over 2} = -1-{d\over 2}$. In this case $\beta + \alpha \approx 1 - {d\over 2}$ so
that, following the procedure of Assumption~\ref{ass:powercount}, the monomials appearing 
are $U^3$ and $\Xi$. The homogeneity of $U^3$ is $3(\beta + \alpha) \approx 3 - {3d \over 2}$,
which is greater than $-1-{d\over 2}$ if and only if $d < 4$.
This is consistent with the fact that $4$ is the critical dimension for Euclidean 
$\Phi^4$ quantum field theory \cite{MR678000}.
Classical fixed point arguments using purely deterministic techniques on the other hand
already fail for dimension $2$, where the homogeneity of $u$ becomes negative, 
which is a well-known fact \cite{MR0378670}. 
In the particular case of $d=2$ however,
provided that one defines the powers $(K * \xi)^k$ ``by hand'', one can write $u = K* \xi + v$,
and the equation for $v$ 
is amenable to classical analysis, a fact that was exploited for example in 
\cite{MR2016604,MR2928722}.
In dimension $3$, this breaks down, but our arguments show that one still
expects to be able to reformulate \eref{e:Phi4} as
a fixed point problem in $\CD^\gamma$, provided that $\gamma > {3\over 2}$. This will be done in Section~\ref{sec:GenFP} below.

For \eref{e:PAM} in dimension $d$ (and therefore space-time $\R^{d+1}$ with the same scaling as above), 
spatial white noise belongs to $\CC_\s^\alpha$ for $\alpha < -{d\over 2}$. As a consequence, 
Assumption~\ref{ass:powercount} does in this case boil down to the condition
$2+\alpha > 0$, which is again the case if and only if $d < 4$. This is again not surprising. 
Indeed, dimension $4$
is precisely such that, if one considers the classical parabolic Anderson model on the lattice $\Z^4$ and
simply rescales the solutions without changing the parameters of the model, one formally converges to
solutions to the continuous model \eref{e:PAM}. On the other hand, as a consequence of Anderson localisation,
one would expect that the rescaled
solution converges to an object that is ``trivial'' in the sense that it could only be described either by the $0$
distribution or by a Dirac distribution concentrated in a random location, which is something that falls
outside of the scope of the theory presented in this article.
In dimensions $2$ and $3$ however, one expects to be able
to formulate and solve a fixed point problem in $\CD^\gamma$ for $\gamma > {3\over 2}$. This time,
one also expects solutions to be global, since the equation is linear.

In the case of \eref{e:KPZ}, one can verify in a similar way that Assumption~\ref{ass:powercount} holds.
As before, if we consider an equation of this type in dimension $d$, we have $|\s| = d+2$, so that one
expects the solution $u$ to be of regularity just below $1-{d\over 2}$. In this case, dimension $2$ is already critical
for three unrelated reasons. First, this is the dimension where $u$ ceases to be function-valued,
so that compositions with smooth functions ceases to make sense. Second, even if the functions $g_i$
were to be replaced by polynomials, $g_4$ would have to be constant in order to satisfy Assumption~\ref{ass:powercount}.
Finally, the homogeneity of the term $|\nabla h|^2$ is $-d$. In dimension $2$, this precisely matches the regularity
$-1-{d\over 2}$ of the noise term. 

We finally turn to the Navier-Stokes equations \eref{e:NS}, which we can write in the form
\eref{e:genSPDE} with $K$ given by the heat kernel, composed with Leray's projection onto
the space of divergence-free vector fields. The situation is slightly more subtle here, as the kernel
is now matrix-valued, so that we really have $d^2$ (or rather $d(d+1)/2$ because of the symmetry)
different convolution operators. Nevertheless, the situation is similar to before and each component of 
$K$ is regularity
improving with $\beta = 2$. The condition for local subcriticality given by Assumption~\ref{ass:powercount}
then states that one should have $(1-{d\over 2}) + (-{d\over 2}) > -1-{d\over 2}$, which is
satisfied if and only if $d < 4$. 

\subsection{General algebraic structure}
\label{sec:genAlg}

The general structure arising in the abstract solution theory for 
semilinear SPDEs of the form \eref{e:Phi4}, \eref{e:PAM}, 
etc is very close to the structure
already mentioned in Section~\ref{sec:Hopf}. The difference however is that $T$ only ``almost''
forms a Hopf algebra, as we will see presently.

In general, we want to build a regularity structure that is sufficiently rich to allow to formulate
a fixed point map for solving our SPDEs. Such a regularity structure will depend on 
the dimension $d$ of the underlying space(-time), the scaling $\s$ of the linear operator,
the degree $\beta$ of the linear operator (which is equal to the regularising index of the 
corresponding Green's function), and the regularity $\alpha$ of the driving noise $\xi$.
It will also depend on finer details of the equation, such as whether the nonlinearity contains
derivatives of $u$, arbitrary functions of $u$, etc. 

At the minimum, our regularity structure should contain polynomials, and it should come with an 
abstract integration map $\CI$ that represents integration against the Green's function $K$ of the linear
operator $\CL$. (Or rather integration against a suitable cut-off version.)
Furthermore, since we might want to represent derivatives of $u$, we
can introduce the integration map\label{lab:defIk} 
$\CI_k$ for a multiindex $k$, which one should think as representing
integration against $D^k K$.
The ``na\"\i ve'' way of building $T$ would be then to consider all possible 
formal expressions $\CF$ 
that can be obtained from the abstract symbols $\Xi$ and $\{X_i\}_{i=1}^d$,
as well as the abstract integration maps $\CI_k$. 
More formally, we can define a set $\CF$
by postulating that $\{\one, \Xi, X_i\} \subset \CF$ and, whenever $\tau, \bar\tau \in \CF$,
we have $\tau \bar \tau \in \CF$ and $\CI_k(\tau) \in \CF$.
(However, we do not include any expression containing a factor of $\CI_k(X^\ell)$, thus
reflecting Assumption~\ref{ass:polynom} at the algebraic level.)
Furthermore, we postulate that the product is commutative and associative by identifying the
corresponding formal expressions (i.e. $X \CI(\Xi) = \CI(\Xi) X$, etc), and that $\one$ is
neutral for the product.

One can then associate to each $\tau \in \CF$ a weight $|\tau|_\s$ which
is obtained by setting $|\one|_\s = 0$,
\begin{equ}
|\tau\bar \tau|_\s = |\tau|_\s + |\tau|_\s\;,
\end{equ}
for any two formal expressions $\tau$ and $\bar \tau$ in $\CF$, and such that
\begin{equ}
|\Xi|_\s = \alpha\;,\quad |X_i|_\s = \s_i\;,\quad |\CI_k(\tau)|_\s = |\tau|_\s + \beta-|k|_\s\;.
\end{equ}
Since these operations are sufficient to generate all of $\CF$, this does indeed define $|\cdot|_\s$.

\begin{example}
These rules yield the weights
\begin{equ}
|\Xi \CI_\ell(\Xi^2 X^k)|_\s = 3\alpha + |k|_\s + \beta - |\ell|_\s\;,\qquad
|X^k \CI(\Xi)^2|_\s = |k|_\s + 2(\alpha+ \beta)\;, 
\end{equ}
for any two multiindices $k$ and $\ell$.
\end{example}

We could then define $T_\gamma$ simply as the set of all formal linear combinations of elements  
$\tau \in \CF$ with $|\tau|_\s = \gamma$.
The problem with this procedure is that  since $\alpha < 0$,
we can build in this way expressions that have arbitrarily negative weight, so that the set of homogeneities
$A\subset\R$ would not be bounded from below anymore. (And it would possibly not even be locally finite.)

The ingredient that allows to circumvent this problem is the assumption of local
subcriticality loosely formulated in Assumption~\ref{ass:powercount}. To make this more formal, assuming 
again for simplicity that the right hand side $F$ of
our problem \eref{e:genSPDE} depends only on $\xi$, $u$, and some partial derivatives
$\d_i u$, we can associate to $F$ a (possibly infinite) collection $\M_F$ of monomials
in $\Xi$, $U$, and $P_i$ in the following way. 

\begin{definition}\label{def:basicBlocks}
For any two integers $m$ and $n$, and multiindex $k$, we have $\Xi^m U^n P^k\in \M_F$
if $F$ contains a term of the type $\xi^{\bar m}u^{\bar n} (Du)^{\bar k}$ for
$\bar m \ge m$, $\bar n \ge n$, and $\bar k \ge k$. 
Here, we consider arbitrary smooth functions as polynomials of ``infinite order'', i.e.
we formally substitute $g(u)$ by $u^\infty$ and similarly for functions involving derivatives of $u$.
Note also that $k$ and $\bar k$ are multiindices since, in general, $P$ is a $d$-dimensional vector.
\end{definition}

\begin{remark}\label{rem:Phi4}
Of course, $\M_F$ is not really well-defined. For example, in the case of \eref{e:Phi4},
we have $F(u,\xi) = \xi-u^3$, so that 
\begin{equ}
\M_F = \{\Xi, U^m \,:\, m \le 3\}\;.
\end{equ}
However, we could of course have rewritten this as $F(u,\xi) = \xi + g(u)$,
hiding  the fact that $g$ actually happens to be a polynomial itself, and this would lead to
adding all higher powers $\{U^n\}_{n > 3}$ to $\M_F$. In practice, it is usually obvious
what the minimal choice of $\M_F$ is.

Furthermore, especially in situations where the solution $u$ is actually vector-valued,
it might be useful to encode into our regularity structure additional structural properties of the equation,
like whether a given function can be written as a gradient. (See the series of works
\cite{Jan,Hendrik,JanHendrik} for situations where this would be of importance.)
\end{remark}

%
\begin{remark}\label{rem:PAMGen}
In the case of \eref{e:PAM}, we have
\begin{equ}
\M_F = \{1, U, U\Xi, \Xi\}\;,
\end{equ}
while in the more general case of \eref{e:PAMGen}, we have
\begin{equ}
\M_F = \bigl\{U^n, U^n\Xi, U^n P_i, U^n P_iP_j\,:\, n \ge 0,\, i,j \in \{1,2\}\bigr\}\;.
\end{equ}
This and \eref{e:Phi4} are the only examples that will be treated in full detail, but
it is straightforward to see what $\M_F$ would be for the remaining examples.
\end{remark}

\begin{remark}
Throughout this whole section, we consider the case where the noise $\xi$ driving our equation
is real-valued and there is only one integral kernel required to describe the fixed point map.
In general, one might also want to consider a finite family $\{\Xi^{(i)}\}$ of formal symbols
describing the driving noises and a family $\{\CI^{(i)}\}$ of symbols describing integration against
various integral kernels.
For example, in the case of \eref{e:NS},
the integral kernel also involves the Leray projection and is therefore matrix-valued,
while the driving noise is vector-valued. 
This is an immediate generalisation that merely requires some additional indices decorating the
objects $\Xi$ and $\CI$ and all the results obtained in the present section trivially extend to this case. 
One could even accommodate the situation where different components
of the noise have different degrees of regularity, but it would then become awkward to state
an analogue to Assumption~\ref{ass:powercount}, although it is certainly possible. 
Since notations are already quite heavy in the current state of things, we 
refrain from increasing our level of generality.
\end{remark}

Given a set of monomials $\M_F$ as in Definition~\ref{def:basicBlocks}, 
we then build subsets $\{\CU_n\}_{n \ge 0}$,
$\{\CP^i_n\}_{n \ge 0}$ and $\{\CW_n\}_{n \ge 0}$ of $\CF$ by the following algorithm.
We set $\CW_0 = \CU_0 = \CP^i_0 = \emptyset$ and,
given subsets $A,B \subset \CF$, we also write $AB$ for the set of all products
$\tau \bar \tau$ with $\tau \in A$ and $\bar \tau \in B$, and similarly for higher order monomials.
(Note that this yields the convention $A^2 = \{\tau\bar \tau\,:\, \tau, \bar \tau \in A\}\neq \{\tau^2\,:\, \tau \in A\}$.)

Then, we define the sets $\CW_n$, $\CU_n$ and $\CP^i_n$ for $n > 0$ recursively by 
\begin{equs}
\CW_{n} &= \CW_{n-1} \cup \bigcup_{\CQ \in \M_F} \CQ(\CU_{n-1},\CP_{n-1},\Xi) \;,\\
\CU_n &= \{X^k\} \cup \bigl\{ \CI(\tau)\,:\, \tau \in \CW_n\bigr\}\;, \label{e:defSetsStruct}\\
\CP^i_n &= \{X^k\} \cup \bigl\{ \CI_i(\tau)\,:\, \tau \in \CW_n\bigr\}\;,
\end{equs}
where in the set $\{X^k\}$, $k$ runs over all possible multiindices.
In plain words, we take any of the monomials in $\M_F$ and build $\CW_{n}$ by 
formally substituting each occurrence of
$U$ by one of
the expressions already obtained in $\CU_{n-1}$ and each occurrence of $P_i$ by one of
the expressions from $\CP^i_{n-1}$. We then apply the maps $\CI$ and $\CI_i$ respectively
to build $\CU_n$ and $\CP_n^i$, ensuring further that they include all monomials involving
only the symbols $X_i$. 
With these definitions at hand, we then set
\begin{equ}[e:defCFF]
\CF_F \eqdef \bigcup_{n \ge 0} \bigl(\CW_n \cup \CU_n\bigr)\;.
\end{equ}
In situations where $F$ depends on $u$ (and not only on $D u$ and $\xi$ like in the 
case of the KPZ equation for example), we furthermore set
\begin{equ}[e:defCUF]
\CU_F \eqdef \bigcup_{n \ge 0} \CU_n\;.
\end{equ}
We similarly define $\CP^i_F = \bigcup_{n \ge 0} \CP_n^i$ in the case when $F$ depends on $\d_i u$.
The idea of this construction is that $\CU_F$ contains those elements of $\CF$ that are required
to describe the solution $u$ to the problem at hand, $\CP^i_F$ contains the elements
appearing in the description of $\d_i u$, and $\CF_F$ contains the elements required to describe
both the solution and the right hand side of \eref{e:genSPDE}, so that $\CF_F$ is rich enough to set up
the whole fixed point map.

The following result then shows that our assumption of local subcriticality, 
Assumption~\ref{ass:powercount}, is really the correct assumption for the theory developed
in this article to apply:

\begin{lemma}\label{lem:finDim}
Let $\alpha < 0$. Then, the set $\{\tau \in \CF_F \,:\, |\tau|_\s \le \gamma\}$ is
finite for every $\gamma \in \R$ if and only if Assumption~\ref{ass:powercount} holds.
\end{lemma}

\begin{proof}
We only show that Assumption~\ref{ass:powercount} is sufficient. Its necessity can be
shown by similar arguments and is left to the reader.
Set $\alpha^{(n)} = \inf \{|\tau|_\s\,:\, \tau \in \CU_n \setminus \CU_{n-1}\}$
and $\alpha_i^{(n)} = \inf \{|\tau|_\s\,:\, \tau \in \CP_n^{(i)} \setminus \CP_{n-1}^{(i)}\}$. 
We claim that under Assumption~\ref{ass:powercount} there exists $\zeta > 0$ 
such that $\alpha^{(n)} > \alpha^{(n-1)} + \zeta$ and similarly for $\alpha_i^{(n)}$,
which then proves the claim.

Note now that $\CW_1 = \{\Xi\}$, so that one has
\begin{equ}
\alpha^{(1)} = (\alpha + \beta) \wedge 0 \;,\qquad \alpha^{(1)}_i = (\alpha + \beta - \s_i)\wedge 0\;.
\end{equ}
Furthermore, Assumption~\ref{ass:powercount} implies that if $\Xi^p U^q P^k \in \M_F \setminus \{\Xi\}$, then
\begin{equ}[e:powCount]
p \alpha + q(\alpha + \beta) + \sum_i k_i (\alpha + \beta - \s_i) > \alpha\;,
\end{equ}
and $k_i$ is allowed to be non-zero only if $\beta > \s_i$.
This immediately implies that one has $|\tau|_\s \ge \alpha$ for every $\tau \in \CF_F$,
$|\tau|_\s \ge (\alpha + \beta)\wedge 0$ for every $\tau \in \CU_F$, and $|\tau|_\s \ge (\alpha + \beta - \s_i)\wedge 0$
for every $\tau \in \CP_F^i$. (If this were to fail, then there would
be a smallest index $n$ at which it fails. But then, since it still holds at $n-1$, condition
\eref{e:powCount} ensures that it also holds at $n$, thus creating a contradiction.)

Let now $\zeta > 0$ be defined as
\begin{equ}
\zeta = \inf_{\Xi^p U^q P^k \in \M_F\setminus \{\Xi\}} \Bigl\{(p-1) \alpha + q(\alpha + \beta) + \sum_i k_i (\alpha + \beta - \s_i)\Bigr\}\;.
\end{equ}
Then we see that $\alpha^{(2)} \ge \alpha^{(1)} + \zeta$ and similarly for $\alpha^{(2)}_i$.
Assume now by contradiction that there is a smallest value $n$ such that 
either $\alpha^{(n)} < \alpha^{(n-1)}+\zeta$
or $\alpha^{(n)}_i < \alpha^{(n-1)}_i+\zeta$ for some index $i$. Note first that one necessarily has
$n \ge 3$ and that, for any such $n$, 
one necessarily has $\alpha^{(n)}_i = \alpha^{(n)} - \s_i$ by \eref{e:defSetsStruct} so that
we can assume that one has $\alpha^{(n)} < \alpha^{(n-1)}+\zeta$.

Note now that there exists some element $\tau \in \CU_n$ with $|\tau|_\s = \alpha^{(n)}$
and that $\tau$ is necessarily of the form $\tau = \CI(\bar \tau)$ with $\bar \tau \in \CW_n \setminus \CW_{n-1}$. In other words, $\bar \tau$ is a product of elements in $\CU_{n-1}$ and $\CP_{n-1}^i$
(and possibly a factor $\Xi$) with at least one factor belonging to either $\CU_{n-1} \setminus \CU_{n-2}$
or $\CP_{n-1}^i \setminus \CP_{n-2}^i$. Denote that factor by $\sigma$, so that $\bar \tau = \sigma u$ for some $u\in \CW_{n}$.

Assume that $\sigma \in \CU_{n-1} \setminus \CU_{n-2}$, the argument being analogous if it
belongs to one of the $\CP_{n-1}^i \setminus \CP_{n-2}^i$. Then, by definition, one has $|\sigma|_\s \ge \alpha^{(n-1)}$. Furthermore, one has $\alpha^{(n-1)} \ge \alpha^{(n-2)} + \zeta$, so that there
exists some element $\hat \sigma \in \CU_{n-2}\setminus \CU_{n-3}$ with $|\hat \sigma|_\s \le |\sigma|_\s - \zeta$. By the same argument, one can find $\hat u \in \CW_{n-1}$ with $|\hat u|_\s \le |u|_\s$.
Consider now the element $\hat \tau = \CI(\hat \sigma \hat u)$.
By the definitions, one has $\hat \tau \in \CU_{n-1}$ and, since $\hat \sigma \not \in \CU_{n-3}$, 
one has $\hat \tau \not \in \CU_{n-2}$. Therefore, we conclude from this that
\begin{equ}
\alpha^{(n-1)} \le |\hat \tau|_\s \le |\tau|_\s - \zeta = \alpha^{(n)} - \zeta\;,
\end{equ}
thus yielding the contradiction required to prove our claim.
\end{proof}

\begin{remark}
If $F$ depends explicitly on $u$, then one has $U \in \M_F$, so that one automatically
has $\CU_F \subset \CF_F$. Similarly, if $F$ depends on $\d_i u$, one has $\CP_F^i \subset \CF_F$.
\end{remark}

\begin{remark}\label{rem:subsets}
If $\tau \in \CF_F$ is such that there 
exists $\tau_1$ and $\tau_2$ in $\CF$ with $\tau = \tau_1 \tau_2$, then one
also has $\tau_1, \tau_2 \in \CF_F$. This is a consequence of the fact that, by
Definition~\ref{def:basicBlocks}, whenever a monomial in $\M_F$ can be written
as a product of two monomials, each of these also belongs to $\M_F$.

Similarly, if $\CI(\tau) \in \CF_F$ or $\CI_i(\tau) \in \CF_F$ for some $\tau \in \CF$, 
then one actually has $\tau \in \CF_F$.
\end{remark}

Given any problem of the type \eref{e:generalProblem}, and under Assumption~\ref{ass:powercount},
this procedure thus allows us to build a candidate $T$ for the model space of a regularity structure,
by taking for $T_\gamma$ the formal linear combinations of elements in $\CF_F$ with
$|\tau|_\s = \gamma$. The spaces $T_\gamma$ are all finite-dimensional by Lemma~\ref{lem:finDim}, 
so the choice of
norm on $T_\gamma$ is irrelevant. For example, we could simply decree that the elements of $\CF_F$
form an orthonormal basis. 
Furthermore, the natural product in $\CF$ extends to a product $\star$ on $T$ by linearity,
and by setting $\tau \star \bar \tau = 0$ whenever $\tau, \bar \tau \in \CF_F$ are such that
$\tau \bar \tau \not \in \CF_F$.

While we now have a candidate for a model space $T$, as well as an index set $A$
(take $A = \{|\tau|_\s\,:\, \tau \in \CF_F\}$), we have not yet constructed the structure group
$G$ that allows to ``translate'' our model from one point to another. 
The remainder of this subsection is devoted to this construction.
In principle, $G$ is completely determined by the action of the group of translations on the $X^k$,
the assumption that $\Gamma \Xi = \Xi$, the requirements
\begin{equ}
\Gamma(\tau  \bar \tau) = (\Gamma \tau) \star (\Gamma \bar \tau)\;,
\end{equ}
for any $\tau , \bar \tau \in \CF_F$ such that $\tau  \bar \tau \in \CF_F$, as well as 
the construction of Section~\ref{sec:extension}. However, since it has a relatively explicit
construction similar to the one of Section~\ref{sec:Hopf}, we give it for the
sake of completeness. This also gives us a much
better handle on elements of $G$, which will be very useful in the next section.
Finally, the construction of $G$ given here exploits the natural relations
between the integration maps $\CI_k$ for different values of $k$ (which are needed
when considering equations involving derivatives of the solution in the right hand side),
which is something that the general construction of Section~\ref{sec:extension} does not do.

In order to describe the structure group $G$, we introduce three different vector spaces.
First, we denote by $\CH_F$ the set of  finite linear combinations of elements in $\CF_F$
and by $\CH$ the set of finite linear combinations of all elements in $\CF$.
We furthermore define a set
$\CF_+$ consisting of all formal expressions of the type
\begin{equ}[e:defF+]
X^k \prod_{j} \CJ_{k_j}\tau_j\;,
\end{equ}
where the product runs over finitely many terms, the $\tau_j$ are elements of $\CF$,
and the $k_j$ are multiindices with the property that $|\tau_j|_\s + \beta - |k_j|_\s > 0$ for
every factor appearing in this product. We should really think of $\CJ_k$ as being
essentially the same as $\CI_k$, so that one can alternatively think of $\CF_+$ as
being the set of all elements $\tau \in \CF$ such that either $\tau = \one$ or $|\tau|_\s > 0$
and such  that,
whenever $\tau$ can be written as $\tau = \tau_1 \tau_2$, one also has either $\tau_i =\one$
or $|\tau_i|_\s > 0$. The notation $\CJ_k$ instead of $\CI_k$ will however serve to reduce
confusion in the sequel, since elements of $\CF_+$ play a role that is distinct from the
corresponding elements in $\CF$. It is no coincidence that the symbol $\CJ$ is
the same as in Section~\ref{sec:integral} since elements of the type $\CJ_k \tau$ are
precisely placeholders for the coefficients $\CJ(x)\tau$ defined in \eref{e:defJx}. 
Similarly, we define $\CF_F^+$ as the set of symbols as in
\eref{e:defF+}, but with the $\tau_j$ assumed to belong to $\CF_F$.
Expressions of the type \eref{e:defF+} come with a natural notion of homogeneity,
given by $|k|_\s + \sum_j \bigl(|\tau_j|_\s + \beta - |k_j|_\s\bigr)$, which is
always positive by definition.

We then denote by\label{lab:defHH} $\CH_+$ the set of all finite linear combinations of all elements in $\CF_+$,
and similarly for $\CH_F^+$.
Note that both $\CH$ and $\CH_+$ are algebras, by simply extending the product $(\tau, \bar \tau) \mapsto \tau \bar \tau$ in a distributive way. While $\CH_F$ is a linear subspace of $\CH$, 
it is \textit{not} in general a subalgebra of $\CH$,
but this will not concern us very much since it is mostly the structure of the larger space $\CH$
that matters. The space $\CH_F^+$ on the other hand is an algebra. (Actually the free algebra
over the symbols $\{X_j, \CJ_k\tau\}$, where $j \in \{1,\ldots,d\}$, $\tau \in \CF_F$,
and $k$ is an arbitrary $d$-dimensional multiindex with $|k|_\s < |\tau|_\s + \beta$.)

We now describe a structure on the spaces $\CH$ and $\CH_+$ that endows $\CH_+$ (resp. $\CH_F^+$)
with
a Hopf algebra structure and $\CH$ (resp. $\CH_F$) with the structure of a comodule over $\CH_+$
(resp. $\CH_F^+$). 
The purpose of these structures is to yield an explicit construction of a regularity
structure that is sufficiently rich to allow to formulate fixed point maps for large classes
of semilinear (stochastic) PDEs. 
This construction will in particular allow us to describe the structure group $G$ in a way that is similar to the
construction in Section~\ref{sec:Hopf}, but with a slight twist since $T=\CH_F$ itself is different from 
both the Hopf algebra $\CH_+$ and the comodule $\CH$.

We first note that for \textit{every} multiindex $k$, we have a natural linear map 
$\hat\CJ_k \colon \CH \to \CH_+$ by setting 
\begin{equ}
\hat \CJ_k (\tau) = \CJ_k \tau \;,\quad |k|_\s < |\tau|_\s + \beta\;,\qquad \hat \CJ_k (\tau) = 0\;,\quad \text{otherwise.}
\end{equ}
Since there can be no scope for confusion, we will make a slight abuse of notation and
simply write again $\CJ_k$ instead of $\hat \CJ_k$.
We then define \textit{two} linear maps $\Delta \colon \CH \to \CH \otimes \CH_+$ and
 $\Deltap \colon \CH_+ \to \CH_+ \otimes \CH_+$ by
\begin{equs}[2]
\Delta \one &= \one \otimes \one\;, &\qquad \Deltap \one &= \one \otimes \one\;,\\
\Delta X_i &= X_i \otimes \one + \one \otimes X_i \;, &\quad  \Deltap X_i &= X_i \otimes \one + \one \otimes X_i\\
\Delta \Xi &= \Xi \otimes \one\;,&&
\end{equs}
and then, recursively, by
\minilab{e:defDelta}
\begin{equs}
\Delta (\tau \bar \tau) &= (\Delta \tau)\,(\Delta \bar \tau) \label{e:prodD1}\\
\Delta (\CI_k \tau)  &= \bigl(\CI_k  \otimes I\bigr)\Delta \tau + \sum_{\ell, m} {X^\ell\over \ell!} \otimes {X^m\over m!}\CJ_{k+\ell+m} \tau \;,
\label{e:int1}
\end{equs}
as well as
\minilab{e:defDelta+}
\begin{equs}
\Deltap (\tau \bar \tau) &= (\Deltap \tau)\,(\Deltap \bar \tau) \label{e:prodD2}\\
\Deltap (\CJ_k \tau)  &= \sum_\ell \Bigl(\CJ_{k+\ell} \otimes {(-X)^\ell\over \ell!}\Bigr)\Delta \tau + \one \otimes \CJ_k \tau \;.
\label{e:int2}
\end{equs}
In both cases, these sums run in principle over all possible multiindices $\ell$ and $m$. 
Note however that these
sums are actually finite since, by definition,
for $|\ell|_\s$ large enough it is always the case that  $\CJ_{k+\ell} \tau = 0$.

\begin{remark}
By construction, for every $\tau \in \CF$, one has the identity
$\Delta \tau = \tau \otimes \one + \sum_i c_i \tau_i^{(1)} \otimes \tau_i^{(2)}$,
for some constants $c_i$ and some elements with $|\tau_i^{(1)}|_\s < |\tau|_\s$ and 
$|\tau_i^{(1)}|_\s + |\tau_i^{(2)}|_\s = |\tau|_\s$. This is a reflection in this context of the
condition \eref{e:coundGroup}.

Similarly, for every $\sigma \in \CF_+$, one has the identity
\begin{equ}
\Deltap \sigma = \sigma \otimes \one + \one \otimes \sigma + \sum_i c_i \sigma_i^{(1)} \otimes \sigma_i^{(2)}\;,
\end{equ}
for some constants $c_i$ and some elements with 
$|\sigma_i^{(1)}|_\s + |\sigma_i^{(2)}|_\s = |\sigma|_\s$. Note also that 
\eref{e:defDelta+} is coherent with our abuse of notation for $\hat \CJ_k$ in the sense that 
if $\tau$ and $k$ are such that $\hat \CJ_k \tau = 0$, then the right hand side automatically vanishes.
\end{remark}

\begin{remark}
The fact that it is $\Delta$ (rather than $\Deltap$) that appears in the right hand side of \eref{e:int2}
is not a typo: there is not much choice since $\tau \in \CF$ and not in $\CF_+$. The motivation for the definitions of $\Delta$ and $\Deltap$ will be given in Section~\ref{sec:realAlg} below
where we show how it allows to canonically lift a continuous realisation $\xi$ of the ``noise'' to
a model for the regularity structure built from these algebraic objects.
\end{remark}

%
%
%
%
\begin{remark}
In the sequel, we will use Sweedler's notation for coproducts. Whenever we write
$\Delta \tau = \sum \tau^{(1)} \otimes \tau^{(2)}$, this should be read as a shorthand for:
``There exists a finite index set $I$, non-zero constants $\{c_i\}_{i \in I}$,
and basis elements $\{\tau^{(1)}_i\}_{i\in I}$, $\{\tau^{(2)}_i\}_{i\in I}$
such that the identity  $\Delta \tau = \sum_{i \in I} c_i  \tau_i^{(1)} \otimes \tau_i^{(2)}$ holds.''
If we then later refer to a joint property of $\tau^{(1)}$ and $\tau^{(2)}$,
this means that the property in question holds for every pair $(\tau_i^{(1)},\tau_i^{(2)})$
appearing in the above sum.
\end{remark}

The structure just introduced has the following nice algebraic properties.

\begin{theorem}\label{theo:projDelta}
The space $\CH_+$ is a Hopf algebra and $\CH$ is a comodule over $\CH_+$.
In particular, one has the identities \minilab{e:coassoc}
\begin{equs}
(I \otimes \Deltap) \Delta \tau &= (\Delta \otimes I)\Delta \tau\;, \label{e:coassoc1}\\
(I \otimes \Deltap) \Deltap  \tau &= (\Deltap \otimes I)\Deltap \tau\;, \label{e:coassoc2}
\end{equs}
for every $\tau \in \CH$. Furthermore, there exists an idempotent antipode $\CA \colon \CH_+\to \CH_+$,
satisfying the identity
\begin{equ}[e:defAnti]
\CM (I \otimes \CA) \Deltap \tau  = \scal{\one^*,\tau} \one = \CM (\CA \otimes I) \Deltap \tau\;,
\end{equ}
where we denoted by $\CM \colon \CH_+ \otimes \CH_+ \to \CH_+$ the multiplication operator
defined by $\CM(\tau \otimes \bar \tau) = \tau \bar \tau$,
and by $\one^*$ the element of $\CH_+^*$ such that $\scal{\one^*,\one} = 1$
and $\scal{\one^*,\tau} = 0$ for all $\tau \in \CF_+ \setminus \{\one\}$.
\end{theorem}

\begin{proof}
We first prove \eref{e:coassoc1}.
Both operators map $\one$ onto $\one \otimes \one \otimes \one$,
$\Xi$ onto $\Xi \otimes \one \otimes \one$, and $X_i$ onto $X_i \otimes \one \otimes \one + \one \otimes X_i \otimes \one + \one \otimes \one \otimes X_i$. Since $\CF$ is then generated by multiplication and action with $\CI_k$,
we can verify \eref{e:coassoc1} recursively by showing that it is stable under products and
applications of the integration maps.

Assume first that, for some $\tau$ and $\bar \tau$ in $\CF$, the identity \eref{e:coassoc1}
holds when applied to both $\tau$ and $\bar \tau$. By \eref{e:prodD1}, \eref{e:prodD2}, 
and the induction hypothesis, one then has
the identity
\begin{equs}
(I \otimes \Deltap) \Delta (\tau \bar \tau) &= (I\otimes \Deltap)(\Delta \tau \Delta \bar \tau)
= \bigl((I\otimes \Deltap)\Delta \tau \bigr) \bigl((I\otimes \Deltap)\Delta \bar \tau \bigr) \\
&= \bigl((\Delta\otimes I)\Delta \tau \bigr)\bigl((\Delta\otimes I)\Delta \bar \tau \bigr) 
= (\Delta\otimes I)(\Delta \tau \Delta \bar \tau) = (\Delta\otimes I)\Delta (\tau \bar \tau)\;,
\end{equs}
as required.

It remains to show that if \eref{e:coassoc1} holds for some $\tau \in \CF$, then it also holds
for $\CI_k \tau$ for every multiindex $k$. First, by \eref{e:int1} and \eref{e:int2}, 
one has the identity
\begin{equs}
(I\otimes \Deltap) \Delta \CI_k \tau &= (I\otimes \Deltap) (\CI_k \otimes I) \Delta \tau + \sum_{\ell,m} {X^\ell\over \ell!} \otimes \Deltap \Bigl({X^m \over m!} \CJ_{k+\ell+m} \tau\Bigr) \\
&= (\CI_k \otimes I \otimes I) (I\otimes \Deltap)\Delta \tau\\
&\qquad + \sum_{\ell,m,n} {X^\ell\over \ell!} \otimes \Bigl({X^m \over m!}\otimes {X^n \over n!} \Bigr)\Deltap \CJ_{k+\ell+m+n} \tau \;,\label{e:someExpression}
\end{equs}
where we used the multiplicative property of $\Deltap$ and the fact that 
\begin{equ}
\Deltap {X^k \over k!} = \sum_{m \le k}{X^m \over m!
}\otimes {X^{k-m} \over (k-m)!}\;.
\end{equ}
(Note again that the seemingly infinite sums appearing in \eref{e:someExpression}
are actually all finite since $\CJ_k \tau = 0$ for $k$ large enough. This will
be the case for every expression of this type appearing below.)
At this stage, we use the recursion relation \eref{e:int2} which yields
\begin{equs}
\sum_{m,n} \Bigl({X^m \over m!}&\otimes {X^n \over n!} \Bigr)\Deltap \CJ_{k+m+n} \tau
= \sum_{m,n} \Bigl({X^m \over m!}\otimes {X^n \over n!}\CJ_{k+m+n}\tau\Bigr) \\
&\qquad + \sum_{\ell, m,n} \Bigl({X^m \over m!}\CJ_{k+\ell+m+n}\otimes {X^n \over n!}{(-X)^\ell \over \ell!}\Bigr) \Delta \tau\\
&=\sum_{m,n} \Bigl({X^m \over m!}\otimes {X^n \over n!}\CJ_{k+m+n}\tau\Bigr) 
+ \sum_{m} \Bigl({X^m \over m!}\CJ_{k+m}\otimes I\Bigr) \Delta \tau\;.
\end{equs}
Here we made use of the fact that $\sum_{\ell+n = k}{X^n \over n!}{(-X)^\ell \over \ell!}$ always vanishes, except
when $k=0$ in which case it just yields $\one$. Inserting this in the above expression, we finally obtain the identity
\begin{equs}[e:idenTripleDelta]
(I\otimes \Deltap) \Delta \CI_k \tau &= (\CI_k \otimes I \otimes I) (I\otimes \Deltap)\Delta \tau \\
&+ \sum_{\ell,m,n} {X^\ell\over \ell!} \otimes {X^m \over m!}\otimes {X^n \over n!}\CJ_{k+\ell + m+n}\tau \\
&+ \sum_{\ell,m} {X^\ell\over \ell!} \otimes  \Bigl({X^m \over m!}\CJ_{k+\ell+m}\otimes I\Bigr) \Delta \tau \;.
\end{equs}

On the other hand, using again \eref{e:int1}, \eref{e:int2}, and the 
binomial identity, we obtain 
\begin{equs}
(\Delta\otimes I) \Delta \CI_k \tau &= (\Delta \CI_k \otimes I) \Delta \tau + \sum_{\ell,m} (\Delta\otimes I) \Bigl({X^\ell \over \ell!} \otimes  {X^m \over m!} \CJ_{k+\ell+m} \tau) \\
&= (\CI_k \otimes I \otimes I) (\Delta \otimes I) \Delta \tau + \sum_{\ell,m} {X^\ell \over \ell!} \otimes  \Bigl({X^m \over m!}\CJ_{k+\ell+m} \otimes I\Bigr)\Delta\tau \\
&\qquad + \sum_{\ell,m,n} {X^\ell\over \ell!} \otimes {X^m \over m!}\otimes {X^n \over n!}\CJ_{k+\ell + m+n}\tau\;.
\end{equs}
Comparing this expression with \eref{e:idenTripleDelta} and using the induction hypothesis, the claim follows at once.

We now turn to the proof of \eref{e:coassoc2}. Proceeding in a similar way as before, we verify that the
claim holds for $\tau = \one$, $\tau = X_i$, and $\tau = \Xi$. Using the fact that $\Deltap$ is a multiplicative morphism, it
follows as before that if \eref{e:coassoc2} holds for $\tau$ and $\bar \tau$, then it also holds for $\tau\bar\tau$.
It remains to show that it holds for $\CJ_k \tau$. One verifies,
similarly to before, that one has the identity
\begin{equs}
(\Deltap \otimes I) \Deltap  \CJ_k \tau &= \one \otimes \one \otimes \CJ_k \tau
+ \one \otimes \sum_{\ell} \Bigl(\CJ_{k+\ell} \otimes {(-X)^\ell \over \ell!}\Bigr) \Delta \tau \\
& \quad + \sum_{\ell,m} \Bigl(\CJ_{k+\ell+m} \otimes {(-X)^\ell\over \ell!} \otimes {(-X)^m\over m!}\Bigr) \bigl(\Delta \otimes I\bigr)\Delta \tau\;,
\end{equs}
while one also has
\begin{equs}
(I \otimes \Deltap) \Deltap  \CJ_k \tau &= \one \otimes \sum_{\ell} \Bigl(\CJ_{k+\ell} \otimes {(-X)^\ell \over \ell!}\Bigr) \Delta \tau + \one \otimes \one \otimes \CJ_k \tau\\
&\quad + \sum_{\ell,m} \Bigl(\CJ_{k+\ell+m} \otimes {(-X)^\ell\over \ell!} \otimes {(-X)^m\over m!}\Bigr) \bigl(I \otimes \Deltap\bigr)\Delta \tau\;.
\end{equs}
The claim now follows from \eref{e:coassoc1}.

It remains to show that $\CH_+$ admits an antipode $\CA \colon \CH_+ \to \CH_+$.
This is automatic for connected graded bialgebras but it turns out that in our case, although 
it admits a natural integer
grading, $\CH_+$ is not connected for it 
(i.e.\ there is more than one basis element with vanishing
degree). It is of course connected for the grading $|\cdot|_\s$, but this is not integer-valued.
The general construction of $\CA$ however still works in essentially the same way.
The natural integer grading $|\cdot|$ on $\CF_+$ for this purpose
is defined recursively by $|X_i| = |\Xi| = |\one| = 0$, and then $|\tau \bar \tau| = |\tau| + |\bar \tau|$
and $|\CJ_k \tau| = |\tau| + 1$. In plain terms, it counts the number of times that an integration
operator arises in the formal expression $\tau$. 

Recall that $\CA$ should be a linear map satisfying \eref{e:defAnti}, and we
furthermore want $\CA$ to be a multiplicative morphism namely, for 
$\tau = \tau_1\tau_2$, we impose that $\CA \tau = (\CA \tau_1)(\CA \tau_2)$.
To construct $\CA$, we start by setting
\begin{equ}[e:basicAnti]
\CA X_i = -X_i\;,\qquad  \CA \one = \one\;.
\end{equ}
Given the construction of $\CH_+$, it then remains to define $\CA$ on elements of the
type $\CJ_k \tau$ with $\tau \in \CH$ and $|\CJ_k \tau|_\s > 0$.
This should be done in such a way that one has
\begin{equ}[e:recAnti]
\CM (I \otimes \CA) \Deltap \CJ_k \tau = 0\;,
\end{equ}
which then guarantees that the first equality in \eref{e:defAnti} holds for all $\tau \in \CH_+$. 
This is because $\CM (I \otimes \CA) \Deltap$ is then a multiplicative morphism 
which vanishes on $X_i$ and every element of the form $\CJ_k \tau$,
and, except for $\tau = \one$, every element of $\CF_+$ has at least one such factor.

To show that it is possible to enforce \eref{e:recAnti} in a coherent way,
we proceed by induction. Indeed, by the definition of
$\Deltap$ and the definition of $\CM$, one has the identity
\begin{equ}
\CM (I \otimes \CA) \Deltap \CJ_k \tau = 
\sum_{\ell}  \CM\Bigl(\CJ_{k+\ell} \otimes {X^\ell\over \ell!} \CA\Bigr)\Delta \tau   + \CA \CJ_k \tau\;.
\end{equ}
Therefore, $\CA \CJ_k \tau$ is determined by \eref{e:recAnti} 
as soon as we know $(I \otimes \CA) \Delta \tau$. 
This can be guaranteed by iterating over $\CF$ in an order of increasing degree. 
(In the sense of the number of times that the integration operator appears in a formal expression,
as defined above.)

We can then show recursively that the antipode also satisfies $\CM (\CA \otimes I) \Deltap \tau = \one^*(\tau)\one$.
Again, we only need to verify it inductively on elements of the form $\CJ_k \tau$. One then has
\begin{equs}
\CM (\CA \otimes I) \Deltap \CJ_k \tau &= \CJ_k \tau + \sum_\ell {(-X)^\ell \over \ell!} \CM \bigl(\CA \CJ_{k+\ell} \otimes I\bigr) \Delta \tau \\
&= \CJ_k \tau - \sum_{\ell,m} {(-X)^\ell X^m \over \ell! m!} \CM \bigl(\CJ_{k+\ell+m} \otimes \CA \otimes I\bigr)\bigl(\Delta \otimes I\bigr) \Delta \tau \\
&= \CJ_k \tau - \CM \bigl(\CJ_{k} \otimes \CA \otimes I\bigr)\bigl(I \otimes \Deltap\bigr) \Delta \tau\;,
\end{equs}
where we used the fact that $\sum_{\ell + m = n} {(-X)^\ell X^m \over \ell! m!} = 0$ unless $n=0$ in which case it is $\one$.
At this stage, we use the fact that 
it is straightforward to verify inductively that 
\begin{equ}[e:propIdstar]
(I \otimes \one^*) \Delta \tau = \tau\;,
\end{equ}
for every $\tau \in \CH$, so that an application of our inductive hypothesis yields 
$\CM (\CA \otimes I) \Deltap \CJ_k \tau = \CJ_k \tau - \CJ_k \tau = 0$ as required.
The fact that $\CA^2 \tau = \tau$ can be verified in a similar way. It is also a consequence of 
the fact that the Hopf algebra $\CH_+$ is commutative \cite{Sweedler}.
\end{proof}

\begin{remark}
Note that $\CH$ is \textit{not} a Hopf module over $\CH_+$ since the identity
$\Delta (\tau \bar \tau) = \Delta \tau \,\Deltap \bar \tau$ does in general \textit{not}
hold for any $\tau \in \CH$ and $\bar \tau \in \CH_+$.
However, $\hat \CH = \CH \otimes \CH_+$ can be turned in a very natural way into a Hopf module over $\CH_+$.
The module structure is given by $(\tau \otimes \bar \tau_1)\bar \tau_2 = \tau \otimes (\bar \tau_1\bar \tau_2)$
for $\tau \in \CH$ and $\bar\tau_1, \bar \tau_2 \in \CH_+$, while the
comodule structure $\hat \Delta\colon \hat \CH \to \hat \CH \otimes \CH_+$ is given by
\begin{equ}
\hat \Delta (\tau \otimes \bar \tau) = \Delta \tau \cdot \Deltap \bar \tau\;,
\end{equ}
where $(\tau_1 \otimes \tau_2)\cdot (\bar \tau_1 \otimes \bar \tau_2) = (\tau_1 \otimes \bar \tau_1) \otimes (\tau_2 \bar \tau_2)$
for $\tau_1 \in \CH$ and $\tau_2, \bar \tau_1, \bar \tau_2 \in \CH_+$.
These structures are then compatible in the sense that $(\hat \Delta \otimes I)\hat \Delta = (I \otimes \Deltap)\hat \Delta$
and $\hat \Delta (\tau \bar \tau) = \hat \Delta \tau \cdot \Deltap \bar \tau$.
It is not clear at this stage
whether known general results on these structures (like the fact that Hopf modules are always free) 
can be of use for the type of analysis performed in this article. 
\end{remark}

We are now almost ready to construct the structure group $G$ in our context.
First, we define a product $\circ$ on $\CH_+^*$, the dual of $\CH_+$, by

\begin{definition}\label{def:convolution}
Given two elements $g, \bar g \in \CH_+^*$, their product $g\circ \bar g$ is given by
the dual of $\Deltap$, i.e., it is the element satisfying
\begin{equ}
\scal{g \circ \bar g,\tau} = \scal{g \otimes \bar g, \Deltap\tau}\;,
\end{equ}
for all $\tau \in \CH_+$.
\end{definition}

From now on, we will use the notations $\scal{g,\tau}$, $g(\tau)$, or even $g\tau$ 
interchangeably for the
duality pairing. We also identify $X \otimes \R$ with $X$ in the usual way ($x \otimes c \sim cx$) 
for any space $X$.
Furthermore, to any $g \in \CH_+^*$, we associate a linear map $\Gamma_g \colon \CH \to \CH$ in
essentially the same way as in \eref{e:actionG}, by setting
\begin{equ}[e:defGammag]
\Gamma_g \tau = (I \otimes g) \Delta \tau\;.
\end{equ}
Note that, by \eref{e:propIdstar}, one has $\Gamma_{\one^*}\tau  = \tau$.
One can also verify inductively that the co-unit $\one^*$ is indeed the neutral element for $\circ$.
With these definitions at hand, we have

\begin{proposition}
For any $g, \bar g \in \CH_+^*$, one has  
$\Gamma_g \Gamma_{\bar g} = \Gamma_{g \circ \bar g}$.
Furthermore, the product $\circ$ is associative.
\end{proposition}

\begin{proof}
One has the identity
\begin{equs}
\Gamma_g \Gamma_{\bar g} \tau &= \Gamma_g (I \otimes \bar g)\Delta \tau
= (I \otimes g \otimes \bar g)(\Delta \otimes I)\Delta \tau\\
&= (I \otimes g \otimes \bar g)(I \otimes \Deltap)\Delta \tau
= \bigl(I \otimes (g \circ \bar g)\bigr)\Delta \tau\;,
\end{equs}
where we first used Theorem~\ref{theo:projDelta} and then the definition of the product $\circ$.
The associativity of $\circ$ is equivalent to the coassociativity \eref{e:coassoc2} of $\Deltap$, which
we already proved in Theorem~\ref{theo:projDelta}.
\end{proof}

We now have all the ingredients in place to define the structure group $G$:

\begin{definition}\label{def:structGroup}
The group $G$ is given by the group-like elements $g \in \CH_+^*$, i.e.\ the elements such that
$g(\tau_1\tau_2) = g(\tau_1)\, g(\tau_2)$
for any $\tau_i \in \CH_+$. Its action on $\CH$ is given by $g \mapsto \Gamma_g$.
\end{definition}

This definition is indeed meaningful thanks to the following standard result:

\begin{proposition}\label{prop:antipode}
Given $g, \bar g \in G$, one has $g \circ \bar g \in G$. Furthermore,
each element $g \in G$ has a unique inverse $g^{-1}$.
\end{proposition}

\begin{proof}
This is standard, see \cite{Sweedler}. The explicit expression
for the inverse is simply $g^{-1}(\tau) = g(\CA \tau)$.
\end{proof}

Finally, we note that our operations behave well when restricting ourselves to the spaces
$\CH_F$ and $\CH_F^+$ constructed as explained previously by only considering those formal
expressions that are ``useful'' for the description of the nonlinearity $F$:

\begin{lemma}
One has $\Delta\colon \CH_F \to \CH_F \otimes \CH_F^+$ and 
$\Deltap \colon \CH_F^+ \to \CH_F^+ \otimes \CH_F^+$.
\end{lemma}

\begin{proof}
We claim that actually, even more is true. Recall the definitions of the sets $\CW_n$,
$\CU_n$ and $\CP^i_n$ from \eref{e:defSetsStruct} and denote by $\scal{\CW_n}$ the linear
span of $\CW_n$ in $\CH_F$, and similarly for $\scal{\CU_n}$ and $\scal{\CP^i_n}$.
Then, denoting by $\CX$ any of these vector spaces, we claim that  
$\Delta$ has the property that $\Delta \CX \subset \CX \otimes \CH_F^+$, which in particular then also
implies that
the action of $G$ leaves each of the spaces $\CX$ invariant.
This can easily be seen by induction over $n$.
The claim is clearly true for $n=0$ by definition. Assuming now that it holds for $\scal{\CU_{n-1}}$
and $\scal{\CP^i_{n-1}}$, it follows from the definition of $\CW_n$ and the morphism property of
$\Delta$ that the claim also holds for $\CW_n$. The identity \eref{e:int1} then also 
implies that the claim is true for $\scal{\CU_{n}}$ and $\scal{\CP^i_{n}}$, as required.

Regarding the property $\Deltap \colon \CH_F^+ \to \CH_F^+ \otimes \CH_F^+$, it follows from
the morphism property of $\Deltap$ (and the fact that $\CH_F^+$ itself is closed under multiplication)
that we only need to check it on elements $\tau$ of the form $\tau = \CI_k \bar \tau$ with
$\bar \tau \in \CF_F$. Using \eref{e:int2}, the claim then immediately follows from the first claim.
\end{proof}

\begin{remark}\label{rem:quotientGroup}
This shows that the action of $G$ onto $\CH_F$ is equivalent to the action of the quotient group
$G_F$ obtained by identifying elements that act in the same way onto $\CH_F^+$.
\end{remark}

This concludes our
construction of the regularity structure associated to a general subcritical semilinear (S)PDE, 
which we summarise as a theorem:

\begin{theorem}\label{theo:genStruct}
Let $F$ be a locally subcritical nonlinearity, let $T = \CH_F$ with $T_\gamma = \scal{\{\tau \in \CF_F\,:\, |\tau|_\s = \gamma\}}$, $A = \{|\tau|_\s \,:\, \tau \in \CF_F\}$, and $G_F$ be defined as above.
Then, $\TT_F = (A,\CH_F,G_F)$, defines a regularity structure $\TT$. 
Furthermore, $\CI$ is an abstract integration map of order $\beta$ for $\TT$.
\end{theorem}

\begin{proof}
To check that $\TT_F$ is a regularity structure, the 
only property that remains to be shown is \eref{e:coundGroup}. This however follows 
immediately from the fact that if one writes $\Delta \tau = \sum \tau^{(1)}\otimes \tau^{(2)}$,
then each of these terms satisfies $|\tau^{(1)}|_\s + |\tau^{(2)}|_\s = |\tau|_\s$
and $|\tau^{(2)}|_\s \ge 0$. Furthermore, one verifies by induction that the term
$\tau \otimes \one$ appears exactly once in this sum, so that for all other terms, $\tau^{(1)}$
is of homogeneity strictly smaller than that of $\tau$.

The map $\CI$ obviously satisfies the first two requirements of an abstract integration 
map by our definitions. The last property follows from the fact that 
\begin{equ}
\Gamma_g \CI_k \tau = (I\otimes g)\Delta \CI_k \tau  = (I\otimes g)(\CI_k \otimes I) \Delta \tau
+ \sum_\ell {(X - x_g)^\ell\over \ell!} g \bigl(\CJ_{k+\ell} \tau\bigr)\;,
\end{equ}
where we defined $x_g \in \R^d$ as the element with coordinates $-g(X_i)$.
Noting that $(I\otimes g)(\CI_k \otimes I) \Delta \tau = \CI_k \Gamma_g \tau$, the claim follows.
\end{proof}

\begin{remark}
If some element of $\M_F$ also contains a factor $\CP_i$, then one can check in the same
way as above that $\CI_i$ is an abstract integration map of order $\beta - \s_i$ for $\TT$. 
\end{remark}

\begin{remark}\label{rem:defTTr}
Given $F$ as above and $r>0$, we will sometimes write $\TT_F^{(r)}$
(or simply $\TT^{(r)}$ when $F$ is clear from the context) for the regularity structure
obtained as above, but with $T_\gamma = 0$ for $\gamma > r$.
\end{remark}

\subsection{Realisations of the general algebraic structure}
\label{sec:realAlg}

While the results of the previous subsection provide a systematic way of constructing a regularity
structure $\TT$ that is sufficiently rich to allow to reformulate \eref{e:genSPDE} as a fixed point
problem which has some local solution $U \in \CD^{\gamma,\eta}_P$ for suitable indices $\gamma$ and $\eta$, 
it does not at all address the 
problem of constructing a model (or family of models) $(\Pi,\Gamma)$ such that
$\CR U$ can be interpreted as a limit of classical solutions to some regularised version of \eref{e:genSPDE}.

It is in the construction of the model $(\Pi,\Gamma)$ that one has to take advantage of additional 
knowledge about $\xi$ (for example that it is Gaussian), which then allows to 
use probabilistic tools, combined with ideas from renormalisation theory, to build a ``canonical
model'' (or in many cases actually a canonical finite-dimensional family of models) 
associated to it. We will see in Section~\ref{sec:Gaussian} below how to do this in the 
particular cases of \eref{e:PAMGen} and \eref{e:Phi4}. For any continuous realisation of the driving noise 
however, it is straightforward to ``lift'' it to the regularity structure that we just built, 
as we will see presently.

Given any \textit{continuous} approximation $\xi_\eps$ to the driving noise $\xi$, we now show 
how one can build a canonical model $(\Pi^{(\eps)}, \Gamma^{(\eps)})$ for the regularity structure
$\TT$ built in the previous subsection. 
First, we set
\begin{equ}
\bigl(\Pi_x^{(\eps)} \Xi\bigr)(y) = \xi_\eps(y)\;,\quad \bigl(\Pi_x^\eps X^k\bigr)(y) = (y-x)^k\;.
\end{equ}
Then, we recursively define $\Pi_x^{(\eps)}\tau$ by
\begin{equ}[e:recursion]
\bigl(\Pi_x^{(\eps)} \tau \bar \tau\bigr)(y) = \bigl(\Pi_x^{(\eps)} \tau \bigr)(y)\,
\bigl(\Pi_x^{(\eps)} \bar \tau\bigr)(y)\;,
\end{equ}
as well as
\begin{equ}[e:defRecursion]
\bigl(\Pi_x^{(\eps)} \CI_k \tau\bigr)(y) = \int D^k_1 K(y,z)\,\bigl(\Pi_x^{(\eps)}\tau\bigr)(z)\,dz 
+ \sum_{\ell} {(y-x)^\ell \over \ell!} f_x^{(\eps)}\bigl(\CJ_{k+\ell} \tau\bigr) \;.
\end{equ}
In this expression, the quantities $f_x^{(\eps)}\bigl(\CJ_\ell \tau\bigr)$ are defined by
\begin{equ}[e:exprIntfx]
f_x^{(\eps)}\bigl( \CJ_\ell \tau\bigr) =  -\int D_1^\ell K(x,z)\,\bigl(\Pi_x^{(\eps)}\tau\bigr)(z)\,dz\;.
\end{equ}
If we furthermore impose that 
\begin{equ}[e:defFbasic]
f_x^{(\eps)} (X_i) = -x_i\;,\qquad f_x^{(\eps)} (\tau \bar \tau) = \bigl(f_x^{(\eps)}\tau\bigr)\bigl(f_x^{(\eps)} \bar \tau\bigr)\;,
\end{equ}
and extend this to all of $\CH_F^+$ by linearity,
then $f_x^{(\eps)}$ defines an element of the group $G_F$ given in 
Definition~\ref{def:structGroup} and Remark~\ref{rem:quotientGroup}. 

Denote by $F_x^{(\eps)}$ the corresponding linear operator on $\CH_F$, i.e.\ $F_x^{(\eps)} = \Gamma_{\!f_{x}^{(\eps)}}$ where the map
$g \mapsto \Gamma_g$ is given by \eref{e:defGammag}.
With these definitions at hand, we then define $\Gamma_{xy}^{(\eps)}$ by 
\begin{equ}[e:defGammaxy]
\Gamma_{xy}^{(\eps)} = \bigl(F_x^{(\eps)}\bigr)^{-1} \circ F_y^{(\eps)}\;.
\end{equ}
Furthermore, for any $\tau \in \CF$, we denote by $V_\tau$ the sector 
given by the linear span of $\{\Gamma \tau\,:\, \Gamma \in G\}$.
This is also given by the projection of $\Delta \tau$ onto its first factor.
We then have:

\begin{proposition}\label{prop:canonicReal}
Let $K$ be as in Lemma~\ref{lem:diffSing} and satisfying Assumption~\ref{ass:polynom} for
some $r > 0$.
Let furthermore $\TT^{(r)}_F$ be the regularity structure obtained 
from any semilinear locally subcritical problem as in Section~\ref{sec:genAlg} and Remark~\ref{rem:defTTr}.
Let finally $\xi_\eps\colon \R^d \to \R$ be a smooth function and let
$(\Pi^{(\eps)},\Gamma^{(\eps)})$ be defined as above.
Then, $(\Pi^{(\eps)},\Gamma^{(\eps)})$ is a model for $\TT^{(r)}_F$.

Furthermore, for any $\tau \in \CF_F$ such that $\CI_k \tau \in \CF_F$, the model 
$(\Pi^{(\eps)},\Gamma^{(\eps)})$
realises the abstract integration operator $\CI_k$ on the sector $V_\tau$.
\end{proposition}

\begin{proof}
We need to verify both the algebraic relations and the analytical bounds of Definition~\ref{def:model}.
The fact that $\Gamma_{xy}^{(\eps)} \Gamma_{yz}^{(\eps)} = \Gamma_{xz}^{(\eps)}$ is immediate from the definition \eref{e:defGammaxy}.
In view of \eref{e:defGammaxy}, the identity $\Pi_x^{(\eps)} \Gamma_{xy}^{(\eps)} = \Pi_y^{(\eps)}$
follows if we can show that
\begin{equ}[e:wantedAlg]
\Pi_x^{(\eps)} \bigl(F_{x}^{(\eps)}\bigr)^{-1}\tau = \Pi_y^{(\eps)} \bigl(F_{y}^{(\eps)}\bigr)^{-1}\tau\;,
\end{equ}
for every $\tau \in \CF_F$ and any two points $x$ and $y$. In order to show that this is the case, it turns
out that it is easiest to simply ``guess'' an expression for $\Pi_x^{(\eps)} \bigl(F_{x}^{(\eps)}\bigr)^{-1}\tau$
that is independent of $x$ and to then verify recursively that our guess was correct. For this,
we define a linear map $\PPi^{(\eps)}\colon \CH_F \to \CC(\R^d)$ by
\begin{equ}
\bigl(\PPi^{(\eps)} \one\bigr)(y) = 1\;,\qquad \bigl(\PPi^{(\eps)} X_i\bigr)(y) = y_i\;, \qquad
\bigl(\PPi^{(\eps)} \Xi\bigr)(y) = \xi_\eps(y)\;,
\end{equ}
and then recursively by
\begin{equ}[e:recPi]
\bigl(\PPi^{(\eps)} \tau \bar \tau\bigr)(y) = \bigl(\PPi^{(\eps)} \tau \bigr)(y)\,
\bigl(\PPi^{(\eps)} \bar \tau\bigr)(y)\;,
\end{equ}
as well as
\begin{equ}[e:RecPiInt]
\bigl(\PPi^{(\eps)} \CI_k \tau\bigr)(y) = \int D^k_1 K(y,z)\,\bigl(\PPi^{(\eps)}\tau\bigr)(z)\,dz \;.
\end{equ}
We claim that one has $\Pi_x^{(\eps)} \bigl(F_{x}^{(\eps)}\bigr)^{-1}\tau = \PPi^{(\eps)}\tau$ for
every $\tau \in \CF_F$ and every $x \in \R^d$. Actually, it is easier to verify the equivalent identity 
\begin{equ}[e:wantedId]
\Pi_x^{(\eps)} \tau = \PPi^{(\eps)} F_{x}^{(\eps)} \tau\;.
\end{equ}

To show this, we proceed by induction. 
The identity obviously holds for $\tau = \Xi$ and $\tau = \one$.
For $\tau = X_i$,
we have by \eref{e:defFbasic}
\begin{equ}
\bigl(\PPi^{(\eps)} F_{x}^{(\eps)} X_i\bigr)(y) = \bigl((\PPi^{(\eps)} \otimes f_x^{(\eps)})(X_i \otimes \one + \one \otimes X_i)\bigr)(y) = y_i - x_i = \bigl(\Pi_x^{(\eps)} X_i\bigr)(y)\;.
\end{equ}
Furthermore, in view of \eref{e:recPi}, \eref{e:recursion}, and the fact that $F_{x}^{(\eps)}$ acts
as a multiplicative morphism, it holds for $\tau \bar \tau$ if it holds for both $\tau$ and $\bar \tau$.

To complete the proof of \eref{e:wantedAlg}, it remains to show that 
\eref{e:wantedId} holds for elements of the form $\CI_k \tau$ if it holds for $\tau$.
It follows from the definitions that
\begin{equs}[e:abstractVersionInt]
F_{x}^{(\eps)}\CI_k \tau &= \CI_k F_x^{(\eps)} \tau + \sum_{\ell,m} {X^\ell \over \ell!} \, f_x^{(\eps)} \Bigl({X^m \over m!}\CJ_{k+\ell+m}\tau\Bigr) \\
&=  \CI_k F_x^{(\eps)} \tau + \sum_{\ell,m} {X^\ell \over \ell!} {(-x)^m \over m!} \, f_x^{(\eps)} \Bigl(P\CJ_{k+\ell+m}\tau\Bigr) \\
&=  \CI_k F_x^{(\eps)} \tau + \sum_{\ell} {(X-x)^\ell \over \ell!} f_x^{(\eps)} \bigl(\CJ_{k+\ell}\tau\bigr) \;,
\end{equs}
where we used \eref{e:defFbasic}, the morphism property of $f_x^{(\eps)}$, and the binomial identity.
The above identity is precisely
the abstract analogue in this context of the identity postulated in Definition~\ref{def:reprK}.

Inserting this into \eref{e:RecPiInt}, we obtain the identity
\begin{equ}[e:PixF]
\bigl(\PPi^{(\eps)} F_{x}^{(\eps)}\CI_k \tau\bigr)(y) = \int D_1^k K(y,z)\,\bigl(\PPi^{(\eps)}F_{x}^{(\eps)} \tau\bigr)(z)\,dz
+ \sum_{\ell} {(y-x)^\ell \over \ell!} f_x^{(\eps)} \bigl(\CJ_{k+\ell}\tau\bigr)\;. 
\end{equ}
Since $\PPi^{(\eps)}F_{x}^{(\eps)} \tau = \Pi_x^{(\eps)}\tau$ by our induction hypothesis, 
this is precisely equal to the right hand side of \eref{e:defRecursion}, as required. 

It remains to show that the required analytical bounds also hold. Regarding $\Pi_x^{(\eps)}$, we actually 
show the 
slightly stronger fact that $(\Pi_x^{(\eps)} \tau)(y) \lesssim \|x-y\|_\s^{|\tau|_\s}$. This is obvious
for $\tau = X_i$ as well as for $\tau = \Xi$ since $|\Xi|_\s < 0$ and we 
assumed that $\xi_\eps$ is continuous. (Of course, such a bound would typically not hold
uniformly in $\eps$!)
Since $|\tau \bar \tau|_\s = |\tau|_\s + |\bar \tau|_\s$, it is also obvious that such a bound holds for $\tau \bar \tau$
if it holds for both $\tau$ and $\bar \tau$.
Regarding elements of the form $\CI_k \tau$, we  note that the second term in \eref{e:defRecursion}
is precisely the truncated Taylor series of the first term, so that the required bound holds 
by Proposition~\ref{prop:Taylor} or, more generally, by Theorem~\ref{theo:extension}. 
To conclude the proof that $(\Pi^{(\eps)}, \Gamma^{(\eps)})$ is a model for our regularity structure, 
it remains to obtain a bound of the type \eref{e:boundPi} for $\Gamma^{(\eps)}_{xy}$. In principle, this
also follows from Theorem~\ref{theo:extension}, but we can also verify it more explicitly in this case.

Note that the required bound follows if we can show that
\begin{equ}
|\gamma_{xy}^{(\eps)} (\tau)| \eqdef \bigl| \bigl(f_x \CA \otimes f_y\bigr)\Deltap \tau \bigr| \lesssim \|x-y\|_\s^{|\tau|_\s}\;,
\end{equ}
for all $\tau \in \CF_F^+$ with $|\tau|_\s \le r$. Again, this can easily be checked for $\tau = X^k$.
For $\tau = \CJ_k \bar \tau$, note that one has the identity
\begin{equ}
\bigl(\CA \otimes I\bigr) \Deltap \CJ_k \bar\tau = 
\one \otimes \CJ_k \bar\tau - \sum_{\ell,m} \bigl(\CM \otimes I\bigr) \Bigl(\CJ_{k+\ell+m}\otimes {X^\ell \over \ell!}\CA
\otimes {(-X)^m \over m!}\Bigr) \bigl(I \otimes \Deltap\bigr) \Delta \bar\tau.
\end{equ}
As a consequence, we have the identity
\begin{equ}
\gamma_{xy}^{(\eps)} (\CJ_k \bar \tau)
= f_y^{(\eps)}(\CJ_k \bar \tau) - \sum_{\ell} {(y-x)^\ell \over \ell!} f_x^{(\eps)}(\CJ_{k+\ell} \Gamma_{xy}^{(\eps)} \bar \tau)\;.
\end{equ}
It now suffices to realise that this is equal to the quantity $\bigl(\Gamma_{yx}^{(\eps)} \CJ_{xy} \bar \tau\bigr)_k$, where 
$\CJ_{xy}$ was introduced in \eref{e:defJxy}, so that the required bound follows from Lemma~\ref{lem:boundGammaxy}. There is an unfortunate notational clash between $\CJ_{xy}$ and $\CJ_k$
appearing here, but since this is the only time in the article that both objects appear
simultaneously, we leave it at that.

The fact that the model built in this way realises $K$ for the abstract integration operator $\CI$
(and indeed for any of the $\CI_k$) follows at once from the definition \eref{e:defRecursion}. 
\end{proof}

\begin{remark}
In general, one does not even need $\xi_\eps$ to be continuous. One just needs it to be in $\CC^\alpha_\s$
for sufficiently large (but possibly negative) $\alpha$ such that all the products appearing in the above
construction satisfy the conditions of Proposition~\ref{prop:multDist}.
\end{remark}

This construction motivates the following definition, where we assume that the kernel $K$
annihilates monomials up to order $r$ and that we are given a regularity structure $\TT_F$ built
from a locally subcritical nonlinearity $F$ as above.

\begin{definition}\label{def:admissible}
A model $(\Pi,\Gamma)$ for $\TT_F^{(r)}$ is \textit{admissible} if it satisfies
$(\Pi_x X^k)(y) = (y-x)^k$, as well as \eref{e:defRecursion}, \eref{e:defFbasic}, \eref{e:defGammaxy},
and \eref{e:exprIntfx}. We denote by $\MM_F$ the set of admissible models.
\end{definition}

Note that the set of admissible models is a closed subset of the set of all models and that
the models built from canonical lifts of smooth functions $\xi^{(\eps)}$ are 
admissible by definition. Admissible models are also adapted to the integration map $K$
(and suitable derivatives thereof) for the integration map $\CI$ (and the maps $\CI_k$ if
applicable). Actually, the converse is also true provided that we \textit{define}
$f$ by \eref{e:exprIntfx}. This can be shown by a suitable recursion procedure, but
since we will never actually use this fact we do not provide a full proof.

\begin{remark}
It is not clear in general whether canonical lifts of smooth functions
are dense in $\MM_F$. As the definitions stand, this will actually never be the case
since smooth functions are not even dense in $\CC^\alpha$! This is however an artificial
problem that can easily be resolved in a manner similar to what we did in the proof of the
reconstruction theorem, Theorem~\ref{theo:reconstruction}. (See also the note \cite{MR2257130}.)
However, even when allowing for some weaker notion of density, it will 
in general \textit{not} be the case that lifts of smooth functions are dense. 
This is because the regularity structure $\TT_F^{(r)}$ built in this section does
not encode the Leibniz rule, so that it can accommodate the type of effects described
in \cite{Trees,Jan,DavidK} (or even just It\^o's formula in one dimension) which cannot arise when only
considering lifts of smooth functions. 
\end{remark}

\subsection{Renormalisation group associated to the general algebraic structure}
\label{sec:renormalisation}

There are many situations where, if we take for $\xi_\eps$ a smooth approximation to $\xi$
such that $\xi_\eps \to \xi$ in a suitable sense, the sequence $(\Pi^{(\eps)},\Gamma^{(\eps)})$
of models built from $\xi_\eps$ as in the previous section fails to converge.
This is somewhat different from the situation encountered in the context of the theory of rough paths
where natural smooth approximations to the driving noise very often do yield 
finite limits without the need for renormalisation \cite{MR1883719,MR2667703}.
(The reason why this is so stems from the fact that if a process $X$ is symmetric
under time-reversal, then the expression $X_i \d_t X_j$ is antisymmetric, thus introducing additional
cancellations. Recall the discussion on the role of symmetries
in Remark~\ref{rem:symmetry}.)

In general, in order to actually build a model associated to the driving noise $\xi$, we will 
need to be able to encode some kind of 
renormalisation procedure. In the context of the regularity structures built in this section,
it turns out that they come equipped with a natural group of continuous transformations on 
their space of admissible models. At the abstract level, this group of transformations 
(which we call $\RR$) will be nothing but a finite-dimensional 
nilpotent Lie group -- in many instances just a copy of $\R^n$ for some $n > 0$. 
As already mentioned in the introduction, a 
renormalisation procedure then consists in finding a sequence
$M_\eps$ of elements in $\RR$ such that $M_\eps(\Pi^{(\eps)}, \Gamma^{(\eps)})$ converges to a finite
limit $(\Pi,\Gamma)$, 
where $(\Pi^{(\eps)}, \Gamma^{(\eps)})$ is the bare model built in Section~\ref{sec:realAlg}.
As previously, different renormalisation procedures yield limits that differ only by an element in $\RR$.

\begin{remark}
The construction outlined in this section, and indeed the whole methodology presented here,
has a flavour that is strongly reminiscent of the theory given in \cite{CK,CK2}. 
The scope however is different: the construction presented here applies to subcritical situations
in which one obtains superrenormalisable theories, so that the group $\RR$ is always finite-dimensional.
The construction of \cite{CK,CK2} on the other hand applies to critical situations and 
yields an infinite-dimensional renormalisation group.
\end{remark}

Assume that we are given some model $(\Pi, \Gamma)$ for our regularity structure $\TT$. As before, we assume that $\Gamma_{xy}$ is provided to us in the form
\begin{equ}[e:defGammaF]
\Gamma_{xy} = F_x^{-1}\circ F_y\;,
\end{equ}
and we denote by $f_x$ the group-like element in the dual of $\CH_F^+$ corresponding to $F_x$.
As a consequence, the operator $\Pi_x F_x^{-1}$ is independent of $x$ and, as in Section~\ref{sec:realAlg}, we will 
henceforth denote it simply by
\begin{equ}[e:defPi]
\PPi \eqdef \Pi_x F_x^{-1}\;.
\end{equ}
Throughout this whole section, we will thus represent a model by the pair $(\PPi, f)$ where $\PPi$
is one single linear map $\PPi \colon T \to \CS'(\R^d)$ and $f$ is a map on $\R^d$ with values
in the morphisms of $\CH_F^+$.

We furthermore make the additional assumption that our model is admissible, 
so that one has the identities
\begin{equs}
\PPi \CI_k \tau &= \int_{\R^d} D^k K(\cdot, y)\, \bigl(\PPi \tau\bigr)(dy)\;, \label{e:idPi}\\
f_x \CJ_k \tau &= -\int_{\R^d} D^k K(x, y)\, \bigl(\Pi_x \tau\bigr)(dy)\;,\label{e:idF}
\end{equs}
where, in view of \eref{e:defPi}, $\PPi$ and $\Pi_x$ are related by
\begin{equ}
\PPi \tau = (\Pi_x \otimes f_x \CA)\Delta \tau\;,\qquad
\Pi_x \tau = (\PPi \otimes f_x)\Delta \tau\;.
\end{equ}
Note that by definition, \eref{e:idF} only ever applies to elements with $|\CJ_k\tau|_\s > 0$,
which implies that the corresponding integral actually makes sense.
In view of \eref{e:int1} and \eref{e:defIa}, this ensures that our model 
does realise $K$ for the abstract integration operator $\CI$ 
(and, if needed, the relevant derivatives of $K$ for the $\CI_k$).
It is crucial that any transformation that we would like to apply to our model
preserves this property, since otherwise the operators $\CK_\gamma$ cannot be constructed
anymore for the new model. 

\begin{remark}
While it is clear that $(\PPi, f)$ is sufficient to determine the corresponding model by \eref{e:defGammaF}
and \eref{e:defPi},
the converse is not true in general if one only imposes \eref{e:defGammaF}. However, if
we also impose \eref{e:idF}, together with the canonical choice $f_x(X) = -x$, then $f$ is uniquely
determined by the model in its usual representation $(\Pi, \Gamma)$.
This shows that although the transformations constructed in this section will be given in
terms of $f$, they do actually define maps defined on the set $\MM_F$ of all admissible models.
\end{remark}

The important feature of $\RR$ is its action on elements $\tau$ of negative
homogeneity. It turns out that, in order to describe it, it is convenient to work on a slightly
larger set $\CF_0 \subset \CF_F$ with some additional properties.
Given any set $\CC \subset \CF_F$, we\label{lab:AlgC} will henceforth denote by 
$\Alg(\CC) \subset \CF_F^+$ the
set of all elements in $\CF_F^+$ of the form $X^k \prod_i \CJ_{\ell_i} \tau_i$, for some
multiindices $k$ and $\ell_i$ such that $|\CJ_{\ell_i} \tau_i|_\s > 0$, and where 
the elements $\tau_i$ all belong to $\CC$.
(The empty product also counts, so that one always has $X^k \in \Alg(\CC)$ and in particular
$\one \in \Alg(\CC)$.) We will also use the notation $\scal{\CC}$ for the linear span of a set $\CC$.
We now fix a subset $\CF_0 \subset \CF_F$ as follows.

\begin{assumption}\label{ass:SetF0}
The set $\CF_0 \subset \CF_F$ has the following properties:
\begin{claim}
\item The set $\CF_0$ contains every $\tau \in \CF_F$ with $|\tau|_\s \le 0$.
\item There exists $\CF_\star\subset \CF_0$
such that, for every $\tau \in \CF_0$, one has 
$\Delta \tau \in \scal{\CF_0}\otimes \scal{\Alg(\CF_\star)}$.
\end{claim}
\end{assumption}

\begin{remark}\label{rem:stableF+}
Similarly to before, we write $\CH_0 = \scal{\CF_0}$, $\CF_0^+ = \Alg(\CF_\star)$, and
 $\CH_0^+ = \scal{\CF_0^+}$. 
Proceeding as in the proofs of Lemmas~\ref{lem:mapDelta} 
and \ref{lem:antiFn} below, one can verify that the second condition automatically
implies that the operators $\Deltap$ and $\CA$ both leave $\CH_0^+$ invariant.
\end{remark}

Let now $M \colon \CH_0 \to \CH_0$
be a linear map such that $M \CI_k \tau = \CI_k M\tau$ for every $\tau \in \CF_0$ such that $\CI_k \tau \in \CF_0$.
Then, we would like to use the map $M$ to build a new model $(\PPi^M, f^M)$ with the property that
\begin{equ}[e:defPiM]
\PPi^M \tau = \PPi M \tau\;.
\end{equ}
(The condition $M \CI_k \tau = \CI_k M\tau$ is required to guarantee that \eref{e:idPi} still holds for $\PPi^M$.)
This is not always possible, but the aim of this section is to provide conditions under which it is.
In order to realise the above identity, we would like to build linear maps $\DeltaM\colon \CH_0 \to \CH_0 \times \CH_0^+$
and $\hat M \colon \CH_0^+\to \CH_0^+$ such that one has
\begin{equ}[e:constrModel]
\Pi^M_x \tau = \bigl(\Pi_x \otimes f_x) \DeltaM \tau\;,\qquad f_x^M \tau = f_x \hat M \tau\;.
\end{equ}

\begin{remark}
One might wonder why we choose to make the ansatz \eref{e:constrModel}. The first identity
really just states that $\Pi^M_x \tau$ is given by a bilinear expression of the type
\begin{equ}
\Pi^M_x \tau = \sum_{\tau_1, \tau_2} C_\tau^{\tau_1,\tau_2} f_x(\tau_1) \, \Pi_x \tau_2\;,
\end{equ}
which is not unreasonable since the objects appearing on the right hand side are
the only objects available as ``building blocks'' for our construction. One might argue that
the coefficients could be given by some polynomial expression in the $f_x(\tau_1)$, but thanks
to the fact that $f_x$ is group-like, this can always be reformulated as a linear expression.
Similarly, the second expression simply states that $f_x^M$ is given by some arbitrary linear
(or polynomial by the same argument as before) expression in the $f_x$.
\end{remark}

Furthermore, we would like to ensure that if the pair $(\PPi,f)$ satisfies the identities
\eref{e:idPi} and \eref{e:idF}, then the pair $(\PPi^M, f^M)$ also
satisfies them. Inserting \eref{e:constrModel} into \eref{e:idF}, we see that this is guaranteed if we impose that
\minilab{e:defMM}
\begin{equ}[e:defMM2]
\hat M \CJ_k = \CM(\CJ_k \otimes I) \DeltaM\;,
\end{equ}
where, as before, $\CM \colon \CH_0^+ \times \CH_0^+ \to \CH_0^+$ denotes the multiplication map.
We also note that if we want to ensure that \eref{e:defPiM} holds, then we should require that, for
every $x \in \R^d$, one has the identity $\PPi M = \Pi_x^M \bigl(F_x^M\bigr)^{-1}$, which we
rewrite as $\Pi_x^M = \PPi M F_x^M $.
Making use of the first identity of \eref{e:constrModel} and of the fact that $\Pi_x = \PPi F_x$, 
the left hand side of this identity can be expressed as
\begin{equ}
\Pi_x^M \tau = \bigl(\PPi \otimes f_x \otimes f_x) (\Delta \otimes I)\DeltaM \tau
= \bigl(\PPi \otimes f_x) (I\otimes \CM) (\Delta \otimes I)\DeltaM \tau\;.
\end{equ} 
Using the second identity of \eref{e:constrModel}, the right hand side on the other 
hand can be rewritten as
\begin{equ}
\PPi M F_x^M\tau = (\PPi \otimes f_x^M) (M \otimes I)\Delta \tau
= (\PPi \otimes f_x) (M \otimes \hat M)\Delta \tau\;.
\end{equ}
We see that these two expressions
are guaranteed to be equal for any choice of $\PPi$ and $f_x$ if we impose the consistency condition
\minilab{e:defMM}
\begin{equ}[e:defMM3]
(I \otimes \CM)(\Delta \otimes I)\DeltaM = (M \otimes \hat M) \Delta\;.
\end{equ}
Finally, we impose that $\hat M$ is a multiplicative morphism and that it leaves $X^k$ invariant, namely that
\minilab{e:defMM}
\begin{equ}[e:defMM1]
\hat M(\tau_1\tau_2) = (\hat M \tau_1)(\hat M \tau_2)\;,\qquad \hat M X^k = X^k\;,
\end{equ}
which is a natural condition given its
interpretation. In view of \eref{e:constrModel}, this 
is required to  ensure that $f_x^M$ is again a group-like element with $f_x^M(X_i) = -x_i$, which is crucial
for our purpose.
It then turns out that equations \eref{e:defMM2}--\eref{e:defMM1} are sufficient to uniquely characterise
$\DeltaM$ and $\hat M$ and that it is always possible to find two operators satisfying these constraints: 

\begin{proposition}\label{prop:constMM}
Given a linear map $M$ as above, there exists a unique choice of $\hat M$ and $\DeltaM$ 
satisfying \eref{e:defMM2}--\eref{e:defMM1}.
\end{proposition}

In order to prove this result, it turns out to be convenient to consider the following 
recursive construction of elements in $\CH_F$.
We define  $\CF^{(0)} = \emptyset$ and then, recursively,
\begin{equ}[e:defFn]
\CF^{(n+1)} = \bigl\{\tau \in \CF_F\,:\, \Delta \tau \in \CH_F \otimes \scal{\Alg(\CF^{(n)})}\bigr\}\;.
\end{equ}

\begin{remark}
In practice, a typical choice for the set $\CF_0$ of Assumption~\ref{ass:SetF0} 
is to take $\CF_0 = \CF^{(n)}$ and $\CF_\star = \CF^{(n-1)}$ 
for some sufficiently large $n$, which then automatically has the required properties
by Lemma~\ref{lem:mapDelta} below. In particular, this also shows that such sets do exist.
\end{remark}

For example, $\CF^{(1)}$ contains all elements of the form $\Xi^n X^k$ that belong to $\CF_F$,
but it might contain more than that depending on the values of $\alpha$ and $\beta$.
The following properties of these sets are elementary:
\begin{claim}
\item One has $\CF^{(n-1)} \subset \CF^{(n)}$. This is shown by induction. For $n = 1$, the statement 
is trivially true. If it holds for some $n$ then one has $\Alg(\CF^{(n-1)}) \subset \Alg(\CF^{(n)})$ and
so, by \eref{e:defFn}, one also has $\CF^{(n)} \subset \CF^{(n+1)}$, as required.
\item If $\tau, \bar \tau \in \CF^{(n)}$ are such that $\tau \bar \tau \in \CF_F$, then 
$\tau \bar \tau \in \CF^{(n)}$ as an immediate consequence of the morphism property of $\Delta$, combined
with the definition of $\Alg$.
\item If $\tau \in \CF^{(n)}$ and $k$ is such that $\CI_k \tau \in \CF_F$, then $\CI_k \tau \in \CF^{(n+1)}$.
As a consequence of this fact, and since all elements in $\CF_F$ can be generated by multiplication and
integration from $\Xi$ and the $X_i$, one has $\bigcup_{n \ge 0} \CF^{(n)} = \CF_F$.
\end{claim}

The following consequence is slightly less obvious:

\begin{lemma}\label{lem:mapDelta}
For every $n \ge 0$ and $\tau \in \CF^{(n)}$, one has $\Delta \tau \in \scal{\CF^{(n)}} \otimes \scal{\Alg(\CF^{(n-1)})}$.
For every $n \ge 0$ and $\tau \in \Alg(\CF^{(n)})$, one has $\Deltap \tau \in \scal{\Alg(\CF^{(n)})} \otimes \scal{\Alg(\CF^{(n)})}$.
\end{lemma}

\begin{proof}
We proceed by induction. For $n = 0$, both statements are trivially true, so we assume that they hold for
all $n \le k$. Take then $\tau \in \CF^{(k+1)}$ and assume by contradiction that
$\Delta \tau \not \in  \scal{\CF^{(k+1)}} \otimes \scal{\Alg(\CF^{(k)})}$.
This then implies that $(\Delta \otimes I) \Delta \tau \not \in \CH_F\otimes \scal{\Alg(\CF^{(k)})} \otimes \scal{\Alg(\CF^{(k)})}$. However, we have $(\Delta \otimes I) \Delta \tau = (I \otimes \Deltap) \Delta \tau$
by Theorem~\ref{theo:projDelta} and $\Deltap$ maps $\scal{\Alg(\CF^{(k)})}$ to $\scal{\Alg(\CF^{(k)})} \otimes \scal{\Alg(\CF^{(k)})}$ by our induction hypothesis, thus yielding the required contradiction.

It remains to show that $\Deltap$ has the desired property for $n = k+1$. 
Since $\Deltap$ is a multiplicative morphism, 
we can assume that $\tau$ is of the form $\tau = \CJ_\ell \bar \tau$ with $\bar \tau \in \CF^{(k+1)}$.
One then has by definition
\begin{equ}
\Deltap \tau = \sum_{m} \Bigl(\CJ_{\ell+m} \otimes {(-X)^m\over m!}\Bigr)  \Delta \bar \tau + \one \otimes \tau\;.
\end{equ}
By the first part of the proof, we already know that $\Delta \bar \tau \in \scal{\CF^{(k+1)}} \otimes \scal{\Alg(\CF^{(k)})}$, so that the first term belongs to $\scal{\Alg(\CF^{(k+1)})} \otimes \scal{\Alg(\CF^{(k)})}$.
The second term on the other hand belongs to $\scal{\Alg(\CF^{(0)})} \otimes \scal{\Alg(\CF^{(k+1)})}$ by definition,
so that the claim follows.
\end{proof}

A useful consequence of Lemma~\ref{lem:mapDelta} is the following.

\begin{lemma}\label{lem:antiFn}
If $\tau \in \Alg(\CF^{(n)})$ for some $n \ge 0$, then $\CA\tau \in \scal{\Alg(\CF^{(n)})}$, where $\CA$ is the
antipode in $\CH_+$ defined in the previous subsection.
\end{lemma}

\begin{proof}
The proof goes by induction, using the relations $\CA(\tau \bar \tau) = \CA(\tau)\,\CA(\bar \tau)$,
as well as the identity
\begin{equ}[e:relIndA]
\CA \CJ_k \bar\tau = - \sum_{\ell}  \CM\Bigl(\CJ_{k+\ell} \otimes {X^\ell\over \ell!} \CA\Bigr)\Delta \bar\tau\;,
%
%
\end{equ} 
which is valid as soon as $|\CJ_k \bar\tau|_\s > 0$. For $n = 0$, the claim is trivially true.
For arbitrary $n > 0$, by the multiplicative property of $\CA$, it suffices to consider the case 
$\tau = \CJ_k \bar \tau$ with $\bar \tau \in \CF^{(n)}$.
Since $\Delta \bar\tau \in \scal{\CF^{(n)}} \otimes \scal{\Alg(\CF^{(n-1)})}$
by Lemma~\ref{lem:mapDelta}, it follows from our definitions and the inductive assumption that
the right hand side of \eref{e:relIndA} does indeed belong to $\scal{\Alg(\CF^{(n)})} \otimes \scal{\Alg(\CF^{(n)})}$ as required.
\end{proof}

We now have all the ingredients in place for the 

\begin{proof}[of Proposition~\ref{prop:constMM}]
We first introduce the map $D \colon \CH_0 \otimes \CH_0^+ \to \CH_0 \otimes \CH_0^+$ 
given by $D = (I \otimes \CM)(\Delta \otimes I)$. It follows immediately from the definition of $\Delta$
and the fact that, by Lemma~\ref{lem:finDim}, homogeneities of elements in $\CF_F$ (and a fortiori
of elements in $\CF_0$) are bounded from below, that
$D$ can be written as 
\begin{equ}
D(\tau\otimes \bar \tau) = \tau \otimes \bar \tau - \bar D(\tau \otimes \bar \tau)\;,
\end{equ}
for some nilpotent map $\bar D$. As a consequence, $D$ is invertible with inverse given by
the Neumann series $D^{-1} = \sum_{k \ge 0} \bar D^k$, which is always finite.

The proof of the statement then goes by induction over $\CF^{(n)} \cap \CF_0$.
Assume that $\hat M$ and $\DeltaM$ are uniquely defined on $\Alg(\CF^{(n)}\cap\CF_\star)$ and
on $\CF^{(n)} \cap \CF_0$ respectively which,
by \eref{e:defMM1},  is trivially true for $n =0$. (For $\DeltaM$ this is also trivial since $\CF^{(0)}$
is empty.)
Take then $\tau \in  \CF^{(n+1)} \cap \CF_0$.
By \eref{e:defMM3}, one has
\begin{equ}
\DeltaM \tau = D^{-1} (M\otimes \hat M) \Delta \tau\;.
\end{equ} 
By Lemma~\ref{lem:mapDelta} and Remark~\ref{rem:stableF+}, 
the second factor of $\Delta \tau$ belongs to $\scal{\Alg(\CF^{(n)}\cap\CF_\star)}$
on which $\hat M$ is already known by assumption, so that this uniquely determines
$\DeltaM \tau$.

On the other hand, in order to determine $\hat M$ on elements of $\Alg(\CF^{(n+1)}\cap\CF_\star)$ it suffices
by \eref{e:defMM1} and Remark~\ref{rem:stableF+} to determine it on elements of the form 
$\tau = \CJ_k \bar \tau$ with $\bar \tau\in \CF^{(n+1)} \cap \CF_\star$. The action of $\hat M$ on
such elements is determined by \eref{e:defMM2} 
so that, since we already know by the first part of the proof that $\DeltaM \bar \tau$ is uniquely determined,
the proof is complete.
\end{proof}

Before we proceed, we introduce a final object whose utility will be clear later on.
Similarly do the definition of $\DeltaM$, we define $\hDeltaM\colon \CH_0^+ \to \CH_0^+\otimes \CH_0^+$
by the identity
\begin{equ}[e:defDeltaMhat]
(\CA\hat M\CA \otimes \hat M)\Deltap = (I\otimes \CM)(\Deltap \otimes I)\hDeltaM\;.
\end{equ}
Note that, similarly to before, one can verify that the map $D^+ = (I\otimes \CM)(\Deltap \otimes I)$
is invertible on $\CH_0^+ \otimes \CH_0^+$, so that this expression does indeed define $\hDeltaM$ uniquely.

\begin{remark}
Note also that in the particular case when $M = I$, the identity, one has $\DeltaM \tau = \tau \otimes \one$,
$\hDeltaM \tau = \tau \otimes \one$, as well as $\hat M  = I$. 
\end{remark}

With these notations at hand, we then give the following description of the ``renormalisation
group'' $\RR$:

\begin{definition}\label{def:renormGroup}
Let $\CF_F$ and $\CF_0$ be as above. Then the corresponding renormalisation group
$\RR$ consists of all linear maps $M \colon \CH_0 \to \CH_0$ such that 
$M$ commutes with the $\CI_k$, such that $M X^k = X^k$, and such that, for every $\tau \in \CF_0$
and every $\hat \tau \in \CF_0^+$,
one can write
\begin{equ}[e:propDM]
\DeltaM \tau = \tau \otimes \one + \sum \tau^{(1)} \otimes \tau^{(2)}\;,\qquad 
\hDeltaM \bar \tau = \bar \tau \otimes \one + \sum \bar \tau^{(1)} \otimes \bar \tau^{(2)}\;,
\end{equ}
where each of the $\tau^{(1)} \in \CF_0$ and $\bar \tau^{(1)} \in \CF_0^+$ 
is such that $|\tau^{(1)}|_\s > |\tau|_\s$ and $|\bar \tau^{(1)}|_\s > |\bar \tau|_\s$. 
\end{definition}

\begin{remark}\label{rem:condDhat}
Note that $\hDeltaM$ is automatically a multiplicative morphism. Since
one has furthermore $\hDeltaM X^k = X^k \otimes  \one$ for every $M$, 
the second condition in \eref{e:propDM} really needs to be verified only for 
elements of the form $\CI_k(\tau)$ with $\tau \in \CF_\star$.
The reason for introducing the quantity $\hDeltaM$ and defining $\RR$ in this way
is that these conditions appear naturally
in Theorem~\ref{theo:renormStruct} below where we check that the renormalised model defined by
\eref{e:constrModel} does again satisfy the analytical bounds of Definition~\ref{def:model}.
\end{remark}

We first verify that our terminology is not misleading, namely that $\RR$ really is a group:

\begin{lemma}
If $M_1, M_2 \in \RR$, then $M_1 M_2 \in \RR$. Furthermore, if $M\in \RR$, then $M^{-1} \in \RR$.
\end{lemma}

\begin{proof}
Note first that if $M = M_1 M_2$ then, due to the identity $\PPi^M = \PPi M_1 M_2$, one obtains the model
$(\PPi^M, F^M)$ by applying the group element corresponding to $M_2$ to $(\PPi^{M_1}, F^{M_1})$.
As a consequence, one can ``guess'' the identities
\minilab{e:idMult}
\begin{equs}
\DeltaM &= (I \otimes \CM)\bigl(\Delta^{\! M_1} \otimes \hat M_1\bigr)\Delta^{\! M_2}\;, 
\label{e:idMult1}\\
\hDeltaM &= (I \otimes \CM)\bigl(\hat \Delta^{\! M_1} \otimes \hat M_1\bigr)\hat \Delta^{\! M_2}\;, 
\label{e:idMult3}\\
\hat M &= \hat M_1 \,\hat M_2\;.
\label{e:idMult2}
\end{equs}
Since we know that \eref{e:defMM} characterises $\DeltaM$ and $\hat M$, \eref{e:idMult} 
can be verified by checking that $\DeltaM$ and $\hat M$ 
defined in this way do indeed satisfy \eref{e:defMM}.
The identity \eref{e:defMM1} is immediate, so we concentrate on the two other ones. 

For \eref{e:defMM2}, we have
\begin{equs}
\CM(\CJ_k \otimes I) \DeltaM &= \CM\bigl((\CJ_k \otimes I) \Delta^{\!M_1} \otimes \hat M_1\bigr) \Delta^{\! M_2} \\
&= \CM\bigl(\hat M_1 \CJ_k \otimes  \hat M_1\bigr) \Delta^{\! M_2}\\
&= \hat M_1 \CM\bigl(\CJ_k \otimes  I\bigr) \Delta^{\! M_2} = \hat M_1 \hat M_2 \CJ_k\;,
\end{equs}
which is indeed the required property. Here, we made use of the morphism property of $\hat M_1$ to go from
the second to the third line. 

For \eref{e:defMM3}, we use \eref{e:idMult1} to obtain
\begin{equs}
(I \otimes \CM)(\Delta \otimes I)\DeltaM &= (I \otimes \CM)(\Delta \otimes I)(I \otimes \CM)\bigl(\Delta^{\! M_1} \otimes \hat M_1\bigr)\Delta^{\! M_2}\\
&= (I \otimes \CM)\bigl((M_1 \otimes \hat M_1)\Delta \otimes \hat M_1\bigr)\Delta^{\! M_2}\\
&= (M_1 \otimes \hat M_1)(I \otimes \CM)\bigl(\Delta \otimes I\bigr)\Delta^{\! M_2}\\
&= (M_1 \otimes \hat M_1)(M_2 \otimes \hat M_2)\Delta = (M \otimes \hat M)\Delta\;,
\end{equs}
as required. Here, we used again the morphism property of $\hat M_1$ to go from
the second to the third line. We also used the fact that, by assumption, \eref{e:defMM3} holds for both
$M_1$ and $M_2$. Finally, we want to verify that the expression \eref{e:idMult3} for $\hDeltaM$
is the correct one. For this, it suffices to proceed in virtually the same way as for $\DeltaM$,
replacing $\Delta$ by $\hat \Delta$ when needed.

To show that $\RR$ is a group and not just a semigroup, 
we first define, for any $\kappa \in \R$, the projection
$\CP_\kappa \colon \CH_0 \to \CH_0$ given by $\CP_\kappa \tau = 0$ if $|\tau|_\s > \kappa$
and $\CP_\kappa \tau = \tau$ if $|\tau|_\s \le \kappa$. We also write 
$\hat \CP_\kappa = \CP_\kappa \otimes I$ as a shorthand.
We then argue by contradiction as follows. 
Assuming that $M^{-1} \not \in \RR$, one of the two conditions in \eref{e:propDM}
must be violated. Assume first that it is the first one, then
there exists a $\tau \in \CF_0$ and a homogeneity $\kappa \le |\tau|_\s$, 
such that $\Delta^{\! M^{-1}} \tau$ can be rewritten as
\begin{equ}
\Delta^{\! M^{-1}} \tau = R^M_- \tau + R^M_+ \tau\;,
\end{equ}
with $\hat \CP_\kappa R^M_- \tau = R^M_- \tau \neq 0$, $\hat \CP_\kappa R^M_+ \tau = 0$,
and $R^M_-\tau \neq \tau \otimes \one$. We furthermore choose for $\kappa$ the smallest 
possible value such that such a decomposition exists, i.e.\ we assume that
$\hat \CP_{\bar \kappa} R^M_- \tau = 0$ for every $\bar \kappa  < \kappa$.

It follows from \eref{e:idMult1} that one has
\begin{equ}
\hat \CP_\kappa(\tau \otimes \one) = \hat \CP_\kappa\Delta^{I} \tau = 
(I \otimes \CM)\bigl(\hat \CP_\kappa\Delta^{\! M} \otimes \CM\hat \Delta^{\! M}\bigr)\Delta^{\! M^{-1}} \tau\;.
\end{equ}
Since, by Definition~\ref{def:renormGroup}, the identity
$\hat \CP_\kappa\DeltaM \bar \tau = \hat \CP_\kappa(\bar \tau \otimes \one)$ 
holds as soon as $|\bar \tau|_\s \ge \kappa$, one eventually obtains
\begin{equ}
\hat \CP_\kappa(\tau \otimes \one) = R^M_- \tau\;,
\end{equ}
which is a contradiction. Therefore, the only way in which one could have $M^{-1}\not\in\RR$
is by violating the second condition in \eref{e:propDM}. This however can also be ruled out
in almost exactly the same way, by making use of \eref{e:idMult3} instead of \eref{e:idMult1}
and exploiting the fact that one also has $\hat \Delta^I \tau = \tau\otimes \one$.
\end{proof}

The main result in this section states that any transformation $M\in \RR$
extends canonically to a transformation on the set of admissible models for $\TT_F^{(r)}$
for arbitrary $r > 0$.

\begin{theorem}\label{theo:renormStruct}
Let $M \in \RR$, where $\RR$ is as in Definition~\ref{def:renormGroup}, let $r > 0$ be
such that the kernel $K$ annihilates polynomials of degree $r$, 
and let $(\PPi, f) \sim (\Pi, \Gamma)$ be 
an admissible model for $\TT_F^{(r)}$ with $f$ and $\Gamma$ related as in \eref{e:defGammaF}. 

Define $\Pi_x^M$ 
and $f^M$ on $\CH_0$ and $\CH_0^+$ as in \eref{e:constrModel} and 
define $\Gamma^M$ via \eref{e:defGammaF}.
Then, $(\Pi^M, \Gamma^M)$ is an admissible model for $\TT_F$ on $\CH_0$. 
Furthermore, it extends uniquely to an admissible model for all of $\TT_F^{(r)}$.
\end{theorem}

\begin{proof}
We first verify that the renormalised model does indeed yield a model for $\TT_F$ on $\CH_0$.
For this, it suffices to show that the bounds \eref{e:boundPi} hold. 
Regarding the bound on $\Pi_x^M$, recall the first identity of \eref{e:constrModel}.
As a consequence of Definition~\ref{def:renormGroup}, this implies that $\bigl(\Pi_x^M \tau\bigr)(\phi_x^\lambda)$
can be written as a finite linear combination of terms of the type $\bigl(\Pi_x \bar \tau\bigr)(\phi_x^\lambda)$
with $|\bar \tau|_\s \ge |\tau|_\s$. The required scaling as a function of $\lambda$ then follows
at once.

Regarding the bounds on $\Gamma_{xy}$, recall that $\Gamma_{xy} \tau = (I\otimes \gamma_{xy})\Delta \tau$
with
\begin{equ}[e:defgammaxy]
\gamma_{xy} = (f_x \CA \otimes f_y) \Deltap\;,
\end{equ}
and similarly for $\gamma_{xy}^M$.
Since we know that $(\Pi,\Gamma)$ is a model for $\TT_F^{(r)}$, this implies that one has the bound
\begin{equ}[e:boundgammaxy]
|\gamma_{xy} \tau| \lesssim \|x-y\|_\s^{|\tau|_\s}\;,
\end{equ}
and we aim to obtain a similar bound for $\gamma_{xy}^M$. Recalling the
definitions \eref{e:defgammaxy} as well as \eref{e:constrModel}, 
we obtain for $\gamma_{xy}^M$ the identity
\begin{equs}
\gamma_{xy}^M &= (f_x \CA \otimes f_y)(\CA\hat M\CA \otimes \hat M) \Deltap
= (f_x\CA \otimes f_y)(I \otimes \CM) (\Deltap \otimes I)\hDeltaM \\
&= (f_x\CA \otimes f_y\otimes f_y)(\Deltap \otimes I)\hDeltaM = \bigl(\gamma_{xy} \otimes f_y\bigr)\hDeltaM\;,
\end{equs}
where the second equality is the definition of $\hDeltaM$, while the last equality 
uses the definition of $\gamma_{xy}$, combined with the morphism property of $f_y$.
It then follows immediately from Definition~\ref{def:renormGroup} and \eref{e:boundgammaxy} 
that the bound \eref{e:boundgammaxy} also holds for $\gamma_{xy}^M$.

Finally, we have already seen that if $(\Pi,\Gamma)$ is admissible, then 
$\Pi_x^M$ and $f_x^M$ satisfy the identities 
\eref{e:idPi} and \eref{e:idF} as a consequence of \eref{e:defMM2}, so that they also form
an admissible model. The fact that the model $(\Pi^M, \Gamma^M)$ 
extends uniquely (and continuously) to all of $\TT_F^{(r)}$ follows
from a repeated application of Theorem~\ref{theo:extension} and Proposition~\ref{prop:extension}.
\end{proof}

\begin{remark}
In principle, the construction of $\RR$ given in this section depends on the choice of 
a suitable set $\CF_0$. It is natural to conjecture that $\RR$ does not actually depend
on this choice (at least if $\CF_0$ is sufficiently large), but it is not clear at this stage
whether there is a simple algebraic proof of this fact.
\end{remark}

\section{Two concrete renormalisation procedures}
\label{sec:renProc}

In this section, we show how the regularity structure and renormalisation group 
built in the previous section can be used concretely to renormalise \eref{e:PAMGen}
and \eref{e:Phi4}.

\subsection{Renormalisation group for \eref{e:PAMGen}}
\label{sec:renPAM}

Consider the regularity structure generated by \eref{e:PAMGen} with $\M_F$ as in Remark~\ref{rem:PAMGen},
$\beta = 2$, and $\alpha \in (-{4\over 3},-1)$. In this case, we can choose 
\begin{equ}
\CF_0 = \{\one, \Xi, X_i \Xi, \CI(\Xi) \Xi, \CI_i(\Xi), \CI_i (\Xi)\CI_j (\Xi)\}\;,\qquad
\CF_\star = \{\Xi\}\;,
\end{equ}
where $i$ and $j$ denote one of the two spatial coordinates.
It is straightforward to check that this set satisfies Assumption~\ref{ass:SetF0}. 
Indeed, provided that $\alpha \in (-{4\over 3}, -1)$, it does contain all the elements of negative
homogeneity. Furthermore, all of the elements $\tau \in \CF_0$ satisfy 
$\Delta \tau = \tau \otimes \one$, except for $\Xi\, \CI (\Xi)$ and $X_i\Xi$ which satisfy
\begin{equ}
\Delta \bigl(\Xi\, \CI (\Xi)\bigr) = \Xi\, \CI (\Xi) \otimes \one + \Xi \otimes \CJ(\Xi)\;,\qquad
\Delta X_i \Xi = X_i \Xi\otimes \one + \Xi \otimes X_i\;.
\end{equ}
It follows that these elements indeed satisfy $\Delta \tau \in \CH_0\otimes \CH_0^+$, 
as required by our assumption.

Then, for any constant $C \in \R$ and $2\times 2$ matrix $\bar C$, one can define 
a linear map $M$ on the span of $\CF_0$ by
\begin{equs}
M \bigl(\CI(\Xi)\Xi\bigr) &= \CI(\Xi)\Xi - C \one\;,\\
M \bigl(\CI_i(\Xi)\CI_j(\Xi)\bigr) &= \CI_i(\Xi)\CI_j(\Xi) - \bar C_{ij} \one\;,
\end{equs}
as well as $M(\tau) = \tau$ for the remaining basis vectors in $\CF_0$.
Denote by $\RR_0$ the set of all linear maps $M$ of this type.

In order to verify that $\RR_0 \subset \RR$ as our notation implies, we need 
to verify that $\DeltaM$ and $\hDeltaM$ satisfy the property required by 
Definition~\ref{def:renormGroup}. Note first that
\begin{equ}
\hat M \CI(\Xi) = \CI(\Xi)\;,
\end{equ}
as a consequence of \eref{e:defMM2}. Since one furthermore has $\hat M X_i = X_i$, this 
shows that one has 
\begin{equ}
(M \otimes \hat M)\Delta \tau = (M \otimes I)\Delta \tau\;,
\end{equ}
for every $\tau \in \CF_0$.
Furthermore, it is straightforward to verify that $(M \otimes I)\Delta \tau = \Delta M \tau$
for every $\tau \in \CF_0$.
Comparing this to \eref{e:defMM3}, we conclude that
in the special case considered here we actually have the identity 
\begin{equ}[e:propertyDeltaM]
\DeltaM \tau = (M\tau) \otimes \one\;,
\end{equ}
for every $\tau \in \CF_0$. Indeed, when plugging \eref{e:propertyDeltaM} into the left hand side
of \eref{e:defMM3}, we do recover the right hand side, which shows the desired claim since we already
know that \eref{e:defMM3} is sufficient to characterise $\DeltaM$. 
Furthermore, it is straightforward to verify that $\hDeltaM \CI(\Xi) = \CI(\Xi)\otimes\one$ so that,
by Remark~\ref{rem:condDhat}, this shows that $M \in \RR$ for every choice of
the matrix $C_{ij}$ and the constant $\bar C$.

Furthermore, 
this $5$-parameter subgroup of $\R$ is canonically isomorphic to $\R^5$ endowed with addition as its
group structure. This is the subgroup $\RR_0$ 
that will be used to renormalise \eref{e:PAMGen} in Section~\ref{sec:PAMGenRigour}.

\subsection{Renormalisation group for the dynamical \texorpdfstring{$\Phi^4_3$}{Phi43} model}
\label{sec:renPhi4}

We now consider the regularity structure generated by \eref{e:Phi4}, which is our
second main example.
Recall from Remark~\ref{rem:Phi4} that this corresponds to the case where
\begin{equ}
\M_F = \{\Xi, U^n \,:\, n \le 3\}\;,
\end{equ}
$\beta = 2$ and $\alpha < -{5\over 2}$. In order for the relevant terms of negative homogeneity
not to depend on $\alpha$, we will choose $\alpha \in (-{18\over 7},-{5\over 2})$.
The reason for this strange-looking value $-{18 \over 7}$ is that this is precisely the value of $\alpha$
at which, setting $\Psi = \CI(\Xi)$ as a shorthand, the homogeneity of the term $\Psi^2 \CI(\Psi^2 \CI(\Psi^3))$ vanishes, so that one would have to modify our choice of $\CF_0$.

In this particular case,  
it turns out that we can choose for $\CF_0$ and $\CF_\star$ the sets
\begin{equs}
\label{e:F0Phi}
\CF_0 &= \{\one, \Xi, \Psi, \Psi^2, \Psi^3, \Psi^2 X_i, \CI(\Psi^3)\Psi, \CI(\Psi^3)\Psi^2,\\
&\qquad \CI(\Psi^2)\Psi^2, \CI(\Psi^2), \CI(\Psi)\Psi, \CI(\Psi)\Psi^2, X_i\}\;,
\qquad \CF_\star = \{\Psi,\Psi^2, \Psi^3\}\;,
\end{equs}
where the index $i$ corresponds again to any of the three spatial directions. 

Then, for any two constants $C_1$ and $C_2$, we define a linear map $M$ on $\CH_0$ by
\begin{equs}[e:defMPhi4]
M \Psi^2 &= \Psi^2 - C_1 \one\;,\\
M \bigl(\Psi^2 X_i\bigr) &= \Psi^2X_i - C_1 X_i\;,\\
M \Psi^3 &= \Psi^3 - 3 C_1 \Psi\;,\\
M \bigl(\CI(\Psi^2)\Psi^2\bigr) &= \CI(\Psi^2) \bigl(\Psi^2 - C_1 \one\bigr) - C_2 \one\;,\\
M \bigl(\CI(\Psi^3)\Psi\bigr) &= \bigl(\CI(\Psi^3) - 3C_1 \CI(\Psi)\bigr) \Psi \;,\\
M \bigl(\CI(\Psi^3)\Psi^2\bigr) &= \bigl(\CI(\Psi^3) - 3C_1 \CI(\Psi)\bigr) \bigl(\Psi^2- C_1 \one) - 3C_2 \Psi\;,\\
M \bigl(\CI(\Psi)\Psi^2\bigr) &= \CI(\Psi) \bigl(\Psi^2 - C_1 \one\bigr)\;,
\end{equs}
as well as $M \tau = \tau$ for the remaining basis elements $\tau \in \CF_0$. We claim
that one has the identity\minilab{e:idenDeltaMPhi}
\begin{equ}[e:idenDeltaMPhi1]
\DeltaM \tau = (M \tau) \otimes \one\;,
\end{equ}
for those elements $\tau \in \CF_0$ which do not contain a factor $\CI(\Psi^3)$. For the remaining
two elements, we claim that one has\minilab{e:idenDeltaMPhi}
\begin{equs}
\DeltaM \CI(\Psi^3)\Psi &= \bigl(M (\CI(\Psi^3)\Psi)\bigr) \otimes \one + 3 C_1 \,\Psi X_i \otimes  \CJ_i(\Psi)\;,\label{e:deltaMPsi1I3}\\
\DeltaM \CI(\Psi^3)\Psi^2 &= \bigl(M (\CI(\Psi^3)\Psi^2)\bigr) \otimes \one + 3 C_1 \,(\Psi^2 - C_1 \one) X_i \otimes \CJ_i(\Psi)\;,\label{e:deltaMPsi2I3}
\end{equs}
where a summation over the spatial components $X_i$ is implicit.

For $\tau \in \{\one, \Xi, \Psi, \Psi^2, \Psi^3\}$, 
one has $\Delta \tau = \tau \otimes \one$, so that $\DeltaM \tau = (M\tau)\otimes \one$ as
a consequence of \eref{e:defMM3}. 
Similarly, it can be verified that \eref{e:idenDeltaMPhi1} holds for $\Psi^2 X_i$ and $X_i$ by using
again \eref{e:defMM3}. For the remaining elements, we first note that, as 
a consequence of this and \eref{e:defMM2}, one has the identities
\begin{equ}\label{e:exprMhatPhi4}
\hat M \CI(\Psi^n) = \CI(M \Psi^n)\;,\qquad \hat M \CI_i(\Psi) = \CI_i(\Psi)\;.
\end{equ}
All the remaining elements are of the form $\tau = \CI(\Psi^n)\Psi^m$, so that \eref{e:defDelta} yields the identity
\begin{equ}
\Delta \tau = \tau \otimes \one + \Psi^m \otimes \CJ(\Psi^n) + \delta_{n1} \bigl(\Psi^m X_i \otimes \CJ_i(\Psi)
+ \Psi^m \otimes X_i \CJ_i(\Psi)\bigr)\;.
\end{equ}
As a consequence of this and of \eref{e:exprMhatPhi4}, one has 
\begin{equs}
(M \otimes \hat M)\Delta \tau &= M\tau \otimes \one + M\Psi^m \otimes \CJ(M\Psi^n)\label{e:FirstDeltaM} \\
&\qquad + \delta_{n1} \bigl(M\Psi^m\otimes \one\bigr) \bigl( X_i \otimes \CJ_i(\Psi)
+ \one \otimes X_i \CJ_i(\Psi)\bigr)\;.
\end{equs}
Furthermore, for each of these elements, one has 
\begin{equ}[e:Mtau]
M \tau = (M \Psi^m) \CI(M \Psi^n) + \bar \tau\;,
\end{equ}
where $\bar \tau$ is an element such that $\Delta \bar \tau = \bar \tau \otimes \one$.
Combining this with the explicit expression for $M$,
one obtains the identity
\begin{equs}
\Delta M \tau &= M\tau \otimes \one  + M\Psi^m \otimes \CJ(M\Psi^n) \\
&\qquad + \delta_{n1} \bigl(M\Psi^m\otimes \one\bigr) \bigl( X_i \otimes \CJ_i(\Psi)
+ \one \otimes X_i \CJ_i(\Psi)\bigr)\\
&\qquad - 3C_1\delta_{n3} (M \Psi^m \otimes \one)(X_i\otimes \CJ_i(\Psi) + \one \otimes X_i \CJ_i(\Psi))  \;.
\end{equs}
Comparing this expression with \eref{e:FirstDeltaM}, we conclude in view of \eref{e:defMM3}
that one does indeed have the identity
\begin{equ}
\DeltaM \tau = M\tau \otimes \one + 3C_1\delta_{n3}\, (M\Psi^m) X_i \otimes \CJ_i(\Psi)\;,
\end{equ}
which is precisely what we claimed. A somewhat lengthy but straightforward calculation along
the same lines yields the identities
\begin{equs}
\Deltap \CJ(M\Psi^n) &= \one \otimes \CJ(M\Psi^n) + \CJ(M\Psi^n) \otimes \one - \delta_{n1} \bigl(\CJ_i(\Psi)\otimes X_i\bigr) \\
&\qquad + 3C_1 \delta_{n3} \bigl(\CJ_i(\Psi) \otimes X_i\bigr)\;,
\end{equs}
as well as
\begin{equs}
\bigl(\CA \hat M \CA \otimes \hat M\bigr)\Deltap \CJ(\Psi^n)&=\one \otimes \CJ(M\Psi^n) + \CJ(M\Psi^n) \otimes \one - \delta_{n1} \bigl(\CJ_i(\Psi)\otimes X_i\bigr) \\
&\qquad - 3C_1 \delta_{n3} \bigl(X_i \CJ_i(\Psi) \otimes \one\bigr)\;.
\end{equs}
Comparing these two expressions with \eref{e:defDeltaMhat}, it follows that $\hDeltaM$ is given by
\begin{equ}
\hDeltaM \CJ(\Psi^n) = \CJ(M\Psi^n) \otimes \one + 3C_1\delta_{n3}\,\bigl(X_i \otimes \CJ_i(\Psi) - X_i \CJ_i(\Psi) \otimes \one\bigr) \;.
\end{equ}
As a consequence of the expressions we just computed for $\DeltaM$ and $\hDeltaM$ 
and of the definition of $M$, 
this shows that one does indeed have $M \in \RR$. Furthermore, it is immediate to verify that
this two-parameter subgroup is canonically isomorphic to $\R^2$ endowed with addition as its
group structure. This is the subgroup $\RR_0 \subset \RR$ 
that will be used to renormalise \eref{e:Phi4} in Section~\ref{sec:Phi4}.

\subsection{Renormalised equations for \eref{e:PAMGen}}
\label{sec:PAMGenRigour}

We now have all the tools required to formulate renormalisation
procedures for the 
examples given in the introduction. We give some details only for the cases of \eref{e:PAMGen}
and \eref{e:Phi4}, but it is clearly possible to obtain analogous
statements for all the other examples. 

The precise statement of our convergence results has to account
for the possibility of finite-time blow-up. (In the case of the $3D$ Navier-Stokes equations,
the existence or absence of such a blow-up is of course a famous open problem even in 
the absence of forcing, which is something that we definitely do not address here.)
The aim of this section is to show what the effect of the renormalisation group $\RR_0$
built in Section~\ref{sec:renPAM} is, when applied to a model used to solve \eref{e:PAMGen}.

Recall that the right hand side of \eref{e:PAMGen} is given by 
\begin{equ}
f_{ij}(u)\, \d_i u \, \d_j u + g(u)\,\xi\;,
\end{equ}
and that the set of monomials $\M_F$ associated with this right hand side 
is given by
\begin{equ}
\M_F = \{U^n, U^n \Xi, U^n \CP_i, U^n \CP_i \CP_j\,:\, n \ge 0,\; i,j \in \{1,2\}\}\;.
\end{equ}
We now let $\TT_F$ be the regularity structure associated to $\M_F$ 
via Theorem~\ref{theo:genStruct} with $d=3$, $\s = (2,1,1)$,
$\alpha = |\Xi|_\s \in (-{4\over 3}, -1)$, and $\beta = 2$.
As already mention when we built it,
the regularity structure $\TT_F$ comes with a sector $V = \scal{\CU_F}\subset T$ which is 
given by the direct sum of the abstract polynomials $\bar T$ with the image of $\CI$:
\begin{equ}[e:defVV]
V = \CI(T) \oplus \bar T\;.
\end{equ}
Since the element in $\CF_F$ with the lowest homogeneity is $\Xi$, the sector $V$ is function-like
and elements $u \in \CD^\gamma(V)$ with $\gamma > 0$ satisfy $\CR u \in \CC^{\gamma \wedge (\alpha+2)_\s}$.
Furthermore, the sector $V$ comes equipped with differentiation maps $\D_i$ given by
$\D_i \CI(\tau) = \CI_i(\tau)$ and $\D_i X^k = k_i X^{k-e_i}$.
It follows immediately from the definitions that any admissible model is compatible
with these differentiation maps in the sense of Definition~\ref{def:compatDiff}.

Assume for simplicity that the symmetry $\SS$ is given by integer translations
in $\R^2$, so that its action on $\TT_F$ is trivial. 
(In other words, we consider the case of periodic boundary conditions
on $[0,1]\times [0,1]$.)
Fix furthermore $\gamma > -\alpha$ and choose one of the decompositions $G = K + R$ 
of the heat kernel given by Lemma~\ref{lem:decompGreen} with $r > \gamma$.

With all this set-up in place, we define the local map $F_\gamma \colon V \to T$ by
\begin{equ}[e:defFgamma]
F_\gamma(\tau) = \hat f_{ij;\gamma}(\tau) \star \D_i \tau \star \D_i \tau + \hat g_\gamma(\tau)\star \Xi\;.
\end{equ}
Here, $\hat f_{ij;\gamma}$ and $\hat g_\gamma$ are defined from $f_{ij}$ and $g$ as in 
Section~\ref{sec:compSmooth}. 
Furthermore, we have explicitly used the symbol $\star$ to emphasise the fact
that this is the product in $T$. We also set as previously $P = \{(t,x) \,:\, t = 0\}$.

We then have the following result:
\begin{lemma}\label{lem:RHSPAM}
Assume that the functions $f_{ij}$ and $g$ are smooth.
Then, for every $\gamma > |\alpha|$ and for $\eta \in (0,\alpha + 2]$, the map $u \mapsto F_\gamma(u)$ is locally Lipschitz
continuous from $\CD^{\gamma,\eta}_P(V)$ into $\CD_P^{\gamma+\alpha,\eta+\alpha}$.
\end{lemma}

\begin{remark}
In fact, we only need sufficient amount of regularity for the results of Section~\ref{sec:compSmooth}
to apply.
\end{remark}

\begin{proof}
Let $u \in \CD^{\gamma,\eta}_P(V)$ and note that $V$ is function-like.
By Proposition~\ref{prop:diffSing}, one then has $\D_i u \in \CD^{\gamma-1,\eta-1}_P(W)$
for some sector $W$ with regularity $\alpha + 1 < 0$.
This is a consequence of the fact that $\D_i \one = 0$, so that the element
of lowest homogeneity appearing in $W$ is given by $\CI_i(\Xi)$. 

Applying Proposition~\ref{prop:multSing}, 
we see that $\D_i u \star \D_j u \in \CD^{\gamma+\alpha,2\eta-2}_P(\bar W)$,
where $\bar W$ is of regularity $2\alpha + 2$.
Since furthermore $\hat f_{ij;\gamma} (u) \in \CD^{\gamma,\eta}_P(V)$
by Proposition~\ref{prop:compSmoothSing} (and similarly for $\hat g_\gamma(u)$), we can apply Proposition~\ref{prop:multSing} once more to conclude that
\begin{equ}
\hat f_{ij;\gamma}(u) \star \D_i u \star \D_i u \in \CD_P^{\gamma+\alpha, 2\eta-2}\;. 
\end{equ}
Similarly, note that we can view the map $z \mapsto \Xi$ as an element of 
$\CD_P^{\gamma,\gamma}$ for every $\gamma > 0$, but taking values in a sector of regularity $\alpha$.
By applying again Proposition~\ref{prop:multSing}, we conclude that one has also
\begin{equ}
\hat g_{\gamma}(u) \star \Xi \in \CD_P^{\gamma+\alpha, 2\eta-2}\;. 
\end{equ}
All of these operations are easily seen to be locally Lipschitz continuous in the sense
of Section~\ref{sec:GenFP}, so the claim follows.
\end{proof}

\begin{corollary}\label{cor:locSolPAMGen}
Denote by $G$ the solution map for the heat equation, let $\eta > 0$, $\alpha \in (-{4\over 3}, -1)$, 
$\gamma > |\alpha|$, and $K$ such that it annihilates polynomials of order $\gamma$.
Then, for every periodic initial condition $u_0 \in \CC^\eta$ with $\eta > 0$ and
every admissible model $Z\in \MM_F$, the fixed point map
\begin{equ}[e:FPPAM]
u = (\CK_{\bar \gamma} + R_{ \gamma}\CR) \PR^+ F_\gamma(u) + G u_0\;,
\end{equ}
where $F_\gamma$ is given by \eref{e:defFgamma},
has a unique solution in $\CD^\gamma$ on $(0,T)$ for $T>0$ sufficiently small. 

Furthermore, setting $T_\infty = T_\infty(u_0; Z)$ 
to be the smallest time for which \eref{e:FPPAM} does not 
have a unique solution, one has either $T_\infty = \infty$ or 
$\lim_{t \to T_\infty} \|\CR u(t,\cdot)\|_\eta = \infty$.
Finally, for every $T < T_\infty$ and every $\delta > 0$, there exists
$\eps > 0$ such that if $\|\bar u_0 - u_0\|_\eta \le \eps$ and $\$Z; \bar Z\$_\gamma \le \eps$,
one has $\$u;\bar u\$_{\gamma,\eta} \le \delta$.
\end{corollary}

\begin{proof}
Since $\alpha > -2$ and $\eta > 0$, it follows from Lemma~\ref{lem:RHSPAM}
that all of the assumptions of Theorem~\ref{theo:fixedPointGen} and
Corollary~\ref{cor:genFixedPoint} are satisfied. 
\end{proof}

Denote now by $\CS^L$ the truncated solution map 
as given in Section~\eref{sec:GenFP}. On the other hand, for any
(symmetric / periodic) \textit{continuous} function $\xi_\eps \colon \R^3 \to \R$
and every (symmetric / periodic) $u_0 \in \CC^\eta(\R^2)$, 
we can build a ``classical'' solution map $u_\eps = \bar \CS^L(u_0, \xi_\eps)$ for the equation
\begin{equ}[e:PAMGen3]
\d_t u_\eps = \Delta u_\eps + f_{ij}(u_\eps) \, \d_i u_\eps \d_j u_\eps 
+ g(u_\eps)\,\xi_\eps\;,\qquad u_\eps(0,x) = u_0(x)\;,
\end{equ}
where the subscript $L$ refers again to the fact that we stop solutions when
$\|u_\eps(t,\cdot)\|_\eta \ge L$. Similarly to before, we also denote by $\bar T^L(u_0, \xi_\eps)$
the first time when this happens. Here, the solution map $\bar \CS^L(u_0, \xi_\eps)$ is the 
standard solution map for \eref{e:PAMGen3} obtained by classical PDE theory \cite{MR1406091,Krylov}.

Given an element $M \in \RR_0$ with the renormalisation group $\RR_0$ defined as in
Section~\ref{sec:renPAM}, we also define a ``renormalised'' solution map $u_\eps = \bar \CS^L_M(u_0, \xi_\eps)$
in exactly the same way, but replacing \eref{e:PAMGen3} by
\begin{equ}[e:PAMRen]
\d_t u_\eps = \Delta u_\eps + f_{ij}(u_\eps) \, \bigl(\d_i u_\eps \d_j u_\eps - g^2(u_\eps) \bar C_{ij}\bigr) 
+ g(u_\eps)\,\bigl(\xi_\eps - C g'(u_\eps)\bigr)\;,
\end{equ}
where $g'$ denotes the derivative of $g$. We then have the following result:

\begin{proposition}\label{prop:identifyRenormPAM}
Given a continuous and symmetric function $\xi_\eps$, denote by $Z_\eps$
the associated canonical model realising $\TT^{(r)}_F$ given by Proposition~\ref{prop:canonicReal}.
Let furthermore $M \in \RR_0$ be as in Section~\ref{sec:renPAM}. Then, for every $L>0$ and symmetric
$u_0 \in \CC^\eta(\R^2)$, one has the identities
\begin{equ}
\CR \CS^L(u_0, Z_\eps) = \bar \CS^L(u_0, \xi_\eps)\;,\quad\text{and}\quad
\CR \CS^L(u_0, M Z_\eps) = \bar \CS^L_M(u_0, \xi_\eps)\;.
\end{equ}
\end{proposition}

\begin{proof}
The fact that $\CR \CS^L(u_0, Z_\eps) = \bar \CS^L(u_0, \xi_\eps)$ is relatively straightforward to see. 
Indeed, we have already seen in the proof of Proposition~\ref{prop:solGen} that the function $v = \CR \CS^L(u_0, Z_\eps)$ satisfies for $t \le T^L(u_0, Z_\eps)$ the identity
\begin{equ}
v(t,x) = \int_0^t \int_{\R^2} G(t-s, x-y) \bigl(\CR F_\gamma(v)\bigr)(s,y)\,dy\,ds + \int_{\R^2} G(t,x-y) u_0(y)\,dy\;,
\end{equ}
where $G$ denotes the heat kernel on $\R^2$. Furthermore, it follows from \eref{e:recursion} 
and Remark~\ref{rem:product} that in the case of the canonical model $Z_\eps$, one has
indeed the identity
\begin{equ}
\bigl(\CR F_\gamma(v)\bigr)(s,y) = f_{ij}(\CR v(s,y)) \, \d_i \CR v(s,y) \d_j \CR v(s,y) 
+ g(\CR v(s,y))\,\xi_\eps(s,y)\;,
\end{equ}
valid for every $v \in \CD^\gamma$ with $\gamma > |\alpha| > 1$. As a consequence, 
$\CR v$ satisfies the same fixed point equation as the classical solution to \eref{e:PAMGen3}.

It remains to find out what fixed point equation $v$ satisfies when we consider instead the
model $M Z_\eps$, for which we denote the reconstruction operator by $\CR^M$. 
Recall first Remark~\ref{rem:continuousModel} 
which states that for every $w \in \CD^\gamma$ with $\gamma > 0$, one has the identity
\begin{equ}
\bigl(\CR^M w\bigr)(z) = \bigl(\Pi_{z}^{M,(\eps)} w(z)\bigr)(z)\;,
\end{equ}
where we have made use of the notation $M Z_\eps = (\Pi^{M,(\eps)}, \Gamma^{M,(\eps)})$.
Furthermore, one has $\bigl(\Pi_{z}^{M,(\eps)} \tau \bigr)(z) = 0$ for any element $\tau$ with
$|\tau|_\s > 0$, so that we only need to consider the coefficients of $w$ belonging to the subspace
spanned by the elements with negative (or $0$) homogeneity.

It follows from Lemma~\ref{lem:RHSPAM} that in order to compute all components of $w = F_\gamma(v)$ with
negative homogeneity, we need to know all components of $v$ with homogeneity less than $|\alpha|$.
One can verify that as long as $\alpha > -{4\over 3}$,
the only elements in $V$ with homogeneity less than $|\alpha|$ are given by $\{\one, X_1, X_2, \CI(\Xi)\}$.
Since $v(z)$ furthermore belongs to the sector $V$, we can find functions 
$\phi \colon \R^3 \to \R$ and $\nabla \Phi\colon \R^3 \to \R^2$ such that 
\begin{equ}
v(z) = \phi(z)\, \one + g(\phi(z)) \CI(\Xi) + \scal{\nabla \phi(z), X} + \rho(z)\;,
\end{equ}
where the remainder $\rho$ consists of terms with homogeneity strictly larger than $|\alpha|$.
Here, the fact that the coefficient of $\CI(\Xi)$ is necessarily given by $g(\phi(z))$ follows 
from the identity \eref{e:idenSol}, combined with an explicit calculation to determine $\FF$.
Furthermore, we make a slight abuse of notation here by denoting by $X$ the spatial 
coordinates of $X$. Note that in general, although $\nabla \phi$ can be interpreted as some
kind of ``renormalised gradient'' for $\phi$, we do not claim any kind of relation between
$\phi$ and $\nabla \phi$.
It follows that
\begin{equ}
\D_i v(z) = g(\phi(z)) \CI_i(\Xi) + \nabla_{\!i} \phi(z)\,\one + \rho_i(z)\;,
\end{equ}
for some remainder $\rho_i$ consisting of terms with homogeneity greater than $|\alpha|-1$.

Regarding $\hat f_{ij;\gamma}(v)$ and $\hat g_\gamma(v)$, we obtain from \eref{e:defFcirc} the expansions
\begin{equs}
\hat f_{ij;\gamma}(v)(z) &= f_{ij}(\phi(z))\,\one + f_{ij}'(\phi(z)) g(\phi(z)) \CI(\Xi) + \rho_f(z)\;,\\
\hat g_{\gamma}(v)(z) &= g(\phi(z))\,\one + g'(\phi(z)) g(\phi(z)) \CI(\Xi) + \rho_g(z)\;,
\end{equs}
where both $\rho_f$ and $\rho_g$ contain terms proportional to $X$,
as well as other components of homogeneities strictly greater than $|\alpha|$.
Note also that when $\alpha > -{4\over 3}$, the elements of negative homogeneity are 
those in $\CF_0$ as in Section~\ref{sec:renPAM}, but that one actually has
$\bigl(\Pi_{z}^{M,(\eps)} X_i \Xi\bigr)(z) = 0$ for every $M \in \RR_0$.

It follows that one has the identity
\begin{equs}
F_\gamma(v)(z) &= f_{ij}(\phi(z)) \bigl(\nabla_{\! i}\phi(z)\nabla_{\! j}\phi(z)\,\one
+ g(\phi(z))\nabla_{\! i}\phi(z) \CI_j(\Xi)\\
&\quad + g(\phi(z))\nabla_{\! j}\phi(z) \CI_i(\Xi) + g^2(\phi(z)) \CI_i(\Xi)\CI_j(\Xi)\bigr)\\
&\quad+ g(\phi(z)) \Xi + g'(\phi(z)) g(\phi(z)) \CI(\Xi)\Xi + \rho_F(z)\;.
\end{equs}
At this stage we use the fact that, by \eref{e:propertyDeltaM}, one has the identity
\begin{equ}
\Pi_z^{M,(\eps)} \tau = \Pi_z^{(\eps)} M\tau\;,
\end{equ}
for all $\tau \in \CF_0$, together with the fact that $\CR^M v(z) = \phi(z)$.
A straightforward calculation  then yields the identity
\begin{equs}
\CR^M F_\gamma(v)(z) &= f_{ij}(\CR^M v(z)) \bigl( \d_i \CR^M v(z) \d_j \CR^M v(z) - \bar C_{ij} g^2(\CR^M v(z))\bigr) \\
&\quad+ g(\CR^M v(z))\bigl(\xi_\eps(z) - C g'(\CR^M v(z))\bigr)\;,
\end{equs}
which is precisely what is required to complete the proof.
\end{proof}

\subsection{Solution theory for the dynamical \texorpdfstring{$\Phi^4_3$}{Phi43} model}

We now turn to the analysis of \eref{e:Phi4}. In this case, one has
$F = \xi - u^3$, so that $\M_F$ is given by $\{1, \Xi, U, U^2, U^3\}$.
This time, spatial dimension is $3$ and the scaling we consider is once 
again the parabolic scaling $\s = (2,1,1,1)$, so that the scaling dimension of space-time is $5$.
Since $\xi$ denotes space-time white noise this time, 
we choose for $\alpha$ some value $\alpha = |\Xi|_\s < -{5\over 2}$. It turns out that in order to be able to
choose the set $\CF_0$ in Section~\ref{sec:renormalisation}
independently of $\alpha$, we should furthermore impose $\alpha > -{18\over 7}$. In this case,
the fixed point equation that we would like to consider is
\begin{equ}[e:localFPPhi]
u = \bigl(\CK_{\bar \gamma} + R_\gamma \CR\bigr)\PR^+ \bigl(\Xi - u^3\bigr) + G u_0\;,
\end{equ}
with $u_0 \in \CC^\eta_{\bar \s}(\R^3)$, $\eta > -{2\over 3}$, $\gamma \in (\bar \gamma, \bar \gamma + 2)$, 
and $\bar \gamma > 0$.

We are then in a situation which is slightly outside of the scope
of the general result of Corollary~\ref{cor:genFixedPoint} for two reasons. First, 
Proposition~\ref{prop:recSing} does \textit{a priori} not apply to the 
singular modelled distribution $\PR^+ \Xi$. Second, the distribution $\CR \CI(\Xi)$
is of negative order, so that there is in principle no obvious way of evaluating
it at a fixed time. 
Fortunately, both of these problems can be solved relatively easily.
For the first problem, we note that multiplying white noise by the indicator function of
a set is of course not a problem at all, so we are precisely in the situation
alluded to in Remark~\ref{rem:senseSingDist}. As a consequence, all we have to
make sure is that the convergence $\xi_\eps \to \xi$ takes place in some
space of distributions that allows multiplication with the relevant indicator function. 
Regarding the distribution $\CR \CI(\Xi)$, it is also possible to verify that
if $\xi$ is space-time white noise, then $K * \xi$ almost surely takes values not only
in $\CC_\s^\eta(\R^4)$ for $\eta < -{1\over 2}$, but it actually takes values in
$\CC(\R, \CC^\eta(\R^3))$, which is precisely what is needed to be able to evaluate it
on a fixed time slice, thus enabling us to extend the argument of 
Proposition~\ref{prop:solGen}.

The simplest way of ensuring that the reconstruction operator yields a well-defined distribution
on $\R^4$ for $\PR^+ \Xi$ is to build a suitable space of distributions
``by hand'' and to show that smooth approximations to space-time white noise also
converge in that space. 
We fix again some final time, which we take to be $1$
for definiteness. We then define for any $\alpha < 0$ and compact $\K$ the norm
\begin{equ}
\/\xi\/_{\alpha;\K} = \sup_{s \in \R} \|\xi \one_{t \ge s} \|_{\alpha;\K}\;,
\end{equ}
and we denote by $\bar \CC^{\alpha}_\s$ the intersections of the completions 
of smooth functions
under ${\/\cdot\/}_{\alpha;\K}$ for all compacts $\K$. 
One motivation for this definition is the following convergence result:

\begin{proposition}\label{prop:convLinearPhi4}
Let $\xi$ be white noise on $\R \times \T^3$, which we extend periodically to $\R^4$.
Let $\rho \colon \R^4 \to \R$ be a smooth compactly supported function integrating to $1$,
set $\rho_\eps = \CS_\s^\eps \rho$, and define $\xi_\eps = \rho_\eps * \xi$.
Then, for every $\alpha \in (-3,-{5\over 2})$, one has $\xi \in \bar \CC_\s^\alpha$ almost surely
and, for every $\eta \in (-1, -{1\over 2})$, one has $K * \xi \in \CC(\R, \CC^\eta(\R^3))$
almost surely.
Furthermore, for every compact $\K \subset \R^4$ and every $\kappa > 0$, one has
\begin{equ}[e:boundConvXi1]
\E \/\xi - \xi_\eps\/_{\alpha;\K} \lesssim \eps^{-{5\over 2} - \alpha - \kappa}\;.
\end{equ}
Finally, for every $\bar \kappa \in (0,-{1\over 2} - \eta)$, the bound
\begin{equ}[e:boundConvXi2]
\E \sup_{t \in [0,1]} \|(K*\xi)(t,\cdot) - (K*\xi_\eps)(t,\cdot)\|_{\eta} \lesssim \eps^{\bar\kappa}\;,
\end{equ}
holds uniformly over $\eps \in (0,1]$.
\end{proposition}

\begin{proof}
In order to show that $\xi \in \bar \CC_\s^\alpha$, note first that it is immediate that $\xi \one_{t \ge s} \in \CC^\alpha_\s$ for every fixed $s \in \R$. It therefore suffices to show that the map 
$s \mapsto \xi \one_{t \ge s}$ is continuous in $\CC^\alpha_\s$. For this, we choose a wavelet basis as
in Section~\ref{sec:spaceC} and, writing $\Psi_\star  = \Psi \cup \{\phi\}$, 
we note that for every $p > 1$, one has the bound
\begin{equs}
\E \|\xi \one_{t \ge s} - \xi \one_{t \ge 0}\|_{\alpha;\K}^{2p} 
&\le \sum_{\psi \in \Psi_\star}\sum_{n \ge 0} \sum_{x \in \Lambda_\s^n\cap \bar \K} \E 2^{2\alpha n p + |\s|np}|\scal{\xi \one_{t \ge s} - \xi \one_{t \ge 0}, \psi_x^{n,\s}}|^{2p} \\
&\le \sum_{\psi \in \Psi_\star} \sum_{n \ge 0} \sum_{x \in \Lambda_\s^n\cap \bar \K} 2^{2\alpha n p+|\s|np} \bigl(\E |\scal{\xi , \one_{t \in [0,s]}\psi_x^{n,\s}}|^2)^{p}\\
&\lesssim \sum_{\psi \in \Psi_\star} \sum_{n \ge 0} 2^{2\alpha n p+|\s|np+|\s|n}\|\one_{t \in [0,s]}\psi_x^{n,\s}\|_{L^2}^{2p}\;.
\end{equs}
Here we wrote $\bar \K$ for the $1$-fattening of $\K$ and we 
used the equivalence of moments for Gaussian random variables to obtain the second line.
We then verify that
\begin{equ}
\|\one_{t \in [0,s]}\psi_x^{n,\s}\|_{L^2}^{2} \lesssim  1 \wedge 2^{2n}s\;.
\end{equ}
Provided that $\alpha \in (-{7\over 2}, -{5\over 2})$, it then follows that
\begin{equ}
\E \|\xi \one_{t \ge s} - \xi \one_{t \ge 0}\|_{\alpha;\K} \lesssim s^{-{5\over 4} - {\alpha \over 2} - {5\over 4p}}\;.
\end{equ}
Choosing first $p$ sufficiently large and then 
applying Kolmogorov's continuity criterion, it follows that one does indeed have 
$\xi \in \bar \CC^\alpha_\s$ as stated.

In order to bound the distance between $\xi$ and $\xi_\eps$, we can proceed in exactly the same way.
We then obtain the same bound, but with $\|\one_{t \in [0,s]}\psi_x^{n,\s}\|_{L^2}^{2}$ replaced by
$\|\one_{t \in [0,s]}\psi_x^{n,\s} - \rho_\eps * \bigl(\one_{t \in [0,s]}\psi_x^{n,\s}\bigr)\|_{L^2}^{2}$.
A straightforward calculation shows that
\begin{equ}
\|\one_{t \in [0,s]}\psi_x^{n,\s} - \rho_\eps * \bigl(\one_{t \in [0,s]}\psi_x^{n,\s}\bigr)\|_{L^2}^{2}
\lesssim 1 \wedge 2^{2n}s \wedge 2^{2n}\eps^2\;.
\end{equ} 
As above, it then follows that, provided that $\alpha + \kappa > -3$,
\begin{equ}
\E \|(\xi - \xi_\eps) \one_{t \in [0,s]}\|_{\alpha;\K} \lesssim \eps^{-{5\over 2} - \alpha - \kappa} s^{{\kappa \over 2} - {5 \over 4p}}\;,
\end{equ}
so that the requested bound 
\eref{e:boundConvXi1} follows at once by choosing $p$ sufficiently large.

In order to show \eref{e:boundConvXi2}, note first that $K * \xi_\eps = \rho_\eps * (K*\xi)$.
As a consequence, it is sufficient to find some space of distributions $\CX \subset \CC([0,1], \CC^\eta)$ such that $K*\xi \in \CX$
almost surely and such that the bound
\begin{equ}[e:wantedBoundX]
\|\rho_\eps * \zeta - \zeta\|_{\CC([0,1], \CC^\eta)} \lesssim \eps^{\bar \kappa} \|\zeta\|_\CX\;,
\end{equ}
holds uniformly over all $\eps \in (0,1]$ and $\zeta \in \CX$.
We claim that $\CX = \CC^{\bar \kappa \over 2}(\R, \CC^{\eta + \bar \kappa})$ is a possible choice. 

To show that \eref{e:wantedBoundX} holds, we use the characterisation 
\begin{equs}
\|\rho_\eps &* \zeta - \zeta\|_{\CC([0,1], \CC^\eta)} \\
&= \sup_{t\in [0,1]} \sup_{\lambda \in (0,1]} \lambda^{-\eta} \sup_{\psi} \int \psi^\lambda(x) \rho_\eps(x-y,t-s) \bigl(\zeta(y,s) - \zeta(x,t)\bigr)\,dx\,dy\,ds\;,
\end{equs}
where the supremum runs over all test functions $\psi \in \CB^1_{\s,0}$ (for $\s$ the Euclidean scaling).
We also wrote $\psi^\lambda$ for the
rescaled test function as previously. One then rewrites the above expressions as a sum
$T_1 + T_2$ with
\begin{equs}
T_1 &= \int \psi^\lambda(x) \rho_\eps(x-y,t-s) \bigl(\zeta(y,s) - \zeta(y,t)\bigr)\,dx\,dy\,ds\;,\\
T_2 &= \int \psi^\lambda(x) \rho_\eps(x-y,t-s) \bigl(\zeta(y,t) - \zeta(x,t)\bigr)\,dx\,dy\,ds\\
&= \int \bigl(\psi^\lambda(x) - \psi^\lambda(y)\bigr) \rho_\eps(x-y,t-s) \zeta(y,t)\,dx\,dy\,ds\;.
\end{equs}
To bound each of these terms, one considers separately the cases $\lambda \le \eps$ and $\lambda > \eps$.
For $T_1$, it is then straightforward to verify that $|T_1| \lesssim (\eps^\eta \wedge \lambda^\eta) |t-s|^{\bar \kappa/2}$. Since one has $|t-s| \lesssim \eps^2$ due to the fact that $\rho$ is compactly supported, 
the requested bound follows for $T_1$. For $T_2$,
arguments similar to those used in Section~\ref{sec:Schauder} yield the bound
$|T_2| \lesssim \lambda^{\eta + \bar \kappa} \lesssim \eps^{\bar \kappa} \lambda^\eta$ in the
case $\lambda \le \eps$ and $|T_2| \lesssim \lambda^{\eta + \bar \kappa -1}\eps \lesssim \eps^{\bar \kappa} \lambda^\eta$ in the case $\eps \le \lambda$. The bound \eref{e:wantedBoundX} then follows at once.

To show that $K * \xi$ belongs to $\CX$ almost surely, the argument is similar. 
Write $K = \sum_{n \ge 0} K_n$ as in the assumption and set $\xi^{(n)} = K_n * \xi$.
We claim that it then suffices to show that there is $\delta > 0$ such that the bound
\begin{equ}[e:wantedBoundXi]
\E \Bigl(\int \psi^\lambda(x) \bigl(\xi^{(n)}(x,t)- \xi^{(n)}(x,0)\bigr)\,dx\Bigr)^2 \lesssim
2^{-\delta n} |t|^{\bar \kappa + \delta} \lambda^{2\eta + 2\bar \kappa + \delta}\;,
\end{equ}
holds uniformly over $n \ge 0$, $\lambda \in (0,1]$, and test functions $\psi \in \CB^1_{\s,0}$.
Indeed, this follows at once by combining the usual Kolmogorov continuity test (in time) with 
Proposition~\ref{prop:charSpaces} (in space) and the equivalence of moments for Gaussian
random variables. 

The left hand side of \eref{e:wantedBoundXi} is equal to
\begin{equ}
\int \Bigl(\int \psi^\lambda(x) \bigl(K_n(x-y,t-r) - K_n(x-y,-r)\bigr)\,dx \Bigr)^2\,dy\,dr =: \|\Psi_n^{\lambda;t} \|_{L^2}^2\;.
\end{equ}
It is immediate from the definitions and the scaling properties of the $K_n$
that the volume of the support of $\Psi_n^{\lambda;t}$ is bounded
by $(\lambda + 2^{-n})^3 2^{-2n}$. The values of $\Psi_n^{\lambda;t}$ inside this support
are furthermore bounded by a multiple of
\begin{equ}
2^{3n} \wedge |t| 2^{5n} \wedge \lambda^{-3}\;.
\end{equ}
For $\lambda < 2^{-n}$ we thus obtain the bound
\begin{equ}
\|\Psi_n^{\lambda;t} \|_{L^2}^2 \lesssim 2^{-5n} |t|^{\bar \kappa +\delta} 2^{6n + 2(\bar \kappa +\delta)n}
= 2^{n + 2(\bar \kappa +\delta)n} |t|^{\bar \kappa + \delta}\;,
\end{equ}
while for $\lambda \ge 2^{-n}$ we obtain
\begin{equ}
\|\Psi_n^{\lambda;t} \|_{L^2}^2 \lesssim \lambda^3 2^{-2n} |t|^{\bar \kappa +\delta}
\lambda^{-6 + \bar \kappa + \delta} 2^{5(\bar \kappa +\delta)n} = |t|^{\bar \kappa +\delta}
\lambda^{3(\bar \kappa +\delta)-3} 2^{-2n + 5(\bar \kappa +\delta)n}\;.
\end{equ}
It follows that since $\eta$ is strictly less than $-{1\over 2}$, it is possible 
to choose $\bar \kappa$ and $\delta$ sufficiently small to guarantee that the bound
\eref{e:wantedBoundXi} holds, thus concluding the proof.
\end{proof}

\begin{remark}
The definition of these spaces of distributions is of course rather \textit{ad hoc}, but
it happens to be one that then allows us to restart solutions, 
which is amply sufficient to apply the same procedure as in Corollary~\ref{cor:genFixedPoint} to
define local solutions to \eref{e:localFPPhi}. 
\end{remark}

As before, the regularity structure $\TT$ comes with a sector $V\subset T$ which is 
given by \eref{e:defVV}.
This time however, the sector $V$ is \textit{not} function-like, but has regularity $2+\alpha \in (-{4\over 7}, -{1\over 2})$.
Assume for simplicity that the symmetry $\SS$ is again given by integer translations
in $\R^3$, so that its action on $\TT$ is trivial. 
Fix furthermore $\gamma > |2\alpha+4|$ and choose one of the decompositions $G = K + R$ 
of the heat kernel given by Lemma~\ref{lem:decompGreen} with $r > \gamma$.

Regarding the nonlinearity, we then have the following bound:

\begin{lemma}\label{lem:RHSPhi}
For every $\gamma > |2\alpha+4|$ and for $\eta \le \alpha + 2$, 
the map $u \mapsto u^3$ is locally Lipschitz continuous  in the strong sense from 
$\CD^{\gamma,\eta}_P(V)$ into $\CD_P^{\gamma+2\alpha+4,3\eta}$.
\end{lemma}

\begin{proof}
This is an immediate consequence of Proposition~\ref{prop:multSing}.
\end{proof}

With these results at hand, our strategy is now as follows. First, we reformulate
the fixed point map \eref{e:localFPPhi} as
\begin{equs}[e:fixedPointPhi4]
u &= - \bigl(\CK_{\bar \gamma} + R_\gamma \CR\bigr)\PR^+ u^3 + G u_0 + v\;,\\
v &= \bigl(\CK_{\bar \gamma} + R_\gamma \CR\bigr)\PR^+ \Xi \;.
\end{equs}
Here, we \textit{define} $\CR \PR^+ \Xi$ as the distribution $\xi \one_{t \ge 0}$, which 
does indeed coincide with $\CR \PR^+ \Xi$ when applied to test functions that are localised
away of the singular line $t = 0$,
and belongs to $\CC_\s^\alpha$ by assumption. This also shows immediately that
$v \in \CD_P^{\gamma,\eta}$ for $\eta$ and $\gamma$ as in Lemma~\ref{lem:RHSPhi}.
We then have the following result:

\begin{proposition}\label{prop:locSolPhi4}
Let $\TT_F$ be the regularity structure associated as above to 
\eref{e:Phi4} with $\alpha \in (-{18\over 7}, -{5\over 2})$, $\beta = 2$
and the formal right hand side $F(U,\Xi,P) = \Xi - U^{3}$. 
Let furthermore $\eta \in (-{2\over 3}, \alpha + 2)$ and let 
$Z=(\Pi,\Gamma)\in \MM_F$ be an admissible model for $\TT$ with the additional 
properties that $\xi \eqdef \CR \Xi$
belongs to $\bar \CC^\alpha_\s$ and that $K* \xi \in \CC(\R, \CC^\eta)$. 

Then, for every $\gamma > 0$ and every 
$L > 0$, one can build a maximal solution map $\CS^L$ for \eref{e:fixedPointPhi4}
with the same properties as in Section~\ref{sec:GenFP}.
Furthermore, $\CS^L$ has the same continuity properties as in Corollary~\ref{cor:genFixedPoint}, provided
that $Z$ and $\bar Z$ furthermore satisfy the bounds
\begin{equ}[e:boundXi1]
\/\xi\/_{\alpha;O} + \/\bar \xi\/_{\alpha;O} \le C\;,\qquad
\sup_{t \in [0,1]} \bigl(\|(K*\xi)(t,\cdot)\|_\eta + \|(K*\bar \xi)(t,\cdot)\|_\eta\bigr) \le C\;,
\end{equ}
as well as
\begin{equ}[e:boundXi2]
\/\xi - \bar \xi\/_{\alpha;O} \le \delta \;,\qquad
\sup_{t \in [0,1]} \bigl(\|(K*\xi)(t,\cdot)-(K*\bar \xi)(t,\cdot)\|_\eta\bigr) \le \delta\;.
\end{equ}
Here, we have set $\bar \xi = \bar \CR \Xi$, where $\bar \CR$ is the reconstruction operator associated to $\bar Z$.
\end{proposition}

\begin{proof}
We claim that, as a consequence of Lemma~\ref{lem:RHSPhi}, 
the nonlinearity $F(u) = -u^3$ satisfies the assumptions of Theorem~\ref{theo:fixedPointGen}
as soon as we choose $\gamma > |2\alpha + 4|$.
Indeed, in this situation, 
$V$ is the sector generated by all elements in $\CF_F$ of the form $\CI \tau$,
while $\bar V$ is the span of $\CF_F \setminus \{\Xi\}$.
As a consequence, one has $\zeta = \alpha + 2$ and 
$\bar \zeta = 3(\alpha+2)$, so that indeed $\zeta < \bar \zeta + 2$. 

Provided that $\eta$ and $\gamma$ are as in Lemma~\ref{lem:RHSPhi}, one then has
$\bar \eta = 3\eta$ and $\bar \gamma = \gamma+2\alpha+4$. 
The condition $\eta < (\bar \eta \wedge \bar \zeta) + 2q$ then reads 
$\eta < 3\eta  + 2$, which translates into the condition $\eta > -1$,
which is satisfied by assumption. 
The condition $\gamma < \bar \gamma + 2q$ reads $\alpha > -3$, which is
also satisfied by assumption. Finally, the assumption $\bar \eta \wedge \bar \zeta > - 2q$
reads $\eta > -{2\over 3}$, which is also satisfied. As a consequence, we can apply 
Theorem~\ref{theo:fixedPointGen} to get a local solution map.

To extend this local map up to the first time where $\|(\CR u)(t,\cdot)\|_\eta$ blows up,
the
argument is virtually identical to the proof of Proposition~\ref{prop:solGen}.
The only difference is that the solution $u$ does not take values in a function-like sector.
However, our local solutions are of the type
$u(t,x) = \CI \Xi + v(t,x)$, with $v$ taking values in a function-like sector.
(As a matter of fact, $v$ takes values in a sector of order $3(\alpha + 2) + 2$.)
The bounds \eref{e:boundXi1} and \eref{e:boundXi2} are then precisely what is required for the
reconstruction operator to still be a continuous map with values in $\CC(\R, \CC_{\bar \s}^\eta)$
and for the fixed point equation
\begin{equs}
u &= - \bigl(\CK_{\bar \gamma} + R_\gamma \CR\bigr)\PR_s^+ u^3 + G u_s + v\;,\\
v &= \bigl(\CK_{\bar \gamma} + R_\gamma \CR\bigr)\PR_s^+ \Xi \;,
\end{equs}
to make sense for all $s > 0$.
\end{proof}

\begin{remark}
The lower bound $-{2\over 3}$ for $\eta$ appearing in this theorem is probably sharp. This is because
the space $\CC^{-{2\over 3}}$ is critical for the deterministic equation so that one
wouldn't even expect to have a continuous solution map for $\d_t u = \Delta u - u^3$ 
in $\CC^{-{2\over 3}}$! If $u^3$ is replaced by $u^2$ however, the critical space is $\CC^{-1}$
and one can build local solutions for any $\eta > -1$.
\end{remark}

As in Section~\ref{sec:PAMGenRigour}, we now identify solutions corresponding to a model
that has been renormalised under the action of the group $\RR_0$ constructed in Section~\ref{sec:renPhi4}
with classical solutions to a modified equation. Recall that this time, elements $M \in \RR_0$ are
characterised by two real numbers $C_1$ and $C_2$.
As before, denote by $u_\eps = \bar \CS^L(u_0, \xi_\eps)$ the classical solution map to the equation
\begin{equ}
\d_t u_\eps = \Delta u_\eps - u_\eps^3 + \xi_\eps\;,
\end{equ}
stopped when $\|u_\eps(t,\cdot)\|_\eta \ge L$. Here, $\xi_\eps$ is a continuous function which is
periodic in space, and $u_0 \in \CC^\eta(\T^3)$. 
This time, it turns out that the renormalised map
$\bar \CS^L_M$ is given by the classical solution map to the equation
\begin{equ}[e:renormEquPhi4]
\d_t u_\eps = \Delta u_\eps + (3C_1 - 9C_2)u_\eps - u_\eps^3 + \xi_\eps\;,
\end{equ}
stopped as before when the norm of the solution reaches $L$. Indeed, one has again: 

\begin{proposition}\label{prop:identifyRenormPhi4}
Given a continuous function $\xi_\eps\colon \R \times \T^3 \to \R$, denote by $Z_\eps = (\Pi^{(\eps)}, \Gamma^{(\eps)})$
the associated canonical model for the regularity structure $\TT_F^{(r)}$ given by Proposition~\ref{prop:canonicReal}.
Let furthermore $M \in \RR_0$ be as in Section~\ref{sec:renPhi4}. Then, for every $L>0$ and symmetric
$u_0 \in \CC^\eta(\R^2)$, one has the identities
\begin{equ}
\CR \CS^L(u_0, Z_\eps) = \bar \CS^L(u_0, \xi_\eps)\;,\quad\text{and}\quad
\CR \CS^L(u_0, M Z_\eps) = \bar \CS^L_M(u_0, \xi_\eps)\;.
\end{equ}
\end{proposition}

\begin{proof}
The proof is similar to the proof of Proposition~\ref{prop:identifyRenormPAM}.
Just like there, we can find periodic 
functions $\phi\colon \R^4 \to \R$ and $\nabla \phi\colon \R^4 \to \R^3$ such that,
writing $\Psi = \CI(\Xi)$ as a shorthand, the solution
$u$ to the abstract fixed point map can be expanded as
\begin{equ}[e:decompuPhi]
u = \Psi + \phi\, \one - \CI(\Psi^3) - 3 \phi\, \CI(\Psi^2) + \scal{\nabla \phi, X} + \rho_u\;,
\end{equ}
where every component of $\rho_u$ has homogeneity strictly greater than $-4-2\alpha$. In particular, 
since $\bigl(\Pi_z^{M,(\eps)} \Psi\bigr)(z) = \bigl(K * \xi_\eps\bigr)(z)$, one has the identity
\begin{equ}
\bigl(\CR u\bigr)(z) = \bigl(K*\xi_\eps\bigr)(z) + \phi(z)\;,
\end{equ}
where we denote by $\CR$
the reconstruction operator associated to $Z_\eps$.
As a consequence of \eref{e:decompuPhi}, $F(u) = \Xi - u^3$ can be expanded in increasing degrees of homogeneity as
\begin{equs}
F(u) &= \Xi - \Psi^3 - 3\phi\, \Psi^2 + 3 \Psi^2 \CI(\Psi^3) - 3 \phi^2\,\Psi + 6 \phi\, \Psi\CI(\Psi^3) \\
&\quad + 9\phi\, \Psi^2 \CI(\Psi^2) - 3 \scal{\nabla\phi, \Psi^2 X}
- \phi^3\,\one + \rho_F\;,
\end{equs}
where every component of $\rho_F$ has strictly positive homogeneity. 
This time, one has the identity $\DeltaM \tau = M\tau \otimes \one + \bar \tau^{(1)} \otimes \bar \tau^{(2)}$
where each of the elements $\bar \tau^{(1)}$ includes at least one 
factor $X_i$. As a consequence, just like in the case of
 \eref{e:PAMGen}, one has again the identity $\bigl(\Pi_z^{M,(\eps)} \tau\bigr)(z) = \bigl(\Pi_z^{(\eps)} M\tau\bigr)(z)$.
It 
follows at once that, for $u$ as in \eref{e:decompuPhi}, one has the identity
\begin{equs}
\bigl(\CR^M F(u)\bigr)(z) &= \xi_\eps(z) - (\CR u)(z)^3 + 3C_1 \bigl(K * \xi_\eps\bigr)(z) + 3 C_1 \phi(z)\\
&\quad -  9C_2 \bigl(K * \xi_\eps\bigr)(z) - 9 C_2 \phi(z)\\
& = \xi_\eps(z) - (\CR u)(z)^3 + (3C_1 - 9C_2) \,\bigl(\CR u\bigr)(z)\;.
\end{equs}
The claim now follows in the same way as in the proof of Proposition~\ref{prop:identifyRenormPAM}.
\end{proof}

\begin{remark}
We could of course have taken for $F$ an arbitrary polynomial of degree $3$. If we take for
example $F(u) = \Xi - u^3 + a u^2$ for some real constant $a$, then we  obtain for
our renormalised equation
\begin{equ}
\d_t u_\eps = \Delta u_\eps + 3(C_1 - 3 C_2)u_\eps - u_\eps^3 + a u_\eps^2 - a (C_1-3C_2) + \xi_\eps\;.
\end{equ}
It is very interesting to note that, again, the renormalisation procedure
formally ``looks like'' simple
Wick renormalisation, except that the renormalisation constant does \textit{not}
equal the variance of the linearised equation. 
It is not clear at this stage whether this is a coincidence or
has a deeper meaning. 

In the case where no term $u^3$ appears, the renormalisation procedure is
significantly simplified since none of the terms involving $\CI(\Psi^3)$ appears. This then
allows to reduce the problem to the methodology of \cite{MR1941997,MR2016604}, 
see also the recent work \cite{ArnulfE}. In this case, the renormalisation is
the usual Wick renormalisation involving only the constant $C_1$.
\end{remark}

\section{Homogeneous Gaussian models}
\label{sec:Gaussian}

One very important class of random models for a given regularity structure is given by 
``Gaussian models'', where
the processes $\Pi_x a$ and $\Gamma_{xy} a$
are built from some underlying Gaussian white noise $\xi$.
Furthermore, we are going to consider the stationary situation where, for any given test 
function $\phi$, any $\tau \in T$, and any $h \in \R^d$, the processes
$x \mapsto \bigl(\Pi_x \tau\bigr)(\phi_x)$ and $x \mapsto \Gamma_{x,x+h}$ are stationary as a
function of $x$. (Here, we wrote $\phi_x$ for the function $\phi$ translated so that it is centred
around $x$.)
Finally, in such a situation, it will be natural to assume that
the random variables $\bigl(\Pi_x \tau\bigr)(\psi)$ and $\Gamma_{xy} \tau$
belong to the (inhomogeneous) Wiener chaos of some fixed order (depending only on $\tau$)
for $\xi$. This is indeed the case for the canonical models $Z_\eps$ built from some
continuous Gaussian process $\xi_\eps$ as in Section~\ref{sec:realAlg}, provided 
that $\xi_\eps(z)$ is a linear functional
of $\xi$ for every $z$. It is also the case for the renormalised model 
$\hat Z_\eps = M^{(\eps)} Z_\eps$, where $M^{(\eps)}$ denotes any element of the renormalisation group
$\RR$ built in Section~\ref{sec:renormalisation}.

Our construction suggests that there exists a general procedure such that, by using the
general renormalisation procedure described in Section~\ref{sec:renormalisation}, 
it is typically
possible to build natural stationary Gaussian models that can then 
be used as input for the abstract solution maps
built in Section~\ref{sec:GenFP}.
As we have seen, the corresponding solutions can then typically be interpreted 
as limits of classical solutions to a renormalised version of the equation as in Section~\ref{sec:renProc}.
Such a completely general statement does unfortunately seem out of reach for the moment, although
someone with a deeper knowledge of algebra and constructive quantum field theory techniques might be
able to achieve this. Therefore, we will only focus on two examples, namely on the case of the  dynamical
$\Phi^4_3$ model, as well as the generalisation of the two-dimensional
continuous parabolic Anderson model given in \eref{e:PAMGen}. 
Several of the intermediate steps in our construction are completely generic though,
and would just as well apply, \textit{mutatis mutandis}, to \eref{e:PAM} in dimension $3$, 
to \eref{e:KPZ}, or to \eref{e:NS}.

\subsection{Wiener chaos decomposition}
\label{sec:Chaos}

In all the examples mentioned in the introduction, the driving noise $\xi$ was Gaussian. 
Actually, it was always given by \textit{white noise} on some copy of $\R^d$ which would always
include the spatial variables and, except for \eref{e:PAM}, would include the temporal
variable as well. Mathematically, white noise is described by a probability space
$(\Omega,\F, \P)$, as well as a Hilbert space $H$ (typically some $L^2$ space) and a collection
$W_h$ of centred jointly Gaussian random variables indexed by $h \in H$ with the property that
the map $h \mapsto W_h$ is a linear isometry from $H$ into $L^2(\Omega,\P)$.
In other words, one has the identity
\begin{equ}
\E W_h W_{\bar h} = \scal{h, \bar h}\;,
\end{equ}
where the scalar product on the right is the scalar product in $H$.

\begin{remark}
We will usually consider a situation where some symmetry group $\SS$ 
acts on $\R^d$. In this case, $H$ is actually given by $L^2(\Dom)$, where
$\Dom \subset \R^d$ is the fundamental domain of the action of $\SS$.
This comes with a natural projection $\pi \colon L^2(\R^d) \to H$ given by
$\bigl(\pi \phi\bigr)(x) = \sum_{g \in \SS} \phi(T_g x)$.
\end{remark}

In the setting of the above remark, this data also yields a random 
distribution, which we denote by $\xi$, defined through $\xi(\phi) \eqdef W_{\pi \phi}$.
If we endow $\R^d$ with some scaling $\s$,
we have the following simple consequence of Proposition~\ref{prop:charSpaces}.

\begin{lemma}\label{lem:approxGauss}
The random distribution $\xi$ defined above 
almost surely belongs to $\CC_\s^{\alpha}$ for every $\alpha < -|\s|/2$.
Furthermore, let $\rho \colon \R^d \to \R$ be a smooth compactly supported function
integrating to one, set $\rho_\eps = \CS^\eps_{\s,0} \rho$, and define $\xi_\eps = \rho_\eps * \xi$.
Then, for every $\alpha < -{|\s|\over 2}$, every $\kappa > 0$, and every compact set $\K \subset \R^d$, 
one has the bound
\begin{equ}
\E \|\xi_\eps - \xi\|_{\alpha;\K} \lesssim \eps^{-{|\s|\over 2} - \alpha - \kappa}\;.
\end{equ}
\end{lemma}

\begin{proof}
The proof is almost identical to the proof of the first part of Proposition~\ref{prop:convLinearPhi4}.
The calculations are actually more straightforward since the indicator functions 
$\one_{t \ge s}$ do not appear, so we leave this as an exercise. 
\end{proof}

It was first remarked by Wiener \cite{WienerChaos} that there exists a natural isometry between
\textit{all} of $L^2(\Omega,\P)$ and the ``symmetric Fock space''
\begin{equ}
\hat H = \bigoplus_{k \ge 0} H^{\otimes_s k}\;,
\end{equ}
where $H^{\otimes_s k}$ denotes the symmetric $k$-fold tensor product of $H$.
Here, we identify $H^{\otimes_s k}$ with $H^{\otimes k}$, quotiented by the
equivalence relations 
\begin{equ}
e_{i_1} \otimes\ldots \otimes e_{i_k} \sim e_{i_{\sigma(1)}} \otimes\cdots 
\otimes e_{i_{\sigma(k)}}\;,
\end{equ}
where $\sigma$ is an arbitrary permutation of $k$ elements.
(This extends by linearity.)

If $\{e_n\}_{n \ge 0}$ denotes an orthonormal basis of $H$ then, for any
sequence $k_0,k_1,\ldots$ of positive integers with only finitely many non-zero elements, 
Wiener's isometry is given
by
\begin{equ}
k! e_k \eqdef k! e_0^{\otimes k_0}\otimes e_1^{\otimes k_1}\otimes \ldots \qquad \Leftrightarrow\qquad
H_{k_0}(W_{e_0})H_{k_1}(W_{e_1})\ldots\;,
\end{equ}
where $H_n$ denotes the $n$th Hermite polynomial, $k! = k_0! k_1!\cdots$, and $e_k$ 
has norm $1$. Random variables in correspondence with elements in $H^{\otimes_s m}$ are
 said to belong to the $m$th homogeneous Wiener chaos. The $m$th inhomogeneous
chaos is the sum of all the homogeneous chaoses of orders $\ell \le m$.
See also \cite[Ch.~1]{Nualart} for more details.

We have a natural projection $H^{\otimes m} \twoheadrightarrow H^{\otimes_s m}$:
just map an element to its equivalence class.
Composing this projection with Wiener's isometry yields a natural family of maps 
$I_m\colon H^{\otimes m} \to L^2(\Omega,\P)$ with the property that
\begin{equ}
\E \bigl(I_m(f)^2\bigr) \le \|f\|^2\;,
\end{equ}
where $f \in H^{\otimes m}$ is identified with an element of $L^2(\Dom^{d})$,
and the right hand side denotes its $L^2$ norm. In the case of an element $f$ that
is symmetric under the permutation of its $m$ arguments, this inequality turns into an equality.
For this reason, many authors restrict themselves to symmetric functions from the start, but
it turns out that allowing ourselves to work with non-symmetric functions will greatly
simplify some expressions later on.

Note that in the case $m=1$, we simply have $I_1(h) = W_h$.
The case $m=0$ corresponds to the natural identification of $H^{0} \sim \R$ with
the constant elements of $L^2(\Omega,\P)$. To state the following result, we denote by 
$\CS(r)$ the set of all permutations of $r$ elements, and by $\CS(r,m) \subset \CS(m)$
the set of all ``shuffles'' of $r$ and $m-r$ elements, namely the set of permutations of 
$m$ elements which preserves the order of the first $r$ and of the last $m-r$ elements.
For $x \in \Dom^m$ and $\Sigma \in \CS(m)$, we write $\Sigma(x) \in \Dom^m$
as a shorthand for $\Sigma(x)_i = x_{\Sigma(i)}$.
For $x \in \Dom^r$ and $y \in \Dom^{m-r}$, we also denote by $x \sqcup y$ the element of $\Dom^{m}$
given by $(x_1,\ldots,x_r,y_1,\ldots,y_{m-r})$.
With these notations, we then have the following formula for the
product of two elements.

\begin{lemma}\label{lem:productChaos}
Let $f \in L^2(\Dom^\ell)$ and $g \in L^2(\Dom^m)$. Then, one has
\begin{equ}[e:formProduct]
I_\ell(f) I_m(g) = \sum_{r =0}^{\ell \wedge m} I_{\ell+m-2r}(f\star_r g)\;,
\end{equ}
where
\begin{equ}
(f \star_r g)(z \sqcup \tilde z) = 
\sum_{\Sigma \in \CS(r,\ell)\atop
\tilde \Sigma \in \CS(r,m)}\sum_{\sigma \in \CS(r)} \int_{\Dom^r} f(\Sigma(x\sqcup z))g(\tilde \Sigma(x \sqcup \sigma(\tilde z)))\,dx\;,
\end{equ}
for all $z \in \Dom^{\ell-r}$ and $\tilde z \in \Dom^{m-r}$.
\end{lemma}

\begin{proof}
See \cite[Prop.~1.1.2]{Nualart}.
\end{proof}

\begin{remark}
Informally speaking, Lemma~\ref{lem:productChaos} states that in order to build the chaos decomposition
of the product $I_\ell(f) I_m(g)$, one should consider all possible ways of pairing $r$ of the $\ell$ 
arguments of $f$ with $r$ of the $m$ arguments of $g$ and integrate over these paired
arguments. This should really be viewed as an extension of Wick's product formula for Gaussian
random variables.
\end{remark}

A remarkable property of the Wiener chaoses is the following equivalence of moments:

\begin{lemma}\label{lem:Nelson}
Let $X \in L^2(\Omega,\P)$ be a random variable in the $k$th inhomogeneous Wiener chaos. Then,
for every $p \ge 1$, there exists a universal constant $C_{k,p}$ such that $\E |X^{2p}| \le C_{k,p} \bigl(\E X^2\bigr)^p$.
\end{lemma}

\begin{proof}
This is a consequence of Nelson's hypercontractive estimate \cite{MR0343816,MR0420249}, combined 
with the fact that the Wiener chaos decomposition diagonalises the Ornstein-Uhlenbeck semigroup.
\end{proof}

\subsection{Gaussian models for regularity structures}

From now on, we assume that we are given a probability space $(\Omega, \F, \P)$,
together with an abstract white noise $h \mapsto W_h$ over the Hilbert space $H = L^2(\Dom)$.
We furthermore assume that we are given a Gaussian random distribution $\xi$ 
which has the property that, for every test function $\psi$, the random variable
$\xi(\psi)$ belongs to the homogeneous first Wiener chaos of $W$.

\begin{remark}
One possible choice of noise $\xi$ is given by $\xi(\psi) = W_\psi$, which corresponds
to white noise. While this is a very natural choice in many physical situations, 
this is not the only choice by far.
\end{remark}

We furthermore assume that we are given a sequence $\xi_\eps$ of continuous
approximations to $\xi$ with the following properties:
\begin{claim}
\item For every $\eps > 0$, the map $x \mapsto \xi_\eps(x)$ is continuous almost surely.
\item For every $\eps > 0$ and every $x \in \R^d$, $\xi_\eps(x)$ is a 
random variable belonging to the first Wiener chaos of $W$.
\item For every test function $\psi$, one has 
\begin{equ}
\lim_{\eps \to 0} \int_{\R^d} \xi_\eps(x)\psi(x)\,dx = \xi(\psi)\;,
\end{equ}
in $L^2(\Omega,\P)$.
\end{claim}

Given such an approximation, one would ideally like to be able to show that the corresponding
sequence $(\Pi^{(\eps)},\Gamma^{(\eps)})$ of canonical models built from $\xi_\eps$
in Section~\ref{sec:realAlg} converges to some limit.
As already mentioned several times, this is simply not the case in general, thus 
the need for a suitable renormalisation procedure.
We will always consider renormalisation procedures based on a sequence $M_\eps$ of
elements in the renormalisation group $\RR$ built in Section~\ref{sec:renormalisation}.
We will furthermore take advantage of the fact that we know \textit{a priori} that the 
models $(\Pi^{(\eps)},\Gamma^{(\eps)})$ belong to some fixed Wiener chaos.

Indeed, we can denote by $\|\tau\|$ the number of occurrences of $\Xi$ in the formal expression $\tau$.
More formally, we set $\|\one\| = \|X\| = 0$, $\|\Xi\| = 1$, and then recursively
\begin{equ}
\|\tau \bar \tau\| = \|\tau\| + \|\bar \tau\|\;,\qquad \|\CI_k \tau\| = \|\tau\|\;.
\end{equ}
Then, as an immediate consequence of Lemma~\ref{lem:productChaos},
for any fixed $\tau \in \CF_F$, $x \in \R^d$, and smooth test function $\psi$, the random variables
$\bigl(\Pi^{(\eps)}_x \tau\bigr)(\psi)$ and $\Gamma_{xy}^{(\eps)}\tau$ belong 
to the (inhomogeneous) Wiener chaos 
of order $\|\tau\|$.
Actually, it belongs to the sum of the homogeneous chaoses of orders $\|\tau\| - 2n$ for
$n$ a positive integer, and this is still true for the renormalised models.
From now on, we denote $\CF_- = \{\tau \in \CF_F\,:\, |\tau|_\s < 0\}$.
The following convergence criterion is the foundation on which all of our convergence
results are built.

\begin{theorem}\label{theo:convRenormGauss}
Let $F$ be a locally subcritical nonlinearity and let $\TT_F^{(r)}$ be the corresponding
regularity structure built in Section~\ref{sec:SPDE}, restricted to $\{\tau\,:\, |\tau|_\s \le r\}$. Let $M_\eps$ be a sequence of
elements in its renormalisation group $\RR$, let $\xi_\eps$ be an approximation
to $\xi$ as in Lemma~\ref{lem:approxGauss} with associated canonical model $Z_\eps=(\Pi^{(\eps)},\Gamma^{(\eps)})$, and let 
$\hat Z_\eps=(\hat \Pi^{(\eps)},\hat \Gamma^{(\eps)}) = M_\eps Z_\eps$
be the corresponding sequence of renormalised models. 

Assume furthermore that there is $\kappa > 0$ such that, for every test function $\phi \in \CB^r_{\s,0}$,
every $x \in \R^d$, and every $\tau \in \CF_-$, there exists 
a random variable $\bigl(\hat \Pi_x \tau\bigr)(\phi)$ belonging to the inhomogeneous
Wiener chaos of order $\|\tau\|$ such that
\begin{equ}[e:aprioriBound]
\E \bigl|\bigl(\hat \Pi_x \tau\bigr)(\phi_x^\lambda)\bigr|^2 \lesssim \lambda^{2|\tau|_\s+\kappa}\;,
\end{equ}
and such that, for some $\theta > 0$,
\begin{equ}[e:convergenceBound]
\E \bigl|\bigl(\hat \Pi_x \tau - \hat \Pi_x^{(\eps)} \tau\bigr)(\phi_x^\lambda)\bigr|^2 
\lesssim \eps^{2\theta} \lambda^{2|\tau|_\s+\kappa}\;.
\end{equ}
Then, there exists a unique admissible random model $\hat Z = (\hat \Pi, \hat \Gamma)$ of $\TT_F^{(r)}$ 
such that, for every compact set $\K \subset \R^d$ and every $p \ge 1$, one has the bounds 
\begin{equ}
\E \$\hat Z\$_{\K}^p \lesssim 1\;,\qquad
\E \$\hat Z; \hat Z_\eps\$_{\K}^p \lesssim \eps^{\theta p}\;.
\end{equ}
\end{theorem}

\begin{remark}
As already seen previously, it is actually sufficient to take for $\phi$ the scaling
function of some sufficiently regular compactly supported wavelet basis. 
\end{remark}

\begin{proof}
Note first that the proportionality constants appearing in \eref{e:aprioriBound} and \eref{e:convergenceBound}
are independent of $x$ by stationarity.
Let now $\CV \subset \CF$ be any finite collection of basis vectors, let $V = \scal{\CV}$,
and assume that $\CV$ is such that $\Delta V \subset V \otimes \CH_+$, so that $V$ is a sector
of $\TT_F$. Then, it follows from Proposition~\ref{prop:wavelet} that, for every compact set $\K$, one has the bound
\begin{equs}
\E \|\hat \Pi\|_{V;\K}^p &\lesssim  \E \Bigl(\bigl(1+\|\Gamma\|_{V;\K}\bigr)^p \sup_{\tau \in \CV} \sup_{n \ge 0} \sup_{x \in \Lambda_\s^n(\bar \K)} 2^{|\tau|_\s p n + {p n|\s|\over 2}}  {\bigl|\bigl(\hat \Pi_x \tau\bigr)(\phi_x^{n,\s})\bigr|^p}\Bigr) \label{e:boundKolmCont}\\
&\lesssim  \sqrt{\E \bigl(1+\|\Gamma\|_{V;\K}\bigr)^{2p}} \sum_{\tau \in \CV} \sum_{n \ge 0} 2^{n|\s| + |\tau|_\s p n + {p n|\s|\over 2}}  \bigl({\E \bigl|\bigl(\hat \Pi_0 \tau\bigr)(\phi_0^{n,\s})\bigr|^2}\bigr)^{p\over 2}\;,\quad
\end{equs}
where the proportionality constant depends on $\K$ and the choice of $\CV$. Here, we used stationarity
and Lemma~\ref{lem:Nelson} to go from the first to the second line.
A similar bound also holds for $\hat \Pi^{(\eps)}$, as well as for the difference between the two models.


The claim will now be proved by induction over $\CF^{(n)}$, where $\CF^{(n)}$
was defined in Section~\ref{sec:renormalisation}.
Recall that for every $n \ge 0$, the linear span $T_n \eqdef \scal{\CF^{(n)}}$ 
forms a sector of $\TT_F$, that these sectors exhaust all of the model space $T$,
and that one has $\Delta T_n \subset T_n \otimes \scal{\Alg(\CF^{(n-1)})}$.
As a consequence, it is sufficient to prove that, for every $p \ge 0$, one has the bounds
\begin{equ}
\E \$\hat Z\$_{T_n;\K}^p \lesssim 1\;,\qquad
\E \$\hat Z; \hat Z_\eps\$_{T_n;\K}^p \lesssim \eps^{\kappa p}\;.
\end{equ}
The claim is trivial for $n = 0$, so we assume from now on that it holds for some
$n\ge 0$.
As a consequence of the definition of $\CF^{(n+1)}$ and the fact that we only consider admissible models, 
the action of $\hat \Gamma_{xy}$ on it 
is determined by the corresponding values $\hat f_x(\tau)$ for $\tau \in \Alg(\CF^{(n)})$.
Since furthermore the functionals $\hat f_x$ are multiplicative and, on elements of the form
$\CI_k \tau$, we know from our definition of the canonical model and
of the renormalisation group that \eref{e:idF} holds, we 
conclude from the finiteness of the set $\CF^{(n)}$
and from Theorem~\ref{theo:extension} 
that there exists some power $k$ (possibly depending on $n$) such that the deterministic bounds
\begin{equs}
\|\hat \Gamma\|_{T_{n+1};\K} &\lesssim 
\bigl(1 + \$\hat Z\$_{T_n;\K}\bigr)^k\;,\\
\|\hat \Gamma - \hat \Gamma^{(\eps)}\|_{T_{n+1};\K} &\lesssim 
\$\hat Z; \hat Z_\eps\$_{T_{n};\K}\bigl(1 + \$\hat Z\$_{T_{n};\K}\bigr)^k\;,
\end{equs}
hold. We now write $\CF^{(n+1)} = \CF^{(n+1)}_- \cup \CF^{(n+1)}_+$, where
$\CF^{(n+1)}_- = \CF^{(n+1)} \cap \CF_-$, while the second set
contains the remainder. Setting $T_{n+1}^- = \scal{\CF^{(n+1)}_-}$, it follows from 
Assumption~\ref{ass:SetF0} and \eref{e:coundGroup} that $\Delta T_{n+1}^- \subset T_{n+1}^- \otimes \scal{\Alg(\CF^{(n)})}$.

It thus follows from \eref{e:boundKolmCont} and \eref{e:aprioriBound} that
\begin{equ}
\E \|\hat \Pi\|_{T_{n+1}^-;\K}^p \lesssim  \sqrt{\E \bigl(1+\|\Gamma\|_{T_{n+1};\K}\bigr)^{2p}} \sum_{\tau \in \CV} \sum_{n \ge 0} 2^{n|\s| - \kappa p n} \;.
\end{equ}
Provided that $p$ is large enough so that $\kappa p > |\s|$, which is something that we can
always assume without any loss of generality since $p$ was arbitrary, it follows that $\hat \Pi$
does indeed satisfy the required bound on $T_{n+1}^-$. Regarding the difference $\hat \Pi - \hat \Pi^{(\eps)}$,
we obtain the corresponding bound in an identical manner.
In order to conclude the argument, it remains to obtain a similar bound on all of $T_{n+1}$.
This however follows by applying Proposition~\ref{prop:extension}, proceeding inductively 
in increasing order of homogeneity. Note that each element we treat in this way has strictly positive
homogeneity since we assume that only $\one$ has homogeneity zero, and $\Pi_x \one = 1$, so nothing needs
to be done there.
\end{proof}

We assume from now on that we are in the setting of Theorem~\ref{theo:convRenormGauss}
and therefore only need to obtain the convergence of $\bigl(\hat \Pi^{(\eps)}_x \tau\bigr)(\phi)$
to a limiting random variable $\bigl(\hat \Pi_x \tau\bigr)(\phi)$ with the required
bounds when considering rescaled versions of $\phi$. 
We also assume that we are in a translation invariant situation in the sense
that $\R^d$ acts onto $H$ via a group of unitary operators $\{S_x\}_{x \in \R^d}$
and there exists an element $\rho_\eps \in H$ such that
\begin{equ}
\xi_\eps(x) = I_1(S_x \rho_\eps)\;,
\end{equ}
where $I_1$ is as in Section~\ref{sec:Chaos}.
As a consequence, $\E \bigl|\bigl(\hat \Pi_x \tau\bigr)(\phi_x)\bigr|^2$ is independent of $x$,
so that we only need to consider the case $x = 0$.

Since the map $\phi \mapsto \bigl(\hat \Pi_x^{(\eps)} \tau\bigr)(\phi)$ is linear, one
can find some functions (or possibly distributions in general) $\hat \CW^{(\eps;k)}\tau$
with 
\begin{equ}[e:defWhat]
\bigl(\hat \CW^{(\eps;k)}\tau\bigr)(x) \in H^{\otimes k}\;,
\end{equ}
where $x \in \R^d$, and such that
\begin{equ}[e:identifPiW]
\bigl(\hat \Pi^{(\eps)}_0 \tau\bigr)(\phi) = \sum_{k \le \|\tau\|} I_k \Bigl(\int_{\R^d}\phi(y)\bigl(\hat \CW^{(\eps;k)}\tau\bigr)(y)\,dy \Bigr)\;,
\end{equ}
where $I_k$ is as in Section~\ref{sec:Chaos}.
The same is of course also true of the bare model $\Pi^{(\eps)}$, and we denote the corresponding functions
by $\CW^{(\eps;k)}\tau$. 

\begin{remark}
Regarding $\hat \Pi^{(\eps)}_x \tau$ for $x \neq 0$, it is relatively straightforward to see that one
has the identity
\begin{equ}[e:identifPiWx]
\bigl(\hat \Pi^{(\eps)}_x \tau\bigr)(\phi_x) = \sum_{k \le \|\tau\|} I_k \Bigl(\int_{\R^d}\phi(y)S_x^{\otimes k}\bigl(\hat \CW^{(\eps;k)}\tau\bigr)(y)\,dy \Bigr)\;,
\end{equ}
which again implies that the law of these random variables is independent of $x$. 
\end{remark}

\begin{remark}
For every $x \in \R^d$, $\bigl(\hat \CW^{(\eps;k)}\tau\bigr)(x)$ is a function on $k$ copies of $\Dom$. 
We will therefore
also denote it by $\bigl(\hat \CW^{(\eps;k)}\tau\bigr)(x; y_1,\ldots, y_k)$.
Note that the dimension of $x$ is not necessarily the same as that of the $y_i$. This is the
case for example in \eref{e:PAMGen} where the equation is formulated in $\R^3$ (one time dimension and
two space dimensions), while the driving noise $\xi$ lives in the Wiener chaos over a subset of $\R^2$.
\end{remark}

We then have the following preliminary result which shows that, in the kind of situations we consider here,
the convergence of the models $\hat Z_\eps$ to some limiting model $\hat Z$ can often be reduced to the convergence
of finitely many quite explicit kernels.

\begin{proposition}\label{prop:convCorFcn}
In the situation just described, fix some $\tau \in \CF_-$ and 
assume that there exists some $\kappa > 0$ such that, 
for every $k \le \|\tau\|$, there exist functions 
$\hat \CW^{(k)}\tau$ with values in $H^{\otimes k}$ and such that
\begin{equ}
\bigl|\scal{ \bigl(\hat \CW^{(k)}\tau\bigr)(z), \bigl(\hat \CW^{(k)}\tau\bigr)(\bar z)}\bigr| \le
C \sum_{\zeta} \bigl(\|z\|_\s + \|\bar z\|_\s\bigr)^{\zeta} \|z- \bar z\|_\s^{\kappa + 2|\tau|_\s - \zeta}\;,
\end{equ} 
where the sum runs over finitely many values $\zeta \in [0, 2|\tau|_\s + \kappa + |\s|)$.
Here, we denoted by
$\scal{\cdot,\cdot}$ the scalar product in $H^{\otimes k}$.

Assume furthermore that there exists $\theta > 0$ such that
\begin{equ}[e:bounddW]
\bigl|\scal{ \bigl(\delta\hat \CW^{(\eps;k)}\tau\bigr)(z), \bigl(\delta \hat \CW^{(\eps;k)}\tau\bigr)(\bar z)}\bigr| \le
C \eps^{2\theta} \sum_{\zeta} \bigl(\|z\|_\s + \|\bar z\|_\s\bigr)^{\zeta} \|z- \bar z\|_\s^{\kappa + 2|\tau|_\s - \zeta}\;,
\end{equ}
where we have set $\delta \hat \CW^{(\eps;k)} = \hat \CW^{(\eps;k)} - \hat \CW^{(k)}$,
and where the sum is as above.
Then, the bounds \eref{e:aprioriBound} and \eref{e:convergenceBound} are satisfied for $\tau$.
\end{proposition}

\begin{proof}
In view of \eref{e:identifPiW} and \eref{e:identifPiWx} we define, for every smooth test function $\psi$ and
every  $x \in \R^d$ the random variable $\bigl(\hat \Pi_x \tau\bigr)(\psi)$ by
\begin{equ}[e:defineLimit]
\bigl(\hat \Pi_x \tau\bigr)(\psi) = \sum_{k \le \|\tau\|}\bigl(\hat \Pi_x^{(k)} \tau\bigr)(\psi) = \sum_{k \le \|\tau\|} I_k \Bigl(\int_{\R^3}\psi(z)S_x^{\otimes k}\bigl(\hat \CW^{(k)}\tau\bigr)(z)\,dz \Bigr)\;.
\end{equ}
We then have the bound
\begin{equs}
\E \bigl|\bigl(\hat \Pi_x^{(k)} \tau\bigr)(\psi_x^\lambda)\bigr|^2 &=
\E \bigl|\bigl(\hat \Pi_0^{(k)} \tau\bigr)(\psi^\lambda)\bigr|^2 \lesssim
\Bigl\|\int_{\R^d}\psi^\lambda(z)\bigl(\hat \CW^{(k)}\tau\bigr)(z)\,dz\Bigr\|^2 \\
&= \int \int \psi^\lambda(z)\psi^\lambda(\bar z) \scal{\bigl(\hat \CW^{(k)}\tau\bigr)(z), \bigl(\hat \CW^{(k)}\tau\bigr)(\bar z)}\,dz\,d\bar z \\
&\lesssim \lambda^{-2|\s|} \sum_\zeta\int_{\|z\|_\s \le \lambda \atop \|\bar z\|_\s \le \lambda}
\bigl(\|z\|_\s + \|\bar z\|_\s\bigr)^{\zeta} \|z- \bar z\|_\s^{\kappa + 2|\tau|_\s - \zeta}
\,dz\,d\bar z\\
&\lesssim \lambda^{-2|\s|} \sum_\zeta \lambda^{\zeta + |\s|} \int_{\|z\|_\s \le 2\lambda}
\|z\|_\s^{\kappa + 2|\tau|_\s - \zeta} \,dz \\
&\lesssim \lambda^{-2|\s|} \sum_\zeta \lambda^{\zeta + 2|\s| + \kappa + 2|\tau|_\s-\zeta}
\lesssim \lambda^{\kappa + 2|\tau|_\s}\;.
\end{equs}
A virtually identical calculation, but making use instead of the bound on $\delta \hat \CW^{(\eps;k)}$,
also yields the bound
\begin{equ}
\E \bigl|\bigl(\hat \Pi^{(\eps)}_x - \hat \Pi_x \tau\bigr)(\psi_x^\lambda)\bigr|^2 \lesssim \eps^{2\theta}\lambda^{\kappa + 2|\tau|_\s}\;,
\end{equ}
as claimed.
\end{proof}

\subsection{Functions with prescribed singularities}

Before we turn to examples of SPDEs for which the corresponding sequence of canonical models for
the regularity structure $\TT_F$ can be successfully renormalised, we 
perform a few preliminary computations on the behaviour of smooth functions having a singularity
of prescribed strength at the origin.

\begin{definition}\label{def:orderKernel}
Let $\s$ be a scaling of $\R^d$ and let $K \colon \R^d\setminus \{0\} \to \R$ be a smooth
function. We say that $K$ is of order
$\zeta$ if, for every sufficiently small
multiindex $k$, there exists a constant $C$ such that
the bound $|D^k K(x)| \le C \|x\|_\s^{\zeta-|k|_\s}$ holds for every $x$ with $\|x\|_\s \le 1$. 

For any $m \ge 0$, we furthermore write
\begin{equ}
\$K\$_{\zeta;m} \eqdef \sup_{|k|_\s \le m} \sup_{x \in \R^d} \|x\|_\s^{|k|_\s-\zeta} |D^k K(x)|\;. 
\end{equ}
\end{definition}

\begin{remark}
Note that this is purely an upper bound on the behaviour of $K$ near the origin.
In particular, if $K$ is of order $\zeta$, then it is also of order $\bar \zeta$ for every $\bar \zeta < \zeta$.
\end{remark}

\begin{lemma}\label{lem:convSing}
Let $K_1$ and $K_2$ be two compactly supported functions of respective
 orders $\zeta_1$ and $\zeta_2$. Then
$K_1 K_2$ is of order $\zeta = \zeta_1 + \zeta_2$ and one has the bound
\begin{equ}
\$K_1 K_2\$_{\zeta;m} \le C \$K_1\$_{\zeta_1;m}\$K_2\$_{\zeta_2;m}\;,
\end{equ}
where $C$ depends on the sizes of the supports of the $K_i$.

If $\zeta_1\wedge \zeta_2  > -|\s|$
 and furthermore $\bar \zeta \eqdef \zeta_1 + \zeta_2 + |\s|$ satisfies $\bar\zeta < 0$,
then $K_1 * K_2$ is of order $\bar \zeta$
and one has the bound
\begin{equ}[e:boundConvKK]
\$K_1 * K_2\$_{\bar \zeta;m} \le C \$K_1\$_{\zeta_1;m}\$K_2\$_{\zeta_2;m}\;.
\end{equ}
In both of these bounds, $m \in \N$ is arbitrary.
In general, if $\bar \zeta \in \R_+ \setminus \N$, then $K_1*K_2$ has derivatives
of order $|k|_\s < \bar \zeta$ at the origin and the function $K$ given by
\begin{equ}[e:renormConv]
K(x) = \bigl(K_1 * K_2\bigr)(x) - \sum_{|k|_\s < \bar\zeta} {x^k\over k!} D^k\bigl(K_1 * K_2\bigr)(0)
\end{equ}
is of order $\bar\zeta$. Furthermore, one has the bound
\begin{equ}[e:boundConvKKRenorm]
\$K\$_{\bar \zeta;m} \le C \$K_1\$_{\zeta_1;\bar m}\$K_2\$_{\zeta_2;\bar m}\;,
\end{equ}
where we set $\bar m = m \vee \bigl(\lfloor \bar \zeta\rfloor + \max \{\s_i\}\bigr)$.
\end{lemma}


\begin{proof}
The claim about the product $K_1 K_2$ is an immediate consequence of the generalised Leibniz rule,
so we only need to bound $K_1 * K_2$. We will first show that, for every $x \neq 0$ and every multiindex $k$ such that
$\bar\zeta < |k|_\s$, one does have the bound
\begin{equ}[e:wantedBoundKK]
\bigl|D^k \bigl(K_1 * K_2\bigr)(x)\bigr| \lesssim \|x\|_\s^{\bar\zeta - |k|_\s}
\$K_1\$_{\zeta_1;|k|_\s}\$K_2\$_{\zeta_2;|k|_\s}\;,
\end{equ}
as required. From such a bound, \eref{e:boundConvKK} follows immediately.
To show that \eref{e:boundConvKKRenorm} follows from \eref{e:wantedBoundKK}, we
note first that $D^k K = D^k (K_1 * K_2)$ for every $k$ such that $|k|_\s > \bar \zeta$,
so that it remains to show that it is possible to find \textit{some} numbers
which we then call $D^k (K_1 * K_2)(0)$ such that if $K$ is defined by \eref{e:renormConv},
then similar bounds hold for
$D^k K$ with $|k|_\s < \bar \zeta$.

For this, we define the set of multiindices $A_{\bar \zeta} = \{k\,:\, |k|_\s < \bar \zeta\}$
and we fix a decreasing enumeration $A_{\bar \zeta} = \{k_0,\ldots,k_M\}$, i.e.\ $|k_m|_\s \ge |k_n|_\s$
whenever $m \le n$.
We then start by setting $\bar K^{(0)}(x) = (K_1 * K_2)(x)$ and we build a sequence of functions
$\bar K^{(n)}(x)$ iteratively as follows.
Assume that we have the bound $|D^{k_n+e_i} K^{(n)}(x)| \lesssim \|x\|_\s^{\bar \zeta - |k_n|_\s - \s_i}$ 
for $i \in \{1,\ldots,d\}$. (This is the case for $n = 0$ by \eref{e:wantedBoundKK}.)
Proceeding as in the proof of Lemma~\ref{lem:vanishPower} it then follows
that one can find a real number $C_n$ such that 
$|D^{k_n} K^{(n)}(x) - C_n| \lesssim \|x\|_\s^{\bar \zeta - |k_n|_\s}$.
We then set $K^{(n+1)}(x) = K^{(n)}(x) - C_n {x^{k_n} \over k_n!}$.
It is then straightforward to verify that if we set $K(x) = K^{(M)}(x)$, it has all the required properties.

It remains to show that \eref{e:wantedBoundKK} does indeed hold.
For this, let $\phi \colon \R^d$ be
a smooth function from $\R^d$ to $[0,1]$ such that $\phi(x) = 0$ for $\|x\|_\s \ge 1$
and $\phi(x) = 1$ for $\|x\|_\s \le {1\over 2}$.
For $r > 0$, we also set $\phi_r(y) = \phi(\CS_\s^r y)$.
Since $K$ is bilinear in $K_1$ and $K_2$, we can assume without loss of generality that
$\$K_i\$_{\zeta_i;|k|_\s} = 1$.
With these notations at hand, we can write
\begin{equs}
\bigl(K_1 * K_2\bigr)(x) &= \int_{\R^d} \phi_r(y) K_1(x-y) K_2(y)\,dy
+ \int_{\R^d} \phi_r(x-y) K_1(x-y) K_2(y)\,dy \\
&\quad + \int_{\R^d} \bigl(1-\phi_r(y)-\phi_r(x-y)\bigr) K_1(x-y) K_2(y)\,dy \\
&= \int_{\R^d} \phi_r(y) K_1(x-y) K_2(y)\,dy
+ \int_{\R^d} \phi_r(y) K_1(y) K_2(x-y)\,dy \\
&\quad + \int_{\R^d} \bigl(1-\phi_r(y)-\phi_r(x-y)\bigr) K_1(x-y) K_2(y)\,dy\;,\label{e:termsConv}
\end{equs}
so that, provided that $r \le \|x\|_\s/2$, say, one has the identity
\begin{equs}
D^k\bigl(K_1 * K_2\bigr)(x) &= \int_{\R^d} \phi_r(y) D^kK_1(x-y) K_2(y)\,dy\\
&\qquad + \int_{\R^d} \phi_r(y) K_1(y) D^k K_2(x-y)\,dy\\
&\qquad + \int_{\R^d} \bigl(1-\phi_r(y)-\phi_r(x-y)\bigr) D^k K_1(x-y) K_2(y)\,dy\\
&\qquad - \sum_{\ell < k} {k!\over \ell! (k-\ell)!}\int_{\R^d} D^\ell\phi_r(x-y) D^{k-\ell} K_1(x-y) K_2(y)\,dy\;.
\end{equs}
It remains to bound these terms separately. For the first term, since the integrand
is supported in the set $\{y\,:\, \|y\|_\s \le \|x\|_\s / 2\}$ (thanks to our choice of $r$), 
we can bound $|D^kK_1(x-y)|$ by $C \|x\|_\s^{\zeta_1-|k|_\s}$ and $K_2(y)$ by $\|y\|_\s^{\zeta_2}$.
Since, for $\zeta > -|\s|$, one has the easily verifiable bound
\begin{equ}[e:boundHomogen]
\int_{\|y\|_\s \le r} \|y\|_\s^{\zeta} \,dy \lesssim r^{|\s|+\zeta}\;,
\end{equ}
it follows that the first term in  \eref{e:termsConv} is bounded by a  multiple
of $\|x\|_\s^{\bar\zeta - |k|_\s}$, as required. The same bound holds for the
second term by symmetry.

For the third term, we use the fact that its integrand is supported in 
the set of points $y$ such that one has both 
$\|y\|_\s \ge \|x\|_\s/4$ and $\|x-y\|_\s \ge \|x\|_\s/4$.
Since $\|x-y\|_\s \ge \|y\|_\s - \|x\|_\s$ by the triangle inequality,
one has
\begin{equ}
\|x-y\|_\s \ge {\eps \|y\|_\s} + \Bigl({{1-\eps \over 4} - \eps}\Bigr) \|x\|_\s
\end{equ}
for every $\eps \in [0,1]$ so that, by choosing $\eps$ small enough,
one has $\|x-y\|_\s \ge C \|y\|_\s$ for some constant $C$.
We can therefore bound the third term by a multiple of
\begin{equ}[e:boundIntLarge]
\int_{C\ge\|y\|_\s \ge \|x\|_\s/4} \|y\|_\s^{\zeta_1+\zeta_2 - |k|_\s}\,dy \sim \|x\|_\s^{\bar\zeta-|k|_\s}\;,
\end{equ}
from which the requested bound follows again at once. (Here, the upper bound on
the domain of integration comes from the assumption that the $K_i$ are compactly supported.)

The last term is bounded in a similar way by using the scaling properties of $\phi_r$ and the fact
that we have chosen $r = \|x\|_\s/2$.
\end{proof}

In what follows, we will also encounter distributions that behave just as if they
were functions of order $\zeta$, but with $\zeta < -|\s|$.
We have the following definition:

\begin{definition}
Let $-|\s|-1 < \zeta \le -|\s|$ and let $K \colon \R^d \setminus \{0\} \to \R$
be a smooth function of order $\zeta$, which is supported in a bounded set. 
We then define the renormalised distribution 
$\Ren K$ corresponding to $K$ by
\begin{equ}
\bigl(\Ren K\bigr)(\psi) = \int_{\R^d} K(x) \bigl(\psi(x) - \psi(0)\bigr)\,dx\;,
\end{equ}
for every smooth compactly supported test function $\psi$.
\end{definition}

The following result shows that these distributions behave under convolution
in pretty much the same way as their unrenormalised counterparts with $\zeta > -|\s|$.

\begin{lemma}\label{lem:convolvRenorm}
Let $K_1$ and $K_2$ be two compactly supported functions of respective
 orders $\zeta_1$ and $\zeta_2$ with $-|\s|-1 < \zeta_1 \le -|\s|$ and 
 $-2|\s|- \zeta_1 < \zeta_2 \le 0$.
Then, the function $(\Ren K_1) * K_2$ is of order $\bar \zeta = 0 \wedge (\zeta_1 + \zeta_2 + |\s|)$
and the bound
\begin{equ}
\$(\Ren K_1) * K_2\$_{\bar \zeta;m} \le C \$K_1\$_{\zeta_1;m}\$K_2\$_{\zeta_2;\bar m}\;,
\end{equ}
holds for every $m \ge 0$, where we have set $\bar m = m+ \max\{\s_i\}$.
\end{lemma}

\begin{proof}
Similarly to before, we can write
\begin{equs}
D^k\bigl((\Ren K_1) * K_2\bigr)(x) &= \int_{\R^d} \phi_r(y) D^kK_1(x-y) K_2(y)\,dy
+ (\Ren K_1) \bigl(\phi_r D^k K_2(x-\cdot)\bigr)\\
&\quad + \int_{\R^d} \bigl(1-\phi_r(y)-\phi_r(x-y)\bigr) D^k K_1(x-y) K_2(y)\,dy\\
&\quad - \sum_{\ell < k} {k!\over \ell! (k-\ell)!} \int_{\R^d} D^\ell\phi_r(x-y) D^{k-\ell} K_1(x-y) K_2(y)\,dy\;.
\end{equs}
Here, we used the fact that, when tested against test functions that vanish at the
origin, $\Ren K_1$ is again nothing but integration against $K_1$.
All these terms are bounded exactly as before, thus yielding the desired bounds,
except for the second term. For this term, we have the identity
\begin{equs}
(\Ren K_1) \bigl(\phi_r D^k K_2(x-\cdot)\bigr)
&= \int_{\R^d} K_1(y) \bigl(\phi_r(y)  D^k K_2(x-y) - D^k K_2(x)\bigr) \,dy\\
&= \int_{\R^d} K_1(y) \phi_r(y) \bigl( D^k K_2(x-y) - D^k K_2(x)\bigr) \,dy\\
&\quad+ D^k K_2(x) \int_{\R^d} K_1(y) \bigl(1-\phi_r(y)\bigr) \,dy\;.\label{e:termRenorm}
\end{equs}
For the first term, we use the fact that the integrand is supported in the region
$\{y\,:\, \|y\|_\s \le \|x\|_\s / 2\}$ (this is the case again by making the choice $r = \|x\|_\s / 2$
as in the proof of Lemma~\ref{lem:convSing}). As a consequence of the gradient theorem,
we then obtain the bound
\begin{equ}
|D^k K_2(x-y) - D^k K_2(x)| \lesssim \sum_{i=1}^d |y_i|\, \|x\|_\s^{\zeta_2 - |k|_\s - \s_i}\$K_2\$_{\zeta_2;\bar k}\;,
\end{equ}
where we have set $\bar k = |k|_\s + \max\{\s_i\}$. Observing that $|y_i| \lesssim \|y\|_\s^{\s_i}$, the required bound then follows
from \eref{e:boundHomogen}. The second term in \eref{e:termRenorm} 
can be bounded similarly as in \eref{e:boundIntLarge}
by making use of the bounds on $K_1$ and $K_2$.
\end{proof}

To conclude this section, we give another two useful results regarding the behaviour
of such kernels. First, we show how a class of natural regularisations of
a kernel of order $\zeta$ converges to it. 
We fix a function $\rho \colon \R^d \to \R$ which is smooth, compactly supported,
and integrates to $1$,
and we write as usual $\rho_\eps(y) = \eps^{-|\s|}\rho(\CS_\s^\eps y)$.
Given a function $K$ on $\R^d$, we then set
\begin{equ}
K_\eps \eqdef K * \rho_\eps\;.
\end{equ}
We then have the following result:

\begin{lemma}\label{lem:mollify}
In the above setting, if $K$ is of order $\zeta \in (-|\s|,0)$, then $K_\eps$ 
has bounded derivatives of all orders. Furthermore, one has the bound
\begin{equ}[e:boundMollif]
\bigl|D^k K_\eps(x)\bigr| \le C \bigl(\|x\|_\s + \eps\bigr)^{\zeta-|k|_\s} \$K\$_{\zeta;|k|_\s}\;.
\end{equ}
Finally, for all $\bar \zeta \in [\zeta-1,\zeta)$ and $m \ge 0$, one has the bound
\begin{equ}[e:boundDiffK]
\$K-K_\eps\$_{\bar \zeta;m} \lesssim \eps^{\zeta - \bar \zeta} \$K\$_{\zeta; \bar m}\;,
\end{equ}
where $\bar m = m + \max \{\s_i\}$.
\end{lemma}

\begin{proof}
Without loss of generality, we assume that $\rho$ is supported in the set $\{x\,:\, \|x\|_\s \le 1\}$.
We first obtain the bounds on $K_\eps$ itself. For $\|x\|_\s \ge 2\eps$, we can write
\begin{equ}
D^k K_\eps(x) = \int_{\R^d} D^k K(x-y) \rho_\eps(y)\,dy\;.
\end{equ}
Since $\rho_\eps$ is supported in a ball of radius $\eps$, it follows from the bound $\|x\|_\s \ge 2\eps$
that whenever the integrand is non-zero, one has $\|x-y\|_\s \ge \|x\|_\s/2$.
We can therefore bound $D^k K(x-y)$ by $\|x\|_\s^{\zeta -|k|_\s}\$K\$_{\zeta;|k|_\s}$,
and the requested bound follows from the fact that $\rho_\eps$ integrates to $1$.

For $\|x\|_\s \le 2\eps$ on the other hand, we use the fact that
\begin{equ}
D^k K_\eps(x) = \int_{\R^d}  K(y) D^k\rho_\eps(x-y)\,dy\;.
\end{equ}
Since $\|x\|_\s \le 2\eps$, the integrand is supported in a ball of radius $3\eps$.
Furthermore, $|D^k\rho_\eps|$ is bounded by a constant multiple of $\eps^{-|\s| - |k|_\s}$
there, so that we have the bound
\begin{equ}
|D^k K_\eps(x)| \lesssim \eps^{-|\s| - |k|_\s} \$K\$_{\zeta;0}\int_{\|y\|_\s \le 3\eps}  \|y\|_\s^\zeta \,dy\;,
\end{equ}
so that \eref{e:boundMollif} follows.

Regarding the bound on $K-K_\eps$, we write
\begin{equ}
D^k K_\eps(x) - D^k K(x) = \int_{\R^d} \bigl(D^kK(x-y) - D^k K(x)\bigr)\,\rho_\eps(y)\,dy\;.
\end{equ}
For $\|x\|_\s \ge 2\eps$, we obtain as previously the bound
\begin{equ}
\bigl|D^kK(x-y) - D^k K(x)\bigr| \lesssim \$K\$_{\zeta;\bar k} \sum_{i=1}^d |y_i|\, \|x\|_\s^{\zeta-|k|_\s-\s_i}\;,
\end{equ}
where we set $\bar k = |k|_\s + \max\{\s_i\}$.
Integrating this bound against $\rho_\eps$, we thus obtain
\begin{equ}
\bigl|D^k K(x) - D^k K_\eps(x)\bigr| \lesssim \$K\$_{\zeta;\bar k} \sum_{i=1}^d \eps^{\s_i}\, \|x\|_\s^{\zeta-|k|_\s-\s_i} \lesssim \eps^{\zeta - \bar \zeta} \$K\$_{\zeta;\bar k} \|x\|_\s^{\bar \zeta-|k|_\s}\;,
\end{equ}
where we used the fact that $\s_i \ge 1$ for every $i$.
For $\|x\|_\s \le 2\eps$ on the other hand, we make use of the bound obtained in the first part,
which implies in particular that 
\begin{equ}
\bigl|D^k K(x) - D^k K_\eps(x)\bigr| \lesssim \$K\$_{\zeta;|k|_\s} \|x\|_\s^{\zeta - |k|_\s}
\lesssim \eps^{\zeta - \bar \zeta}\$K\$_{\zeta;|k|_\s} \|x\|_\s^{\bar \zeta - |k|_\s}\;,
\end{equ}
which is precisely the requested bound.
\end{proof}

Finally, it will be useful to have a bound on the difference between the values of a singular kernel,
evaluated at two different locations. The relevant bound takes the following form:

\begin{lemma}\label{lem:boundDiffKernel}
Let $K$ be of order $\zeta \le 0$. Then, for every $\alpha \in [0,1]$, one has the bound
\begin{equ}
\bigl|K(z) - K(\bar z)\bigr| \lesssim \|z-\bar z\|_\s^\alpha \bigl(\|z\|_\s^{\zeta-\alpha} + \|\bar z\|_\s^{\zeta-\alpha}\bigr) \$K\$_{\zeta;m}\;,
\end{equ}
where $m = \sup_i \s_i$.
\end{lemma}

\begin{proof}
For $\alpha = 0$, the bound is obvious, so we only need to show it for $\alpha = 1$; the other values 
then follow by interpolation.

If $\|z-\bar z\|_\s \ge \|z\|_\s \wedge \|\bar z\|_\s$, we use the ``brutal'' bound
\begin{equs}
\bigl|K(z) - K(\bar z)\bigr| &\le
|K(z)| + |K(\bar z)|
\le \bigl(\|z\|_\s^{\zeta} + \|\bar z\|_\s^{\zeta}\bigr) \$K\$_{\zeta;m}\\
&\le 2\bigl(\|z\|_\s^{\zeta} \wedge \|\bar z\|_\s^{\zeta}\bigr) \$K\$_{\zeta;m}
\le 2\|z-\bar z\|\bigl(\|z\|_\s^{\zeta-1} \wedge \|\bar z\|_\s^{\zeta-1}\bigr) \$K\$_{\zeta;m} \\
&\le 2\|z-\bar z\|\bigl(\|z\|_\s^{\zeta-1} + \|\bar z\|_\s^{\zeta-1}\bigr) \$K\$_{\zeta;m}\;,
\end{equs}
which is precisely what is required.

To treat the case $\|z-\bar z\|_\s \le \|z\|_\s \wedge \|\bar z\|_\s$,
we use the identity
\begin{equ}[e:idenDiffK]
K(z) - K(\bar z) = \int_\gamma \scal{\nabla K(y), dy}\;,
\end{equ}
where $\gamma$ is any path connecting $\bar z$ to $z$.
It is straightforward to verify that it is always possible to find $\gamma$
with the following properties:
\begin{enumerate}
\item The path $\gamma$ is made of finitely many line segments that are parallel to
the canonical basis vectors $\{e_i\}_{i=1}^d$. 
\item There exists $c > 0$ such that one has $\|y\|_\s \ge c (\|z\|_\s \wedge \|\bar z\|_\s)$ for every $y$ on $\gamma$.
\item There exists $C>0$ such that the total (Euclidean) length of the line segments
parallel to $e_i$ is bounded by $C\|z-\bar z\|_\s^{\s_i}$. 
\end{enumerate} 
Here, both constants $c$ and $C$ can be chosen uniform in $z$ and $\bar z$.
It now follows from the definition of $\$K\$_{\zeta;m}$ that one has
\begin{equ}
|\d_i K(y)| \le \$K\$_{\zeta;m} \,\|y\|_\s^{\zeta - \s_i}\;.
\end{equ}
It follows that the total contribution to \eref{e:idenDiffK}
coming from the line segments parallel to $e_i$ is bounded by
a multiple of 
\begin{equ}
\$K\$_{\zeta;m} \|z-\bar z\|_\s^{\s_i} \bigl(\|z\|_\s^{\zeta-\s_i}
+ \|\bar z\|_\s^{\zeta-\s_i}\bigr)\le
\$K\$_{\zeta;m} \|z-\bar z\|_\s \bigl(\|z\|_\s^{\zeta-1}
+ \|\bar z\|_\s^{\zeta-1}\bigr)\;,
\end{equ}
where, in order to obtain the inequality, 
we have used the fact that $\s_i \ge 1$ and that we are
considering the regime $\|z-\bar z\|_\s \le \|z\|_\s \wedge \|\bar z\|_\s$.
\end{proof}

\subsection{Wick renormalisation and the continuous parabolic Anderson model}
\label{sec:PAMGenRen}

There is one situation in which it is possible to show without much effort that
bounds of the type \eref{e:aprioriBound} and \eref{e:convergenceBound} hold, which is
when $\tau = \tau_1 \tau_2$ and one has
identity 
\begin{equ}
\bigl(\hat \Pi^{(\eps)}_z \tau\bigr)(\bar z) \approx \bigl(\hat \Pi^{(\eps)}_z \tau_1\bigr)(\bar z) \diamond
\bigl(\hat \Pi^{(\eps)}_z \tau_2\bigr)(\bar z)\;,
\end{equ}
either as an exact identity or as an approximate identity with a ``lower-order'' error term,
where $\diamond$ denotes the Wick product between elements of some fixed Wiener chaos. 
Recall that if $f \in H^{\otimes k}$ and $g \in H^{\otimes \ell}$, then the Wick
product between the corresponding random variables is \textit{defined} by
\begin{equ}
I_k(f)\diamond I_\ell(g) = I_{k+\ell}(f \otimes g)\;.
\end{equ}
In other words, the Wick product only keeps the ``dominant'' term in the product formula
\eref{e:formProduct} and discards all the other terms. 

We have seen in Section~\ref{sec:PAMGenRigour} how to associate to
\eref{e:PAMGen} a renormalisation group $\RR_0$ and how to interpret the solutions
to the fixed point map associated to a renormalised model. In this section, we
perform the final step, namely we show that if $\xi_\eps$ is a smooth approximation
to our spatial white noise $\xi$ and $Z_\eps$ denotes the corresponding canonical
model, then one can indeed find a sequence of elements $M_\eps \in \RR_0$ such that
one has $M_\eps Z_\eps \to \hat Z$.
Recalling that elements in $\RR_0$ are characterised by a real number $C$ and a $2\times 2$ matrix
$\bar C$, we show furthermore that it is possible to choose the sequence $M_\eps$ in such a way
that the corresponding constant $C$ is given by a logarithmically diverging constant $C_\eps$,
while the corresponding $2\times 2$ matrix $\bar C$ is given by $\bar C_{ij} = -{1\over 2} C_\eps \delta_{ij}$.

We are in the setting of Theorem~\ref{theo:convRenormGauss} 
and Proposition~\ref{prop:convCorFcn} with $H = L^2(\T^2)$, and where 
the action of $\R^3$ onto $H$ is given by translation in the spatial directions. More
precisely, for $z = (t,x) \in \R \times \R^2$ and $\phi \in H$, one has
\begin{equ}
\bigl(S_z \phi\bigr)(y) = \phi(y-x)\;.
\end{equ}
It turns out that in this case, writing as before $z = (t,x)$ and $\bar z = (\bar t, \bar x)$,
the random variables $\bigl(\hat \Pi^{(\eps)}_z \tau\bigr)(\bar z)$ are not only independent of
$t$, but they are also independent of $\bar t$. So we really view our model as a model
on $\R^2$ endowed with the Euclidean scaling, rather than on $\R^3$ endowed with the 
parabolic scaling. The corresponding integral kernel $\bar K$ is obtained from $K$ by simply
integrating out the temporal variable. 

Since the temporal integral of the heat kernel yields the Green's function of the Laplacian,
we can choose $\bar K$ in such a way that
\begin{equ}[e:singK]
\bar K(x) = -{1\over 2\pi} \log \|x\|\;,
\end{equ}
for values of $x$ in some sufficiently small neighbourhood of the origin. 
Outside of that neighbourhood,
we choose $\bar K$ as before in such a way that it is smooth, compactly supported, and such that
$\int_{\R^2} x^k \bar K(x)\,dx = 0$, for every multiindex $k$ with $|k| \le r$ for some fixed and sufficiently
large value of $r$. These properties can always be ensured by a suitable choice for the original space-time
kernel $K$.
In particular, $\bar K$ is of order $\zeta$ for every $\zeta < 0$ in the sense of Definition~\ref{def:orderKernel}.

Recall now that we define $\xi_\eps$ by $\xi_\eps = \rho_\eps * \xi$, where $\rho$ is a smooth compactly 
supported function integrating to $1$ and $\rho_\eps$ denotes the rescaled function as usual. 
From now on, we consider everything in $\T^2$, so that $\rho \colon \R^2 \to \R$. With this definition, we then
have the following result, which is the last missing step for the proof of Theorem~\ref{theo:mainConvPAM}.

\begin{theorem}\label{theo:convModelPAM}
Denote by $\TT$ the regularity structure associated to \eref{e:PAMGen} with $\alpha \in (-{4\over 3}, -1)$ and $\beta = 2$. Let furthermore $M_\eps$ be a sequence of elements in $\RR_0$ and define the renormalised model
$\hat Z_\eps = M_\eps Z_\eps$. Then, there exists a limiting model $\hat Z$ independent of 
the choice of mollifier $\rho$, as well as a choice of $M_\eps \in \RR_0$ 
such that $\hat Z_\eps \to \hat Z$ in probability. 
More precisely, for any $\theta < -1 - \alpha$,
any compact set $\K$, and any $\gamma < r$, one has the bound
\begin{equ}
\E \$M_\eps Z_\eps; \hat Z\$_{\gamma;\K} \lesssim \eps^\theta\;,
\end{equ}
uniformly over $\eps \in (0,1]$.

Furthermore, it is possible to renormalise the model in such a way that the family 
of all solutions to \eref{e:PAMGen} with respect to the model $\hat Z$
formally satisfies the chain rule.
\end{theorem}

\begin{remark}
Note that we do not need to require that the mollifier $\rho$ be symmetric, although
a non-symmetric choice might require a renormalisation sequence $M_\eps$ which does
not satisfy the identity $\bar C_{ij} = -{1\over 2} C \delta_{ij}$.
\end{remark}

\begin{proof}
As already seen in Section~\ref{sec:renPAM}, the only elements in the regularity structure associated
to \eref{e:PAMGen} that have negative homogeneity are 
\begin{equ}
\{\Xi, X_i \Xi, \CI(\Xi)\Xi, \CI_i(\Xi)\CI_j(\Xi)\}\;.
\end{equ}
By Theorem~\ref{theo:convRenormGauss}, we thus only need to identify the random variables
$\bigl(\Pi_x \tau\bigr)(\psi)$ and
to obtain the bounds \eref{e:aprioriBound} 
and \eref{e:convergenceBound}
for elements $\tau$ in the above set. For $\tau = \Xi$,
it follows as in the proof of Proposition~\ref{prop:convLinearPhi4} that
\begin{equ}
\E \bigl|\bigl(\hat \Pi_x^{(\eps)} \Xi\bigr)(\phi_x^\lambda)\bigr|^2 \lesssim \lambda^{-2}\;,\qquad
\E \bigl|\bigl(\hat \Pi_x^{(\eps)} \Xi - \hat \Pi_x \Xi\bigr)(\phi_x^\lambda)\bigr|^2 \lesssim \eps^{2\theta}\lambda^{-2-2\theta}\;,
\end{equ}
provided that $\theta < {1\over 2}$, which is precisely the required bound. 
For $\tau = X_i\Xi$, the required bound
follows immediately from the corresponding bound for $\tau = \Xi$, so it only remains
to consider $\tau = \CI(\Xi)\Xi$ and $\tau = \CI_i(\Xi)\CI_j(\Xi)$.

We start with $\tau = \CI(\Xi)\Xi$, in which case we aim to show that
\begin{equ}[e:wantedBound]
\E \bigl|\bigl(\hat \Pi_x^{(\eps)} \tau\bigr)(\phi_x^\lambda)\bigr|^2 \lesssim \lambda^{-\kappa}\;,\qquad
\E \bigl|\bigl(\hat \Pi_x^{(\eps)} \tau - \hat \Pi_x \tau\bigr)(\phi_x^\lambda)\bigr|^2 \lesssim \eps^{2\theta}\lambda^{-\kappa-2\theta}\;.
\end{equ}
For this value of $\tau$, one has the identity
\begin{equ}
\bigl(\hat \Pi^{(\eps)}_x \tau\bigr)(y) = \xi_\eps(y)\int \bigl(\bar K(y-z) - \bar K(x-z)\bigr) \xi_\eps(z)\,dz - C^{(\eps)}\;,
\end{equ}
where $C^{(\eps)}$ is the constant appearing in the characterisation of $M_\eps \in \RR_0$. 
Note now that
\begin{equ}
\E \xi_\eps(y)\xi_\eps(z) = \int_{\R^2} \rho_\eps(y-x)\rho_\eps(z-x)\,dx \eqdef \rho_\eps^{\star 2}(y-z)\;,
\end{equ}
and define the kernel $\bar K_\eps$ by
\begin{equ}
\bar K_\eps(y) = \int \rho_\eps(y-z)\bar K(z)\,dz\;.
\end{equ}
With this notation, provided that we make the choice $C^{(\eps)} = \scal{\rho_\eps, \bar K_\eps}$,
we have the identity
\begin{equ}
\bigl(\hat \Pi^{(\eps)}_x \tau\bigr)(y) = \int \bigl(\bar K(y-z) - \bar K(x-z)\bigr) \bigl(\xi_\eps(z)\diamond \xi_\eps(y)\bigr)\,dz - \int \rho_\eps^{\star 2}(y-z)\bar K(x-z)\,dz\;.
\end{equ}
In the notation of Proposition~\ref{prop:convCorFcn}, we thus have
\begin{equs}
\bigl(\hat \CW^{(\eps;0)} \tau\bigr)(y) &= \bigl(\rho_\eps * \bar K_\eps\bigr)(y)\;,\\
\bigl(\hat \CW^{(\eps;2)} \tau\bigr)(y;z_1,z_2) &= \rho_\eps(z_2-y)\bigl(\bar K_\eps(y-z_1) - \bar K_\eps(-z_1)\bigr) \;.
\end{equs}
This suggests that one should define the $L^2$-valued distributions
\begin{equs}[e:defWtau1]
\bigl(\hat \CW^{(0)} \tau\bigr)(y) &= \bar K(y)\;,\\
\bigl(\hat \CW^{(2)} \tau\bigr)(y;z_1,z_2) &= \delta(z_2-y) \bigl(\bar K(y-z_1) - \bar K(-z_1)\bigr)\;,
\end{equs}
and use them to define the limiting random variables $\bigl(\hat \Pi_x^{(k)} \tau\bigr)(\psi)$ via
\eref{e:defineLimit}.

A simple calculation then shows that, for any two points $y$ and $\bar y$ in $\R^2$, one has
\begin{equs}
\scal{\bigl(\hat \CW^{(\eps;2)} \tau\bigr)(y)&,\bigl(\hat \CW^{(\eps;2)} \tau\bigr)(\bar y)} \\
&= \rho_\eps^{\star 2}(y-\bar y)\int \bigl(\bar K_\eps(y-z) - \bar K_\eps(-z)\bigr)\bigl(\bar K_\eps(\bar y- z) - \bar K_\eps(- z)\bigr) \,dz\\
&\eqdef \rho_\eps^{\star 2}(y-\bar y) W_\eps(y, \bar y)\;.\label{e:exprScalPAM2}
\end{equs}
Writing $Q_\eps(y) \eqdef \int \bar K(y-z) K_\eps(- z) \,dz$ and using furthermore the
shorthand notation
\begin{equ}[e:defhatQeps]
\hat Q_\eps(y) \eqdef Q_\eps(y) - Q_\eps(0) - \scal{y, \nabla Q_\eps(0)}\;,
\end{equ}
we obtain
\begin{equ}
W_\eps(y, \bar y) = \hat Q_\eps(y-\bar y)- \hat Q_\eps(y) - \hat Q_\eps(-\bar y)\;.
\end{equ}
As a consequence of Lemmas~\ref{lem:convSing} and \ref{lem:mollify}, we obtain 
for any $\delta > 0$ the bound $|\hat Q_\eps(z)| \lesssim \|z\|^{2-\delta}$
uniformly over $\eps \in (0,1]$. This then immediately implies that
\begin{equ}
|W_\eps(y, \bar y)| \lesssim \|y\|^{2-\delta} + \|\bar y\|^{2-\delta}\;,
\end{equ}
uniformly over $\eps \in (0,1]$. 

It follows immediately from these bounds that
\begin{equ}
\Bigl|\int \scal{\bigl(\hat \CW^{(\eps;2)} \tau\bigr)(y),\bigl(\hat \CW^{(\eps;2)} \tau\bigr)(\bar y)} \psi^\lambda(y)\psi^\lambda(\bar y) \,dy\,d\bar y\Bigr| \lesssim \lambda^{-\delta}\;,
\end{equ}
uniformly over $\eps \in (0,1]$.
In the same way, it is straightforward to obtain an analogous bound on
$\hat \CW^{(2)} \tau$, so it remains to find similar bounds on the quantity 
\begin{equ}
\bigl(\delta \hat \Pi^{(\eps;2)}_x \tau\bigr)(\psi^\lambda) \eqdef
\bigl(\hat \Pi^{(\eps;2)}_x \tau\bigr)(\psi^\lambda) - \bigl(\hat \Pi_x^{(2)} \tau\bigr)(\psi^\lambda)\;.
\end{equ}
Writing $\delta \hat\CW^{(\eps;2)}\tau = \hat\CW^{(\eps;2)}\tau - \hat\CW^{(2)}\tau$, 
we can decompose this as
\begin{equs}
\bigl(\delta \hat\CW^{(\eps;2)}\tau\bigr)(y;z_1,z_2) 
&= \bigl(\delta(z_2-y) - \rho_\eps(z_2-y)\bigr) 
\bigl(\bar K(y-z_1) - \bar K(-z_1)\bigr) \\
&\quad + \rho_\eps(z_2-y) \bigl(\delta \bar K_\eps(y-z_1) - \delta\bar K_\eps(-z_1)\bigr) \\
&\eqdef \bigl(\delta \hat\CW^{(\eps;2)}_1\tau\bigr)(y;z_1,z_2) + \bigl(\delta \hat\CW^{(\eps;2)}_2\tau\bigr)(y;z_1,z_2)\;.
\end{equs}
where we have set $\delta \bar K_\eps = \bar K - \bar K_\eps$.
Accordingly, at the level of the corresponding random variables, 
we can write 
\begin{equ}
\delta \hat \Pi^{(\eps;2)}_x \tau = \delta \hat \Pi^{(\eps;2)}_{x;1} \tau + \delta \hat \Pi^{(\eps;2)}_{x;2} \tau\;,
\end{equ}
and it suffices to bound each of these separately. Regarding
$\delta \hat \Pi^{(\eps;2)}_{x;2} \tau$, it is straightforward to bound it
exactly as above, but making use of Lemma~\ref{lem:mollify} 
in order to bound $\delta \bar K_\eps$. The result of this calculation is that
the second bound in \eref{e:wantedBound} does indeed hold for $\delta \hat \Pi^{(\eps;2)}_{x;2}$,
for every $\theta < {1\over 2}$
and $\kappa > 0$, uniformly over $\eps, \lambda\in (0,1]$.

Let us then turn to $\delta \hat \Pi^{(\eps;2)}_{x;1} \tau$.
It follows from the definitions that one has the identity
\begin{equs}
\scal{\bigl(\delta\hat \CW^{(\eps;2)}_{0;1} \tau\bigr)(y),&\bigl(\delta\hat \CW^{(\eps;2)}_{0;1} \tau\bigr)(\bar y)} \\ &= \bigl(\delta (y-\bar y) - \rho_\eps(\bar y- y) -  \rho_\eps(y-\bar y) + \rho_\eps^{\star 2}(y-\bar y)\bigr) W(y, \bar y)\;.
\end{equs}
At this stage, we note that we can decompose this as a sum of $9$ terms of the form
\begin{equ}[e:oneTerm]
\bigl(\delta (y-\bar y) - \tilde \rho_\eps(y-\bar y)\bigr) \hat Q(x) \;,
\end{equ}
where $\tilde \rho_\eps$ is one of $\rho_\eps^{\star 2}$, $\rho_\eps$, or $\rho_\eps(-\cdot)$, $x$ is one of $y$, $\bar y$ and $y-\bar y$, and $\hat Q$ is defined analogously
to \eref{e:defhatQeps}.
Let us consider the case $x = y$. One then has the identity
\begin{equs}
\int \bigl(\tilde \rho_\eps(y-\bar y)-\delta (y-\bar y)\bigr) &\hat Q(y) \,\psi^\lambda(y)\psi^\lambda(\bar y)\,dy \,d\bar y \label{e:something}\\
&= \int \tilde \rho_\eps(h) \hat Q(y) \,\psi^\lambda(y) \bigl(\psi^\lambda(y-h)- \psi^\lambda(y)\bigr)\,dy \,dh\;.
\end{equs}
Since the integrand vanishes as soon as $\|h\| \gtrsim \eps$, we have the bound
$\bigl|\psi^\lambda(y-h)- \psi^\lambda(y)\bigr| \lesssim \lambda^{-3}\eps$.
Combining this with the bound on $\hat Q$ obtained previously, this immediately
yields for any such term the bound $\eps \lambda^{-1-\delta}$,
provided that $\eps \le \lambda$. However, a bound proportional to $\lambda^{-\delta}$
can be obtained by simply bounding each term in 
\eref{e:something} separately, so that for every $\theta < {1\over 2}$,
one has again a bound of the type $\eps^{2\theta} \lambda^{-2\theta-\kappa}$, uniformly over all
$\eps, \lambda \in (0,1]$.

The case $x = \bar y$ is analogous by symmetry, so it remains to consider the
case $x = y-\bar y$. In this case however, \eref{e:something} reduces to
\begin{equ}
\int \rho_\eps(y-\bar y) \hat Q(y-\bar y) \,\psi^\lambda(y)\psi^\lambda(\bar y)\,dy \,d\bar y \;,
\end{equ}
which is even bounded by $\eps^{2-\delta}\lambda^{-2}$, so that the requested bound follows
again.
This concludes our treatment of the component in the second Wiener chaos for $\tau = \CI(\Xi)\Xi$.

Regarding the term $\hat \CW^{(\eps;0)} \tau$ in the $0$th Wiener chaos, it follows 
immediately from Lemma~\ref{lem:mollify} that, for any $\delta > 0$, 
one has the uniform bound
\begin{equ}
\Bigl|\int \scal{\bigl(\hat \CW^{(\eps;0)} \tau\bigr)(y),\bigl(\hat \CW^{(\eps;0)} \tau\bigr)(\bar y)} \psi^\lambda(y)\psi^\lambda(\bar y) \,dy\,d\bar y\Bigr| \lesssim \lambda^{-\delta}\;,
\end{equ}
as required. For the difference $\delta \hat \CW^{(\eps;0)} \tau$, we obtain
immediately from Lemma~\ref{lem:mollify} that, for any $\kappa < 1$ and $\delta > 0$, one 
has indeed the bound 
\begin{equ}
\Bigl|\int \scal{\bigl(\hat \CW^{(\eps;0)} \tau\bigr)(y),\bigl(\hat \CW^{(\eps;0)} \tau\bigr)(\bar y)} \psi^\lambda(y)\psi^\lambda(\bar y) \,dy\,d\bar y\Bigr| \lesssim \eps^{\kappa-\delta}\lambda^{-\kappa}\;,
\end{equ}
uniformly over $\eps, \lambda \in (0,1]$. This time, the corresponding bound
on the difference between $\hat \CW^{(\eps;0)} \tau$ and $\hat \CW^{(0)} \tau$
is an immediate consequence of Lemma~\ref{lem:mollify}.

We now turn to the case $\tau = \CI_i(\Xi)\CI_j(\Xi)$.
This is actually the easier case, noting that one has the identity
\begin{equ}
\bigl(\hat \Pi^{(\eps)}_x \tau\bigr)(y) = \int \d_i \bar K(y-z) \xi_\eps(z)\,dz \int \d_j \bar K(y-z) \xi_\eps(z)\,dz - \bar C_{ij}^{(\eps)}\;,
\end{equ}
independently of $x$.
If we now choose $\bar C_{ij}^{(\eps)} = \scal{\d_i \bar K_\eps, \d_j \bar K_\eps}$, one has similarly to before
the identity
\begin{equ}
\bigl(\hat \Pi^{(\eps)}_x \tau\bigr)(y) = \int \d_i \bar K(y-z_1)\d_j \bar K(y-z_2) \bigl(\xi_\eps(z_1)\diamond \xi_\eps(z_2)\bigr)\,dz_1\,dz_2\;,
\end{equ}
so that in this case $\bigl(\hat \Pi^{(\eps)}_x \tau\bigr)(y)$ belongs to the homogeneous chaos of order $2$
with
\begin{equ}
\bigl(\hat \CW^{(\eps;2)} \tau\bigr)(y;z_1,z_2) = \d_i \bar K_\eps(y-z_1)\,\d_i \bar K_\eps(y-z_2)\;.
\end{equ}
It then follows at once from Lemma~\ref{lem:mollify} that the required bounds \eref{e:aprioriBound} 
and \eref{e:convergenceBound} do hold in this case as well.

Let us recapitulate what we have shown so far. If we choose the renormalisation map $M_\eps$
associated to $C^{(\eps)} = \scal{\rho_\eps, \bar K_\eps}$ and 
$\bar C_{ij}^{(\eps)} = \scal{\d_i \bar K_\eps,\d_j \bar K_\eps}$, which certainly does depend on the choice
of mollifier $\rho$, then the renormalised model $\hat Z_\eps$ converges in probability to a limiting
model $\hat Z$ that is independent of $\rho$. However, this is not the only possible choice 
for $M_\eps$: we could just as well have added to $C^{(\eps)}$ and $\bar C_{ij}^{(\eps)}$ some
constants independent of $\eps$ and $\rho$ (or converging to such a limit as $\eps \to 0$) and we would
have obtained a different limiting model $\hat Z$, so that we do in principle obtain a $4$-parameter
family of possible limiting models.

We now lift some of this indeterminacy by imposing that the limiting model yields a family of solutions
to \eref{e:PAMGen} which obeys the usual chain rule. As we have seen in \eref{e:PAMRenorm},
this is the case if we obtain $\hat Z$ as a limit of renormalised models where 
$\bar C_{ij} = -{1\over 2} C \delta_{ij}$, thus yielding a one-parameter family of models. 
Since we already know that with the choices mentioned above
the limiting model is independent of $\rho$, it suffices to find \textit{some} $\rho$ such that
the constants $E_{ij}$ defined by
\begin{equ}[e:defEij]
E_{ij} = -\lim_{\eps \to 0} \Bigl(\bar C_{ij}^{(\eps)} + {1\over 2} C^{(\eps)} \delta_{ij}\Bigr)\;,
\end{equ}
are finite. If we then define the model $\tilde Z$ by $\tilde Z = M_E \hat Z$, where $M_E$ 
denotes the action of the element of $\RR_0$ determined by $C = 0$ and $\bar C_{ij} = E_{ij}$,
then the model $\tilde Z$ leads to a solution theory for \eref{e:PAMGen} that does 
obey the chain rule.

It turns out that in order to show that the limits \eref{e:defEij} exist and are finite,
it is convenient to choose a mollifier $\rho$ which has sufficiently many symmetries so that 
\begin{equ}[e:symrho]
\rho(x_1,x_2) = \rho(x_2, x_1) = \rho(x_1,-x_2) = \rho(-x_1,x_2)\;,
\end{equ}
for all $x \in \R^2$. (For example, choosing a $\rho$ that is radially symmetric will do.)
Indeed, by the symmetry of the singularity of $\bar K$ given by
\eref{e:singK}, it follows in this case that 
\begin{equ}
\d_1 \bar K_\eps(x_1, x_2) = -\d_1 \bar K_\eps(-x_1, x_2) = \d_1 \bar K_\eps(x_1, -x_2)\;,
\end{equ}
for $x$ in some sufficiently small neighbourhood of the origin, and similarly for $\d_2 \bar K_\eps$.
As a consequence, the function $\d_1 \bar K_\eps\, \d_2 \bar K_\eps$ integrates to $0$ in any
sufficiently small symmetric neighbourhood of the origin.
It follows at once that in this case, one has 
\begin{equ}[e:convOnDiag]
\lim_{\eps \to 0} C_{12}^{(\eps)} = \int_{\|x\| \ge \delta} \d_1 \bar K(x)\, \d_2 \bar K(x)\,dx\;,
\end{equ}
which is indeed finite (and independent of $\delta>0$, provided that it is sufficiently small) 
since the integrand is a smooth function.

It remains to treat the on-diagonal elements. For this, note that one has
\begin{equ}
\int \bigl(\bigl(\d_1 \bar K_\eps(x)\bigr)^2 + \bigl(\d_2 \bar K_\eps(x)\bigr)^2\bigr)\,dx = 
-\int \bar K_\eps(x)\, \Delta \bar K_\eps(x)\, dx\;.
\end{equ}
It follows from \eref{e:singK} that, as a distribution, one has the identity
$\Delta \bar K = \delta_0 + \bar R$, where $\bar R$ is a smooth function. As a consequence, we obtain
the identity
\begin{equ}
\scal{\d_1 \bar K_\eps, \d_1 \bar K_\eps} + \scal{\d_2 \bar K_\eps, \d_2 \bar K_\eps} = 
-\scal{\bar K_\eps, \rho_\eps} + \int \bar K_\eps(x)\, \bigl(\rho_\eps * \bar R\bigr)(x)\, dx\;,
\end{equ}
so that
\begin{equ}[e:convOffDiag]
\lim_{\eps \to 0} \bigl(\scal{\d_1 \bar K_\eps, \d_1 \bar K_\eps} + \scal{\d_2 \bar K_\eps, \d_2 \bar K_\eps} + \scal{\bar K_\eps, \rho_\eps}\bigr) = \scal{\bar K, \bar R}\;.
\end{equ}
On the other hand, writing $(x_1, x_2)^\perp = (x_2, x_1)$, it follows from
\eref{e:singK} and the symmetries of $\rho$ that $\bar K_\eps(x^\perp) = \bar K_\eps(x)$
for all values of $x$ in a sufficiently small neighbourhood of the origin, so that
$(\d_1 \bar K_\eps)^2 - (\d_2 \bar K_\eps)^2$ integrates to $0$ there.
It follows that 
\begin{equ}
\lim_{\eps \to 0} \bigl(\scal{\d_1 \bar K_\eps, \d_1 \bar K_\eps} - \scal{\d_2 \bar K_\eps, \d_2 \bar K_\eps}\bigr)
= \int_{\|x\| \ge \delta} \bigl(\bigl(\d_1 \bar K(x)\bigr)^2 - \bigl(\d_2 \bar K(x)\bigr)^2\bigr)\,dx\;.
\end{equ}
Combining this with \eref{e:convOffDiag} and \eref{e:convOnDiag}, it immediately follows that the right hand
side of \eref{e:defEij} does indeed converge to a finite limit. Furthermore, since the singularity is 
avoided in all of the above expressions, the convergence rate is of order $\eps$.
\end{proof}

\begin{remark}
The value $C^{(\eps)}$ can be computed very easily. Indeed, for $\eps$ small enough,
one has the identity
\begin{equs}[e:CEpsPAM]
C^{(\eps)} &= \int \rho^{\star 2}_\eps(z) \bar K(z)\,dz = -{1\over 2\pi} \int \rho^{\star 2}_\eps(z) \log \|z\|\,dz\\
&= -{1\over \pi} \log \eps  -{1\over 2\pi} \int \rho^{\star 2}(z) \log \|z\|\,dz\;,
\end{equs}
which shows that only the finite part of $C^{(\eps)}$ actually depends on the choice of $\rho$.
Since this expression does not depend explicitly on $K$ either, 
it also shows that in this case there is a unique canonical choice of renormalised model $\hat Z$.
This is unlike in case of the dynamical $\Phi^4_3$ model where no such canonical choice exists.
\end{remark}

\subsection{The dynamical \texorpdfstring{$\Phi^4_3$}{Phi43} model}
\label{sec:Phi4}

We now finally turn to the analysis of the renormalisation procedure for 
\eref{e:Phi4} in dimension $3$. The setting is very similar to the previous section, 
but this time we work in full space-time, so that the ambient space is $\R^4$, endowed
with the parabolic scaling $\s = (2,1,1,1)$. Our starting point is the canonical model
built from $\xi_\eps = \rho_\eps * \xi$, where $\xi$ denotes space-time white noise
on $\R \times \T^3$
and $\rho_\eps$ is a parabolically rescaled mollifier similarly to before.

We are then again in the setting of Theorem~\ref{theo:convRenormGauss} 
and Proposition~\ref{prop:convCorFcn} but with $H = L^2(\R \times \T^3)$.
This time, the kernel $K$ used for building the canonical model
is obtained by excising the singularity from the heat kernel,
so we can choose it in such a way that
\begin{equ}
K(t,x) =  {\one_{t > 0} \over (4\pi t)^{3\over 2}} \exp\Bigl(-{\|x\|^2 \over 4t}\Bigr)\;,
\end{equ}
for $(t,x)$ sufficiently close to the origin. Again, we extend this to all of $\R^4$ in a way which is
compactly supported and smooth away from the origin, and such that it annihilates all polynomials
up to some degree $r > 2$. The following convergence result is the last missing ingredient for
the proof of Theorem~\ref{thm:Phi4}.

\begin{theorem}\label{theo:convPhi4}
Let $\TT_F$ be the regularity structure associated to the dynamical $\Phi^4_3$ model for $\beta = 2$ and
some $\alpha \in (-{18\over 7}, -{5\over 2})$, let $\xi_\eps$ as above, and let $Z_\eps$ be the 
associated canonical model, where the kernel $K$ is as above. 
Then, there exists a random model $\hat Z$ independent of the
choice of mollifier $\rho$ and elements $M_\eps \in \RR_0$ such that $M_\eps Z_\eps \to \hat Z$
in probability. 

More precisely, for any $\theta < -{5\over 2} - \alpha$,
any compact set $\K$, and any $\gamma < r$, one has the bound
\begin{equ}
\E \$M_\eps Z_\eps; \hat Z\$_{\gamma;\K} \lesssim \eps^\theta\;,
\end{equ}
uniformly over $\eps \in (0,1]$.
\end{theorem}

\begin{proof}
Again, we are in the setting of Theorem~\ref{theo:convRenormGauss}, 
so we only need to show that the suitably 
renormalised model converges for those elements $\tau \in \CF_F$ with non-positive homogeneity.
It can be verified that in the case of the dynamical $\Phi^4_3$ model, these elements are given by
\begin{equ}
\CF_- = \{\Xi, \Psi, \Psi^2, \Psi^3, \Psi^2 X_i, \CI(\Psi^3)\Psi, \CI(\Psi^2)\Psi^2, \CI(\Psi^3)\Psi^2\}\;.
\end{equ}
Regarding $\tau = \Xi$, the claim follows exactly as in the proof of Theorem~\ref{theo:convModelPAM}.
Regarding $\tau = \Psi = \CI(\Xi)$, the relevant bound follows at once from Proposition~\ref{prop:convCorFcn} and Lemma~\ref{lem:mollify},
noting that $\bigl(\hat \Pi_z^{(\eps)} \Psi\bigr)(\bar z) = \bigl(\Pi_z^{(\eps)} \Psi\bigr)(\bar z)$ belongs to the first Wiener chaos with
\begin{equ}
\bigl(\hat \CW^{(\eps;1)} \Psi\bigr)(z, \bar z) = K_\eps(\bar z - z)\;,
\end{equ}
where we have set similarly to before
$K_\eps = \rho_\eps * K$. This is because $|\Psi|_\s < 0$, so that the second term 
appearing in \eref{e:defIa} vanishes in this case. In particular, $\bigl(\hat \Pi_z^{(\eps)} \Psi\bigr)(\bar z)$
is independent of $z$, so we also denote this random variable by $\bigl(\hat \Pi^{(\eps)} \Psi\bigr)(\bar z)$. 
Here, we used the fact that both $K$ and $K_\eps$ are of order $-3$.

The cases $\tau = \Psi^2$ and $\tau = \Psi^3$ 
then follow very easily. Indeed, denote by $C_1^{(\eps)}$ and $C_2^{(\eps)}$
the two constants characterising the element $M_\eps \in \RR_0$ used to renormalise our model.
Then, provided that we make the choice 
\begin{equ}[e:defC1eps]
C_1^{(\eps)} = \int_{\R^4} \bigl(K_\eps(z)\bigr)^2\,dz\;,
\end{equ}
we do have the identities
\begin{equ}
\bigl(\hat \Pi^{(\eps)} \Psi^2\bigr)(z) = \bigl(\hat \Pi^{(\eps)} \Psi\bigr)(z)\diamond \bigl(\hat \Pi^{(\eps)} \Psi\bigr)(z)\;,\qquad 
\bigl(\hat \Pi^{(\eps)} \Psi^3\bigr)(z) = \bigl(\hat \Pi^{(\eps)} \Psi\bigr)(z)^{\diamond 3}\;.
\end{equ}
As a consequence, $\bigl(\hat \Pi^{(\eps)} \Psi^k\bigr)(z)$ belongs to the $k$th homogeneous
Wiener chaos and one has
\begin{equ}[e:Psik]
\bigl(\hat \CW^{(\eps;k)} \Psi^k\bigr)(z; \bar z_1,\ldots,\bar z_k) = K_\eps(\bar z_1 - z)\cdots K_\eps(\bar z_k - z)\;,
\end{equ}
for $k \in \{2,3\}$ so that the relevant bounds follow again from 
Proposition~\ref{prop:convCorFcn} and Lemma~\ref{lem:mollify}.
Regarding $\tau = \Psi^2 X_i$, the corresponding bound follows again at 
once from those for $\tau = \Psi^2$.

In order to treat the remaining terms, it will be convenient to introduce the following
graphical notation, which associates a function to a graph with two types of edges.
The first type of edge, drawn as \mhpastefig{edge1}, represents a factor  
$K$, while the second type of edge, drawn as \mhpastefig{edge2}, represents a factor  $K_\eps$.
Each vertex of the graph denotes a variable in $\R^4$, and the kernel is always evaluated at the 
difference between the variable that the arrow points from and the one that it points to.
For example, $z_1  \mhpastefig{edge1} z_2$ is a shorthand for $K(z_1 - z_2)$. 
Finally, we use the convention that if a vertex is drawn in grey,
then the corresponding variable is integrated out. As an example, the identity \eref{e:Psik}
with $k=3$ and the identity \eref{e:defC1eps} translate into
\begin{equ}[e:defC1graphic]
\bigl(\hat \CW^{(\eps;3)}\Psi^3\bigr)(z) = \mhpastefig{Psi3}\;,\qquad
C_1^{(\eps)} = \mhpastefig{DefC1}\;.
\end{equ}
Here, we made a slight abuse of notation, since the second picture actually defines a function of
one variable, but this function is constant by translation invariance.
With this graphical notation, Lemma~\ref{lem:productChaos} has a very natural graphical 
interpretation as follows. The function $f$ is given by a graph with $\ell$ unlabelled
black vertices and similarly for $g$ with $m$ of them. Then, the contribution of $I_\ell(f) I_m(g)$
in the $(\ell+m-2r)$th Wiener chaos is obtained by summing over all possible ways 
of contracting $r$ vertices of $f$ with $r$ vertices of $g$.

We now treat the case $\tau = \CI(\Psi^3)\Psi$. Combining the comment we just made on the interpretation
of Lemma~\ref{lem:productChaos} with \eref{e:deltaMPsi1I3} 
and the definition \eref{e:defC1graphic} of $C_1^{(\eps)}$, we then have 
\begin{equ}
\bigl(\hat \CW^{(\eps;4)}\tau\bigr)(z) = \MHFigArg{z}{Psi3I1} - \MHFigArg{z}{Psi3I12}\;,
\end{equ}
while the contribution to the second Wiener chaos is given by
\begin{equ}[e:Wtau2]
\bigl(\hat \CW^{(\eps;2)}\tau\bigr)(z) = 3 \biggl(\MHFigArg{z}{Psi3I1contr} - \MHFigArg{z}{Psi3I1contr2}\biggr)
\eqdef 3\,\bigl(\hat \CW^{(\eps;2)}_1\tau - \hat \CW^{(\eps;2)}_2\tau\bigr)\;.
\end{equ}
The reason why no contractions appear between the top vertices is that,
thanks to the definition of $C_1^{(\eps)}$ in \eref{e:defC1graphic}, these have been taken care of 
by our renormalisation procedure. 

We first treat the quantity $\hat \CW^{(\eps;4)}\tau$. The obvious guess is that, in a suitable sense, one has
the convergence $\hat \CW^{(\eps;4)}\tau \to \hat \CW^{(4)}\tau $, where
\begin{equ}
\bigl(\hat \CW^{(4)}\tau\bigr)(z) = \MHFigArg{z}{Psi3I1full} - \MHFigArg{z}{Psi3I12full}\;.
\end{equ}
In order to apply Proposition~\ref{prop:convCorFcn}, 
we first need to obtain uniform bounds on the quantity $\scal[b]{\bigl(\hat \CW^{(\eps;4)}\tau\bigr)(z), \bigl(\hat \CW^{(\eps;4)}\tau\bigr)(\bar z)}$. This can be obtained in a way similar to what we did for bounding $\hat \CW^{(\eps;2)} \CI(\Xi)\Xi$ in
Theorem~\ref{theo:convModelPAM}. Defining kernels $Q_\eps^{(3)}$ and $P_\eps$ by
\begin{equ}
Q_\eps^{(3)}(z-\bar z) = \mhpastefig{Propagate3}\;,\qquad
P_\eps(z-\bar z) = \mhpastefig{Propagate1}\;,
\end{equ}
we have the identity
\begin{equ}
\scal[b]{\bigl(\hat \CW^{(\eps;4)}\tau\bigr)(z), \bigl(\hat \CW^{(\eps;4)}\tau\bigr)(\bar z)}
= P_\eps(z-\bar z) \, \delta^{(2)} Q_{\eps}^{(3)}(z,\bar z)\;,
\end{equ}
where, for any function $Q$ of two variables, we have set
\begin{equ}
\delta^{(2)} Q(z,\bar z) = Q(z,\bar z) - Q(z,0) - Q(0,\bar z) + Q(0,0)\;.
\end{equ}
(Here, we have also identified 
a function of one variable with a function of two variables by 
$Q(z,\bar z) \leftrightarrow Q(z-\bar z)$.)
It follows again from a combination of Lemmas~\ref{lem:convSing} and \ref{lem:mollify} that, for every $\delta > 0$, one has the bounds
\begin{equ}
\bigl|Q_\eps^{(3)}(z) - Q_\eps^{(3)}(0)\bigr| \lesssim \|z\|_\s^{1-\delta}\;,\qquad 
\bigl|P_\eps(z)\bigr| \lesssim \|z\|_\s^{-1}\;.
\end{equ}
Here, in the first term, we used the notation $z = (t,x)$ and we write $\nabla_x$ for 
the spatial gradient. As a consequence, 
we have the desired \textit{a priori} bounds for $\hat \CW^{(\eps;4)}\tau$, namely
\begin{equ}
\bigl|\scal[b]{\bigl(\hat \CW^{(\eps;4)}\tau\bigr)(z),\bigl(\hat \CW^{(\eps;4)}\tau\bigr)(\bar z)}\bigr| \lesssim \|z-\bar z\|_\s^{-1} \bigl(\|z-\bar z\|_\s^{1-\delta}
+ \|z\|_\s^{1-\delta} + \|\bar z\|_\s^{1-\delta}\bigr)\;,
\end{equ}
which is valid for every $\delta > 0$.

To obtain the required bounds on $\delta \hat \CW^{(\eps;4)}\tau$, we proceed in
a similar manner. For completeness, we provide some details for this term. Once
suitable \textit{a priori} bounds are established, all subsequent terms of the 
type $\delta \hat \CW^{(\ldots)}\tau$ can be bounded 
in a similar manner, so we will no longer treat them in detail.
Let us introduce a third kind of arrow, denoted by \mhpastefig{edge3}, which
represents the kernel $K - K_\eps$. With this notation, one has the identity
\begin{equs}
\bigl(\delta\hat \CW^{(\eps;4)}\tau\bigr)(z) &= \biggl(\mhpastefig{Psi3I1diff1} - \mhpastefig{Psi3I12diff1}\biggr) + 
\biggl(\mhpastefig{Psi3I1diff2} - \mhpastefig{Psi3I12diff2}\biggr)\\
&\quad + \biggl(\mhpastefig{Psi3I1diff3} - \mhpastefig{Psi3I12diff3}\biggr)
+ \biggl(\mhpastefig{Psi3I1diff4} - \mhpastefig{Psi3I12diff4}\biggr)
\eqdef \sum_{i=1}^4 \bigl(\delta\hat \CW^{(\eps;4)}_i\tau\bigr)(z)\;.
\end{equs}
It thus remains to show that each of the four terms $\bigl(\delta\hat \CW^{(\eps;4)}_i\tau\bigr)(z)$ satisfies a bound of the type \eref{e:bounddW}. 
Note now that each term is of exactly the same form as
$\bigl(\hat \CW^{(\eps;4)}\tau\bigr)(z)$, except that some of the factors $K_\eps$
are replaced by a factor $K$ and exactly one factor $K_\eps$ is replaced by a factor
$(K-K_\eps)$. Proceeding as above, but making use of the bound \eref{e:boundDiffK}, we
then obtain for each $i$ the bound
\begin{equs}
\bigl|\scal[b]{\bigl(\delta \hat \CW_i^{(\eps;4)}\tau\bigr)(z),&\bigl(\delta\hat \CW_i^{(\eps;4)}\tau\bigr)(\bar z)}\bigr| \\
&\lesssim \eps^{2\theta}\|z-\bar z\|_\s^{-1} \bigl(\|z-\bar z\|_\s^{1-2\theta-\kappa}
+ \|z\|_\s^{1-2\theta-\kappa} + \|\bar z\|_\s^{1-2\theta-\kappa}\bigr)\;,
\end{equs}
which is valid uniformly over $\eps \in (0,1]$, provided that
$\theta < 1$ and that $\kappa>0$. Here, we made use of \eref{e:boundDiffK} and the fact that 
each of these terms always contains exactly two factors $(K-K_\eps)$.

We now turn to $\hat \CW^{(\eps;2)}\tau$, which we decompose according to \eref{e:Wtau2}.
For the first term, it follows from Lemmas~\ref{lem:convSing} and \ref{lem:mollify} 
that we have the bound
\begin{equ}
\bigl|\scal{ \bigl(\hat \CW^{(\eps;2)}_1\tau\bigr)(z),\bigl(\hat \CW^{(\eps;2)}_1\tau\bigr)(\bar z)}\bigr|
= \biggl| \mhpastefig{Propagate2}  \biggr| \lesssim \|z-\bar z\|_\s^{-\delta}\;,
\end{equ}
valid for every $\delta > 0$. (Recall that both $K$ and $K_\eps$ are of order $-3$, with norms uniform in $\eps$.)
In order to bound $\hat \CW^{(\eps;2)}_2\tau$, we introduce the notation $z\mhpastefig{edgealpha}\bar z$ 
as a shorthand for $\|z-\bar z\|_\s^\alpha \one_{\|z-\bar z\|_\s \le C}$ for an unspecified constant $C$.
(Such an expression will always appear as a bound and means that there exists a choice 
of $C$ for which the bound holds true.) 
We will also make use of the inequalities\minilab{e:boundz}
\begin{equs}
\|z\|_\s^{-\alpha} \|\bar z\|_\s^{-\beta} &\lesssim \|z\|_\s^{-\alpha-\beta} + \|\bar z\|_\s^{-\alpha-\beta}\;,\label{e:boundz1}\\
\|z\|_\s^{-\alpha} \|\bar z\|_\s^{-\alpha} &\lesssim \|z-\bar z\|_\s^{-\alpha}\bigl(\|z\|_\s^{-\alpha} + \|\bar z\|_\s^{-\alpha}\bigr)\;,\label{e:boundz2}
\end{equs}
which are valid for every $z$, $\bar z$ in $\R^4$ and any two exponents $\alpha, \beta > 0$. 
The first bound is just a reformulation of Young's inequality. The second bound follows immediately
from the fact that $\|z\|_\s \vee \|\bar z\|_\s \ge {1\over 2} \|z-\bar z\|_\s$.

With these bounds at hand, we obtain for every $\delta \in (0,1)$ the bound
\begin{equs}
\bigl|\scal{ \bigl(\hat \CW^{(\eps;2)}_2\tau\bigr)(z),\bigl(\hat \CW^{(\eps;2)}_2\tau\bigr)(\bar z)}\bigr|
&\lesssim \biggl|\mhpastefig{PropagateBound1}\biggr| \label{e:boundW2}\\
&\lesssim \|z\|_\s^{-\delta} \bigl(G(z) + G(\bar z) + G(z-\bar z) + G(0)\bigr)\;,
\end{equs}
where the function $G$ is given by
\begin{equ}
G(z-\bar z) = \mhpastefig{PropagateBound}\;.
\end{equ}
Here, in order to go from the first to the second line in \eref{e:boundW2}, we used
\eref{e:boundz2} with $\alpha = \delta$, followed by \eref{e:boundz1}.
As a consequence of Lemma~\ref{lem:convSing}, the function $G$ is bounded, so that
the required bound follows from \eref{e:boundW2}. 
Defining as previously $\hat \CW^{(2)}_i\tau$ like $\hat \CW^{(\eps;2)}_i\tau$
but with each instance of $K_\eps$ replaced by $K$, one then also obtains
as before the bound
\begin{equ}
\bigl|\scal{ \bigl(\delta \hat \CW^{(\eps;2)}_i\tau\bigr)(z),\bigl(\delta \hat \CW^{(\eps;2)}_i\tau\bigr)(\bar z)}\bigr|
\lesssim \eps^{2\theta}\bigl(\|z\|_\s^{-2\theta-\kappa} + \|\bar z\|_\s^{-2\theta-\kappa}
+ \|\bar z-z\|_\s^{-2\theta-\kappa}\bigr)\;,
\end{equ}
which is exactly what we require.

We now turn to the case $\tau = \CI(\Psi^2)\Psi^2$. Denoting by $\psi_{\eps}$ the random function
$\psi_{\eps}(z) = (K * \xi_\eps)(z) = (K_\eps * \xi)(z)$, one has the identity
\begin{equ}[e:exprPsi22]
\bigl(\hat \Pi_0\tau\bigr)(z) = \bigl(\bigl(K * \psi_\eps^{\diamond 2}\bigr)(z) - \bigl(K * \psi_\eps^{\diamond 2}\bigr)(0)\bigr)\cdot \bigl(\psi_\eps(z) \diamond \psi_\eps(z)\bigr) - C_2^{(\eps)}\;.
\end{equ}
Regarding $\hat \CW^{(\eps;4)}\tau$, we therefore obtain similarly to before the identity
\begin{equ}
\bigl(\hat \CW^{(\eps;4)}\tau\bigr)(z) = \mhpastefig{Psi2I2} - \mhpastefig{Psi2I22}\;.
\end{equ}
Similarly to above, we then have the identity
\begin{equ}
\scal[b]{\bigl(\hat \CW^{(\eps;4)}\tau\bigr)(z), \bigl(\hat \CW^{(\eps;4)}\tau\bigr)(\bar z)}
= P_\eps^2(z-\bar z)   \,\delta^{(2)} Q_{\eps}^{(2)}(z,\bar z)\;,
\end{equ}
where $P_\eps$ is as above and $Q_\eps^{(2)}$ is defined by  
\begin{equ}
Q_\eps^{(2)}(z-\bar z) = \mhpastefig{Propagate3b}\;.
\end{equ}
This time, it follows from Lemmas~\ref{lem:convSing} and \ref{lem:mollify} that
\begin{equ}
\bigl|Q_\eps^{(2)}(z) - Q_\eps^{(2)}(0) - \scal{x, \nabla_x  Q_\eps^{(2)}(0)}\bigr| \lesssim \|z\|_\s^{2-\delta}\;,\end{equ}
for arbitrarily small $\delta > 0$ and otherwise the same notations as above. Combining this with the bound
already obtained for $P_\eps$ immediately yields the
bound
\begin{equ}
\bigl|\scal[b]{\bigl(\hat \CW^{(\eps;4)}\tau\bigr)(z), \bigl(\hat \CW^{(\eps;4)}\tau\bigr)(\bar z)}\bigr| 
\lesssim \|z-\bar z\|_\s^{-\delta}\;,
\end{equ}
as required. Again, the corresponding bound on $\delta \hat \CW^{(\eps;4)}$ then follows in 
exactly the same fashion as before.

Regarding $\hat \CW^{(\eps;2)}\tau$, it follows from Lemma~\ref{lem:productChaos} 
and \eref{e:exprPsi22} that
one has the identity
\begin{equ}
\bigl(\hat \CW^{(\eps;2)}\tau\bigr)(z) = 4 \biggl(\mhpastefig{Psi2I2contr} - \mhpastefig{Psi2I2contr2}\biggr)\;.
\end{equ}
We then obtain somewhat similarly to above
\begin{equ}
\scal{ \bigl(\hat \CW^{(\eps;2)}\tau\bigr)(z),\bigl(\hat \CW^{(\eps;2)}\tau\bigr)(\bar z)}
= P_\eps(z-\bar z) \, \delta^{(2)} Q_{z,\bar z}^\eps(z,\bar z) \;,
\end{equ}
where we have set
\begin{equ}
Q_{z,\bar z}^\eps(a,b) = \mhpastefig{Propagate5}\;.
\end{equ}
At this stage, we make use of Lemma~\ref{lem:boundDiffKernel}. 
Combining it with Lemma~\ref{lem:convSing}, 
this immediately yields, for any $\alpha \in [0,1]$, the bound
\begin{equ}
\bigl|\delta^{(2)}Q_{z,\bar z}^\eps(z,\bar z)\bigr|
\lesssim \|z\|_\s^{\alpha} \|\bar z\|_\s^{\alpha} 
\bigl(G(z,\bar z) + G(z,0) + G(0,\bar z) + G(0,0)\bigr)\;,
\end{equ}
where this time the function $G$ is given by
\begin{equ}
G(a,b) = \mhpastefig{PropagateBound2}\;.
\end{equ}
As a consequence of Lemma~\ref{lem:convSing}, we see that $G$ is bounded as soon
as $\alpha < {1\over 2}$, which yields the required bound. 
The corresponding bound on $\delta \hat \CW^{(\eps;2)}\tau$ is obtained
as usual.

Still considering $\tau = \CI(\Psi^2)\Psi^2$,
we now turn to $\hat \CW^{(\eps;0)}\tau$, the component in the $0$th Wiener chaos. From 
the expression \eref{e:exprPsi22} and the definition of the Wick product, we deduce that
\begin{equ}[e:Wtau0]
\bigl(\hat \CW^{(\eps;0)}\tau\bigr)(z) = 2\biggl(\mhpastefig{Psi2I2W0bis} - \mhpastefig{Psi2I2W0}\biggr) - C_2^{(\eps)}\;.
\end{equ}
The factor two appearing in this expression arises because there are two equivalent ways of pairing the two ``top'' arrows
with the two ``bottom'' arrows.
At this stage, it becomes clear 
why we need the second renormalisation constant $C_2^{(\eps)}$: the first term
in this expression diverges as $\eps \to 0$ and needs to be cancelled out. (Here, we omitted the label $z$ 
for the first term since it doesn't depend on it by translation invariance.)
This suggests the
choice
\begin{equ}[e:defC2eps]
C_2^{(\eps)} = 2\mhpastefig{Psi2I2W0bis}\;,
\end{equ}
which then reduces \eref{e:Wtau0} to
\begin{equ}[e:defW0]
-{1\over 2} \bigl(\hat \CW^{(\eps;0)}\tau\bigr)(z) =   \mhpastefig{Psi2I2W0}\;.
\end{equ}
This expression is straightforward to deal with, and it
follows immediately from Lemmas~\ref{lem:convSing} 
and \ref{lem:mollify} that we have the bound
$\bigl|\bigl(\hat \CW^{(\eps;0)}\tau\bigr)(z)\bigr|\lesssim \|z\|_\s^{-\delta}$ for every
exponent $\delta > 0$. 

This time, we postulate that $\hat \CW^{(0)}\tau$ is given by \eref{e:defW0}
with every occurrence of $K_\eps$ replaced by $K$.
The corresponding bound on $\delta \hat \CW^{(\eps;0)}\tau$ is then again obtained
as above. 
This concludes our treatment of the term $\tau = \CI(\Psi^2)\Psi^2$.

We now turn to the last element with negative homogeneity, 
which is $\tau = \CI(\Psi^3)\Psi^2$. This is treated in a way which 
is very similar to the previous term; in particular one has an identity similar to
\eref{e:exprPsi22}, but with $\psi_\eps^{\diamond 2}$ replaced by $\psi_\eps^{\diamond 3}$
and $C_2^{(\eps)}$ replaced by $3C_2^{(\eps)} \psi_\eps(z)$. One verifies that one has the identity
\begin{equ}
\scal[b]{\bigl(\hat \CW^{(\eps;5)}\tau\bigr)(z), \bigl(\hat \CW^{(\eps;5)}\tau\bigr)(\bar z)}
= P_\eps^2(z-\bar z)   \,\delta^{(2)} Q_{\eps}^{(3)}(z,\bar z)\;,
\end{equ}
where both $P_\eps$ and $Q_\eps^{(3)}$ were defined earlier. The relevant bounds then follow
at once from the previously obtained bounds.   

The component in the third Wiener chaos is also very similar to what was obtained previously. Indeed,
one has the identity
\begin{equ}
\bigl(\hat \CW^{(\eps;3)}\tau\bigr)(z) = 6 \biggl(\mhpastefig{Psi3I2contr} - \mhpastefig{Psi3I2contr2}\biggr)\;,
\end{equ}
so that
\begin{equ}
\scal{ \bigl(\hat \CW^{(\eps;3)}\tau\bigr)(z),\bigl(\hat \CW^{(\eps;3)}\tau\bigr)(\bar z)}
= P_\eps(z-\bar z) \, \delta^{(2)} \tilde Q_{z,\bar z}^\eps(z,\bar z) \;,
\end{equ}
where we have set
\begin{equ}
\tilde Q_{z,\bar z}^\eps(a,b) = \mhpastefig{PropagateLast}\;.
\end{equ}
This time however, we simply use \eref{e:boundz1} in conjunction with 
Lemmas~\ref{lem:convSing} and \ref{lem:mollify} to
obtain the bound
\begin{equ}
|\tilde Q_{z,\bar z}^\eps(a,b)| \lesssim \|z-\bar z\|_\s^{-\delta} + \|z-b\|_\s^{-\delta} + \|a-\bar z\|_\s^{-\delta} + \|b-a\|_\s^{-\delta}\;.
\end{equ}
The required a priori bound then follows at once, and the corresponding bounds on
$\delta \hat \CW^{(\eps;3)}\tau$ are obtained as usual.

It remains to bound the component in the first Wiener chaos. For this, one verifies the identity
\begin{equs}
\bigl(\hat \CW^{(\eps;1)}\tau\bigr)(z) &= \biggl(6\, \mhpastefig{Psi3I2W1bis} - 3C_2^{(\eps)}\mhpastefig{StochConv}  \biggr)
-\mhpastefig{Psi3I2W1} \\[0.5em]
&\eqdef 6\, \bigl(\hat \CW^{(\eps;1)}_1\tau\bigr)(z) - \bigl(\hat \CW^{(\eps;1)}_2\tau\bigr)(z)\;.
\end{equs}
Recalling that we chose $C_2^{(\eps)}$ as in \eref{e:defC2eps}, we see that
\begin{equ}
\bigl(\hat \CW^{(\eps;1)}_1\tau\bigr)(z; \bar z) = \bigl((\Ren L_\eps)* K_\eps\bigr)(\bar z - z)\;,
\end{equ} 
where the kernel $L_\eps$ is given by $L_\eps(z) = P_\eps^2(z) K(z)$.
It follows from Lemma~\ref{lem:convolvRenorm} that, for every $\delta > 0$, the bound
\begin{equ}
\bigl|\scal{ \bigl(\hat \CW^{(\eps;1)}_1\tau\bigr)(z),\bigl(\hat \CW^{(\eps;1)}_1\tau\bigr)(\bar z)}\bigr|
\lesssim \|z-\bar z\|_\s^{-1-\delta}\;,
\end{equ}
holds uniformly for $\eps \in (0,1]$ as required. Regarding $\hat \CW^{(\eps;1)}_2\tau$, we can again
apply the bounds \eref{e:boundz} to obtain
\begin{equ}
\bigl|\scal{ \bigl(\hat \CW^{(\eps;1)}_2\tau\bigr)(z),\bigl(\hat \CW^{(\eps;1)}_2\tau\bigr)(\bar z)}\bigr|
\lesssim \|z\|_\s^{-{1\over 2}-\delta}\|\bar z\|_\s^{-{1\over 2}-\delta}\;,
\end{equ}
as required. Regarding $\hat \CW^{(1)}\tau$, we define it as
\begin{equ}
\hat \CW^{(1)}\tau = \hat \CW^{(1)}_1\tau + \hat \CW^{(1)}_2\tau\;,
\end{equ}
where $\hat \CW^{(1)}_2\tau$ is defined like $\hat \CW^{(1)}_2\tau$, but with $K_\eps$ replaced by $K$, and where
\begin{equ}
\bigl(\hat \CW^{(1)}_1\tau\bigr)(z;\bar z) = \bigl((\Ren L)* K\bigr)(\bar z - z)\;.
\end{equ}
Again, $\delta \hat \CW^{(1)}\tau$ can be bounded in a manner similar to before,
thus concluding the proof.
\end{proof}

\begin{remark}
It is possible to show that $C_1^{(\eps)} \sim \eps^{-1}$ and
$C_2^{(\eps)} \sim \log \eps$, but the precise values of these constants
do not really matter here. See \cite{MR0384003,MR0416337} for an expression for these constants
in a slightly different context.
\end{remark}

\appendix

\section{A generalised Taylor formula}
\label{app:Taylor}

Classically, Taylor's formula for functions on $\R^d$ is obtained by applying the 
one-dimensional formula to the function obtained by evaluating the original
function on a line connecting the start and endpoints. This however does not yield the
``right'' formula if one is interested in obtaining the correct scaling behaviour when applying
it to functions with inhomogeneous scalings.
In this section, we provide a version of Taylor's formula with a remainder
term having the correct scaling behaviour for any non-trivial scaling $\s$ of $\R^d$.
Although it is hard to believe that this formula isn't known (see \cite{MR2504877} for some
formulae with a very similar flavour)
it seems difficult to find it in the literature in the form stated here. 
Furthermore, it is of course very easy to prove, so we provide a complete proof.

In order to formulate our result, we introduce the following kernels on $\R$:
\begin{equ}
\mu_\ell(x,dy) = \one_{[0,x]}(y){(x-y)^{\ell-1}\over (\ell-1)!}\,dy\;,\quad
\mu_\star(x,dy) = \delta_0(dy)\;.
\end{equ}
For $\ell = 0$, we extend this in a natural way by setting $\mu_0(x,dy) = \delta_x(dy)$.
With these notations at hand, any multiindex $k \in \N^d$ gives rise
to a kernel $\CQ^k$ on $\R^d$ by
\begin{equ}[e:massQk]
\CQ^k(x,dy) = \prod_{i=1}^d \mu^{k}_i(x_i,dy_i)\;, 
\end{equ} 
where we define
\begin{equ}
 \mu^{k}_i(z,\cdot) = 
 \left\{\begin{array}{cl}
 	\mu_{k_i}(z, \cdot) & \text{if $i \le \m(k)$,} \\
 	{z^{k_i} \over k_i !} \mu_{\star}(z, \cdot) & \text{otherwise,}
 \end{array}\right.
\end{equ}
where we defined the quantity
\begin{equ}
\m(k) = \min\{j\,:\, k_j \neq 0\}\;.
\end{equ}
Note that, in any case, one has the identity $\mu_i^k(z, \R) = {z^{k_i}\over k_i !}$, so that
\begin{equ}
\CQ^k(x,\R^d) = {x^k \over k!}\;.
\end{equ}

Recall furthermore that $\N^d$ is endowed with a natural partial order
by saying that $k \le \ell$ if $k_i \le \ell_i$ for every $i \in \{1,\ldots,d\}$. Given $k \in \N^d$,
we use the shorthand $k_< = \{\ell \neq k\,:\, \ell \le k\}$.

\begin{proposition}\label{prop:Taylor}
Let $A \subset \N^d$ be such that $k\in A \Rightarrow k_{<} \subset A$ and define
$\d A = \{k \not \in A\,:\, k - e_{\m(k)} \in A\}$. Then, the identity
\begin{equ}[e:Taylor]
f(x) = \sum_{k \in A} {D^k f(0) \over k!} x^k + \sum_{k \in \d A} \int_{\R^d} D^k f(y)\, \CQ^k(x,dy)\;,
\end{equ}
holds for every smooth function $f$ on $\R^d$.
\end{proposition}

\begin{proof}
The case $A = \{0\}$ is straightforward to verify ``by hand''.
Note then that, for every set $A$ as in the statement, one can
find a sequence $\{A_n\}$ of sets such as in the statement with $A_a = \{0\}$, $A_{|A|} = A$, and 
$A_{n+1} = A_n \cup \{k_n\}$ for some $k_n \in  \d A_n$. It is therefore sufficient to show that 
if \eref{e:Taylor} holds for some set $A$, then it also holds for $\bar A = A \cup \{\ell\}$ for any
$\ell \in \d A$. 

Assume from now on that \eref{e:Taylor} holds for some $A$ and we choose some 
$\ell \in \d A$. Inserting the first-order Taylor expansion (i.e.\ \eref{e:Taylor} with $A = \{0\}$)
into the term involving $D^\ell f$ and using \eref{e:massQk}, we then obtain
the identity
\begin{equs}
f(x) &= \sum_{k \in \bar A} {D^k f(0) \over k!} x^k + \sum_{k \in \d A \setminus\{\ell\}} \int_{\R^d} D^k f(y)\, \CQ^k(x,dy)\\ 
&\qquad + \sum_{i=1}^d \int_{\R^d} D^{\ell + e_i} f(y)\, \bigl(\CQ^{e_i}\star \CQ^\ell\bigr)(x,dy)\;.
\end{equs}
It is straightforward to check that one has the identities
\begin{equ}
\mu_m \star \mu_n = \mu_{m+n}\;,\quad \bigl(\mu_ \star \star \mu_{n}\bigr)(x,\cdot) = {x^n \over n!}\mu_\star\;,\quad
\mu_\star \star \mu_\star = \mu_\star\;,\quad \mu_n \star \mu_\star = 0\;,
\end{equ}
valid for every $m,n \ge 0$.
As a consequence, it follows from the definition of $\CQ^k$ that one has the identity
\begin{equ}
\CQ^{e_i}\star \CQ^\ell = 
\left\{\begin{array}{cl}
	\CQ^{\ell+e_i} & \text{if $i \le \m(\ell)$,} \\
	0 & \text{otherwise.}
\end{array}\right.
\end{equ}
The claim now follows from the fact that, by definition, $\d \bar A$ is precisely given by $(\d A \setminus \{\ell\}) \cup \{\ell + e_i \,:\, i \le \m(\ell)\}$.
\end{proof}

\newpage
\section{Symbolic index}

In this appendix, we collect the most used symbols of the article, together
with their meaning and the page where they were first introduced.

\begin{center}

\renewcommand{\arraystretch}{1.1}
\begin{tabular}{lll}\toprule
Symbol & Meaning & Page\\
\midrule
$\one$ & Unit element in $T$ & \pageref{def:regStruct}\\
$\one^*$ & Projection onto $\one$ & \pageref{e:defAnti}\\
$|\cdot|_\s$ & Scaled degree of a multiindex & \pageref{e:scaledk}\\
$\|\cdot\|_\s$ & Scaled distance on $\R^d$ & \pageref{e:defds} \\
$\$\cdot\$_{\gamma;\K}$ & ``Norm'' of a model / modelled distribution & \pageref{e:normModel}, \pageref{def:Dgamma} \\
$\circ$ & Dual of $\Delta^+$ & \pageref{def:convolution}\\
$\star$ & Generic product on $T$ & \pageref{e:defProd}\\
$*$ & Convolution of distributions on $\R^d$ & \pageref{e:wanted}\\
$A$ & Set of possible homogeneities for $T$ & \pageref{def:regStruct} \\
$\CA$ & Antipode of $\CH_+$ & \pageref{e:defAnti}\\
$\Alg(\CC)$ & Subalgebra of $\CH_+$ determined by $\CC$ & \pageref{lab:AlgC}\\
$\beta$ & Regularity improvement of $K$ & \pageref{def:regK}\\
$\CC^\alpha_\s$ & $\alpha$-H\"older continuous functions with scaling $\s$ & \pageref{def:Calphas}\\
$\DD$ & Abstract gradient & \pageref{def:Di} \\
$\CD^\gamma$ & Modelled distributions of regularity $\gamma$ & \pageref{def:Dgamma} \\
$\CD_P^{\gamma,\eta}$ & Singular modelled distributions of regularity $\gamma$ & \pageref{def:singDist} \\
$\Delta$ & Comodule structure of $\CH$ over $\CH_+$ & \pageref{e:defDelta} \\
$\Delta^+$ & Coproduct in $\CH_+$ & \pageref{e:defDelta} \\
$\DeltaM$ & Action of $M$ on $\Pi$ & \pageref{e:constrModel}\\
$\hDeltaM$ & Action of $M$ on $\Gamma$ & \pageref{e:defDeltaMhat}\\
$F$ & Nonlinearity of the SPDE under consideration & \pageref{e:generalProblem}\\
$F_x$ & Factor in $\Gamma_{xy} = F_x^{-1} F_y$ & \pageref{e:defGammaF}\\
$\CF$ & All formal expressions for the model space & \pageref{e:defCFF} \\
$\CF_+$ & All formal expressions representing Taylor coefficients & \pageref{e:defCFF} \\
$\CF_F$ & Subset of $\CF$ generated by $F$ & \pageref{e:defCFF} \\
$\CF_F^+$ & Subset of $\CF_+$ generated by $F$ & \pageref{e:defCFF} \\
$G$ & Structure group & \pageref{def:regStruct} \\
$\Gamma_g$ & Action of $\CH_+^*$ onto $\CH$ & \pageref{def:convolution} \\
$\CH_+^*$ & Dual of $\CH_+$ & \pageref{def:convolution}\\
$\CH_+$ & Linear span of $\CF_+$ & \pageref{lab:defHH}\\
$\CH_F^+$ & Linear span of $\CF_F^+$ & \pageref{lab:defHH}\\
$\CH$ & Linear span of $\CF$ & \pageref{lab:defHH}\\
$\CI$ & Abstract integration map for $K$ & \pageref{def:abstractI}\\
$\CI_k$ & Abstract integration map for $D^k K$ & \pageref{lab:defIk}\\
$\CJ(x)$ & Taylor expansion of $K * \Pi_x$ & \pageref{e:defJx}\\
$\CJ_k \tau$ & Abstract placeholder for Taylor coefficients of $\Pi_x \tau$ & \pageref{e:defF+}\\
$\K$ & Generic compact set in $\R^d$ & \pageref{e:boundAlpha}\\
$K$ & Truncated Green's function of $\CL$ & \pageref{e:wanted}, \pageref{e:scalingK}\\
$K_n$ & Contribution of $K$ at scale $2^{-n}$ & \pageref{def:regK}\\
\bottomrule
\end{tabular}

\newpage
\begin{tabular}{lll}\toprule
Symbol & Meaning & Page\\
\midrule
$K_{n,xy}^\alpha$ & Remainder of Taylor expansion of $K_n$ & \pageref{e:TaylorKn}\\
$\CK,\CK_\gamma$ & Operator such that $\CR \CK f = K* \CR f$ & \pageref{e:wanted}\\
$\CL$ & Linearisation of the SPDE under consideration & \pageref{e:generalProblem}\\
$\Lambda_n^\s$& Diadic grid at level $n$ for scaling $\s$ & \pageref{lab:Lambdan} \\
$\CM$ & Multiplication operator on $\CH_+$ & \pageref{e:defAnti}\\
$M$ & Renormalisation map & \pageref{e:defPiM} \\
$\hat M$ & Action of $M$ on $f$ & \pageref{e:constrModel} \\
$M_g$ & Action of $\SS$ on $T$ & \pageref{def:symmetric}\\
$\M_F$ & Basic building blocks of $F$ & \pageref{def:basicBlocks}\\
$\MM_\TT$ & All models for $\TT$ & \pageref{def:extension} \\
$\MM_F$ & All admissible models associated to $F$ & \pageref{def:admissible} \\
$\CN_\gamma$ & Operator such that $\CK_\gamma = \CI + \CJ + \CN_\gamma$ & \pageref{e:defI}\\
$(\PPi,f)$ & Alternative representation of an admissible model & \pageref{e:defGammaF} \\
$(\PPi^M,f^M)$ & Renormalised model & \pageref{e:defPiM} \\
$(\Pi,\Gamma)$ & Model for a regularity structure & \pageref{def:model} \\
$\phi_x^n$& Scaling function at level $n$ around $x$ & \pageref{thm:wavelet}\\
$\psi_x^n$& Wavelet at level $n$ around $x$ & \pageref{thm:wavelet}\\
$P$ & Time $0$ hyperplane & \pageref{e:defST} \\
$\CP_F$ & Formal expressions required to represent $\d u$ & \pageref{e:defCUF} \\
$\CQ_\alpha$ & Projection onto $T_\alpha$ & \pageref{e:defTalpha+} \\
$\CQ_\alpha^-$ & Projection onto $T_\alpha^-$ & \pageref{e:defTalpha+} \\
$\PR^+$ & Restriction to positive times & \pageref{e:defST} \\
$R$ & Smooth function such that ``$G = K + R$'' & \pageref{e:propertyGKR} \\
$R_\gamma$ & Convolution by $R$ on $\CD^\gamma$ & \pageref{e:defRgamma} \\
$\CR$ & Reconstruction operator & \pageref{theo:reconstruction} \\
$\RR$ & Renormalisation group & \pageref{def:renormGroup}\\
$\s$ & Scaling of $\R^d$ & \pageref{lab:scaling}\\
$\CS_{\s,x}^\delta$ & Scaling by $\delta$ around $x$ & \pageref{e:scaling}\\
$\CS_P^\delta$ & Scaling by $\delta$ in directions normal to $P$ & \pageref{lab:defSPdelta} \\
$\SS$ & Discrete symmetry group & \pageref{def:symmetric}\\
$T$ & Model space & \pageref{def:regStruct} \\
$T_\alpha$ & Elements of $T$ of homogeneity $\alpha$ & \pageref{def:regStruct} \\
$T_\alpha^+$ & Elements of $T$ of homogeneity $\alpha$ and higher & \pageref{e:defTalpha+} \\
$T_\alpha^-$ & Elements of $T$ of homogeneity strictly less than $\alpha$ & \pageref{e:defTalpha+} \\
$T_g$ & Action of $\SS$ on $\R^d$ & \pageref{def:symmetric}\\
$\bar T$ & Abstract Taylor polynomials in $T$ & \pageref{rem:polynomStructures} \\
$\TT$ & Generic regularity structure $\TT = (A,T,G)$ & \pageref{def:regStruct} \\
$\CT_\beta$ & Classical Taylor expansion of order $\beta$ & \pageref{e:TaylorNormal}\\
$\CU_F$ & Formal expressions required to represent $u$ & \pageref{e:defCUF} \\
$V,W$ & Generic sector of $T$ & \pageref{def:sector}\\
$X^k$ & Abstract symbol representing Taylor monomials & \pageref{sec:canonical}\\
$\Xi$ & Abstract symbol for the noise & \pageref{lab:Xi}\\
\bottomrule
\end{tabular}

\end{center}

\endappendix

\bibliographystyle{./Martin}
\bibliography{./refs}

\end{document}